\documentclass[a4paper,11pt]{article}
\usepackage{pdfsync}
 \usepackage[plainpages=false]{hyperref}
 \usepackage{amsfonts,amsmath,amssymb,amsthm,mathtools}
 \usepackage{latexsym,lscape,rawfonts,mathrsfs,slashed}
\usepackage{enumitem}
\usepackage{scalerel}[2014/03/10]
\usepackage[usestackEOL]{stackengine}
\usepackage{CJKutf8}
\usepackage{fmtcount}
\usepackage{makeidx}
\usepackage{titletoc}
\allowdisplaybreaks[4]
\usepackage[thinc]{esdiff}
 \usepackage[all]{xy}
 \usepackage{graphicx,psfrag}
 \usepackage{pstool}
 \usepackage{tikz-cd}
 \usepackage{upgreek}

 \usepackage{array,tabularx}

 \usepackage{setspace}

 \usepackage[titletoc,title]{appendix}
\usepackage{txfonts}
 \newcommand{\N}{\ensuremath{\mathbb{N}}}
 \newcommand{\Q}{\ensuremath{\mathbb{Q}}}
 \newcommand{\R}{\ensuremath{\mathbb{R}}}
 \newcommand{\Z}{\ensuremath{\mathbb{Z}}}
  \def\III{\mathbb{I}}
 \def\XX{\mathcal{X}}
 \def\RR{\mathcal{R}}
 \def\MM{\mathcal{M}}
 \def\LL{\mathcal{L}}
 
  \def\cc{\mathfrak{c}}

 \def\PP{\mathcal{P}}
 \def\EE{\mathcal{E}}
 
 \def\TT{\mathcal{T}}

 \def\scal{\mathrm{R}}

 \def\loc{\mathrm{loc}}
 \def\rcd{\mathrm{RCD}}

 \def\lip{\mathrm{Lip}}
 
 \def\Id{\mathrm{Id}}
 
 \def\supp{\mathrm{supp}}

 \newcommand{\RP}{\ensuremath{\mathbb{RP}}}

 \newcommand{\ba}{\begin{align*}}
 \newcommand{\ea}{\end{align*}}
 \newcommand{\flow}{\mathcal{X}=\{M^n,(g(t))_{t\in I}\}}
 \newcommand{\na}{\nabla}
\newcommand{\la}{\langle}
\newcommand{\ra}{\rangle}

\newcommand{\lc}{\left(}
\newcommand{\rc}{\right)}
 
\newcommand{\ep}{\epsilon}

\def\Th{\Theta}

\def\lam{\lambda}

\def\HHH{\mathscr{H}}%use \AAA to denote \mathscr{A}
\def\MMM{\mathscr{M}}
\def\CCC{\mathscr{C}}
\def\CC{\mathcal{C}}
\def\N{\mathbb{N}}
\def\NN{\mathcal{N}}
\def\NNN{\mathscr{N}}

\def\WW{\mathcal{W}}

\def\delf{\mathrm{div}_f}

\def\Div{\mathrm{div}}
\def\dd{\mathrm{d}}

\newcommand{\di}{\text{div}}
\newcommand{\Rm}{\ensuremath{\mathrm{Rm}}}
\newcommand{\Ric}{\ensuremath{\mathrm{Ric}}}

\renewcommand{\t}{\mathfrak{t}}

\def\MS{\mathcal{S}}

 \makeatletter
 \def\ExtendSymbol#1#2#3#4#5{\ext@arrow 0099{\arrowfill@#1#2#3}{#4}{#5}}
 
 \makeatother

 \makeatletter
 \def\ExtendSymbol#1#2#3#4#5{\ext@arrow 0099{\arrowfill@#1#2#3}{#4}{#5}}
 
 \makeatother
 
\def\aint{\,\ThisStyle{\ensurestackMath{%
  \stackinset{c}{.2\LMpt}{c}{.5\LMpt}{\SavedStyle-}{\SavedStyle\phantom{\int}}}%
  \setbox0=\hbox{$\SavedStyle\int\,$}\kern-\wd0}\int}

\DeclarePairedDelimiter\abs{\lvert}{\rvert}%
\makeatletter
\let\oldabs\abs
\def\abs{\@ifstar{\oldabs}{\oldabs*}}
\makeatother

\numberwithin{equation}{section}
\newtheorem{thm}{Theorem}[section]
\newtheorem{cor}[thm]{Corollary}
\newtheorem{prop}[thm]{Proposition}
\newtheorem{lem}[thm]{Lemma}
\newtheorem{conj}[thm]{Conjecture}

\newtheorem{rem}[thm]{Remark}

\newtheorem{defn}[thm]{Definition}

\newtheorem{claim}[thm]{Claim}
\newtheorem{notn}[thm]{Notation}

\titlecontents{part}[0pt]  % 作用于part条目
{\addvspace{10pt}\bfseries} % 条目前格式（粗体+垂直间距��
{\makebox[3em][l]{PART-\thecontentslabel}}      % 标签（如“Part I”）
{}                          % 无编号时的占位符
{}                          % 页面数字格式
\makeindex
\title{Singular sets in noncollapsed Ricci flow limit spaces}
\author{Hanbing Fang \quad and \quad Yu Li} 
\date{\today}

\begin{document}
	\begin{CJK}{UTF8}{gbsn}

\maketitle
	\begin{abstract}
In this paper, we study the singular set $\mathcal{S}$ of a noncollapsed Ricci flow limit space, arising as the pointed Gromov--Hausdorff limit of a sequence of closed Ricci flows with uniformly bounded entropy. The singular set $\mathcal{S}$ admits a natural stratification:
\begin{equation*}
	\mathcal S^0 \subset \mathcal S^1 \subset \cdots \subset \mathcal S^{n-2}=\mathcal S,
\end{equation*}
where a point $z \in \mathcal S^k$ if and only if no tangent flow at $z$ is $(k+1)$-symmetric. In general, the Hausdorff dimension of $\mathcal S^k$ with respect to the spacetime distance is at most $k$. We show that the subset $\mathcal{S}^k_{\mathrm{qc}} \subset \mathcal{S}^k$, consisting of points at which some tangent flow is given by a standard cylinder or its quotient, is parabolic $k$-rectifiable. 

In dimension four, we prove the stronger statement that each stratum $\mathcal{S}^k$ is parabolic $k$-rectifiable for $k \in \{0, 1, 2\}$. Furthermore, we establish a sharp uniform $\HHH^2$-volume bound for $\mathcal{S}$ and show that, up to a set of $\HHH^2$-measure zero, the tangent flow at any point in $\mathcal{S}$ is backward unique. In addition, we derive $L^1$-curvature bounds for four-dimensional closed Ricci flows. As an application, we resolve Perelman’s bounded diameter conjecture for three-dimensional closed Ricci flows.
	\end{abstract}
	\tableofcontents
\section{Introduction}

Ricci flow, introduced by Richard Hamilton in the early 1980s, is defined by the evolution equation
\begin{equation*}
	\partial_tg(t)=-2\Ric(g(t)).
\end{equation*}
Since its inception, Ricci flow has proven to be a powerful tool in geometric analysis, enabling deep insights into the topology and geometry of manifolds. 

This paper investigates the pointed Gromov--Hausdorff limit space $(Z, d_Z,  \t)$, obtained from a sequence of closed Ricci flows $\XX^i=\{M_i^n,(g_i(t))_{t  \in \III^{++}}\} \in \MM(n, Y, T)$. As in \cite[Definition 3.1]{fang2025RFlimit}, $\MM(n,Y,T)$ is the entropy-bounded subspace of the moduli space $\MM(n,T)$ of all $n$-dimensional closed Ricci flows defined on $\III^{++}:=[-T,0]$ (see Definition \ref{def:entropybound}). Equipped with the spacetime distance $d^*_i$ from \cite[Definition 3.2]{fang2025RFlimit} (see Definition \ref{defnd*distance}), we have the convergence
	\begin{align} \label{eq:convpgh}
		(M_i \times \III, d^*_i, p_i^*,\t_i) \xrightarrow[i \to \infty]{\quad \mathrm{pGH} \quad} (Z, d_Z, p_{\infty},\t),
	\end{align}
	where $\III:=[-0.98T,0]$ and $p_i^* \in M_i \times \III$. The limit space $(Z,d_Z,\t)$ is a \textbf{Ricci flow limit space} over $\III$ and is \textbf{noncollapsed} because the approximating sequence lies in $\MM(n,Y,T)$. In \cite[Theorem 1.3]{fang2025RFlimit}, weak compactness is proved for the larger moduli space $\MM(n,T)$, while the corresponding structure theory applies in the noncollapsed setting. In particular, $(Z,d_Z,\t)$ is a parabolic space (see \cite[Definition 3.19]{fang2025RFlimit}), meaning that it is a metric space endowed with a time-function $\t:Z \to \III$ satisfying $|\t(x)-\t(y)| \le d_Z^2(x,y)$ for any $x,y \in Z$.
	
The noncollapsed Ricci flow limit space $(Z, d_Z,\t)$, when restricted to the time interval $\III^-:=(-0.98T, 0]$, admits the following regular–singular decomposition:
	\begin{align*}
	Z_{\III^-}=\RR_{\III^-} \sqcup \MS,
\end{align*}
	where $\RR_{\III^-}$ is a dense open subset of $Z_{\III^-}$ that carries the structure of a Ricci flow spacetime $(\RR, \t, \partial_\t, g^Z)$. Furthermore, the convergence in \eqref{eq:convpgh} can be improved to be
	\begin{equation*} 
(M_i \times \III, d^*_i, p_i^*,\t_i) \xrightarrow[i \to \infty]{\quad \hat C^\infty \quad} (Z, d_Z, p_{\infty},\t).
	\end{equation*}
Roughly speaking, $\hat C^\infty$ means that the convergence is smooth on the regular part $\RR$ (see Theorem \ref{thm:intro3} and Notation \ref{not:2} for more details). In addition, it is shown in \cite[Theorem 1.13]{fang2025RFlimit} that the singular set $\MS$ has Minkowski dimension at most $n-2$.
	
Understanding the structure of the singular set $\MS$ is a central problem in the analysis of Ricci flow limit spaces. For any point $z \in \MS$, a useful method for probing the local geometry near $z$ is blow-up analysis. Specifically, one considers the pointed Gromov--Hausdorff limit, known as a \textbf{tangent flow}, of the rescaled spaces $(Z, r_j^{-1} d_Z, z, r_j^{-2}(\t-\t(z)))$ as $r_j \searrow 0$. It was shown in \cite[Section 7]{fang2025RFlimit} that any such tangent flow is a \textbf{Ricci shrinker space} (see Definition \ref{def:rss}). Roughly speaking, a Ricci shrinker space is self-similar for $t \in (-\infty, 0)$ and satisfies the Ricci shrinker equation on its regular part. Moreover, the singular set of each negative time slice in a Ricci shrinker space has Minkowski codimension at least four.

By analyzing the symmetry properties of tangent flows, one obtains a natural stratification of the singular set:
\begin{equation*}
	\mathcal S^0 \subset \mathcal S^1 \subset \cdots \subset \mathcal S^{n-2}=\mathcal S,
\end{equation*}
where a point $z \in \MS^k$ if and only if no tangent flow at $z$ is $(k+1)$-symmetric (see Definition \ref{defnsymmetricsoliton}). In general, the Hausdorff dimension of $\MS^k$ with respect to $d_Z$ satisfies $\dim_{\HHH} \mathcal S^k \le k$; see Corollary \ref{cor:kthcover}. 

A sharp dimension bound, however, does not by itself provide detailed geometric information about $\mathcal S^k$. While blow-up analysis effectively captures the local behavior of singularities, it does not, on its own, describe the global geometry of the singular set $\mathcal S$. The optimal structural picture is that each stratum $\mathcal S^k$ should resemble a $k$-dimensional smooth manifold in an appropriate measure-theoretic or geometric sense. In this paper, we focus on a particularly significant subset $\mathcal S^k_{\mathrm{qc}}\subset \mathcal S^k$, consisting of points for which some tangent flow is a standard cylinder, or a finite quotient thereof. For the precise definition, see Definition \ref{def:cylindrq}.

It was shown in \cite[Theorem 8.11]{FLloja05} that for any point $z \in \mathcal{S}^k_{\mathrm{qc}}$, the tangent flow at $z$ is unique. The importance of this subset is highlighted by the decomposition of the top stratum $\MS^{n-2}\setminus \MS^{n-3}$, which admits the disjoint union:
\begin{equation}\label{intro:decomp1}\index{$\MS^{n-2}_{\mathrm{F}}$}
\MS^{n-2}\setminus \MS^{n-3}=\lc \MS_{\mathrm{qc}}^{n-2}\setminus \MS_{\mathrm{qc}}^{n-3}\rc \sqcup \MS^{n-2}_{\mathrm{F}},
\end{equation}
where the set $\MS_{\mathrm{qc}}^{n-2}\setminus \MS_{\mathrm{qc}}^{n-3}$ consists of points $z$ whose tangent flow, restricted to time-slice $-1$, is isometric to the standard $\R^{n-2} \times S^2$ or $\R^{n-2} \times \RP^2$. In contrast, any point $z \in \MS^{n-2}_{\mathrm{F}}$ admits at least one tangent flow given by a static flat cone of the form $\R^{n-4} \times (\R^4/\Gamma) \times \R$, where $\Gamma \leqslant \mathrm{O}(4)$ acts freely on $S^3$.

Moreover, in the four-dimensional case, the classification of two- and three-dimensional Ricci shrinkers (see \cite{hamilton1993formations, perelman2002entropy, naber2010noncompact, ni2008classification, CCZ08}) implies that
\begin{equation}\label{intro:decomp2}
\MS^1 \setminus \MS^0=\MS_{\mathrm{qc}}^{1}\setminus \MS_{\mathrm{qc}}^{0},
\end{equation}
meaning that the entire first stratum consists of quotient cylindrical-type singularities. These results underscore the central role of $\MS_{\mathrm{qc}}^k$ in the structure theory of the singular set.

Before stating the structure theorem for $\MS_{\mathrm{qc}}^k$, we first introduce the notion of parabolic $k$-rectifiability. Here, $\HHH^k$ denotes the $k$-dimensional Hausdorff measure with respect to $d_Z$.

\begin{defn}[Parabolic $k$-rectifiability]\label{defnrectifiablity}\index{parabolic $k$-rectifiability}
A parabolic space $(Z, d_Z, \t)$ is said to be \textbf{horizontally parabolic $k$-rectifiable} if, for any $\ep>0$, there exists a countable collection of $\HHH^k$-measurable subsets $Z_i \subset Z$ such that 
\begin{equation*}
\HHH^k \lc Z \setminus \bigcup_i Z_i \rc = 0,
\end{equation*}
and for each $i$, there exists a bi-Lipschitz map $\phi_i : Z_i \to \R^k$, where $\R^k$ is equipped with the standard Euclidean distance, and the time-function $\t$ satisfies 
\begin{equation} \label{eq:reif0}
\sqrt{|\t(x)-\t(y)|} \le \ep d_Z(x, y) \quad \text{for all} \quad x, y \in Z_i.
\end{equation}

Similarly, $(Z, d_Z, \t)$ is said to be \textbf{vertically parabolic $k$-rectifiable} if there exists a countable collection of $\HHH^k$-measurable subsets $Z_i \subset Z$ such that 
\begin{equation*}
\HHH^k \lc Z \setminus \bigcup_i Z_i \rc  = 0,
\end{equation*}
and for each $i$, there exists a \textbf{time-preserving}, bi-Lipschitz map $\phi_i : Z_i \to \R^{k-2} \times \R$, where $\R^{k-2} \times \R$ is equipped with the standard parabolic distance and the time-function is the projection onto the last $\R$.

In general, the space $(Z, d_Z, \t)$ is said to be \textbf{parabolic $k$-rectifiable} if it can be written as a countable union of subspaces, each of which is horizontally or vertically parabolic $k$-rectifiable.
\end{defn}

We are now ready to state our second main result, whose proof will be given in Theorems \ref{recticylsing2} and \ref{thmrectquotient}.

\begin{thm}\label{intro:thm1}
Let $(Z, d_Z, \t)$ be a noncollapsed Ricci flow limit space arising as the pointed Gromov--Hausdorff limit of a sequence in $\MM(n, Y, T)$. Then, for any $k\in \{0,1,\ldots, n-2\}$, the set $\MS_{\mathrm{qc}}^k$ is horizontally parabolic $k$-rectifiable with respect to the $d_Z$-distance.
\end{thm}

The notion of rectifiability plays an important role in geometric measure theory, with wide-ranging applications in geometric analysis and metric geometry. In the classical Euclidean setting, a set $E \subset \R^n$ is $k$-rectifiable if it can be covered, up to a set of Hausdorff $k$-measure zero, by countably many images of Lipschitz maps from subsets of $\R^k$ into $\R^n$. Intuitively, this means that $E$ can be approximated almost everywhere by pieces of $k$-dimensional $C^1$ submanifolds. A detailed and systematic exposition of rectifiable sets in the Euclidean setting can be found in \cite[Chapter 3]{Federer69}.

For a general metric space $(X, d)$, the definition is analogous: $X$ is $k$-rectifiable if it can be covered, up to a set of Hausdorff $k$-measure zero, by countably many Lipschitz images of subsets of $\R^k$ into $X$. The concept, formalized by Federer and others in the mid-20th century, serves as a cornerstone for the modern theory of currents, minimal surfaces, and the study of singular sets arising in variational problems.

In the theory of harmonic maps, rectifiability provides a precise description of the structure of singularities. Simon \cite{Simon95} proved that the singular set of an energy-minimizing harmonic map from an $n$-dimensional manifold is $(n-3)$-rectifiable. Lin \cite{Lin99} extended this to stationary harmonic maps, showing that their singular sets are $(n-2)$-rectifiable. These foundational results were later refined by Naber and Valtorta \cite{NV17}, who developed quantitative stratification techniques to show that each singular $k$-stratum is $k$-rectifiable.

In the context of noncollapsed Ricci limit spaces---the Gromov--Hausdorff limits of Riemannian manifolds with uniform lower Ricci curvature bounds and noncollapsing conditions---Cheeger, Jiang, and Naber \cite{Cheeger2018RectifiabilityOS} established that the $k$-th stratum of the singular set, consisting of all points at which no tangent cone splits off an $\R^{k+1}$, is $k$-rectifiable. Notably, in settings with only a lower Ricci curvature bound, rectifiability is essentially the best one can hope for—the singular set may fail to be manifold-like and may even resemble a Cantor set; see \cite{LN20} for such examples.

In the setting of mean curvature flow, which is closely related to the present work, Colding and Minicozzi \cite{colding2016singular} established a foundational result concerning the structure of the singular set. Specifically, for a mean curvature flow $\{M_t\}$ of hypersurfaces in $\R^{n+1}$ with generic singularities—that is, tangent flows at singular points are multiplicity-one shrinking cylinders—they proved that the $k$-th stratum $\MS^k \setminus \MS^{k-1}$ can be written as countably many bi-Lipschitz images of subsets of $\R^k$. 

In fact, the result in \cite{colding2016singular} goes further: each such Lipschitz image is shown to be $2$-H\"older with vanishing constant (see \cite[Equation (4.16)]{colding2016singular}). Consequently, each stratum $\MS^k$ is parabolic $k$-rectifiable in the $\R^{n+1} \times \R$ equipped with the standard parabolic metric, in the sense of \cite[Definition 1.2]{Mattila22}. Building on this, one can further show that $\MS^k\setminus\MS^{k-1}$ is locally contained in a $k$-dimensional $C^1$-submanifold (see also the recent work \cite{sun2025regularity} of Sun--Wang--Xue for the improved regularity). The definition of parabolic rectifiability introduced in Definition \ref{defnrectifiablity} is in part motivated by these results and concepts, adapted to the Ricci flow setting.

Next, we turn to the four-dimensional case, where a more refined structural result for the singular set can be established. Our third main theorem is as follows:

\begin{thm}\label{intro:thm2}
Let $(Z, d_Z, \t)$ be a noncollapsed Ricci flow limit space arising as the pointed Gromov--Hausdorff limit of a sequence in $\MM(4, Y, T)$. Then the following statements hold.
	\begin{enumerate}[label=\textnormal{(\roman{*})}]
		\item For each $k \in \{0, 1,2\}$, the stratum $\MS^k$ is parabolic $k$-rectifiable.
		
		\item Each connected component of $\MS^2_{\mathrm{qc}}$ is contained in a single time slice.
		
		\item For $\HHH^2$-a.e. $x \in \MS$, the tangent flow at $x$ is backward unique.
	\end{enumerate}
\end{thm}

Part (i) of Theorem \ref{intro:thm2} builds on Theorem \ref{intro:thm1}. As observed in the decompositions \eqref{intro:decomp1} and \eqref{intro:decomp2}, it suffices to analyze the rectifiability of $\MS^0$ and $\MS^2_{\mathrm{F}}$. In Proposition \ref{prop:countable}, we show that $\MS^0$ is a countable set, and hence trivially parabolic $0$-rectifiable. In Theorem \ref{thm:rec1}, we prove that $\MS^2_{\mathrm{F}}$ is vertically parabolic $2$-rectifiable. By definition, any point $z \in \MS^2_{\mathrm{F}}$ has one tangent flow given by $\R^4/\Gamma \times \R$, and thus $\MS^2_{\mathrm{F}}$ can be thought of as being contained in a vertical line, justifying the vertical rectifiability.

Part (ii) of the theorem holds in any dimension (see Corollary \ref{conntimesliceaq}). Since $\MS^2_{\mathrm{qc}}$ is horizontally parabolic $2$-rectifiable by Theorem \ref{intro:thm1}, its time projection into $\R$ has zero Lebesgue measure. Therefore, any connected component must be entirely contained in a single time slice. A similar result in the setting of mean curvature flow was established by Colding and Minicozzi in \cite[Theorem 1.2]{colding2016singular}.

Finally, part (iii) follows from the fact that tangent flows at points in $\MS^k_{\mathrm{qc}}$ are unique (see \cite[Theorem 8.11]{FLloja05}), and that, for $\HHH^2$-a.e. $x \in \MS_{\mathrm{F}}^2$, the negative part of the tangent flow at $x$ is unique and given by $\R^4/\Gamma \times \R_-$. This will be proved in Corollary \ref{cor:tangunique}.

Motivated by the results in the four-dimensional case, we propose the following conjecture in general dimension. If true, it would imply—when combined with Theorem \ref{intro:thm1}—that the entire singular set $\MS$ is parabolic $(n-2)$-rectifiable.

\begin{conj}
Let $(Z, d_Z, \t)$ be a noncollapsed Ricci flow limit space arising as the pointed Gromov--Hausdorff limit of a sequence in $\MM(n, Y, T)$. Then the set $\MS_{\mathrm{F}}^{n-2}$ is vertically parabolic $(n-2)$-rectifiable with respect to the $d_Z$-distance. Moreover, for $\HHH^{n-2}$-a.e. $x \in \MS_{\mathrm{F}}^{n-2}$, the tangent flow at $x$ is backward unique, and its negative part is given by $\R^{n-4} \times (\R^4/\Gamma) \times \R_-$.
\end{conj}

Next, in the four-dimensional setting, we provide a sharp volume estimate for the singular set.

\begin{thm}\label{intro:thm3}
Let $(Z, d_Z, \t)$ be a noncollapsed Ricci flow limit space arising as the pointed Gromov--Hausdorff limit of a sequence in $\MM(4, Y, T)$. Then for any $z_0 \in Z$ with $\t(z_0)-2 r_0^2 \in \III^-$, we have
	\begin{align*}
\HHH^2 \lc \MS \bigcap B^*(z_0, r_0) \rc \leq C(Y) r_0^{6}.
	\end{align*}
\end{thm}

Here, $B^*$ denotes a metric ball in $Z$. In fact, we prove a more refined estimate than that stated in Theorem \ref{intro:thm3}. Recall that in \cite{fang2025RFlimit}, we introduced the quantitative singular strata $\MS^{\ep,k}_{r_1,r_2}$, inspired by the framework in \cite{cheeger2013lower}. More precisely, a point $z \in  \MS^{\ep,k}_{r_1,r_2}$ if and only if $\t(z)-\ep^{-1} r_2^2 \in  \III^-$ and for all $r \in [r_1, r_2]$, $z$ is not $(k+1,\ep,r)$-symmetric (see Definition \ref{defnstratification}). A basic identity follows from the definition: for any $L \ge 1$,
\begin{align*}
\MS^{k}=\bigcup_{\ep \in (0, L^{-1})} \bigcap_{0<r<\ep L} \MS^{\ep,k}_{r, \ep L}.
\end{align*}
It was shown in \cite[Theorem 1.12]{fang2025RFlimit} that for any $r \in (0, \ep)$,
	\begin{align*}
\abs{B^*_{rr_0} \lc \MS^{\ep, n-2}_{rr_0, \ep r_0}\rc \bigcap B^*(z_0, r_0) } \le C(n, Y, \ep) r^{4-\ep} r_0^{n+2}
	\end{align*}
provided that $\t(z_0)-2 r_0^2 \in \III^-$. Moreover, for any $t \in \R$,
	\begin{align*}
\abs{B^*_{rr_0} \lc \MS^{\ep, n-2}_{r r_0, \ep r_0}\rc \bigcap B^*(z_0, r_0) \bigcap Z_t }_t \le C(n, Y, \ep) r^{2-\ep} r_0^{n}.
	\end{align*}
In this paper, we sharpen these estimates in the four-dimensional case.

\begin{thm}\label{intro:thm4}
Under the same assumptions of Theorem \ref{intro:thm3}, for $k\in\{1,2\}$, we have
	\begin{align*}
\abs{B^*_{rr_0}\lc\MS^{\ep,k}_{rr_0, \ep r_0}\rc\bigcap B^*(z_0, r_0) } \leq C(Y,\ep)r^{6-k} r_0^{6}
	\end{align*}
for any $r \in (0, \ep)$. In particular, we have
	\begin{align*}
		\HHH^k\left(\bigcap_{0<r<\ep} \MS^{\ep,k}_{r r_0, \ep r_0}\bigcap B^*(z_0, r_0)\right)\leq C(Y,\ep) r_0^{6}.
	\end{align*}
	Moreover, for any $t \in \R$ and $r \in (0, \ep)$, 
		\begin{align*}
\abs{B^*_{rr_0}\lc\MS^{\ep,k}_{rr_0,\ep r_0}\rc\bigcap B^*(z_0, r_0) \bigcap Z_t}_t \leq C(Y,\ep) r^{4-k} r_0^{4}.
	\end{align*} 
\end{thm}

The proof of Theorem \ref{intro:thm4} will be provided in Theorems \ref{thmrectifiabledim4}, \ref{thmrectifiabledim4x} and \ref{thmvolumeslice}. The volume estimate in Theorem \ref{intro:thm3} follows immediately from Theorem \ref{intro:thm4}, since
	\begin{align*}
\MS \bigcap B^*(z_0, r_0) = \bigcap_{0<r<\ep} \MS^{\ep,2}_{r r_0,\ep r_0} \bigcap B^*(z_0, r_0)
	\end{align*}
for a sufficiently small $\ep=\ep(Y)$. 

Note that similar sharp and uniform volume estimates for the singular set have been obtained in other geometric flow settings. In the context of mean curvature flow starting from a mean-convex closed hypersurface, such an estimate was proved in \cite[Theorem 1.2]{fang2025volume}. In the setting of noncollapsed Ricci limit spaces, a corresponding result was established in \cite[Theorem 1.7]{Cheeger2018RectifiabilityOS}.

As an application of the methods developed for singularity analysis, we also obtain the following $L^1$-curvature bounds for four-dimensional closed Ricci flows. Here, $r_{\Rm}$ denotes the curvature radius; see \cite[Definition 2.10]{fang2025RFlimit}.

\begin{thm}\label{intro:thm5}
Let $\XX=\{M^4, (g(t))_{t\in [-T, 0)}\}$ be a four-dimensional closed Ricci flow with $T<\infty$, where $t=0$ is the first singular time. Then there exists a constant $C$, depending only on the flow, such that for any $t\in [-T,0)$,
	\begin{align*}
		\int_{M} |\Rm|\,\mathrm{d}V_{g(t)} \le \int_{M} r^{-2}_{\Rm}\,\mathrm{d}V_{g(t)} \leq C.
	\end{align*}
\end{thm}

The proof of Theorem \ref{intro:thm5} will be given in Theorem \ref{thm:RmL1y}. Moreover, we obtain integral estimates for $\na^k \Rm$ for all $k \ge 1$; see Theorem \ref{thm:RmLky}.

In \cite[Theorem 9.2]{fang2025RFlimit}, it is shown that for any $n$-dimensional closed Ricci flow $\{M^n, (g(t))_{t\in [-T, 0)}\}$, one has the uniform bound:
	\begin{align*}
		\int_{M} |\Rm|^{1-\ep}\,\mathrm{d}V_{g(t)}\leq C_{\ep}
	\end{align*}
for any $t \in [-T, 0)$ and every $\ep>0$, where $C_{\ep}$ depends only on the flow and $\ep$. Theorem \ref{intro:thm5} refines this in dimension four and may be viewed as a parabolic analogue of the $L^2$-curvature bounds established in \cite{jiang20212}. Related $L^2$-curvature estimates for four-dimensional closed Ricci flows under a bounded scalar curvature assumption appear in \cite{bamler2017heat,simon20}. See also the recent work \cite{GiLe25} for $L^1$-curvature bounds for closed Ricci flow under the Type I assumption.

By taking the product of a three-dimensional closed Ricci flow with a flat circle $S^1$ and applying Theorem \ref{intro:thm5}, we obtain the corresponding $L^1$-curvature bounds in dimension three. These bounds, together with \cite[Theorem~2.4]{Peter05}, yield the following result---often referred to as \textbf{Perelman’s bounded diameter conjecture} (see \cite[Section 13.2]{perelman2002entropy}):
\begin{thm}\label{intro:thm6}
Let $\XX=\{M^3, (g(t))_{t\in [-T, 0)}\}$ be a three-dimensional closed Ricci flow. Then there exists a constant $C$ depending on the Ricci flow such that 
	\begin{align*}
\sup_{t \in [-T, 0)} \mathrm{diam}_{g(t)}(M) \le C.
	\end{align*}
\end{thm}

For comparison, Gianniotis \cite{Gi25} established this conclusion under the additional Type I curvature assumption. For related results in mean curvature flow, see \cite{GiHa20, HuJi25}.

\subsection*{Some Ingredients in the Proofs of the Main Results}

In the following, we outline several key ideas and techniques that play a central role in the proofs of our main results.

\subsubsection*{The Reifenberg Method}

Originally developed in the 1960s by Edward Reifenberg to address the Plateau problem for minimal surfaces (see \cite{Rei60}), the Reifenberg method has since become a useful tool in geometric measure theory, with far-reaching applications in rectifiability, regularity theory, and the analysis of metric spaces.

In \cite{Rei60}, Reifenberg considered closed subsets of $\R^n$ that are well approximated by affine $k$-planes at every scale and location. For a set $S \subset \R^n$ containing the origin, the flatness of $S$ near a point $x$ at scale $r>0$ is quantified by
	\begin{align} \label{eq:reif1}
\beta(x, r):=r^{-1} \inf_{\Gamma} d_{\mathrm{H}} \lc S \cap B_r(x), \Gamma \cap B_r(x) \rc
	\end{align}
where $d_{\mathrm{H}}$ denotes the Hausdorff distance, and the infimum is taken over all $k$-dimensional affine subspaces $\Gamma$ in $\R^n$. 

The classical Reifenberg's theorem asserts that if $\beta(x, r)<\delta$ for every $x \in S \cap B_1(\vec{0}^n)$ and all $r \in (0, 1)$, where $\delta$ is a sufficiently small constant, then there exists a bi-H\"older homeomorphism $\phi: S \cap B_{1/2}(\vec{0}^n) \to B_{1/2}(\vec{0}^k)\subset \R^k$. In particular, this establishes a form of topological regularity for sets under quantitative flatness assumptions. 

In general, the bi-H\"older homeomorphism $\phi$ in Reifenberg’s theorem cannot be improved to a \textbf{bi-Lipschitz} map. A classical counterexample illustrating this limitation is given by the snowflake construction (see \cite[Example 3.3]{Naber20}). However, if the approximation by affine $k$-planes improves sufficiently fast across scales, then a bi-Lipschitz parametrization becomes possible.

Indeed, Toro \cite{Toro95} proved that if the \textbf{square-summability} condition on the $\beta$-numbers holds:
	\begin{align} \label{eq:reif2}
\sum_{i=1}^\infty \sup_{x \in S\cap B_{1}(\vec{0}^n)} \beta(x, 2^{-i})^2<\delta,
	\end{align}
for a sufficiently small constant $\delta>0$, then there exists a bi-Lipschitz homeomorphism $\phi: S \cap B_{1/2}(\vec{0}^n) \to B_{1/2}(\vec{0}^k)\subset \R^k$.

Reifenberg’s theorem can be extended to complete metric spaces; see \cite[Theorem A.1.2]{cheeger1997structure}. Roughly speaking, for a metric space $(X, d)$, one defines a scale-invariant flatness quantity using the Gromov--Hausdorff distance:
	\begin{align} \label{eq:reif3}
\beta'(x, r):=r^{-1} d_{\mathrm{GH}} \lc B_r(x), B_r(\vec{0}^n) \rc
	\end{align}
for every $x \in X$ and $r>0$. If $\beta'(x, r)$ is uniformly small for all $x$ and $r$ in a neighborhood, then one obtains a local bi-H\"older map $\phi: X \to \R^n$. Using this result, Cheeger and Colding proved that the regular part of a noncollapsed Ricci limit space admits a manifold structure. Furthermore, if a summability condition on $\beta'$ holds, the bi-H\"older map can be upgraded to a bi-Lipschitz map; see \cite[Theorem 2.7]{Gigli25}.

In the setting of a noncollapsed Ricci flow limit space $(Z, d_Z, \t)$, applying the Reifenberg method to the set $\MS^k_{\mathrm{qc}}$ requires constructing a suitable flatness quantity, analogous to \eqref{eq:reif1} or \eqref{eq:reif3}. For illustrative purposes, we restrict attention to the subset $\MS_{\mathrm c}^k\setminus \MS^{k-1}_{\mathrm c} \subset \MS^k_{\mathrm{qc}}$ (see Definition \ref{def:cylindr}), consisting of points at which one (and hence every) tangent flow is isometric to the model space $\bar{\mathcal C}^k$ (see Subsection \ref{cacp}), whose time slice at $-1$ is the standard cylinder $\R^k \times S^{n-k}$. 

Intuitively, for any point $z_0 \in \MS_{\mathrm c}^k\setminus \MS^{k-1}_{\mathrm c}$, the local geometry around $z_0$ nearly splits off an $\R^k$, and the singular set $\MS_{\mathrm c}^k\setminus \MS^{k-1}_{\mathrm c}$ near $z_0$ behaves approximately like a horizontal $k$-plane. To quantify this almost $k$-splitting, we consider a $(k, \ep, r)$-splitting map $\vec u=(u_1, \ldots, u_k): Z_{(\t(z_0)-10r^2,\t(z_0)]}\to \R^k$ at $z_0$ (see Definitions \ref{defnsplittingmap} and \ref{defnsplittingmap1}).

But how does one construct such an approximate splitting map at $z_0$? To this end, we introduce the following notion of almost self-similarity:
\begin{defn}[$(\delta,r)$-selfsimilar]
	A point $z \in Z_{\III^-}$ is called \textbf{$(\delta,r)$-selfsimilar} if $\t(z)-\delta^{-1} r^2 \in \III^-$ and 
	\begin{align*}
	\widetilde 	\WW_{z}(\delta r^2)-\widetilde \WW_{z}(\delta^{-1} r^2) \leq \delta.
	\end{align*}
\end{defn}
Here, we use the \textbf{modified pointed entropy} $\widetilde \WW_{z}(\tau)$, rather than the conventional pointed entropy $\WW_z(\tau)$, for technical reasons. Although the two agree for almost every $\tau$, the modified version is better suited for analysis: it is nonincreasing in $\tau$ and more accurately controls the spacetime integral of the Ricci shrinker operator (see Lemma \ref{lem:monolimit}). In fact, in dimension four they agree identically; see Proposition \ref{prop:agree}.

A point $z$ being $(\delta, r)$-selfsimilar roughly means that, at scale $r$, the geometry around $z$ in $Z$ is close to that of a Ricci shrinker space. As shown in \cite[Appendix D]{fang2025RFlimit}, the set of selfsimilar points in a Ricci shrinker space forms the so-called \textbf{spine}, which characterizes the degree of symmetry of the shrinker. For instance, if two selfsimilar points lie on distinct time slices, then the Ricci shrinker must be a static or quasi-static cone; see \cite[Lemma D.5]{fang2025RFlimit}. In Theorem \ref{staticestimate}, we establish a quantitative version of this result, yielding an integral estimate of $|\Ric|^2$ in spacetime, which is crucial for constructing almost splitting maps.

We now introduce the concept of entropy pinching:

\begin{defn}[Strongly entropy pinching] 
	For $k\in \{1, \ldots ,n\}$, $\alpha \in (0, 1)$, $\delta>0$ and $r>0$, the \textbf{$(k,\alpha,\delta,r)$-entropy pinching at $z_0$} is defined as
	\begin{equation*}
		\mathfrak{S}_r^{k,\alpha,\delta}(z_0):= \inf \sum_{i=0}^k \left(\widetilde \WW_{z_i}(r^2/40)-\widetilde \WW_{z_i}(40r^2)\right)^{\frac{1}{2}},
	\end{equation*}
\end{defn}
\noindent where the infimum is taken for all sets $\{z_i\}_{1 \le i \le k}$ which are strongly $(k,\alpha,\delta,r)$-independent at $z_0$; see Definition \ref{defnindependentpointsRFlimit}. Roughly speaking, a collection $\{z_i\}_{1 \le i \le k}$ is strongly $(k,\alpha,\delta,r)$-independent at $z_0$ if the points are almost selfsimilar and well-distributed around $z_0$ in a way as if they generate a $k$-dimensional space. 

A key feature of the definition of $\mathfrak{S}_r^{k,\alpha,\delta}(z_0)$ is the \textbf{exponent $1/2$}, which is critical for obtaining sharp estimates. With this quantity, we construct a sharp splitting map $\vec u$ at $z_0$, whose spacetime integral of $|\na^2 u_i|^2$ is controlled by $\mathfrak{S}_r^{k,\alpha,\delta}(z_0)$; see Theorem \ref{thmsharpsplittingRFlimit}. Similar sharp splitting maps have been constructed previously: for Ricci limit spaces, see \cite[Theorem 6.1]{Cheeger2018RectifiabilityOS}; and for Type-I Ricci flows, see \cite[Theorem 8.1]{gianniotis2024splitting}, where related entropy pinching quantities are formulated differently.

A fundamental question is whether a $(k, \ep, r)$-splitting map at $z_0$ remains a $(k, \ep, r')$-splitting map at $z_0$ for smaller scales $r'<r$. In general, this does not hold without additional assumptions. However, under an appropriate summability condition on $\mathfrak{S}_r^{k,\alpha,\delta}(z_0)$, the splitting property can be propagated to smaller scales.

In Theorem \ref{nondegenrerationthm}, we prove a general nondegeneration result for almost splitting maps, assuming the following summability condition:
	\begin{equation}  \label{eq:reif4}
		\sum_{\bar r \leq r_j=2^{-j}\leq 1}\mathfrak{S}^{k,\alpha,\ep}_{r r_j}(z_0) \ll 1,
	\end{equation}
where $\bar r$ is a scale such that, for all $s \in [\bar r r,  r]$, the point $z_0$ is $(\delta,s)$-selfsimilar and $(k,\delta,s)$-splitting but not $(k+1,\eta,s)$-splitting, where $\delta\leq\delta (n,Y,\eta , \alpha, \ep)$.

This summability condition can be viewed as a pointwise analog of the Reifenberg summability condition in \eqref{eq:reif2}. The proof of Theorem \ref{nondegenrerationthm} follows a strategy similar to that of \cite[Theorem 8.1]{Cheeger2018RectifiabilityOS}, though the settings and technical execution differ substantially. The key steps involve establishing a Hessian decay estimate for the limiting heat flow (see Theorem \ref{hessiandecayRFlimit}) and then using the sharp splitting map construction to derive the desired nondegeneration result. These results are first proven in the setting of closed Ricci flows, and then extended to the noncollapsed Ricci flow limit space via a limiting argument.

\subsubsection*{Construction of Cylindrical Neck Regions}

To prove the horizontally parabolic $k$-rectifiability of $\MS^k_{\mathrm{c}}$, we aim to decompose $\MS_{\mathrm c}^k\setminus \MS^{k-1}_{\mathrm c}$ into countably many subsets, each of which admits a bi-Lipschitz map into $\R^k$. To achieve this, we introduce the concept of a \textbf{$(k, \delta, \cc, r)$-cylindrical neck region} (see Definition \ref{defiofcylneckregion}), analogous to the definition introduced by the present authors in the mean curvature flow setting \cite[Definition 4.1]{fang2025volume}. Related notions also appear in the study of Ricci limit spaces, such as \cite[Definition 3.1]{jiang20212} and \cite[Definition 2.4]{Cheeger2018RectifiabilityOS}.

In Proposition \ref{cylneckdecomp}, we show that for any point $z_0 \in \MS_{\mathrm c}^k\setminus \MS^{k-1}_{\mathrm c}$, there exists a sufficiently small $r_0>0$ such that 
	\begin{equation}  \label{eq:reif5}
\lc\MS_{\mathrm c}^k\setminus \MS^{k-1}_{\mathrm c} \rc \bigcap B^*(z_0, r_0) \subset \CCC_0,
	\end{equation}
where $\NNN=B^*(z_0,2r_0)\setminus B^*_{r_x}(\CCC)$ is a $(k,\delta, \cc, r_0)$-cylindrical neck region. The key fact used to establish \eqref{eq:reif5} is that any $z_0 \in \MS_{\mathrm c}^k\setminus \MS^{k-1}_{\mathrm c}$ is $(k, \ep, r)$-cylindrical for all sufficiently small $r>0$, as shown in Proposition \ref{prop:uniformc}. This relies crucially on the uniqueness of cylindrical singularities established in \cite{FLloja05}.

Given \eqref{eq:reif5}, it suffices to prove the rectifiability of $\CCC_0$ for any general $(k, \delta, \cc, r)$-cylindrical neck region $\NNN=B^*(z,2r)\setminus B_{r_x}^*(\CCC)$. On this region, we define the \textbf{packing measure} $\mu$ as:
	\begin{equation*}
		\mu:=\sum_{x\in \CCC^+}r_x^{k}\delta_x+\HHH^{k}|_{\CCC_0},
	\end{equation*}
	where $\HHH^{k}$ is the $k$-dimensional Hausdorff measure with respect to $d_Z$. We show in Proposition \ref{choosenondegeneratingpoints} that if $\mu$ satisfies the Ahlfors regularity condition:
	\begin{align}  \label{eq:reif6}
D^{-1}s^k\leq \mu(B^*(x,s))\leq D s^k
	\end{align} 
for any $x\in\CCC$ and $r_x\leq s\leq r-d_Z(x, z)/2$, with some constant $D>1$, then one can construct a map $\vec u: \CCC \to \R^k$ that is bi-Lipschitz away from a set of small measure. A standard covering argument then implies that $\CCC_0$ is horizontally parabolic $k$-rectifiable. Note that the conclusion \eqref{eq:reif0} follows from the rapid clearing-out property (see Lemma \ref{clearout}).

To guarantee the nondegeneracy of the map $\vec u$, we verify the summability condition:
		\begin{equation}\label{eq:reif7}
			\sum_{r_x\leq r_i=2^{-i} \leq 2^{-5} r}\aint_{B^*(x,r_i)}\left|\widetilde \WW_y(r_i^2/40)-\widetilde \WW_y(40 r_i^2)\right|^{1/2}\,\mathrm{d}\mu(y)\leq\epsilon
		\end{equation}
for most $x \in \CCC$, which in turn ensures the summability condition \eqref{eq:reif4} at $x=z_0$. A key estimate used to derive \eqref{eq:reif7} is that for any $y\in\CCC$,
	\begin{align*}
		\sum_{r_y\leq r_i=2^{-i}\leq 2^{-5}r} \abs{\widetilde \WW_y(r_i^2/40)-\widetilde \WW_y(40 r_i^2) }^{1/2}\leq\Psi(\delta),
	\end{align*}
	where $\Psi(\delta) \to 0$ as $\delta \to 0$. This summability inequality is a consequence of a discrete Lojasiewicz inequality, originally established for closed Ricci flows in \cite[Theorem 1.3, Corollary 1.4]{FLloja05}, where it underpinned the strong uniqueness result. Here we generalize this inequality to noncollapsed Ricci flow limit spaces; see Proposition \ref{sumWonRFlimit}.
	
	Finally, in Theorem \ref{ahlforsregucyl1}, we establish the Ahlfors regularity estimate \eqref{eq:reif6} with constant $D=D(n, Y, \cc)$. The proof of this theorem is technically intricate and combines an inductive covering argument, a limiting process, and a geometric transformation theorem proved in Theorem \ref{thm:existtransRFL}.
	
\subsubsection*{Neck Decomposition in Dimension Four}

The study of quotient cylindrical singularities and the construction of quotient cylindrical neck regions are not sufficient to analyze general singularities. In dimension $4$, however, most singularities are modeled either on a quotient cylinder or on a flat cone of the form
\begin{align*}
\mathcal F(\Gamma)=\R^4/\Gamma \times \R,
\end{align*}
where $\Gamma \leqslant \mathrm{O}(4)$ is a nontrivial finite group acting freely on $S^3$.

The quantitative analysis of this model space presents certain difficulties. First, even if one tangent flow at a point is $\R^4/\Gamma \times \R$, other tangent flows at the same point need not coincide. Second, a region that is quantitatively close to $\R^4/\Gamma \times \R$ at small scales may, on larger scales, resemble a quasi-static cone $\R^4/\Gamma \times (-\infty, a]$ for some $a \in [0,\infty)$. Roughly speaking, this indicates that the model $\R^4/\Gamma \times \R$ is not quantitatively stable.

To address this, we consider a more stable model:
\begin{align*}
\mathcal F^0(\Gamma)=\R^4/\Gamma \times \R_-,
\end{align*}
which captures only the negative-time behavior. In this setting, we introduce the notion of a \textbf{$(\delta, \cc, r)$–flat neck region} $\NNN$ (see Definition \ref{defnneckgeneral}), defined analogously to a $(k, \delta, \cc, r)$–cylindrical neck region, but with the model space replaced by $\mathcal F^0(\Gamma)$. For any $(\delta, \cc, r)$–flat neck regions, we prove Ahlfors regularity in Proposition \ref{ahlforsregforstaticneck}, and further show in Lemma \ref{lem:bilitpsta} that their centers $\CCC$ are vertically parabolic $2$-rectifiable. By Definition \ref{defnrectifiablity}, this is equivalent to stating that the time-function $\t \vert_{\CCC}: \CCC \to \R$ is bi-Lipschitz, where $\R$ is equipped with the parabolic metric.

Another key fact is that quotient cylindrical neck regions and flat neck regions are mutually exclusive. In other words, there is no need to consider neck regions of mixed type---regions that resemble a quotient cylinder at one scale but $\R^4/\Gamma \times \R_-$ (on the negative part) at another. Such a scenario is ruled out by the quantitative uniqueness of quotient cylinders; see Proposition \ref{prop:uniformcq}.

We now state the neck decomposition theorem in dimension four. For simplicity, we only state the case $k=2$ in Theorem \ref{neckdecomgeneral}.

\begin{thm}[Neck decomposition theorem]\label{neckdecomgeneral4dintro}
Let $(Z, d_Z, \t)$ be a noncollapsed Ricci flow limit space arising as the pointed Gromov--Hausdorff limit of a sequence in $\MM(4, Y, T)$. For any constants $\delta>0$ and $\eta>0$, if $\zeta \le \zeta(Y, \delta, \eta)$, then the following holds.

Given $z_0 \in Z$ with $\t(z_0)-2 \zeta^{-2} r_0^2 \in \III^-$, we have the decomposition\emph{:}
	\begin{align*}
		&B^*(z_0,r_0)\subset \bigcup_a\big(\NNN'_a\bigcap B^*(x_a,r_a)\big)\bigcup \bigcup_b B^*(x_b,r_b)\bigcup S^{2,\delta,\eta},\\
		&S^{2,\delta,\eta}\subset \bigcup_a\big(\CCC_{0,a}\bigcap B^*(x_a,r_a)\big)\bigcup\tilde{S}^{2,\delta,\eta},
	\end{align*}
with the following properties\emph{:}
	\begin{enumerate}[label=\textnormal{(\alph{*})}]
		\item For each $a$, $\NNN_a=B^*(x_a,2r_a)\setminus B^*_{r_x}(\CCC_a)$ is either a $(2, \delta, \cc, r_a)$-quotient cylindrical neck region or a $(\delta, \cc, r_a)$-flat neck region, where $\cc=\cc(Y)$. In the former case, we set $\NNN_a'=\NNN_a$\emph{;} in the latter, $\NNN_a'$ denotes the modified $\delta$-region associated with $\NNN_a$.
		
		\item For each $b$, there exists a point in $B^*(x_b,2 r_b)$ which is $(3,\eta,r_b)$-symmetric.
		
		\item The following content estimates hold\emph{:}
		\begin{align*}
\sum_a r_a^2+\sum_b r_b^2+\HHH^2(S^{2,\delta,\eta})\leq C(Y) r_0^2 \quad \text{and} \quad \HHH^2(\tilde{S}^{2,\delta,\eta})=0.
	\end{align*}	
	\end{enumerate}
\end{thm}

The definition of the modified $\delta$-region is given in Definition \ref{def:modi}. The motivation for introducing it is as follows. Roughly speaking, if one has constructed a quotient cylindrical neck region, then any point away from the center locally resembles $\R^4 \times \R$ and should therefore be covered by a $b$-ball $B^*(x_b,r_b)$ as above. By contrast, if one has constructed a flat neck region, then any point in this region carries information only on its negative part, while the model space on the positive part remains a priori unclear. Nevertheless, in a quasi-static cone, one already knows that any point with time coordinate exceeding the arrival time must have larger pointed entropy. Thus, to cover the positive part of a flat neck region, one needs to introduce another collection of balls, each of which has larger pointed entropy. One can then perform a further decomposition on these new balls and repeat the process finitely many times. Consequently, we are led to modify the definition of flat neck regions so as to exclude the newly introduced balls.

The proof of Theorem \ref{neckdecomgeneral4dintro} is technically involved. Related ideas can be found in the literature (see, e.g., \cite{Cheeger2018RectifiabilityOS}), but our argument must handle the quasi-static cone $\R^4/\Gamma \times (-\infty, a]$ and the resulting complications near the boundary $\R^4/\Gamma \times \{a\}$.

As a consequence of Theorem \ref{neckdecomgeneral4dintro}, we deduce that $\MS^2_{\mathrm{F}}$ is contained in the centers of countably many flat neck regions. Hence, $\MS^2_{\mathrm{F}}$ is vertically parabolic $2$-rectifiable. Moreover, Theorem \ref{intro:thm2}(iii) follows directly from Theorem \ref{neckdecomgeneral4dintro}. The volume estimates in Theorem \ref{intro:thm4} also arise from this decomposition, and the $L^1$-curvature bounds in Theorem \ref{intro:thm5} are likewise derived from the neck decomposition.

\subsection*{Organization of the Paper}

The paper is organized as follows.

\begin{itemize}
\item Section \ref{secprel} introduces the basic conventions and foundational results for Ricci flows. We also review key results on Ricci flow limit spaces from \cite{fang2025RFlimit}, and describe our model spaces—cylinders—along with their elementary properties.

\item Section \ref{secnondegeneration} focuses on almost splitting maps in the setting of closed Ricci flows. We construct sharp splitting maps controlled by entropy pinching, establish a Hessian decay estimate for heat flows, and prove a covering lemma for independent points.

\item Section \ref{sec:limit} extends the main results of the previous section to noncollapsed Ricci flow limit spaces. We first introduce a modified version of the pointed $\WW$-entropy and prove its basic properties. We then extend the construction of sharp splitting maps and the Hessian decay estimate to the limit space. Finally, by generalizing the covering lemma, we obtain the Hausdorff dimension bound for each singular stratum.

\item Section \ref{secrectifiablitycyl} is devoted to the parabolic rectifiability of $\MS^k_{\mathrm{c}}$ via the construction of cylindrical neck regions. We also prove Ahlfors regularity for the associated packing measure and generalize the results to the case of quotient cylinders.

\item Section \ref{sec:nd} focuses on the four-dimensional case. It begins with the countability of $\MS^0$. We then introduce flat neck regions, establish their Ahlfors regularity and the rectifiability of their centers, and present the proof of the neck decomposition theorem, with applications to the rectifiability of $\MS^k$, sharp volume estimates of $\MS$, and $L^1$-curvature bounds for closed four-dimensional Ricci flows.

\item Appendix \ref{app:A} contains the proof of a geometric transformation theorem for almost splitting maps, which plays an important role in establishing the Ahlfors regularity of cylindrical neck regions.

\item Finally, we include a list of notations for reference.
\end{itemize}

\textbf{Acknowledgements}: Hanbing Fang would like to thank his advisor, Prof. Xiuxiong Chen, for his encouragement and support. Hanbing Fang is supported by the Simons Foundation. Yu Li is supported by National Key R\&D Program of China 2025YFA1018200, NSFC-12522105, YSBR-001 and research funds from University of Science and Technology of China and Chinese Academy of Sciences.

\section{Preliminaries}\label{secprel}

\subsection{Basic conventions and results for Ricci flows}

Throughout this paper, we consider a closed Ricci flow solution $\XX=\{M^n,(g(t))_{t\in I}\}$, where $M$ is an $n$-dimensional closed manifold, $I \subset \R$ is a closed interval, and $(g(t))_{t\in I}$ is a family of smooth metrics on $M$ satisfying for all $t\in I$ the Ricci flow equation:
\begin{equation*}
	\partial_tg(t)=-2\Ric(g(t)).
\end{equation*}

Following the notation in \cite{fang2025RFlimit}, we use $x^*\in\XX$ to denote a spacetime point $x^* \in M \times I$, and define $\t(x^*)$ to be its time component. We denote by $d_t$ the distance function and by $\mathrm{d}V_{g(t)}$ the volume form induced by the metric $g(t)$. For $x^*=(x,t)\in \XX$, we write $B_t(x,r)$ for the geodesic ball centered at $x$ with radius $r$ at time $t$. The Riemannian curvature, Ricci curvature, and scalar curvature are denoted by $\Rm$, $\Ric$, and $\scal$, respectively, with the time parameter omitted when there is no ambiguity.

For the smooth closed Ricci flow $\XX$, we set $K(x,t;y,s)$ to be the heat kernel, which is determined by:
  \begin{align*}
  \begin{cases}
    &\square K(\cdot,\cdot;y,s) =0,\\
    &\square^{*} K(x,t;\cdot,\cdot) =0,   \\
    &\lim_{t\searrow s} K(\cdot,t;y,s)=\delta_{y},\\
    &\lim_{s\nearrow t} K(x,t;\cdot,s)=\delta_{x}.
        \end{cases}
  \end{align*}\index{$K(x, t;y,s)$}
where $\square:=\partial_t-\Delta$\index{$\square$} and $\square^*:=-\partial_t-\Delta+\scal$\index{$\square^*$}. 

\begin{defn}
	The \textbf{conjugate heat kernel measure} $\nu_{x^*;s}$\index{$\nu_{x^*;s}$} based at $x^*=(x, t)$ is defined as
	\begin{align*}
		\mathrm{d}\nu_{x^*;s}=\mathrm{d}\nu_{x,t;s}:=K(x,t;\cdot,s)\,\mathrm{d}V_{g(s)}.
	\end{align*}		
	It is clear that $\nu_{x^*;s}$ is a probability measure on $M$. If we set 
	\begin{align*}
		\mathrm{d}\nu_{x^*;s}=(4\pi(t-s))^{-n/2}e^{-f_{x^*}(\cdot,s)}\,\mathrm{d}V_{g(s)},
	\end{align*}	
	then the function $f_{x^*}$\index{$f_{x^*}$} is called the \textbf{potential function} at $x^*$ which satisfies:
	\begin{equation*}
		-\partial_s f_{x^*}=\Delta f_{x^*}-|\nabla f_{x^*}|^2+\scal-\frac{n}{2(t-s)}.
	\end{equation*}	
\end{defn}

Next, we recall the definitions of Nash entropy and $\WW$-entropy based at a spacetime point $x^*$.

\begin{defn}\label{defnentropy}
	The \textbf{Nash entropy} based at $x^*\in \XX$ is defined by
	\begin{equation*}
		\NN_{x^*}(\tau):= \int_M f_{x^*}\,\mathrm{d}\nu_{x^*;\t(x^*)-\tau}-\frac{n}{2},
	\end{equation*}\index{$\NN_{x^*}(\tau)$}
	for $\tau>0$ with $\t(x^*)-\tau\in I$, where $f_{x^*}$ is the potential function at $x^*$. 
	Moreover, the $\WW$-entropy based at $x^*$ is defined by
	\begin{equation}\label{defWentropy}
		\WW_{x^*}(\tau):= \int_M\tau(2\Delta f_{x^*}-|\na f_{x^*}|^2+\scal)+f_{x^*}-n \,\mathrm{d}\nu_{x^*;\t(x^*)-\tau}.
	\end{equation}\index{$\WW_{x^*}(\tau)$}
\end{defn}

 The following proposition gives basic properties of Nash entropy (see \cite[Section 5]{bamler2020entropy}):
\begin{prop}\label{propNashentropy}
	For any $x^* \in \XX$ with $\t(x^*)-\tau \in I$ and $\scal(\cdot, \t(x^*)-\tau) \ge R_{\min}$, we have the following inequalities:
	\begin{enumerate}[label=\textnormal{(\roman{*})}]
		\item $\displaystyle -\frac{n}{2\tau}+R_{\min} \leq \frac{\dd}{\dd\tau}\NN_{x^*}(\tau)\leq 0;$
		\item $\displaystyle \frac{\dd}{\dd\tau}\lc\tau \NN_{x^*}(\tau)\rc=\WW_{x^*}(\tau)\leq 0;$
		\item $\displaystyle \frac{\dd^2}{\dd\tau^2}\lc\tau\NN_{x^*}(\tau)\rc=-2\tau\int_M \abs{\Ric+\nabla^2 f_{x^*}-\frac{1}{2\tau}g}^2\,\mathrm{d}\nu_{x^*;\t(x^*)-\tau}\leq 0$.
	\end{enumerate}	
\end{prop}

As in \cite[Definition 2.20]{fang2025RFlimit}, we have the following definition.

\begin{defn}\label{def:entropybound}
	A closed Ricci flow $\XX=\{M^n,(g(t))_{t\in I}\}$ is said to have entropy bounded below by $-Y$ at $x^*\in \XX$ if 
	\begin{align}\label{entropybd-Y}
		\inf_{\tau>0}\NN_{x^*}(\tau)\geq -Y,
	\end{align}
	where the infimum is taken over all $\tau>0$ whenever the Nash entropy $\NN_{x^*}(\tau)$ is well-defined. 
	
	Moreover, we say that the Ricci flow $\XX$ has entropy bounded below by $-Y$ if \eqref{entropybd-Y} holds for all $x^*\in\XX$.
\end{defn}

We have the following result from \cite[Propositions 3.12, 3.13]{bamler2020entropy}. Here, the definition of $H$-center can be found in \cite[Definition 2.13]{fang2025RFlimit}. Note that by \cite[Corollary 3.8]{bamler2020entropy}, an $H_n$-center of $x_0^*$, where $H_n:=(n-1)\pi^2/2+2$\index{$H_n$}, must exist for any $t<\t(x_0^*)$.

\begin{prop}\label{existenceHncenter}
Any two $H$-centers $(z_1, t)$ and $(z_2, t)$ of $x_0^*$ satisfy $d_t(z_1, z_2) \le 2 \sqrt{H(\t(x_0^*)-t)}$. Moreover, if $(z,t)$ is an $H$-center of $x_0^* \in \XX$, then for any $L>0$, we have
	\begin{align*}
		\nu_{x^*_0;t}\lc B_t \lc z,\sqrt{LH (\t(x_0^*)-t)} \rc \rc\geq 1-L^{-1}.
	\end{align*}
\end{prop}

For later applications, we need the following $L^p$-Poincar\'e inequality, proved by \cite[Theorem 1.10]{hein2014new} and \cite[Theorem 11.1]{bamler2020entropy}.

\begin{thm}[Poincar\'e inequality]\label{poincareinequ}
	Let $\flow$ be a closed Ricci flow with $x_0^*=(x_0,t_0)\in\XX$. Suppose $\tau>0$ with $t_0-\tau \in I$, and $h\in C^1(M)$ with $\int_M h \,\mathrm{d}\nu_{x_0^*;t_0-\tau}=0$. Then for any $p\geq 1$, 
	\begin{align*}
		\int_M |h|^p\,\mathrm{d}\nu_{x_0^*;t_0-\tau}\leq C(p)\tau^{p/2}\int_M |\nabla h|^p\,\mathrm{d}\nu_{x_0^*;t_0-\tau}.
	\end{align*}
	Here, we can choose $C(1)=\sqrt{\pi}$ and $C(2)=2$. 
\end{thm}

Also, the following integral estimates from \cite[Proposition 6.2]{bamler2020structure} will be frequently used.

\begin{prop}\label{integralbound}
	There exists $\bar \theta=\bar \theta(n)>0$ such that the following holds. Let $\flow$ be a smooth closed Ricci flow. Assume $x_0^*=(x_0,t_0)\in\XX$ with $[t_0-2r^2,t_0]\subset I$ and set $\mathrm{d}\nu_t=\mathrm{d}\nu_{x_0,t_0;t}=(4\pi\tau)^{-n/2}e^{-f}\mathrm{d}V_{g(t)}$, where $\tau=t_0-t$. Assume $\NN_{x_0^*}(2r^2)\geq -Y$ for some $r>0$, then for any $0<\chi \leq 1/2$ and $\theta\in [0,\bar \theta]$,
	\begin{align*}
		\int_{t_0-r^2}^{t_0-\chi r^2}\int_M\lc\tau|\Ric|^2+\tau |\na^2 f|^2+|\na f|^2+\tau |\na f|^4+\tau^{-1}e^{\theta f}+\tau^{-1}\rc e^{\theta f}\mathrm{d}\nu_t \mathrm{d}t&\leq C(n,Y)|\log\chi|,\\
		\int_M\lc \tau |\scal|+\tau |\Delta f|+\tau|\na f|^2+e^{\theta f}+1\rc e^{\theta f}\mathrm{d}\nu_{t_0-r^2}&\leq C(n,Y).
	\end{align*}
\end{prop}

We conclude this section with the following weighted Bianchi identity, whose proof can be found in \cite[Lemma 2.5]{FLloja05}. Here, we use $\di$ to denote the divergence operator and $\delf:=e^f \circ \mathrm{div} \circ e^{-f}$ the weighted divergence operator.

\begin{lem}\label{weightedBianchilem}
Given a spacetime point $x_0^*\in \XX$ with $f=f_{x_0^*}$ and $\tau=\t(x_0^*)-t$, the following weighted Bianchi identity holds:
	\begin{equation*}
		\nabla \lc \tau(2\Delta f-|\nabla f|^2+\scal)+f-n \rc=2\delf \lc \tau \Ric+\nabla^2(\tau f)-\frac{g}{2} \rc.
	\end{equation*}
\end{lem}

\subsection{Noncollapsed Ricci flow limit spaces}

In this subsection, we review the construction of noncollapsed Ricci flow limit spaces and their key properties, as developed in \cite{fang2025RFlimit}.

Following \cite[Section 3]{fang2025RFlimit}, fix $T\in(0,+\infty]$. When $T<+\infty$, set
\begin{align*}
\III^+:=[-0.99T,0],\quad \III:=[-0.98T,0],\quad \III^-:=(-0.98T,0],\quad \III^{++}:=[-T,0].\index{$\III^{++}, \III^{+}, \III, \III^-$}
\end{align*}
When $T=+\infty$, set $\III^-=\III=\III^+=\III^{++}=(-\infty,0]$. More generally, in the finite-time case these intervals may be chosen using $\sigma\in(0,1/100]$ as $[-(1-\sigma)T,0]$, $[-(1-2\sigma)T,0]$, and $(-(1-2\sigma)T,0]$, respectively. For simplicity, we fix $\sigma=0.01$.

As in \cite[Definition 3.1]{fang2025RFlimit}, we define the moduli spaces as follows.

\begin{defn}[Moduli space]\label{defnmoduli}
For a fixed constant $T\in(0,+\infty]$, the moduli space $\MM(n,T)$\index{$\MM(n,T)$} consists of all $n$-dimensional closed Ricci flows
\begin{align*}
\XX=\{M^n,(g(t))_{t\in\III^{++}}\}.
\end{align*}
For $Y>0$, the subspace $\MM(n,Y,T)\subset\MM(n,T)$\index{$\MM(n,Y,T)$} consists of all such flows with entropy bounded below by $-Y$ (cf. Definition \ref{def:entropybound}).
\end{defn}

For any $\XX\in\MM(n,T)$, the spacetime distance on $M\times\III^+$ is defined as follows (cf. \cite[Definition 3.2]{fang2025RFlimit}).
\begin{defn}\label{defnd*distance}
For any $x^*,y^*\in M\times\III^+$, set
\begin{align*}
t_+:=\max\{\t(x^*),\t(y^*)\},\qquad t_-:=\min\{\t(x^*),\t(y^*)\}.
\end{align*}
We define
\begin{align*}
d^*(x^*,y^*):=\inf_{-0.99T\leq\tau\leq t_-}
\max\left\{\sqrt{t_+-\tau},\ d_{W_1}^{\tau}(\nu_{x^*;\tau},\nu_{y^*;\tau})\right\}.\index{$d^*(x^*,y^*)$}
\end{align*}
If $T=+\infty$, the infimum is taken over $\tau\in(-\infty,t_-]$.
\end{defn}
By \cite[Lemma 3.4]{fang2025RFlimit}, $d^*$ defines a distance function on $M\times\III^+$, and its induced topology coincides with the standard topology on $M\times\III^+$ (see \cite[Corollary 3.8]{fang2025RFlimit}). The metric ball defined by $d^*$ is denoted by $B^*$.

The following weak compactness theorem is proved in \cite[Theorem 1.3]{fang2025RFlimit}.

\begin{thm}[Weak compactness]\label{thm:intro1}
Given any sequence $\XX^i=\{M_i^n,(g_i(t))_{t\in\III^{++}}\}\in\MM(n,T)$ with base points $p_i^*\in M_i\times\III$ \emph{(}when $T=+\infty$, we additionally assume $\limsup_{i\to\infty}\t_i(p_i^*)>-\infty$\emph{)}, after passing to a subsequence, we obtain
\begin{align}\label{eq:constar}
(M_i\times\III,d_i^*,p_i^*,\t_i)\xrightarrow[i\to\infty]{\quad\mathrm{pGH}\quad}(Z,d_Z,p_\infty,\t),
\end{align}
where $d_i^*$ denotes the restriction of the $d^*$-distance to $M_i\times\III$. The limit $(Z,d_Z,\t)$ is a complete, separable, locally compact parabolic space.
\end{thm}

The limit $(Z,d_Z,\t)$ is referred to as a \textbf{Ricci flow limit space} over $\III$. It is called \textbf{noncollapsed} if the approximating sequence lies in $\MM(n,Y,T)$ for some fixed $Y>0$. From this point through the structure theory, we assume this noncollapsed setting. For the distance $d_Z$, we write $B_Z^*$ for the corresponding metric ball. $Z$ contains a regular part $\RR$, whose restriction $\RR_{\III^-}$ is a dense open subset of $Z_{\III^-}$ (see \cite[Corollary 5.7]{fang2025RFlimit}) and carries the structure of a Ricci flow spacetime $(\RR, \t, \partial_\t, g^Z)$. On this regular part, the convergence described in Theorem \ref{thm:intro1} is smooth, in the following sense (cf. \cite[Theorem 1.5]{fang2025RFlimit}):

\begin{thm}[Smooth convergence]\label{thm:intro3}
	There exists an increasing sequence $U_1 \subset U_2 \subset \ldots \subset \RR$ of open subsets with $\bigcup_{i=1}^\infty U_i = \RR$, open subsets $V_i \subset M_i \times \III$, time-preserving diffeomorphisms $\phi_i : U_i \to V_i$ and a sequence $\ep_i \to 0$ such that the following holds:
	\begin{enumerate}[label=\textnormal{(\alph{*})}]
		\item We have
		\begin{align*}
			\Vert \phi_i^* g^i - g^Z \Vert_{C^{[\ep_i^{-1}]} ( U_i)} & \leq \ep_i, \\
			\Vert \phi_i^* \partial_{\t_i} - \partial_{\t} \Vert_{C^{[\ep_i^{-1}]} ( U_i)} &\leq \ep_i,
		\end{align*}
		where $g^i$ is the spacetime metric induced by $g_i(t)$, and $\partial_{\t_i}$ is the standard time vector field induced by $\t_i$.
		
		\item Let $y \in \RR$ and $y_i^* \in M_i \times \III$. Then $y_i^* \to y$ in the Gromov--Hausdorff sense if and only if $y_i^* \in V_i$ for large $i$ and $\phi_i^{-1}(y_i^*) \to y$ in $\RR$.
		
		\item For $U_i^{(2)}=\{(x,y) \in U_i \times U_i \mid \t(x)> \t(y)+\ep_i\}$, $V_i^{(2)}=\{(x^*,y^*) \in V_i \times V_i \mid \t_i(x^*)> \t_i(y^*)+\ep_i\}$ and $\phi_i^{(2)}:=(\phi_i, \phi_i): U_i^{(2)} \to V_i^{(2)}$, we have
		\begin{align*}
			\Vert  (\phi_i^{(2)})^* K^i-K_Z \Vert_{C^{[\ep_i^{-1}]} ( U_i^{(2)})} \le \ep_i,
		\end{align*}	\index{$K_Z(\cdot ;\cdot)$}
		where $K^i$ and $K_Z$ denote the heat kernels on $(M_i \times \III, g_i(t))$ and $(\RR, g^Z)$, respectively.
		
			\item If $z_i^* \in M_i \times \III$ converge to $z \in Z$ in Gromov--Hausdorff sense, then
\begin{align*}
K^i(z_i^*;\phi_i(\cdot)) \xrightarrow[i \to \infty]{C^{\infty}_\mathrm{loc}} K_Z(z;\cdot) \quad \text{on} \quad \RR_{(-\infty,\t(z))}.
\end{align*}	
		\item For each $t \in \III$, there are at most countable connected components of the time-slice $\RR_t$.
	\end{enumerate}
\end{thm}

For each $z \in Z$, we can assign a conjugate heat kernel measure $\mathrm{d}\nu_{z;s}:=K_Z(z;\cdot) \,\mathrm{d}V_{g^Z_s}$\index{$\nu_{z;s}$} based at $z$ for $s \le \t(z)$, which is a probability measure on $\RR_s$. All these probability measures together satisfy the reproduction formula (cf. \cite[Equation (5.5)]{fang2025RFlimit}).

For each $t \in \III^-$, one can define an extended distance function $d^{Z}_t(x,y):=\lim_{s \nearrow t} d_{W_1}^{\RR_s}(\nu_{x;s},\nu_{y;s}) \in [0,\infty]$ for any $x,y \in Z_t$, where $d_{W_1}^{\RR_s}$ denotes the $W_1$-Wasserstein distance on $(\RR_s, g^Z_s)$. It can be proved, see \cite[Theorem 1.7]{fang2025RFlimit}, that $(Z_t, d_t^Z)$ is a complete extended metric space. Moreover, for all but countably many times $t \in \III^-$, $d^Z_t=d_{g^Z_t}$ on each connected component of $\RR_t$. For more properties of $d^Z_t$, we refer readers to \cite[Section 6]{fang2025RFlimit}.

Next, we recall the following definition and notation from \cite[Definition 5.37, Notation 5.38]{fang2025RFlimit}.

\begin{defn}[$\ep$-close]\label{defn:close}
Suppose $(Z, d_Z, z, \t)$ and $(Z', d_{Z'}, z',\t')$ are two pointed noncollapsed Ricci flow limit spaces, with regular parts given by the Ricci flow spacetimes $(\RR, \t, \partial_\t, g^Z)$ and $(\RR', \t', \partial_{\t'}, g^{Z'})$, respectively, such that $J$ is a time interval.

We say that $(Z, d_Z, z, \t)$ is \textbf{$\ep$-close} to $(Z', d_{Z'}, z',\t')$ \textbf{over $J$} if there exists an open set $U \subset \RR'_J$ and a smooth embedding $\phi: U \to \RR_J$ satisfying the following properties.
\begin{enumerate}[label=\textnormal{(\alph{*})}]
\item $\phi$ is time-preserving.

\item $U \subset B^*_{Z'}(z', \ep^{-1}) \bigcap \RR'_J$ and $U$ is an $\ep$-net of $B^*_{Z'}(z', \ep^{-1}) \bigcap Z'_J$ with respect to $d_{Z'}$.

\item For any $x, y \in U$, we have
	\begin{align*}
\abs{d_Z(\phi(x), \phi(y))-d_{Z'}(x, y)} \le \ep.
	\end{align*}
	
\item The $\ep$-neighborhood of $\phi(U)$ with respect to $d_Z$ contains $B^*_{Z}(z, \ep^{-1}-\ep) \bigcap Z_J$.

\item There exists $x_0 \in U$ such that $d_{Z'}(x_0, z') \le \ep$ and $d_{Z}(\phi(x_0), z) \le \ep$.

\item On $U$, the following estimates hold:
  \begin{align*}
  	\rVert \phi^* g^Z-g^{Z'}\rVert_{C^{[\ep^{-1}]}(U)}+\rVert \phi^* \partial_\t-\partial_{\t'}\rVert_{C^{[\ep^{-1}]}(U)} \le \ep.
  \end{align*} 
\end{enumerate}

Note that $\phi$ can be extended to a map $\tilde \phi: B^*_{Z'}(z', \ep^{-1}) \bigcap Z'_J \to Z_J$ such that for any $x, y \in B^*_{Z'}(z', \ep^{-1}) \bigcap Z'_J$, we have
	\begin{align*}
\abs{d_Z(\tilde \phi(x), \tilde \phi(y))-d_{Z'}(x, y)} \le 3\ep.
	\end{align*}
We call the extension $\tilde \phi$ an \textbf{$\ep$-map}\index{$\ep$-map}. In general, $\tilde \phi$ is neither unique nor continuous, but it serves as a Gromov--Hausdorff approximation from $Z'_J$ to $Z_J$.
\end{defn}

\begin{notn}\label{not:2}
For a sequence of noncollapsed Ricci limit spaces $(Z_i, d_{Z_i}, z_i, \t_i)$, $i \in \mathbb N \cup \{\infty\}$, we write
	\begin{equation*} 
		(Z_i, d_{Z_i}, z_i, \t_i) \xrightarrow[i \to \infty]{\quad \hat C^\infty \quad} (Z_\infty, d_{Z_\infty}, z_\infty,\t_\infty),
	\end{equation*}
	if there exists a sequence $\ep_i \to 0$ such that $(Z_i, d_{Z_i}, z_i, \t_i)$ is $\ep_i$-close to $(Z_\infty, d_{Z_\infty}, z_\infty,\t_\infty)$ over $[-\ep_i^{-1}, \ep_i^{-1}]$.
\end{notn}

In particular, it is clear by Theorem \ref{thm:intro3} that the convergence \eqref{eq:constar} can be improved to be
	\begin{equation*} 
(M_i \times \III, d^*_i, p_i^*,\t_i) \xrightarrow[i \to \infty]{\quad \hat C^\infty \quad} (Z, d_Z, p_{\infty},\t).
	\end{equation*}

Now, we recall the following definitions.

\begin{defn}[Ricci shrinker space]\label{def:rss}
A pointed parabolic metric space $(Z',d_{Z'},z',\t')$ with $\t'(z')=0$ is called an $n$-dimensional \textbf{Ricci shrinker space} with entropy bounded below by $-Y$ if it satisfies $\R_- \subset \mathrm{image}(\t')$ and arises as the pointed Gromov--Hausdorff limit of a sequence of Ricci flows in $\mathcal M(n, Y, T_i)$ with $T_i\to +\infty$ \emph{(}see \emph{\cite[Remark 3.22]{fang2025RFlimit}}\emph{)}. Moreover, $\NN_{z'}(\tau)$ remains constant for all $\tau>0$.

For any Ricci shrinker space $(Z',d_{Z'},z',\t')$, we call $(Z'_{(-\infty, 0]},d_{Z'},z',\t')$ its \textbf{negative part}\index{negative part}.
	\end{defn}

\begin{defn}[Tangent flow] \label{def:tf}
	For any $z \in Z_{\III^-}$, a \textbf{tangent flow} at $z$ is a pointed Gromov--Hausdorff limit of $(Z, r_j^{-1} d_Z, z, r_j^{-2}(\t-\t(z)))$ for a sequence $r_j  \searrow 0$.
\end{defn}

It is clear that any tangent flow is a Ricci shrinker space. For Ricci shrinker spaces, we have (cf. \cite[Theorem 1.9]{fang2025RFlimit})

\begin{thm}\label{thm:intro5}
	Let $(Z', d_{Z'}, z', \t')$ be a Ricci shrinker space so that its regular part is given by a Ricci flow spacetime $(\RR', \t', \partial_{\t'}, g^{Z'}_t)$. Then the following statements hold.
	\begin{enumerate}[label=\textnormal{(\alph{*})}]
		\item On $\RR'_{(-\infty,0)}$, the following equation holds:
		\begin{align*}
			\Ric(g^{Z'})+\na^2 f_{z'}=\frac{g^{Z'}}{2 \tau},
		\end{align*}
		where $f_{z'}$ is the potential function of $\nu_{z';\cdot}$ and $\tau(\cdot)=-\t'(\cdot)$.
		
		\item For any $t<0$, the slice $\RR'_t$ is connected. Moreover, the distance $d^{Z'}_t$, when restricted on $\RR'_t$, coincides with the Riemannian distance induced by the metric $g^{Z'}_t$.
		
		\item $Z'_{(0, \infty)}=\emptyset$ if $(Z', d_{Z'}, z', \t')$ is \textbf{collapsed} (see \emph{\cite[Definition 7.18]{fang2025RFlimit}}).
		
		\item For any $t<0$, $Z'_{t} \setminus \RR'_t$ has Minkowski dimension at most $n-4$ with respect to $d^{Z'}_t$.
	\end{enumerate}
\end{thm}

For tangent flows, we have the following rough classification.

\begin{defn}[$k$-symmetric]\label{defnsymmetricsoliton}
	A Ricci shrinker space $(Z',d_{Z'},z',\t')$ is called \textbf{$k$-symmetric}\index{$k$-symmetric} if one of the following holds:
		\begin{enumerate}[label=\textnormal{(\arabic*)}]
		\item  $(Z',d_{Z'},z',\t')$ is $k$-splitting and is not a static cone.
	
		\item $(Z',d_{Z'},z',\t')$ is a static cone that is $(k-2)$-splitting.
	\end{enumerate} 
\end{defn}
	Here, $(Z',d_{Z'},z',\t')$ is $k$-splitting if the regular part $\RR'$ splits off an $\R^k$ at $\t=-1$ (see \cite[Definition 8.1, Proposition 8.2]{fang2025RFlimit}). Roughly speaking, a \textbf{static cone} is characterized by $\mathrm{image}(\t')=\R$ and vanishing Ricci curvature on $\RR'$. Case (1) above may include a \textbf{quasi-static cone}, which has vanishing Ricci curvature only on $\RR'_{(-\infty, t_a]}$ for some constant $t_a \in [0, \infty)$ called the \textbf{arrival time}, but not beyond. For precise definitions and related properties of static and quasi-static cones, see \cite[Definition 7.17, Theorem 7.21, Corollary 7.22, Proposition 7.23]{fang2025RFlimit}.

On $Z_{\III^-}$, we have the following regular-singular decomposition:
\begin{align*} 
	Z_{\III^-}=\RR_{\III^-} \sqcup \MS,
\end{align*}
where $\RR_{\III^-}$ denotes the restriction of $\RR$ on $\III^-$. It can be proved (see \cite[Theorem 7.15]{fang2025RFlimit}) that a point $z$ is a regular point if and only if any of its tangent flows is isometric to $(\R^n\times\R,d_E^*,(\vec 0^n,0),\t)$ or $(\R^n\times\R_-,d_E^*,(\vec 0^n,0),\t)$, where $d_E^*$ denotes the spacetime distance induced by Definition \ref{defnd*distance} on the static Euclidean flow (see \cite[Example 3.6]{fang2025RFlimit}). Equivalently, $z$ is a regular point if and only if $\NN_z(0):=\lim_{\tau \nearrow 0} \NN_z(\tau)\geq-\ep_n$ (see \cite[Proposition 7.7]{fang2025RFlimit}).

The singular set $\MS$ admits a natural stratification:
\begin{equation}\label{defnstratification}
	\mathcal S^0 \subset \mathcal S^1 \subset \cdots \subset \mathcal S^{n+1}=\mathcal S,
\end{equation}
where a point $z \in \MS^k$ if and only if no tangent flow at $z$ is $(k+1)$-symmetric. It can be proved, see \cite[Theorems 1.10, 1.13]{fang2025RFlimit}) that
	\begin{align*}
		\MS=\MS^{n-2},
	\end{align*}
	and the Minkowski dimension of $\mathcal S$ with respect to $d_Z$ satisfies
		\begin{align*}
		\dim_{\MMM} \mathcal S \le n-2.
	\end{align*}

\begin{defn}\label{defnalmostsymmetric}
A point $z\in Z_{\III^-}$ is called \textbf{$(k,\ep,r)$-symmetric}\index{$(k,\ep,r)$-symmetric} if $\t(z)-\ep^{-1}r^2 \in \III^-$, and there exists a $k$-symmetric Ricci shrinker space $(Z',d_{Z'},z',\t')$ such that 
	  \begin{align*}
(Z, r^{-1} d_Z, z, r^{-2}(\t-\t(z))) \quad \text{is $\ep$-close to} \quad (Z',d_{Z'},z',\t') \quad \text{over} \quad [-\ep^{-1}, \ep^{-1}].
  \end{align*} 

Furthermore, if $k \in \{n-3,n-2\}$, then the model space $(Z',d_{Z'},z',\t')$ cannot be a quasi-static cone. If $k\geq n-1$, then the model space $(Z',d_{Z'},z',\t')$ is isometric to $(\R^n\times(-\infty,t_a],d_E^*,(\vec 0,0),\t)$ for some constant $t_a\in[0,+\infty]$.
\end{defn}

Next, we recall the definition of the quantitative singular strata.

\begin{defn} \label{introdefnquantiSS} \index{$\MS^{\ep,k}_{r_1,r_2}$}
	For $\ep > 0$ and $0<r_1<r_2<\infty$, the quantitative singular strata
	\[  \MS^{\ep,0}_{r_1,r_2} \subset \MS^{\ep,1}_{r_1,r_2} \subset   \ldots \subset  \MS^{\ep,n-2}_{r_1,r_2} \subset Z_{\III^-} \]
	are defined as follows:
	$z \in  \MS^{\ep,k}_{r_1,r_2}$ if and only if $\t(z)-\ep^{-1} r_2^2 \in \III^-$ and for all $r \in [r_1, r_2]$, $z$ is not $(k+1,\ep,r)$-symmetric. 
\end{defn}

The following identity is clear from the above definitions: for any $L>1$,
\begin{align*}
\MS^{k}=\bigcup_{\ep \in (0, L^{-1})} \bigcap_{0<r<\ep L} \MS^{\ep,k}_{r, \ep L}.
\end{align*}

Next, we recall the following definition from \cite[Definition 10.5]{fang2025RFlimit}.

\begin{defn}\label{defnksplitting}
	A point $z\in Z_{\III^-}$ is called \textbf{$(k,\ep,r)$-splitting}\index{$(k,\ep,r)$-splitting} if $\t(z)-10  r^2 \in \III^-$ and there exists a noncollapsed Ricci flow limit space such that its regular part $\RR'_{[-10, 0]}$ splits off an $\R^k$ as a Ricci flow spacetime. Moreover, 
	  \begin{align*}
(Z, r^{-1} d_Z, z, r^{-2}(\t-\t(z))) \quad \text{is $\ep$-close to} \quad (Z',d_{Z'},z',\t') \quad \text{over} \quad [-10, 0].
  \end{align*} 
\end{defn}

We also have the following concept of almost splitting maps as in \cite[Definition 10.1]{fang2025RFlimit}, which will be of crucial importance in proving the rectifiability of the singular set. 

\begin{defn}[$(k,\ep, r)$-splitting map]\label{defnsplittingmap}\index{$(k,\ep, r)$-splitting map}
	Let $\XX=\{M^n,(g(t))_{t\in \III^{++}}\}\in \MM(n,Y, T)$ and $x_0^*=(x_0,t_0)\in\XX, r>0$ with $[t_0-10r^2,t_0]\subset \III$.
	A map $\vec u=(u_1,\ldots, u_k)$ is called a \textbf{$(k,\ep, r)$-splitting map at $x_0^*$} if for all $i,j\in \{1, \ldots, k\}$, the following properties hold:
	\begin{enumerate}[label=\textnormal{(\roman{*})}]
		\item $u_i(x^*_0)=0$.
		\item $\square u_i=0$ on $M\times [t_0-10r^2,t_0]$.
		\item $\displaystyle	\int_{t_0-10 r^2}^{t_0-r^2/10} \int_{M} |\na^2 u_i|^2 \,\mathrm{d}\nu_{x_0^*;t} \mathrm{d}t \le \ep.$		
		\item $\displaystyle	\int_{t_0-10 r^2}^{t_0-r^2/10} \int_{M} \la \na u_i ,\na u_j \ra-\delta_{ij} \,\mathrm{d}\nu_{x_0^*;t} \mathrm{d}t=0.$
	\end{enumerate}
\end{defn}
For the basic properties of almost splitting maps, we refer readers to \cite[Section 10]{fang2025RFlimit}. On a Ricci flow limit space $(Z, d_Z, p_{\infty},\t)$ obtained in Theorem \ref{thm:intro1}, we can generalize the above definition.

\begin{defn}\label{defnsplittingmap1}
A map $\vec u=(u_1,\cdots, u_k)$ is called a \textbf{$(k,\ep, r)$-splitting map at $z \in Z_{\III^-}$} if $\t(z)-10 r^2 \in \III^-$, and $\vec u$ is obtained as the limit of a sequence of $(k, \ep, r)$-splitting maps $\vec u^i=(u^i_1,\ldots,u^i_k)$ at $z_i^*$ with $z_i^* \to z$ in the Gromov--Hausdorff sense. Note that $\vec u$ is defined on $Z_{(\t(z)-10 r^2, +\infty)}$ by reproduction formula.
\end{defn}

Note that by taking the limit, all properties for smooth almost splitting maps hold for almost splitting maps on $Z$. We also have the following definition.

\begin{defn}\label{defnstatic}
	A point $z\in Z_{\III^-}$ is \textbf{$(\ep,r)$-static}\index{$(\ep,r)$-static} if $\t(z)-2r^2 \in \III^-$ and
	\begin{align*}
r^2\int_{\t(z)-2r^2}^{\t(z)-r^2/2} \int_{\RR_t} |\Ric_{g^Z}|^2 \,\mathrm{d}\nu_{z;t} \mathrm{d}t \le \ep.
	\end{align*}
\end{defn}

We end this section with the following volume estimate from \cite[Proposition 5.35]{fang2025RFlimit}, which will be frequently used.

\begin{prop}\label{prop:volumebound}
For any $x \in Z$ and $r>0$ with $\t(x)-r^2 \in \III^-$, we have
	\begin{align*}
0<c(n,Y) r^{n+2} \le 	|B^*_Z(x ,r)|\le C(n) r^{n+2},
	\end{align*}
where $|\cdot|$ denotes the volume in $Z$; see \emph{\cite[Definition 5.33]{fang2025RFlimit}}. 
\end{prop}

\subsection{Cylindrical and almost cylindrical points}\label{cacp}

We consider the standard Ricci flow solution on the cylinder:
\begin{align*}\index{$\mathcal C^k$}
\mathcal C^k:=(\bar M,(\bar g(t))_{t<0},(\bar f(t))_{t<0})=\left(\R^{k}\times S^{n-k}, g_E \times |t| g_{S^{n-k}}, \frac{|\vec{x}|^2}{4|t|}+\frac{n-k}{2}+\Theta_{n-k} \right),
\end{align*}
where $g_E$ is the Euclidean metric on $\R^{k}$, $g_{S^{n-k}}$ is the round metric on $S^{n-k}$ such that $\Ric(g_{S^{n-k}})=g_{S^{n-k}}/2$. The vector $\vec{x}=(x_1,\ldots,x_{k})$ denotes the standard coordinate function on $\R^{k}$. The constant $\Theta_{n-k}$ is chosen to ensure that for any $t<0$
\begin{align*}
(4\pi |t|)^{-\frac n 2}\int_{\mathcal C^k_t} e^{-\bar f(t)} \,\mathrm{d}V_{\bar g(t)}=1.
\end{align*}

We denote by $d_{\mathcal C}^*$\index{$d_{\mathcal C}^*$} the spacetime distance on $\mathcal C^k$ induced by the variational formula in Definition \ref{defnd*distance}.

Then, we set the completion of $\mathcal C^k$ under $d_{\mathcal C}^*$ by $\bar{\mathcal C}^k$\index{$\bar{\mathcal C}^k$}. It is straightforward to verify that the metric completion adds only the singular set $\R^k \times \{0\}$, which is the spine (see \cite[Definition D.4]{fang2025RFlimit}) of $\bar{\mathcal C}^k$.

We then define the base point $p^*$ as the limit of $(\bar p, t)$ as $t \nearrow 0$ with respect to $d_{\mathcal C}^*$, where $\bar p \in \bar M$ is a minimum point of $\bar f(-1)$. It is clear that $p^*$ is independent of the choice of $\bar p$. Moreover, for any $t<0$,
	\begin{align*}
\nu_{p^*;t}=(4\pi |t|)^{-\frac n 2} e^{-\bar f(t)} \,\mathrm{d}V_{\bar g(t)}.
	\end{align*}

For later application, we need the following lemma from \cite[Lemma 8.5]{fang2025RFlimit}.

\begin{lem} \label{lem:comparedis1}
For any $x^*, y^* \in \bar{\mathcal C}^k_0$, we have
	\begin{align*}
0<c(n)|\vec{x}-\vec{y}|\leq d_{\mathcal C}^*(x^*,y^*)\leq |\vec{x}-\vec{y}|,
	\end{align*}
where $\vec{x}$ and $\vec{y}$ are components of $x^*$ and $y^*$ in $\R^k$, respectively.
\end{lem}

Next, we consider a general noncollapsed Ricci flow limit space $(Z, d_Z, \t)$ over $\III$, obtained as the limit of a sequence in $\mathcal M(n, Y, T)$. Then, we have the following definition from \cite[Definition 2.22]{FLloja05}.

\begin{defn}\label{def:ccc}
A point $z \in Z_{\III^-}$ is called a \textbf{cylindrical point with respect to $\bar{\mathcal C}^k$} if a tangent flow at $z$ (see Definition \ref{def:tf}) is isometric to $\bar{\mathcal C}^k$.
\end{defn}

Note that by \cite[Theorem 2.23]{FLloja05}, the tangent flow at a cylindrical point is unique. Next, we introduce the following definitions.

\begin{defn}\label{def:cylindr}
Let $\MS^k_{\mathrm{c}}$\index{$\MS^k_{\mathrm{c}}$} be the subset of $\MS^k$ consisting of all cylindrical points with respect to $\bar{\mathcal C}^l$ for some $l \le k$.
\end{defn}

\begin{defn}\label{def:almost0}
Let $(Z, d_Z, \t)$ be a noncollapsed Ricci flow limit space arising as the pointed Gromov--Hausdorff limit of a sequence in $\MM(n, Y, T)$. A point $z \in Z_{\III^-}$ is called \textbf{$(k,\ep,r)$-cylindrical}\index{$(k,\ep,r)$-cylindrical} if $\t(z)-\ep^{-1} r^2 \in \III^-$ and
	  \begin{align*}
(Z, r^{-1} d_Z, z, r^{-2}(\t-\t(z))) \quad \text{is $\ep$-close to} \quad (\bar{\mathcal C}^k ,d^*_{\mathcal C}, p^*,\t) \quad \text{over} \quad [-\ep^{-1}, \ep^{-1}].
  \end{align*} 

Let $\tilde \phi$ be an $\ep$-map from Definition \ref{defn:close}, which maps $B^*(p^*,\ep^{-1}) \cap \bar{\mathcal C}^k_{[-\ep^{-1}, \ep^{-1}]}$ to $Z_{[\t(z)-\ep^{-1} r^2, \t(z)+\ep^{-1} r^2]}$, where $B^*(p^*,\ep^{-1})$ is the metric ball in $\bar{\mathcal C}^k$ with respect to $d_{\mathcal C}^*$. Then, we define\emph{:}
\begin{align*}
\LL_{z,r}:=\tilde \phi \lc B^*(p^*,\ep^{-1}) \cap \bar{\mathcal C}^k_0 \rc,\index{$\LL_{z,r}$}
\end{align*}
	and say that $z$ is \textbf{$(k,\ep,r)$-cylindrical with respect to $\LL_{z,r}$}. Note that $\bar{\mathcal C}^k_0$ is exactly the spine of $\bar{\mathcal C}^k$.
\end{defn}

\section{Almost splitting maps on Ricci flows}\label{secnondegeneration}

Throughout this section, we assume $\XX \in \mathcal M(n, Y, T)$ (see Definition \ref{defnmoduli}) and set
\begin{align*}
\III^{++}:=[-T, 0], \quad \III^+:=[-0.99T,0], \quad \III:=[-0.98 T,0], \quad \III^-:=(-0.98 T,0].
	\end{align*}
Moreover, the spacetime distance $d^*$ on $M \times \III^+$ is defined as in Definition \ref{defnd*distance}, and the corresponding metric balls are denoted by $B^*$.

\subsection{Construction of auxiliary functions}

In this subsection, we fix a spacetime point $x_0^*=(x_0,t_0)\in \XX$ and set
\begin{align*}
\mathrm{d}\nu_t=\mathrm{d}\nu_{x^*_0;t}=(4\pi\tau)^{-n/2}e^{-f}\mathrm{d}V_{g(t)},
\end{align*}
where $\tau=t_0-t$ and $f=f_{x_0^*}$. For simplicity, we define
  \begin{equation*}
  \begin{dcases}
    &w:= \tau(2\Delta f-|\nabla f|^2+\scal)+f-n,   \\
    &\TT:=\tau \Ric+\nabla^2(\tau f)-\frac{g}{2},\\
    & F:=\tau f.
        \end{dcases}
  \end{equation*}
Recall that Perelman's differential Harnack inequality (see \cite[Section 9]{perelman2002entropy}) states $w \le 0$.

For constructing almost splitting maps, we need the following lemma, which is similar to \cite[Proposition 6.7]{FLloja05}.

\begin{lem}\label{constructionofsplittingfunction1}
Suppose that $[t_0-30r^2,t_0] \subset \III$ and $u_0:M\times [t_0-30r^2,t_0]\to \R$ is a function satisfying:
	\begin{align*}
		\square u_0=-\frac{n}{2} \quad \text{and} \quad u_0=F \quad \mathrm{at}\quad t=t_0-30r^2.
	\end{align*}
Then, the following statements hold.
	\begin{enumerate}[label=\textnormal{(\roman{*})}]
		\item We have
		\begin{align*}
			\int_{t_0-30r^2}^{t_0-r^2/30}\int_{M} \left|\tau \Ric+\nabla^2 u_0-\frac{g}{2}\right|^2\,\mathrm{d}\nu_t \mathrm{d}t\leq 75 r^{2}\left(\mathcal{W}_{x_0^*}( r^2/30)-\mathcal{W}_{x_0^*}(30r^2)\right).
		\end{align*} 
		\item For any $t\in [t_0-30r^2,t_0-r^2/30]$, we have
		\begin{align*}
			\int_M\left|\na (u_0-F)\right|^2\,\mathrm{d}\nu_{t}\leq 30 r^{2}\left(\mathcal{W}_{x_0^*}( r^2/30)-\mathcal{W}_{x_0^*}(30r^2)\right).
		\end{align*}
		\item For $0<\theta \le \theta(n)$, we have
			\begin{align*}
	\sup_{t \in [t_0-10r^2, t_0-r^2/10]} \int_M |\na u_0|^2 e^{\theta f} \,\mathrm{d}\nu_t+ \int_{t_0-10r^2}^{t_0-r^2/10}\int_{M} |\na^2 u_0|^2 e^{\theta f} \,\mathrm{d}\nu_t \mathrm{d}t \le C(n, Y) r^{2}.
		\end{align*}
	\end{enumerate}
	
\end{lem}
\begin{proof}
Without loss of generality, we assume $t_0=0$ and $r=1$.

It is clear from $\partial_t f=-\Delta f+|\na f|^2-\scal+\frac{n}{2\tau}$ that
		\begin{align} \label{evolutionoff}
\square F=-w-\frac{n}{2}.
	\end{align}	
	
Set $u:= u_0-F$. Then by \eqref{evolutionoff}, the following evolution equation holds:
	\begin{align*}
		\square u=w \quad \text{and} \quad u=0 \quad \mathrm{at}\quad t=-30.
	\end{align*}
	By the weighted Bianchi identity (see Lemma \ref{weightedBianchilem}), we calculate 
	\begin{align*}
		\frac{\dd}{\dd t}\int_{M} \left|\nabla u\right|^2\,\mathrm{d}\nu_t&=\int_{M} \square \left|\nabla u\right|^2\,\mathrm{d}\nu_t\nonumber\\
		&=2\int_{M} \la\nabla\square u,\nabla u\ra \,\mathrm{d}\nu_t-2\int_{M} \left|\nabla^2 u\right|^2\,\mathrm{d}\nu_t\nonumber\\
		&=2\int_{M} \la\nabla w,\nabla u\ra \,\mathrm{d}\nu_t-2\int_{M} \left|\nabla^2 u\right|^2\,\mathrm{d}\nu_t\nonumber\\
		&=4\int_{M} \la \Div_f\TT,\nabla u\ra \,\mathrm{d}\nu_t-2\int_{M} \left|\nabla^2 u\right|^2\,\mathrm{d}\nu_t.
	\end{align*}
	Using integration by parts and integrating in time, we obtain that for any $t_1\in [ -30, 0]$,
	\begin{align}\label{splitting3}
		2\int_{-30}^{t_1}\int_{M} \left|\na^2u \right|^2\,\mathrm{d}\nu_t \mathrm{d}t&=-4\int_{-30}^{t_1}\int_{M}\la\TT,\nabla^2 u\ra \,\mathrm{d}\nu_t \mathrm{d}t-\int_{M}\left|\nabla u\right|^2\,\mathrm{d}\nu_{t_1}\nonumber\\
		&\leq \int_{-30}^{t_1}\int_{M}4\left|\mathcal{T}\right|^2+\left|\nabla^2u\right|^2\,\mathrm{d}\nu_t \mathrm{d}t
	\end{align}
	and thus we get
	\begin{equation}\label{constructsplittingmap1}
		\int_{-30}^{-1/30}\int_{M} \left|\nabla^2u \right|^2\,\mathrm{d}\nu_t \mathrm{d}t\leq 4 \int_{-30}^{-1/30}\int_{M} \left|\TT\right|^2\,\mathrm{d}\nu_t \mathrm{d}t.
	\end{equation}
	
	Note that by Proposition \ref{propNashentropy} (iii),
	$$\WW_{x_0^*}(1/30)-\WW_{x_0^*}(30)=	2 \int_{-30}^{-1/30}\int_M\tau^{-1}\left|\TT\right|^2\,\mathrm{d}\nu_t \mathrm{d}t \ge \frac{1}{15} \int_{-30}^{-1/30}\int_M\left|\TT\right|^2\,\mathrm{d}\nu_t \mathrm{d}t.$$
	Combining this identity with \eqref{constructsplittingmap1} and using the definition of $u$ and $\TT$, we have
	\begin{align*}
\int_{-30}^{-1/30}\int_{M} \left|\tau \Ric+\nabla^2 u_0-\frac{g}{2}\right|^2\,\mathrm{d}\nu_t \mathrm{d}t \le  5 \int_{-30}^{-1/30}\int_{M} \left|\TT\right|^2\,\mathrm{d}\nu_t \mathrm{d}t \le 75 \left(\mathcal{W}_{x_0^*}( 1/30)-\mathcal{W}_{x_0^*}(30)\right).
	\end{align*}	
	
Moreover, by \eqref{splitting3}, for any $t_1\in [ -30,-1/30]$,
	\begin{align*}
		\int_{M}\left|\nabla u\right|^2\,\mathrm{d}\nu_{t_1}&\leq -4\int_{-30}^{t_1}\int_{M}\la\TT,\nabla^2 u\ra \,\mathrm{d}\nu_t \mathrm{d}t-2 \int_{-30}^{t_1}\int_{M} |\na^2 u|^2 \,\mathrm{d}\nu_t \mathrm{d}t \nonumber\\
		&\leq 2\int_{-30}^{-1/30}\int_M |\mathcal{T}|^2\,\mathrm{d}\nu_t \mathrm{d}t  \le 30 \left(\mathcal{W}_{x_0^*}( 1/30)-\mathcal{W}_{x_0^*}(30)\right).
	\end{align*}
	
Thus, $u_0$ satisfies properties (i) and (ii) above. Next, we focus on (iii).

Since $\square |\na u_0| \le 0$, it follows from the hypercontractivity (see \cite[Theorem 12.1]{bamler2020entropy}) that for any $t \in [-10, -1/10]$, we have
	\begin{align}\label{splixtra0}
\int_M |\na u_0|^4 \,\mathrm{d}\nu_t \le \lc \int_M |\na u_0|^2 \,\mathrm{d}\nu_{-30} \rc^2=  \lc \int_M |\na F|^2 \,\mathrm{d}\nu_{-30} \rc^2 \le C(n, Y),
	\end{align}
where we used Proposition \ref{integralbound} for the last inequality. Then, it follows from Proposition \ref{integralbound} that
	\begin{align*}
\int_M |\na u_0|^2 e^{\theta f} \,\mathrm{d}\nu_t \le \lc \int_M |\na u_0|^4 \,\mathrm{d}\nu_{t} \rc^{\frac 1 2} \lc \int_M e^{2 \theta f} \,\mathrm{d}\nu_{t} \rc^{\frac 1 2} \le C(n, Y)
	\end{align*}
for any $t \in [-10, -1/10]$. In particular, by Proposition \ref{integralbound} again, we have
	\begin{align} \label{splixtra1}
\sup_{t \in [-10, -1/10]} \int_M |\na u|^2 e^{\theta f} \,\mathrm{d}\nu_t \le C(n, Y).
	\end{align}

Moreover, we obtain by Proposition \ref{integralbound} and \eqref{splixtra0},
	\begin{align} \label{splixtra3}
\int_{-10}^{-1/10} \int_M |\na u|^4 \,\mathrm{d}\nu_t \mathrm{d}t \le C(n, Y)+C \int_{-10}^{-1/10} \int_M |\na F|^4 \,\mathrm{d}\nu_t \mathrm{d}t \le C(n, Y).
	\end{align}

For any small $\theta>0$, we compute
\begin{align*}
		&\frac{\dd}{\dd t}\int_{M} \left|\nabla u\right|^2 e^{\theta f}\,\mathrm{d}\nu_t=\int_{M} \square \lc \left|\nabla u\right|^2 e^{\theta f} \rc\,\mathrm{d}\nu_t \\
		=& \int_M -2|\na^2 u|^2 e^{\theta f} +4 \la \Div_f\TT,\nabla u\ra e^{\theta f}+|\na u|^2 \square(e^{\theta f})-2 \theta \la \na |\na u|^2, \na f \ra e^{\theta f} \,\mathrm{d}\nu_t \\
				=& \int_M -2|\na^2 u|^2 e^{\theta f} -4  \la \TT,\nabla^2 u\ra e^{\theta f}-4 \theta \la \TT, du \otimes df \ra e^{\theta f}+|\na u|^2 \square(e^{\theta f})-2 \theta \la \na |\na u|^2, \na f \ra e^{\theta f} \,\mathrm{d}\nu_t \\
				\le&  \int_M -|\na^2 u|^2 e^{\theta f}+C|\TT|^2 e^{\theta f}+C|\na u|^4+C\lc |\na f|^4+|\Delta f|^2+|\scal|^2+1\rc e^{2 \theta f} \,\mathrm{d}\nu_t,
	\end{align*}
where we used the fact that $\abs{\square(e^{\theta f})} \le C \lc |\Delta f|+|\na f|^2+|\scal|+1 \rc e^{\theta f}$. Using Proposition \ref{integralbound} and \eqref{splixtra3}, we conclude that
\begin{align} \label{splixtra4}
\int_{-10}^{-1/10} \int_M |\na^2 u|^2 e^{\theta f} \,\mathrm{d}\nu_t \mathrm{d}t \le C(n, Y)+ \int_M |\na u|^2 e^{\theta f} \,\mathrm{d}\nu_{-10} \le C(n, Y),
	\end{align}
where for the last inequality we used \eqref{splixtra1}.

Combining \eqref{splixtra4} with Proposition \ref{integralbound}, we obtain
\begin{align*}
\int_{-10}^{-1/10} \int_M |\na^2 u_0|^2 e^{\theta f} \,\mathrm{d}\nu_t \mathrm{d}t \le C(n, Y),
	\end{align*}
which completes the proof.
\end{proof}

\subsection{Almost self-similar points}

We begin by defining the following concept of almost selfsimilar points, which is similar to \cite[Definition 5.1]{bamler2020structure}.

%\begin{defn}[$(\delta,r)$-selfsimilar]\label{defnselfsimilarity1}
%	A point $x^* \in \XX_\III$ is called \textbf{$(\delta,r)$-selfsimilar} if $\t(x^*)-100r^2 \in \III$ and 
	%\begin{align*}
	%	\NN_{x_0^*}(99r^2)-\NN_{x_0^*}(100r^2) \leq \delta.
	%\end{align*}
%\end{defn}

\begin{defn}[$(\delta,r)$-selfsimilar]\label{defnselfsimilarity}\index{$(\delta,r)$-selfsimilar}
	A point $x^* \in \XX_\III$ is called \textbf{$(\delta,r)$-selfsimilar} if $\t(x^*)-\delta^{-1} r^2 \in \III^-$ and 
	\begin{align*}
		\WW_{x^*}(\delta r^2)-\WW_{x^*}(\delta^{-1} r^2) \leq \delta.
	\end{align*}
	Here, we implicitly assume $\delta \ll 1$.
\end{defn}

\begin{lem} \label{lem:imply1}
For any $\ep > 0$, if $x^* \in \XX_\III$ is $(\delta,r)$-selfsimilar and $\delta \le \delta(n, Y, \ep)$, then
	\begin{align*}
		\NN_{x^*}(\ep r^2)-\NN_{x^*}(\ep^{-1} r^2) \leq  \ep.
	\end{align*}
\end{lem}

\begin{proof}
Without loss of generality, we may assume that $\t(x^*) = 0$ and $r = 1$.

Suppose the statement is false. Then there exist a constant $\ep > 0$ and a sequence $\XX^i \in \mathcal M(n, Y, T_i)$ with $(i^{-2}, 1)$-selfsimilar points $x_i^* \in \XX^i_0$ such that
\begin{align} \label{eq:asump001}
\NN_{x_i^*}(\ep ) - \NN_{x_i^*}(\ep^{-1}) > \ep.
\end{align}

Passing to a subsequence if necessary, it follows from Theorems \ref{thm:intro1} and \ref{thm:intro3} that
\begin{align} \label{eq:conv001}
(M_i \times \III, d^*_i, x_i^*, \t_i) \xrightarrow[i \to \infty]{\quad \hat C^\infty \quad} (Z, d_Z, z, \t),
\end{align}
where $(Z, d_Z, z, \t)$ is a noncollapsed Ricci flow limit space over $(-\infty,0]$, whose regular part is a Ricci flow spacetime $(\RR, \t, \partial_\t, g^Z)$.

By Proposition \ref{propNashentropy} (iii) and the smooth convergence in Theorem \ref{thm:intro3}, it follows that
\begin{align*}
\Ric(g^Z) + \nabla^2 f_z = \frac{g^Z}{2|\t|}
\end{align*}
on $\iota_z(\RR^z_{(-\infty,0)})$, where $\RR^z$ is the regular part of a metric flow $\XX^z$ associated with $z$, and $\iota_z$ is the embedding in \cite[Theorem 1.4]{fang2025RFlimit}. 

It follows from the proof of \cite[Theorem 15.69]{bamler2020structure} that on $\iota_z(\RR^z_{(-\infty,0)})$, the heat kernel $K_Z$ satisfies
\begin{align*}
X_x K_Z(x;y)+X_y K_Z(x;y)=\frac{n}{2} K_Z(x, y)
\end{align*}
for any $x, y \in \iota_z(\RR^z_{(-\infty,0)})$, where $X:=|\t|(\partial_\t-\na f_z)$. 

As a result, applying the same argument as in \cite[Theorem 15.69]{bamler2020structure} (see also \cite[Lemma 7.28]{fang2025RFlimit}), we obtain a one-parameter family of diffeomorphisms $\boldsymbol{\psi}^s$ on $\iota_z(\RR^z)$ generated by the vector field $X$, satisfying
	\begin{align*}
	d_{g^Z_{e^{-s}t}} (\boldsymbol{\psi}^s(x), \boldsymbol{\psi}^s(y))=e^{-\frac s 2} d_{g^Z_t}(x, y)
	\end{align*}
 for any $x, y \in \iota_z(\RR^z_t)$ and any $s \in \R$. In other words, $\boldsymbol{\psi}^s$ corresponds to the rescaling by $e^{-\frac s 2}$. Since $f_z$ is constant along the flow lines of $\boldsymbol{\psi}^s$, it is clear that the Nash entropy $\NN_z(\tau)$ remains constant for $\tau >0$. Thus, $(Z, d_Z, z, \t)$ is a Ricci shrinker space.

Since the Nash entropy is continuous under the convergence in \eqref{eq:conv001} (see \cite[Lemma 7.2]{fang2025RFlimit}), we conclude that
\begin{align*}
\abs{\NN_{x_i^*}(\ep) - \NN_{x_i^*}(\ep^{-1})} \to 0,
\end{align*}
which contradicts \eqref{eq:asump001}. This completes the proof.
\end{proof}

Conversely, we have

\begin{lem} \label{lem:imply1a}
For any $\ep>0$ and $\delta \in (0, \ep^4/2)$, if $x^* \in \XX_\III$ satisfies $\t(x^*)-\delta^{-1} r^2 \in \III^-$ and
	\begin{align} \label{eq:improvenash}
		\NN_{x^*}(\delta r^2)-\NN_{x^*}(\delta^{-1} r^2) \leq  \delta,
	\end{align}
then $x^*$ is $(\ep, r)$-selfsimilar.
\end{lem}

\begin{proof}
Without loss of generality, we may assume that $\t(x^*) = 0$ and $r = 1$.

Since $s \mathcal N_{x^*}'(s)=\mathcal W_{x^*}(s)-\mathcal N_{x^*}(s) \le 0$, it follows from \eqref{eq:improvenash} and the mean value theorem that there exists a time $s_0 \in [\ep^{-1}, \ep^{-1}+1]$ such that
\begin{align*}
\mathcal N_{x^*}(s_0)-\mathcal W_{x^*}(s_0) \le (\ep^{-1}+1) \delta.
\end{align*}
In other words, by Proposition \ref{propNashentropy} (ii), we have
\begin{align*}
 \int_0^{s_0} \mathcal W_{x^*}(s)-\mathcal W_{x^*}(s_0) \,\mathrm{d}s \le (\ep^{-1}+1)^2 \delta \le 2\ep^{-2}\delta.
\end{align*}
This inequality is equivalent to
\begin{align*}
 \int_0^{s_0} \int_{-s_0}^{-s} \int_M \tau^{-1}|\TT|^2 \,\mathrm{d}\nu_{x^*;t} \mathrm{d}t \mathrm{d}s \le \ep^{-2}\delta,
\end{align*}
where $\TT:=\tau \Ric+\nabla^2(\tau f)-g/2$ and $\tau=-t$. Switching the order of $\mathrm{d}t$ and $\mathrm{d}s$, we have
\begin{align*}
\int_{-s_0}^{0} \int_M |\TT|^2 \,\mathrm{d}\nu_{x^*;t} \mathrm{d}t= \int_{-s_0}^{0} \int_{0}^{-t} \int_M \tau^{-1}|\TT|^2 \,\mathrm{d}\nu_{x^*;t} \mathrm{d}s \mathrm{d}t \le \ep^{-2}\delta.
\end{align*}
Thus, we have
\begin{align*}
\mathcal W_{x^*}(\ep)-\mathcal W_{x^*}(\ep^{-1})=2 \int_{-\ep^{-1}}^{-\ep } \int_M \tau^{-1}|\TT|^2 \,\mathrm{d}\nu_{x^*;t} \mathrm{d}t\le  2 \ep^{-1}  \int_{-\ep^{-1}}^{-\ep} \int_M |\TT|^2 \,\mathrm{d}\nu_{x^*;t} \mathrm{d}t \le 2\ep^{-3}\delta \le \ep.
\end{align*}

This completes the proof.
\end{proof}

Lemmas \ref{lem:imply1}, \ref{lem:imply1a} and \cite[Proposition 7.1]{bamler2020structure} indicate that the definition of $(\delta, r)$-selfsimilarity in \cite{bamler2020structure} is equivalent to the one used here.

Our next result follows from Lemma \ref{lem:imply1} and \cite[Propositions 7.1, 9.1]{bamler2020structure}. 

\begin{lem} \label{lem:nonexpand}
Given a constant $\beta \in (0, 1)$, let $x^* \in \XX_\III$ be $(\delta,r)$-selfsimilar with $\delta \le \delta(n, Y, \beta)$. Suppose $d^*(x^*,y^*) \le r$, then for any $s \in [\min\{\t(x^*), \t(y^*)\}-\beta^{-1} r^2, \min\{\t(x^*), \t(y^*)\}-\beta r^2]$,
\begin{align*}
d_{W_1}^{s}(\nu_{x^*;s},\nu_{y^*;s}) \le C(n, Y, \beta) r.
\end{align*}
\end{lem}

\begin{proof}
Without loss of generality, we assume $\t(x^*) \ge \t(y^*)$. By our assumption and Definition \ref{defnd*distance} (see \cite[(3.6)]{fang2025RFlimit}), we have
	\begin{align*}
d_{W_1}^{\t(x^*)-r^2}(\nu_{x^*;\t(x^*)-r^2},\nu_{y^*;\t(x^*)-r^2})\leq r.
	\end{align*}
Thus, the conclusion follows from Lemma \ref{lem:imply1} and \cite[Propositions 7.1, 9.1]{bamler2020structure}.
\end{proof}

\subsection{Static estimates in Ricci flow}

In this subsection, we prove the following static estimate, which gives a quantitative version of \cite[Proposition 10.1]{bamler2020structure} and will play a crucial role later. 

\begin{thm}[Static estimate]\label{staticestimate}
Suppose that $x_0^*=(x_0, t_0),\,x_1^*=(x_1,t_1) \in \XX_\III$ satisfy $d^*(x_0^*, x_1^*) \le r$ and $\beta:=|t_1-t_0| \le r^2/100$.	Set 
\begin{align*}
\Lambda:= \max\{\WW_{x_0^*}(r^2/30)-\WW_{x_0^*}(30 r^2), \WW_{x_1^*}(r^2/30)-\WW_{x_1^*}(30r^2)\}.
	\end{align*}
	There exists a constant $\bar \delta=\bar \delta(n, Y)>0$\index{$\bar \delta$} such that if either $x_0^*$ or $x_1^*$ is $(\bar \delta, r)$-selfsimilar, then the following estimate holds\emph{:}
	\begin{align*}
		\beta^2\int_{t_0-20r^2}^{t_0-r^2/20} \int_M |\Ric|^2 \,\mathrm{d}\nu_{x_0^*;t} \mathrm{d}t \le C(n,Y) \Lambda^{\frac 1 2}r^2.
	\end{align*}
\end{thm}
\begin{proof}
	Without loss of generality, we assume $r=1$ and $t_0=0$. Moreover, we set $\nu_t=\nu_{x_0^*;t}$, $f_i=f_{x_i^*}$, $\tau_i=t_i-t$, $\mathcal{T}_i:= \tau_i \Ric+\nabla^2 (\tau_i f_i)-g/2$ and $w_i=\tau_i(2\Delta f_i-|\na f_i|^2+\scal)+f_i-n$ for $i\in\{0,1\}$. Note that $|\tau_0-\tau_1|=\beta$.

By Proposition \ref{propNashentropy} (iii), we have
	\begin{equation}\label{perpendicularofRichessianf0}
		\int_{-20}^{-1/20}\int_M|\TT_0|^2 \, \mathrm{d}\nu_{t} \mathrm{d}t \leq 10 \Lambda.
	\end{equation}
	
By our assumption, $x_0^*$ or $x_1^*$ is $(\bar \delta, 1)$-selfsimilar, it follows from Lemma \ref{lem:nonexpand} that
\begin{align} \label{eq:extranonex001}
d_{W_1}^{s}(\nu_{x_0^*;s},\nu_{x_1^*;s}) \le C(n, Y)
\end{align}	
for any $s \in [-20, -1/20]$, provided $\bar \delta \le \bar \delta(n, Y)$.
	
	By Proposition \ref{integralbound}, \eqref{eq:extranonex001} and \cite[Proposition A.2]{fang2025RFlimit}, we obtain for a small constant $\theta=\theta(n)>0$,
	\begin{align}\label{perpendicularofRichessianf1}
		&\int_{-20}^{-1/20}\int_M|\TT_1|^2\, \mathrm{d}\nu_{t} \mathrm{d}t\leq C(n,Y)\int_{-20}^{-1/20}\int_M|\TT_1|^2e^{\theta f_1} \,\mathrm{d}\nu_{x^*_1;t} \mathrm{d}t \nonumber\\
		&\leq C(n,Y) \lc\int_{-20}^{-1/20}\int_M |\TT_1|^2\,\mathrm{d}\nu_{x^*_1;t} \mathrm{d}t\rc^{1/2} \lc \int_{-20}^{-1/20}\int_M |\TT_1|^2e^{2 \theta f_1}\,\mathrm{d}\nu_{x^*_1;t} \mathrm{d}t\rc^{1/2}\leq C(n,Y)\Lambda^{\frac{1}{2}}.
	\end{align}
	Here, we used the fact that $[-20, -1/20] \subset [t_1-30, t_1-1/30]$ by our assumption.
	
Now we calculate:
	\begin{align}
		&\int_M\la \Ric,\nabla ^2 f_1\ra \,\mathrm{d}\nu_t 	=\int_M\Ric(\nabla f_0,\nabla f_1)-\la\nabla f_1, \Div(\Ric) \ra \, \mathrm{d}\nu_t	=\int_M \Ric (\nabla f_0,\nabla f_1)-\frac{\la\nabla f_1,\nabla \scal\ra}{2} \,\mathrm{d}\nu_t\nonumber\\
		=&\int_M \lc\Ric+\nabla^2 f_0-\frac{g}{2\tau_0}\rc (\nabla f_0,\nabla f_1)-(\nabla ^2 f_0)(\nabla f_1,\nabla f_0)+\frac{\la\nabla f_1,\nabla f_0\ra}{2\tau_0}-\frac{\la\nabla f_1,\nabla \scal\ra}{2}\, \mathrm{d}\nu_t\nonumber\\
		=& \int_M \lc\Ric+\nabla^2 f_0-\frac{g}{2\tau_0}\rc(\nabla f_0,\nabla f_1) +\frac{1}{2\tau_0} \la \nabla \lc-\tau_0(|\nabla f_0|^2+\scal)+f_0\rc, \nabla f_1 \ra \, \mathrm{d}\nu_t\nonumber\\
		=&\int_M \tau_0^{-1} \TT_0(\nabla f_0,\nabla f_1)- \la \nabla \lc \Delta f_0+\scal-\frac{n}{2\tau_0}\rc, \nabla f_1 \ra+\frac{1}{2\tau_0} \la \na w_0, \na f_1 \ra \, \mathrm{d}\nu_t.\label{staticesti1}
	\end{align}

Similarly to \eqref{perpendicularofRichessianf1}, we have
	\begin{align}\label{staticesti2}
		\int_{-20}^{-1/20}\int_M |\nabla f_1|^4+|\nabla^2 f_1|^2\,\mathrm{d}\nu_t \mathrm{d}t\leq\int_{-20}^{-1/20}\int_M \lc|\nabla f_1|^4+|\nabla^2 f_1|^2\rc e^{\theta f_1}\,\mathrm{d}\nu_{x^*_1;t} \mathrm{d}t \leq C(n,Y).
	\end{align}
	Combining \eqref{perpendicularofRichessianf0} and \eqref{staticesti2}, we get
	\begin{align}\label{staticesti3}
		\left|\int_{-20}^{-1/20}\int_M \TT_0(\nabla f_0,\nabla f_1)\,\mathrm{d}\nu_t \mathrm{d}t\right|\leq C(n,Y)\Lambda^{1/2}.
	\end{align}
	
Similarly, using the integration by parts, we have
	\begin{align}\label{staticesti4}
		&\left|\int_{-20}^{-1/20}\int_M \big\la\nabla \lc \Delta f_0+\scal-\frac{n}{2\tau_0}\rc,\nabla f_1\big\ra \,\mathrm{d}\nu_t \mathrm{d}t\right| \notag\\
		\le & \left|\int_{-20}^{-1/20}\int_M \lc  \Delta f_0+\scal-\frac{n}{2\tau_0} \rc \lc \Delta f_1-\la \na f_1, \na f_0 \ra \rc \,\mathrm{d}\nu_t \mathrm{d}t\right|	\leq C(n, Y)\Lambda^{1/2}.
	\end{align}

By the weighted Bianchi identity (see Lemma \ref{weightedBianchilem}), we also obtain
	\begin{align}\label{staticesti5}
		\left|\int_{-20}^{-1/20}\int_M \la\nabla w_0,\nabla f_1\ra \,\mathrm{d}\nu_t \mathrm{d}t\right|=&2\left|\int_{-20}^{-1/20}\int_M \la \Div_{f_0}\TT_0,\nabla f_1\ra \,\mathrm{d}\nu_t \mathrm{d}t\right| \notag \\
		=&2\left|\int_{-20}^{-1/20}\int_M \la\mathcal{T}_0,\nabla^2 f_1\ra \,\mathrm{d}\nu_t \mathrm{d}t\right|\leq C(n,Y)\Lambda^{1/2}.
	\end{align}
	Combining \eqref{staticesti3}, \eqref{staticesti4}, \eqref{staticesti5} with \eqref{staticesti1}, we conclude that
	\begin{equation}\label{perpendicularofRicandhessian1}
		\left|\int_{-20}^{-1/20}\int_M \la \Ric,\nabla^2f_1\ra \,\mathrm{d}\nu_{t} \mathrm{d}t\right|\leq C(n,Y)\Lambda^{1/2}.
	\end{equation}
	A similar argument also gives
	\begin{equation}\label{perpendicularofRicandhessian2}
		\left|\int_{-20}^{-1/20}\int_M \la \Ric,\nabla^2f_0\ra \, \mathrm{d}\nu_{t} \mathrm{d}t\right|\leq C(n,Y) \Lambda^{1/2}.
	\end{equation}
	
	Now, using 
		\begin{align*}
		\mathcal{T}_0-\mathcal{T}_1-\nabla^2(\tau_0f_0-\tau_1f_1)=(\tau_0-\tau_1) \Ric, 
	\end{align*}
	we obtain
	\begin{align*}
		|\mathcal{T}_0-\mathcal{T}_1|^2= \beta^2|\Ric|^2+|\nabla^2(\tau_0f_0-\tau_1f_1)|^2+2(\tau_0-\tau_1)\la \Ric,\nabla^2 (\tau_0f_0-\tau_1f_1)\ra.
	\end{align*}
	Then, by \eqref{perpendicularofRichessianf0}, \eqref{perpendicularofRichessianf1}, \eqref{perpendicularofRicandhessian1} and \eqref{perpendicularofRicandhessian2}, we finally get
	\begin{align*}
		&\beta^2 \int_{-20}^{-1/20}\int_M |\Ric|^2 \,\mathrm{d}\nu_t \mathrm{d}t \\
		\leq &\int_{-20}^{-1/20}\int_M |\mathcal{T}_0-\mathcal{T}_1|^2-2(\tau_0-\tau_1)\la \Ric, \nabla ^2 (\tau_0f_0-\tau_1f_1)\ra\, \mathrm{d}\nu_t \mathrm{d}t\leq C(n,Y)\Lambda^{1/2}.
	\end{align*}
	This completes the proof.
\end{proof}

\subsection{Construction of sharp splitting maps}

Fix a base point $x_0^*=(x_0, t_0) \in \XX_\III$. For simplicity, we set $\tau=t_0-t$ and
	\begin{align*}
\mathrm{d}\nu_t:=\mathrm{d}\nu_{x_0^*;t}=(4\pi \tau)^{-\frac n 2}e^{-f}\,\mathrm{d}V_{g(t)}.
	\end{align*}

 The following definition of independent points can be viewed as a parabolic counterpart considered in prior work (see \cite[Definition 4.4]{Cheeger2018RectifiabilityOS}). Here, $\bar \delta=\bar \delta(n, Y)$ is the constant from Theorem \ref{staticestimate}.
  
\begin{defn}[$(k,\alpha,\delta,r)$-independent points]\label{defnindependentpoints}\index{$(k,\alpha,\delta,r)$-independent points}
Given constants $\alpha\in (0, 1)$, $\delta>0$, $r>0$, $k \in \{1, \ldots, n\}$ and a $(\bar \delta, r)$-selfsimilar point $x_0^*$, a set of spacetime points $\{x_i^*=(x_i,t_i)\}_{1\leq i\leq k}$ is called \textbf{$(k,\alpha,\delta,r)$-independent} at $x_0^*$ if the following conditions hold.
	\begin{enumerate}[label=\textnormal{(\roman{*})}]	
		\item $\displaystyle \WW_{x_i^*}(r^2/30)-\WW_{x_i^*}(30 r^2) \le \delta$ for all $i \in \{0,\ldots,k\}$.
		
		\item $\displaystyle d^*(x_i^*,x_0^*) \leq r$ for all $i \in \{1,\ldots,k\}$.

		\item $\displaystyle |t_0-t_i|\leq \frac{r^2}{100}$ for all $i \in \{1,\ldots,k\}$.
		\item Set $f_i=f_{x_i^*}$, $\tau_i=t_i-t$ and $F_i=\tau_if_i$. For $h_i=r^{-1}(F_i-F_0)$ and the symmetric matrix
		\begin{align*}
		A=(a_{ij}), \quad \text{where} \quad	a_{ij}:=\aint_{t_0-10r^2}^{t_0-r^2/10} \int_M \la\nabla h_i ,\nabla h_j\ra \,\mathrm{d}\nu_t \mathrm{d}t,
 		\end{align*}
the first eigenvalue $\lambda_1(A)$ of $A$ satisfies
 				\begin{align*}
			\lambda_1(A) \geq \alpha^2.
		\end{align*}
	\end{enumerate}
\end{defn}

For the rest of the section, we always assume $\delta \ll \bar \delta$. Next, we prove the existence of almost splitting maps constructed from $(k,\alpha,\delta,r)$-independent sets.

\begin{prop}\label{constructionsplitting2}
	Assume $\{x_i^*=(x_i,t_i)\}_{i=1,\ldots,k} $ is $(k,\alpha,\delta,r)$-independent at $x_0^*$. If $\delta\leq\delta(n,Y,\alpha)$, then we can find a $(k,\ep,r)$-splitting map $\vec u=(u_1,\ldots, u_k):M\times [t_0-10r^2,t_0]\to \R^k$ at $x_0^*$, where
	\begin{align*}
\ep=C(n,Y, \alpha)\sum_{i=0}^k\lc\mathcal{W}_{x_i^*}(r^2/30)-\mathcal{W}_{x_i^*}(30r^2)\rc^{\frac{1}{2}}.
	\end{align*}
\end{prop}
\begin{proof}
	For each $i \in \{0,\ldots,k\}$, by Lemma \ref{constructionofsplittingfunction1}, we can find $v_i$ so that $\square v_i=-n/2$ and
	\begin{align*}
		\int^{t_0-r^2/10}_{t_0-10r^2}\int_M \left|\tau_i\Ric+\nabla^2 v_i-\frac{g}{2}\right|^2\,\mathrm{d}\nu_{x_i^*;t}\mathrm{d}t\leq 75 r^2\left(\mathcal{W}_{x_i^*}(r^2/30)-\mathcal{W}_{x_i^*}(30r^2)\right),
	\end{align*} 
where $\tau_i=t_i-t$. By Lemma \ref{lem:nonexpand}, we can apply \cite[Proposition A.2]{fang2025RFlimit} to obtain for $\theta=\theta(n)$ and $i\in \{1,\ldots,k\}$,
	\begin{align}\label{sharpsplitting11}
		&\int^{t_0-r^2/10}_{t_0-10r^2}\int_M \left|\tau_i\Ric+\nabla^2 v_i-\frac{g}{2}\right|^2\,\mathrm{d}\nu_t \mathrm{d}t\leq \int^{t_0-r^2/10}_{t_0-10r^2}\int_M \left|\tau_i\Ric+\nabla^2 v_i-\frac{g}{2}\right|^2e^{\theta f_i}\,\mathrm{d}\nu_{x_i^*;t}\mathrm{d}t\nonumber\\
		\leq& \lc\int^{t_0-r^2/10}_{t_0-10r^2}\int_M \left|\tau_i\Ric+\nabla^2 v_i-\frac{g}{2}\right|^2\,\mathrm{d}\nu_{x_i^*;t}\mathrm{d}t\rc^{1/2}\lc \int^{t_0-r^2/10}_{t_0-10r^2}\int_M \left|\tau_i\Ric+\nabla^2 v_i-\frac{g}{2}\right|^2e^{2 \theta f_i}\,\mathrm{d}\nu_{x_i^*;t} \mathrm{d}t\rc^{1/2}\nonumber\\
		\leq& C(n,Y)r^2\left(\mathcal{W}_{x_i^*}(r^2/30)-\mathcal{W}_{x_i^*}(30r^2)\right)^{1/2},
	\end{align}
where for the last inequality, we have used Proposition \ref{integralbound} and Lemma \ref{constructionofsplittingfunction1} (iii).

	Therefore, it follows that
	\begin{align*}
		\sum_{i=1}^k\int^{t_0-r^2/10}_{t_0-10r^2}\int_M \left|\tau_i\Ric+\nabla^2 v_i-\frac{g}{2}\right|^2\,\mathrm{d}\nu_t \mathrm{d}t\leq C(n,Y)r^2\sum_{i=1}^k\left(\mathcal{W}_{x_i^*}(r^2/30)-\mathcal{W}_{x_i^*}(30r^2)\right)^{\frac{1}{2}}.
	\end{align*}
	
For $i \in \{1,\ldots,k\}$, set $u_i:=r^{-1}(v_i-v_0)$. Thus, by taking difference of \eqref{sharpsplitting11} and applying Theorem \ref{staticestimate} to the integral of the $\Ric$-term, we obtain
	\begin{align*}
		\sum_{i=1}^k\int^{t_0-r^2/10}_{t_0-10r^2}\int_M \left|\nabla^2 u_i\right|^2\,\mathrm{d}\nu_t \mathrm{d}t\leq C(n,Y)\sum_{i=0}^k \left(\mathcal{W}_{x_i^*}(r^2/30)-\mathcal{W}_{x_i^*}(30r^2)\right)^{\frac{1}{2}}.
	\end{align*}
It follows from Lemma \ref{constructionofsplittingfunction1} (ii) that for any $t\in [t_0-10r^2,t_0-r^2/10]$,
		\begin{align}\label{sharpsplittingex:001}
		\int_{M}|\na (F_0-v_0)|^2\,\mathrm{d}\nu_{t}\leq 30 r^2\left(\mathcal{W}_{x_0^*}(r^2/30)-\mathcal{W}_{x_0^*}(30r^2)\right).
	\end{align}
	By Lemma \ref{lem:nonexpand} and \cite[Proposition A.2]{fang2025RFlimit} again, we have for $\theta=\theta(n)>0$, $t\in [t_0-10r^2,t_0-r^2/10]$ and $i\in \{1,\ldots,k\}$,
			\begin{align}\label{sharpsplittingex:002}
		&\int_{M}|\na (F_i-v_i)|^2\,\mathrm{d}\nu_{t}\leq \int_{M}|\na (F_i-v_i)|^2 e^{\theta f_i}\,\mathrm{d}\nu_{x_i^*;t} \notag \\
		\le & \lc \int_{M}|\na (F_i-v_i)|^2 \,\mathrm{d}\nu_{x_i^*;t} \rc^{\frac 1 2} \lc \int_{M}|\na (F_i-v_i)|^2 e^{2\theta f_i}\,\mathrm{d}\nu_{x_i^*;t} \rc^{\frac 1 2} \notag\\
		\le & C(n, Y) r^2 \lc \mathcal{W}_{x_i^*}(r^2/30)-\mathcal{W}_{x_i^*}(30r^2) \rc^{\frac 1 2},
	\end{align}
where the last inequality holds by Lemma \ref{constructionofsplittingfunction1} (ii)(iii).	Combining \eqref{sharpsplittingex:001} and \eqref{sharpsplittingex:002}, we obtain for any $t\in [t_0-10r^2,t_0-r^2/10]$ and $i\in \{1,\ldots,k\}$,
	\begin{align}\label{eq:extraint}
\int_{M}|\na (u_i-h_i)|^2\,\mathrm{d}\nu_{t} \le C(n, Y) \lc \lc \mathcal{W}_{x_i^*}(r^2/30)-\mathcal{W}_{x_i^*}(30r^2) \rc^{\frac 1 2}+ \mathcal{W}_{x_0^*}(r^2/30)-\mathcal{W}_{x_0^*}(30r^2)\rc.
	\end{align}

	Now we can use Definition \ref{defnindependentpoints} to modify the function $\{u_i\}$ to make them satisfy the orthogonality condition. In fact, by Definition \ref{defnindependentpoints}, we see that $A=(a_{ij})$ satisfies $\lam_1(A)\geq \alpha^2 $. Set $A'=(a_{ij}')$, where 
	\begin{align*}
		a_{ij}'=\aint_{t_0-10r^2}^{t_0-r^2/10}\int_M\la \na u_i,\na u_j\ra \,\mathrm{d}\nu_t \mathrm{d}t.
	\end{align*}
	By \eqref{eq:extraint}, if $\delta\leq\delta(n,Y,\alpha)$, the first eigenvalue $\lambda_1(A')\geq \alpha^2/2$.
	Thus, we can find a matrix $B=(b_{ij})$ with $\|B\|\leq C(\alpha)$ so that for $u'_i:=b_{ij}u_j$, the following holds:
	\begin{align*}
		\int_{t_0-10r^2}^{t_0 -r^2/10}\int_M \la\na u'_i,\na u'_j\ra -\delta_{ij} \,\mathrm{d}\nu_t \mathrm{d}t=0.
	\end{align*}
	
	After adding a constant vector to $\vec u'=(u'_1,\ldots,u'_k)$ so that $\vec u'(x_0^*)=0$, we conclude that $\vec u'$ satisfies all the requirements of Definition \ref{defnsplittingmap}. This finishes the proof.
\end{proof}

\begin{defn}[Entropy pinching]\label{defnentropypinching}
	For $k\in \{1,2, \ldots,n\}$, $\alpha \in (0, 1)$, $\delta>0$ and $r>0$, the \textbf{$(k,\alpha,\delta,r)$-entropy pinching at $x_0^*$} is defined as
	\begin{equation*}\index{$\EE_r^{k,\alpha,\delta}$}
		\EE_r^{k,\alpha,\delta}(x_0^*):= \inf_{\{x_i^*\}_{1 \le i \le k} \,\mathrm{are}\, (k,\alpha,\delta,r)-\mathrm{independent}\ \mathrm{at}\ x_0^*}\sum_{i=0}^k \left(\WW_{x_i^*}(r^2/30)-\WW_{x_i^*}(30r^2)\right)^{\frac{1}{2}}.
	\end{equation*}
	Here, we implicitly assume that $x_0^*$ is $(\bar\delta, r)$-selfsimilar and there is at least one $ (k,\alpha,\delta,r)$-independent set at $x_0^*$.
\end{defn}

Note that our definition of entropy pinching is different from \cite[Definition 4.23]{Cheeger2018RectifiabilityOS}. As a corollary of Proposition \ref{constructionsplitting2}, we have:
\begin{cor}\label{corsharpsplitting}
Suppose $x_0^*$ is $(\bar\delta, r)$-selfsimilar and there exists a $ (k,\alpha,\delta,r)$-independent set at $x_0^*$. If $\delta \le \delta(n, Y, \alpha)$, then there exists a $(k,\ep,r)$-splitting map $\vec u=(u_1,\ldots, u_k):M\times [t_0-10r^2,t_0]\to \R^k$ at $x_0^*$, where
	\begin{align*}
\ep=C(n,Y, \alpha)\EE_r^{k,\alpha,\delta}(x_0^*).
	\end{align*}
\end{cor}

\subsection{Hessian decay of heat flows}

We establish the following proposition, demonstrating that the almost splitting property at an almost self-similar point can be propagated to all scales.

\begin{prop}\label{equivalencesplitsymetric1}
For any $\ep>0$, we can find $\delta\leq \delta(n,Y,\ep)$ such that the following hold:
	\begin{enumerate}[label=\textnormal{(\roman{*})}]
		\item If $z \in \XX_{\III^-}$ is $(\delta,r)$-selfsimilar, and there exists a $(k, \delta, r)$-splitting map at $z$, then $z$ is $(k, \ep, s)$-splitting (cf. Definition \ref{defnksplitting}) for any $s \in [\ep r, \ep^{-1} r]$.

	\item If $z \in \XX_{\III^-}$ is $(\delta,r)$-selfsimilar and $(k, \delta, r)$-splitting, then $z$ is $(k, \ep, s)$-splitting for any $s \in [\ep r, \ep^{-1} r]$.
	\end{enumerate}
\end{prop}

\begin{proof}
Without loss of generality, we assume $r=1$. We will only prove (i), since the proof of (ii) is similar.

Suppose the conclusion is false. Then we can find $\XX^l \in \mathcal M(n, Y, T_l)$ with $\III_l=[-0.98T_l, 0]$ and $z_l^* \in \XX^l$ such that $z_l^*$ is $(l^{-2}, 1)$-selfsimilar. In addition, one can find $(k,l^{-2},1)$-splitting map $\vec u^l=(u^l_1,\ldots,u^l_k)$ at $z_l^*$. However, $z_l^*$ is not $(k,\ep, s_l)$-splitting for some $s_l \in [\ep, \ep^{-1}]$. 

By taking a subsequence, we assume 
\begin{align*}
		(\XX^l, d^*_l, z_l^*,\t_l-\t_l(z_l^*)) \xrightarrow[l \to \infty]{\quad \hat C^\infty \quad} (Z, d_Z, z,\t).
	\end{align*}

By the smooth convergence, we assume that $u_i^l\to y_i$ in $C_{\loc}^\infty(\RR_{(-10,0)})$ such that on $\RR_{(-10, -1/10)}$,
	\begin{align} \label{eq:entroto1}
	\la \na y_i, \na y_j \ra=\delta_{ij},\quad \na^2 y_i=0,\quad \partial_{\t} y_i=0.
	\end{align}

Arguing as in the proof of Lemma \ref{lem:imply1}, we conclude that $Z$ is a Ricci shrinker space. Furthermore, it follows from \cite[Proposition 8.2, Remark 8.3]{fang2025RFlimit} that one can extend functions $\{y_1, \ldots, y_k\}$ on $\RR_{(-\infty, 0)}$ so that \eqref{eq:entroto1} still holds. Thus, by \cite[Proposition 8.4]{fang2025RFlimit}, $\RR_{(-\infty, 0]}$ admits an $\R^k$-splitting. However, this implies that for sufficiently large $l$, $z_l^*$ is $(k, \ep, s_l)$-splitting. Consequently, we obtain a contradiction.
\end{proof}

Next, we prove a Hessian decay of heat flows. At the first reading, it may be helpful for the readers to assume that the Ricci flow limit space $Z$ in the proof is $\bar{\mathcal C}^k$ or $\bar{\mathcal C}^m_k(\Gamma)$ (see Subsection \ref{subsec:quo}); this specialization already suffices for the main results of the paper.

\begin{thm}[Hessian decay]\label{hessiandecay}
Let $\eta>0$ be a fixed constant. Suppose that $x_0^*=(x_0,t_0)\in \XX_\III$ is $(\delta,r)$-selfsimilar and is not $(k+1,\eta, r)$-splitting (cf. Definition \ref{defnksplitting}).
	
Let $\vec u=(u_1,\ldots,u_k):M\times [t_0-10r^2,t_0]\to\R^k$ be a $(k,\delta,r)$-splitting map, and $h: M\times [t_0-10r^2,t_0]\to\R$ be a smooth function satisfying $\square h=0$. Define
	\begin{align*}
		v:= h-\sum_{i=1}^k b_iu_i, \quad \text{where} \quad b_i:= \aint_{t_0-8r^2}^{t_0-8r^2/10}\int_M\la \nabla h,\nabla u_i\ra \,\mathrm{d}\nu_t \mathrm{d}t.
	\end{align*}
	Then there exists a constant $\theta=\theta(n,Y,\eta)\in (0,1)$ such that the following holds\emph{:} if $\delta\leq\delta(n,Y,\eta)$, then
	\begin{equation}\label{hessiandecay11}
		\int_{t_0-2r^2}^{t_0-2r^2/10}\int_M|\nabla^2 v|^2\,\mathrm{d}\nu_t \mathrm{d}t\leq \theta \int_{t_0-8r^2}^{t_0-8r^2/10}\int_M|\nabla^2 v|^2\,\mathrm{d}\nu_t \mathrm{d}t.
	\end{equation}
\end{thm}
\begin{proof}
Without loss of generality, we assume $r=1$ and $t_0=0$. 

In addition, we assume the following normalization holds:
	\begin{align}\label{normalization2}
		\int_M|\nabla v|^2\,\mathrm{d}\nu_{-8}=1,\quad \int_Mv\,\mathrm{d}\nu_{-8}=0.
	\end{align}
	By Theorem \ref{poincareinequ}, we have
	\begin{align}\label{normalization2aa}
		\int_M v^2\,\mathrm{d}\nu_{-8}\leq 16.
	\end{align}
	Since $\diff{}{t} \int_M v^2\,\mathrm{d}\nu_t=-2\int_M |\na v|^2\,\mathrm{d}\nu_t\leq 0$, we obtain for any $t \in [-8, 0]$,
		\begin{align*}
		\int_M v^2\,\mathrm{d}\nu_{t}\leq 16.
	\end{align*}
	Similarly, by $\diff{}{t} \int_M |\na v|^2 \,\mathrm{d}\nu_t=-2\int_M |\na^2 v|^2 \,\mathrm{d}\nu_t$, we have
		\begin{align*} 
2\int_{-8}^0 \int_M |\na^2 v|^2 \,\mathrm{d}\nu_t \mathrm{d}t +\sup_{t \in [-8, 0)} \int_M |\na v|^2 \,\mathrm{d}\nu_t \le 2.
	\end{align*}	

If $\displaystyle \int_M |\na v|^2 \,\mathrm{d}\nu_{-2} \le 1/10$, then we have
		\begin{align*}
2\int_{-2}^{-2/10}\int_{M} |\na^2 v|^2\,\mathrm{d}\nu_t \mathrm{d}t\leq \int_{M} |\na v|^2 \,\mathrm{d}\nu_{-2}\leq \frac{1}{10}.
	\end{align*}

On the other hand,
	\begin{align*}
		2\int_{-8}^{-8/10}\int_{M}\left|\nabla^2 v\right|^2\,\mathrm{d}\nu_t \mathrm{d}t&=\int_{M} \left|\nabla v\right|^2\,\mathrm{d}\nu_{-8}-\int_{M} \left|\nabla v\right|^2\,\mathrm{d}\nu_{-8/10}\nonumber\\
		&\geq \int_{M} \left|\nabla v\right|^2\,\mathrm{d}\nu_{-8}-\int_{M} \left|\nabla v\right|^2\,\mathrm{d}\nu_{-2}\geq 1-\frac{1}{10}=\frac{9}{10}.
	\end{align*}
Thus, the conclusion \eqref{hessiandecay11} holds for $\theta=1/9$. In the following proof, we may also assume
		\begin{align} \label{eq:hessianeq002}
		\frac{1}{10}\leq \int_{M}\left|\nabla v\right|^2\,\mathrm{d}\nu_{-2}\leq 1.
	\end{align}		
	
We prove the theorem by contradiction. Suppose the conclusion fails. Then there exists a sequence of smooth, closed Ricci flows $\XX^l=\{M^n_l,(g_l(t))_{t \in \III_l^{++}}\} \in \MM(n, Y, T_l)$, where $\III_l^{++}=[-T_l, 0]$, and $x_{0, l}^* \in \XX^l_0$ is $(l^{-2},1)$-selfsimilar, but not $(k+1,\eta,1)$-splitting.

Set $\mathrm{d}\nu^l_t=\mathrm{d}\nu_{x_{0,l}^*;t}=(4\pi\tau)^{-n/2}e^{-f_l}\,\mathrm{d}V_{g_l(t)}$, where $\tau=-t$. Assume the corresponding functions $\vec u^l=(u^l_1,\ldots,u^l_k), h^l, v^l$, as well as the constants $b_i^l$, are defined as in the statement of Theorem \ref{hessiandecay}, and satisfy the normalization conditions \eqref{normalization2} and \eqref{eq:hessianeq002}. Furthermore, suppose $\vec u^l$ is a $(k, l^{-2}, 1)$-splitting map.

However, inequality \eqref{hessiandecay11} is violated; that is,
	\begin{equation}\label{contradictioninequality1}
		\int_{-2}^{-\frac{2}{10}}\int_{M_l} |\nabla^2 v^l|^2\,\mathrm{d}\nu_t^l \mathrm{d}t> (1-l^{-2})\int_{-8}^{-\frac{8}{10}}\int_{M_l} |\nabla^2 v^l|^2\,\mathrm{d}\nu_t^l \mathrm{d}t.
	\end{equation}
	
Passing to a subsequence if necessary, it follows from Theorems \ref{thm:intro1} and \ref{thm:intro3} that
\begin{align*}
(M_l \times \III_l, d^*_l, x_{0, l}^*, \t_l) \xrightarrow[l \to \infty]{\quad \hat C^\infty \quad} (Z, d_Z, z, \t),
\end{align*}
where $(Z, d_Z, z, \t)$, by our assumption on $x_{0,l}^*$, is a Ricci shrinker space, whose regular part is a Ricci flow spacetime $(\RR, \t, \partial_\t, g^Z)$. In particular, we have on $\RR_{(-\infty, 0)}$,
\begin{align*}
\Ric(g^Z) + \nabla^2 f_z = \frac{g^Z}{2 \tau},
\end{align*}
where $\tau=-\t$. Moreover, $\RR_t$ is connected for any $t<0$. Let $\phi_l$ denote the diffeomorphisms given in Theorem \ref{thm:intro3}.

By the smooth convergence in Theorem \ref{thm:intro3}, we assume $v^l \to v^{\infty}$ and $u_i^l \to y_i$ smoothly on $\RR_{(-8, 0)}$. On the other hand, since $\vec u^l$ is a $(k, l^{-2}, 1)$-splitting map, we conclude from Proposition \ref{equivalencesplitsymetric1} that $\RR_{(-\infty, 0)}$ splits off an $\R^k$ isometrically, on which the coordinates are given by $\{y_1, \ldots, y_k\}$.
	
Next, we consider the family of cutoff functions $\{\eta_{r, A}\}$ from \cite[Proposition 8.23]{fang2025RFlimit}. Then, it is clear from \cite[Proposition 8.23 (3)]{fang2025RFlimit} that on $\RR_{[-100, -1/100]}$,
	\begin{align}\label{esticutoff}
		r^2\left(|\partial_\t \eta_{r,A}|+|\Delta \eta_{r,A}|\right)+r|\nabla \eta_{r,A}|\leq C(n).
	\end{align}
Also, it follows from \cite[Proposition 8.23 (4)]{fang2025RFlimit} that for any $\ep>0$,
	\begin{align} \label{eq:hessianeq006a}
\iint_{\RR_{[-100, -1/100]} \bigcap \{0 <\eta_{r, A}<1\}} 1 \,\mathrm{d}V_{g^Z_t} \mathrm{d}t \le C(n, Y, A, \ep) r^{4-\ep}.
\end{align}
	
\begin{claim}\label{takelimit1}
For any $s_1 \in [1/100, 2]$, we have
		\begin{equation}\label{smootheq}
			\lim_{l\to\infty}\int_{M_l}|\nabla v^l|^2\,\mathrm{d}\nu^l_{-s_1}=\int_{\RR_{-s_1}}|\nabla v^\infty|^2\,\mathrm{d}\nu_{z;-s_1}.
		\end{equation}
	\end{claim}

%By \eqref{eq:hessianeq001} and taking the limit, we have
	%\begin{align} \label{eq:hessianeq006}
%\int_{-8}^0\int_{\RR_t}|\nabla^2 v^\infty|^2\,\mathrm{d}\nu_{z;t} \mathrm{d}t \le \frac{1}{2}.
%\end{align}
It follows from the hypercontractivity (see \cite[Theorem 12.1]{bamler2020entropy}) and \eqref{normalization2} that for any $s \in (0, 2]$, 
	\begin{align}\label{estihess1}
\int_{M_l}|\nabla v^l|^{5}\,\mathrm{d}\nu^l_{-s} \le 1.
	\end{align}
Taking the limit, we also have
	\begin{align*}
\int_{\RR_{-s}}|\nabla v^\infty|^{5}\,\mathrm{d}\nu_{z;-s} \le 1.
	\end{align*}
	
For fixed $s_1\in [1/100, 2]$, set $B^l_{r,A}:=\phi_l\lc \supp\big(\eta_{r,A}\big)\cap \RR_{-s_1}\rc$, which is well defined for fixed $r$ and $A$, provided that $l$ is sufficiently large.

By the smooth convergence on $\RR$, it follows that
	\begin{align}\label{smootheq1}
		\lim_{l\to\infty}\int_{B^l_{r,A}}\left|\na v^l\right|^2\eta_{r,A}\circ \phi_l^{-1}\,\mathrm{d}\nu_{-s_1}^l=\int_{\RR_{-s_1}}\left|\na v^\infty\right|^2\eta_{r,A}\, \mathrm{d}\nu_{z;-s_1}.
	\end{align}
	
	On the one hand, it follows from Theorem \ref{thm:intro3} that once $r$ and $A$ are fixed, 
		\begin{align}\label{eq:compare1}
B^*(x_{0,l}^*;A/3) \cap \phi_l\lc \supp\big(\eta_{r,A}\big)\rc \cap M_l\times \{-s_1\} \subset B^l_{r,A}\subset B^*(x_{0,l}^*; 3A)\cap M_l\times \{-s_1\}
	\end{align}
	for sufficiently large $l$. Thus, it is clear that
		\begin{align*}
		\int_{M_l \setminus B_{r, A}^l}\left|\na v^l\right|^2 \, \mathrm{d}\nu^l_{-s_1} \le \int_{M_l \setminus B^*(x_{0,l}^*;A/3)}\left|\na v^l\right|^2 \, \mathrm{d}\nu^l_{-s_1}+\int_{B^*(x_{0,l}^*;A/3) \setminus \phi_l\lc \supp(\eta_{r,A})\rc}\left|\na v^l\right|^2 \, \mathrm{d}\nu^l_{-s_1}.
	\end{align*}	
	
By Proposition \ref{existenceHncenter} and \eqref{estihess1}, we have
		\begin{align}\label{smootheq3a}
\int_{M_l \setminus B^*(x_{0,l}^*;A/3)}\left|\na v^l\right|^2 \, \mathrm{d}\nu^l_{-s_1} 
\le  \lc\int_{M_l \setminus B^*(x_{0,l}^*;A/3)}\left|\na v^l\right|^{5}\, \mathrm{d}\nu^l_{-s_1}\rc^{\frac{2}{5}}\lc\int_{M_l \setminus B^*(x_{0,l}^*;A/3)} 1 \, \mathrm{d}\nu^l_{-s_1}\rc^{\frac{3}{5}} \le \Psi(A^{-1}).
	\end{align}	
	
By \cite[Proposition 8.23 (2)]{fang2025RFlimit}, we know that any $x \in \lc B^*(x_{0,l}^*;A/3) \setminus \phi_l\lc \supp(\eta_{r,A})\rc \rc \cap M_l \times \{-s_1\}$ satisfies $r_{\Rm}(x) \le 3r$ for all sufficiently large $l$, where $r_{\Rm}$ denotes the curvature radius. Thus, by using \cite[Theorem 1.12 (b)]{fang2025RFlimit}, we obtain
		\begin{align}\label{smootheq3b}
\int_{B^*(x_{0,l}^*;A/3) \setminus \phi_l\lc \supp(\eta_{r,A})\rc}\left|\na v^l\right|^2 \, \mathrm{d}\nu^l_{-s_1} \le \lc\int_{B^*(x_{0,l}^*;A/3) \setminus \phi_l\lc \supp(\eta_{r,A})\rc}1 \, \mathrm{d}\nu^l_{-s_1}\rc^{\frac{3}{5}} \le C(n, A, Y) r.
	\end{align}	
Combining \eqref{smootheq3a} and \eqref{smootheq3b}, we obtain
		\begin{align}\label{smootheq3}
		\int_{M_l \setminus B_{r, A}^l}\left|\na v^l\right|^2 \, \mathrm{d}\nu^l_{-s_1} \le \Psi(A^{-1})+C(n, A, Y) r.
	\end{align}	

In addition, for sufficiently large $l$, we have
	\begin{align}\label{convintegral3}
		&\left|\int_{B_{r, A}^l}\left|\na v^l\right|^2\left(1-\eta_{r,A}\circ \phi_l^{-1}\right) \,\mathrm{d}\nu^l_{-s_1}\right|\nonumber\\
		\leq &  \lc\int_{B_{r, A}^l}\left|\na v^l\right|^{5}\, \mathrm{d}\nu^l_{-s_1}\rc^{\frac{2}{5}} \lc\int_{B_{r, A}^l} \left(1-\eta_{r,A}\circ \phi_l^{-1}\right)^{\frac{5}{3}}\, \mathrm{d}\nu^l_{-s_1}\rc^{\frac{3}{5}} \le  C(n, A, Y) r,
	\end{align}
	where the last inequality holds by the same reason as in \eqref{smootheq3b}.
	
	Thus, we can first choose a large $A$ and then choose a small $r$ so that all integrals \eqref{smootheq3} and \eqref{convintegral3} are as small as we want. Combining this fact with \eqref{smootheq1}, we obtain \eqref{smootheq} and thus complete the proof of Claim \ref{takelimit1}.

\begin{claim}\label{takelimit2}
For any $1/100 \le s_1 \le s_2 <8$, we have on $\RR_{[-s_2, -s_1]} \cap \supp(\eta_{r, A})$,
		\begin{align}\label{timediff2b1}
|\na v^{\infty}| \le C(n, A, Y, s_2).
	\end{align}
	\end{claim}

We set $D^l_{r, A}=\phi_l \lc \supp(\eta_{r, A}) \rc \cap M_l \times [-s_2, -s_1]$. Similar to \eqref{eq:compare1}, we have
\begin{align*}
D^l_{r,A}\subset B^*(x_{0,l}^*; 3A)\cap M_l\times [-s_2, -s_1].
	\end{align*}
By the reproduction formula, for any $x^* \in D^l_{r,A}$, we have for $s_2'=(s_2+8)/2$,
	\begin{align}\label{timediff2b1ex001}
|\na v^l|^2(x^*) \le \int_{M_l} |\na v^l|^2 \,\mathrm{d}\nu_{x^*; -s_2'}.
	\end{align}
	
Now, we choose $t'<\t_l(x^*)$ such that $\beta:=\t_l(x^*)-t'$ is small and will be determined later. By our assumption on $x_{0, l}^*$ and Lemma \ref{lem:nonexpand}, we conclude that 
	\begin{align}\label{timediff2b1ex001a}
d_{W_1}^{t'}(\nu_{x^*;t'}, \nu^l_{t'}) \le C(n, Y, A, \beta).
	\end{align}
	
For a small constant $\gamma \in (0, 1)$, it follows from \eqref{timediff2b1ex001a} and \cite[Proposition A.2]{fang2025RFlimit} that if $\beta=\beta(n, \gamma, s_2)$ is small, then
	\begin{align}\label{timediff2b1ex002}
\nu_{x^*;-s_2'} \le C(n, Y, A, \gamma, s_2) e^{\gamma f_l} \nu^l_{-s_2'}.
	\end{align}	
	
By hypercontractivity (see \cite[Theorem 12.1]{bamler2020entropy}), we have
	\begin{align*}
\int_{M_l} |\na v^l|^{1+8/s'_2} \,\mathrm{d}\nu^l_{-s'_2} \le 1
	\end{align*}	
and hence
	\begin{align}\label{timediff2b1ex003}
\int_{M_l} |\na v^l|^2 e^{\gamma f_l} \,\mathrm{d}\nu^l_{-s_2'} \le  \lc \int_{M_l} |\na v^l|^{1+8/s_2'} \,\mathrm{d}\nu^l_{-s_2'} \rc^{\frac{2s_2'}{8+s_2'}} \lc \int_{M_l} \exp \lc \gamma \frac{8+s_2'}{8-s_2'} f_l \rc\,\mathrm{d}\nu^l_{-s_2'}  \rc^{\frac{8-s_2'}{8+s_2'}} \le C(n, Y, s_2),
	\end{align}	
	where we used Proposition \ref{integralbound}, provided that $\gamma=\gamma(n, Y, s_2)$ is small. 
	
Combining \eqref{timediff2b1ex001}, \eqref{timediff2b1ex002}, \eqref{timediff2b1ex003} and using smooth convergence, we have completed the proof of Claim \ref{takelimit2}.

\begin{claim}\label{takelimit3}
For any $1/100\le s_1 \le s_2 <8$,
		\begin{equation*}
			2\int_{-s_2}^{-s_1}\int_{\RR_t}|\nabla^2 v^\infty|^2\,\mathrm{d}\nu_{z;t} \mathrm{d}t=\int_{\RR_{-s_2}}|\nabla v^\infty|^2\,\mathrm{d}\nu_{z;-s_2}-\int_{\RR_{-s_1}}|\nabla v^\infty|^2\,\mathrm{d}\nu_{z;-s_1}.
		\end{equation*}
	\end{claim}
	
	By the evolution equation on the regular part, we have
	\begin{align*}
		\frac{\mathrm{d}}{\mathrm{d}t}\int_{\RR_t}\left|\nabla v^\infty\right|^2\eta_{r,A} \,\mathrm{d}\nu_{z;t}=-2\int_{\RR_t}\left|\nabla^2 v^\infty\right|^2\eta_{r,A} \,\mathrm{d}\nu_{z;t}+\int_{\RR_t}\big|\nabla v^\infty\big|^2\square \eta_{r,A}-2\la \nabla \big|\nabla v^\infty\big|^2, \nabla \eta_{r,A}\ra \,\mathrm{d}\nu_{z;t}.
	\end{align*}
	
Consequently, for fixed $1/100 \le s_1 \le s_2 <8$,
		\begin{align}
& \int_{\RR_{-s_2}}|\nabla v^\infty|^2\eta_{r,A} \,\mathrm{d}\nu_{z;-s_2}-\int_{\RR_{-s_1}}|\nabla v^\infty|^2\eta_{r,A} \,\mathrm{d}\nu_{z;-s_1} \notag\\
=& 2 \int_{-s_2}^{-s_1} \int_{\RR_t} |\na^2 v^\infty|^2 \eta_{r,A} \, \mathrm{d}\nu_{z;t} \mathrm{d}t-\int_{-s_2}^{-s_1} \int_{\RR_t}\big|\nabla v^\infty\big|^2\square \eta_{r,A}-2\la \nabla \big|\nabla v^\infty\big|^2, \nabla \eta_{r,A}\ra \,\mathrm{d}\nu_{z;t} \mathrm{d}t. \label{timediff2a}
	\end{align}

Now we estimate the error terms. First, we compute
	\begin{align*}
		&\int_{-s_2}^{-s_1}\int_{\RR_t}\left|\nabla v^\infty\right|^2\left|\square \eta_{r,A}\right|\,\mathrm{d}\nu_{z;t} \mathrm{d}t \\
	\leq & C(n, Y, A,s_2) r^{-2} \lc \int_{\RR_{[-s_2, -s_1]}\bigcap \{0 <\eta_{r, A}<1\}} 1 \,\mathrm{d}\nu_{z;t} \mathrm{d}t \rc \le C(n, Y, A,s_2) r
	\end{align*}
	where we used \eqref{esticutoff}, \eqref{eq:hessianeq006a} and \eqref{timediff2b1} for the last line.

  Similarly, we have
	\begin{align*}
		&\int_{-s_2}^{-s_1}\int_{\RR_t}\left|\nabla\big|\nabla v^\infty\big|^2\right|\left|\nabla\eta_{r,A}\right|\,\mathrm{d}\nu_{z;t} \mathrm{d}t \\
		\leq& C(n,Y, A)r^{-1}\iint_{\RR_{[-s_2, -s_1]}\bigcap \{0 <\eta_{r, A}<1\}}\left|\nabla ^2 v^\infty\right|\left|\nabla v^\infty\right|\,\mathrm{d}\nu_{z;t} \mathrm{d}t\\
		\leq& C(n,Y, A)r^{-1} \lc\iint_{\RR_{[-s_2, -s_1]}\bigcap \{0 <\eta_{r, A}<1\}}\left|\nabla ^2 v^\infty\right|^2\,\mathrm{d}\nu_{z;t} \mathrm{d}t\rc^{1/2}\lc\iint_{\RR_{[-s_2, -s_1]}\bigcap \{0 <\eta_{r, A}<1\}}\left|\nabla v^\infty\right|^2\,\mathrm{d}\nu_{z;t} \mathrm{d}t\rc^{1/2} \\
		 \le & C(n, Y, A,s_2) r^{1/2},
	\end{align*}
where we used \eqref{esticutoff}, \eqref{eq:hessianeq006a} and \eqref{timediff2b1}.
	
	Therefore, the second integral on the right-hand side of \eqref{timediff2a} can be made arbitrarily small if we first choose a large $A$ and then choose a small $r$. Thus, we obtain
	\begin{align*}
& \int_{\RR_{-s_2}}|\nabla v^\infty|^2 \,\mathrm{d}\nu_{z;-s_2}-\int_{\RR_{-s_1}}|\nabla v^\infty|^2 \,\mathrm{d}\nu_{z;-s_1}=2 \int_{-s_2}^{-s_1} \int_{\RR_t} |\na^2 v^\infty|^2  \, \mathrm{d}\nu_{z;t} \mathrm{d}t.
	\end{align*}
	This completes the proof of Claim \ref{takelimit3}.

By \cite[Lemma D.1]{fang2025RFlimit}, $\lc \iota_z(\XX^z_{-1}), d^Z_{-1}, \nu_{z;-1} \rc$ is a $\rcd (1/2,\infty)$-space. Thus, the eigenvalues and eigenfunctions of $\Delta_{f_z(-1)}$ are well defined (see \cite{gigli2013ConvergenceOP}), where $\Delta_{f_z(-1)}$ denotes the weighted Laplacian on the $\rcd$-space.

	\begin{claim}\label{claimeigen}
Let $\lambda_i$ be all nonzero eigenvalues of $\Delta_{f_z(-1)}$ with associated eigenfunctions $\psi_i$, then the following holds for some constant $c_0=c_0(n,Y,\eta)>0$, 
		\begin{align*}
			0<\lambda_1=\cdots=\lambda_k=\frac{1}{2}<\frac{1}{2}+c_0\leq \lambda_{k+1}.
		\end{align*} 
	\end{claim}

Since $(\iota_z(\XX^z_{-1}), d^Z_{-1})$ splits off an $\R^k$, it follows from \cite[Theorem 1.1]{GKKO} that $\lambda_1=\cdots=\lambda_k=1/2$. If $\lambda_{k+1} \leq 1/2+c_0$ for some small $c_0$, then we can, by using $\eta_{r,A}$ and $\phi_l$, pull back the first $k+1$ eigenfunctions to $M_l\times \{-1\}$ to get $\vec \psi^l=(\psi^l_1,\ldots,\psi^l_{k+1})$ satisfying for $1\le i, j  \le k+1$,
	\begin{align*}
\frac{1}{2} \le \int_{M_l} (\psi_i^l)^2 \,\mathrm{d}\nu_{-1}^l \le 2, \quad  \abs{\int_{M_l} \la \na \psi^l_i, \na \psi^l_j \ra \,\mathrm{d}\nu_{-1}^l-\delta_{ij}} \le \frac{1}{100}
	\end{align*}
and
	\begin{align*}
\dfrac{\int_{M_l}|\na \psi_i^l|^2\,\mathrm{d}\nu_{-1}^l}{\int_{M_l}(\psi_i^l)^2\,\mathrm{d}\nu_{-1}^l}\leq \frac{1}{2}+2c_0.
	\end{align*}
Consequently, if $c_0$ is sufficiently small, it follows from the same argument as in \cite[Proposition C.3, Corollary C.4]{fang2025RFlimit} that there exists a $(k+1,\Psi(c_0),1/\sqrt{10})$-splitting map at $x_{0,l}^*$. By Proposition \ref{equivalencesplitsymetric1} (i), this implies that if $c_0 \le c_0(n, Y, \eta)$, then $x_{0,l}^*$ is $(k+1, \eta, 1 )$-splitting, which contradicts our assumption. This completes the proof of Claim \ref{claimeigen}. Without loss of generality, we may assume $\psi_i=y_i/\sqrt{2}$ on $\RR_{-1}$ for $1 \le i \le k$.
	
We now extend $\psi_i$ to a spacetime-function, still denoted by $\psi_i$, on $\RR_{(-\infty, 0)}$. Let $\boldsymbol{\psi}^s$ be the flow generated by $X:=\tau(\partial_\t-\na f_z)$, which is well-defined for all $s \in \R$ on $\RR_{(-\infty, 0)}$ by \cite[Theorem 1.9]{fang2025RFlimit}. Define a function $u$ on $\RR_{(-\infty, 0)}$ such that $u=\psi_i$ on $\RR_{-1}$ and $u$ is invariant under $\boldsymbol{\psi}^s$. Then, define the function $\psi_i$ on $\RR_{(-\infty, 0)}$ by $\psi_i:=\tau^{\lambda_i} u$. A direct computation shows that $\square \psi_i=0$ on $\RR_{(-\infty, 0)}$ and
	\begin{align} \label{eq:compare4aa}
\Delta_{f_z} \psi_i+\tau^{-1} \lambda_i \psi_i=0.
	\end{align}
Moreover, for any $t<0$, we have the identity
	\begin{align}\label{eq:compare4a}
\int_{\RR_t} \psi_i(t) \psi_j (t) \,\mathrm{d}\nu_{z;t}=\tau^{2\lambda_i} \delta_{ij}.
	\end{align}
In other words, $\{\tau^{-\lambda_i} \psi_i(t)\}$ is an orthonormal basis with respect to $\mathrm{d}\nu_{z;t}$.

Next, we define
	\begin{align*}
\tilde v(x, s)=v^\infty(\boldsymbol{\psi}^s(x)) \qquad (x,s)\in \RR_{-1}\times[0,\infty).
	\end{align*}
A direct computation yields
	\begin{align*}
\partial_s \tilde v=\Delta_{f_z(-1)} \tilde v
	\end{align*}
on $\RR_{-1} \times [0, \infty)$. Since $W^{1, 2}(\RR_{-1}, \nu_{z;-1})$ is dense in $W^{1, 2}(\iota_z(\XX^z_{-1}), \nu_{z;-1})$ (see the proof of \cite[Lemma D.1]{fang2025RFlimit}), we may regard $\tilde v$ as the heat flow (cf. \cite[Section 4]{AGS14}) on the $\mathrm{RCD}(1/2, \infty)$-space $\lc \iota_z(\XX^z_{-1}), d^Z_{-1}, \nu_{z;-1} \rc$ with initial data $v^{\infty} \vert_{\RR_{-1}}$.

Define the coefficients $a_i=\int_{\RR_{-1}}v^\infty \psi_i\,\mathrm{d}\nu_{z;-1}$. Since $\int_{\RR_{-1}}v^\infty \,\mathrm{d}\nu_{z;-1}=0$, we have the $L^2(\nu_{z;-1})$-expansion
	\begin{align*}
v^\infty \vert_{\RR_{-1}}=\sum_{i=1}^\infty a_i \psi_i(-1).
	\end{align*}	
As $\psi_i(-1)$ is an eigenfunction with eigenvalue $\lambda_i$, the heat flow from $\psi_i(-1)$ is $e^{-\lambda_i s} \psi_i(-1)$. By the contraction property of the heat flow,
	\begin{align*}
\tilde v(\cdot, s)=\sum_{i=1}^\infty e^{-\lambda_i s} a_i \psi_i(-1) \quad \text{in} \quad L^2(\nu_{z;-1}).
	\end{align*}	
	
Using the definitions of $\tilde v$ and $\psi_i(t)$, this implies
	\begin{align}\label{eq:expandextra001}
v^{\infty}\vert_{\RR_t} =\sum_{i=1}^\infty  a_i \psi_i(t)
	\end{align}	
in $L^2(\nu_{z;t})$, for any $t \in [-1, 0)$. 

Repeating the argument for any base time in $(-8, 0)$ in place of $-1$, we conclude that the $L^2$-expansion \eqref{eq:expandextra001} holds for all $t \in (-8, 0)$. Combining this with Claims \ref{takelimit3}, \ref{claimeigen}, \eqref{eq:compare4aa}, and \eqref{eq:compare4a}, we obtain for any $s \in [1/100, 8)$,
	\begin{align}
		\int_{-s}^{-\frac{s}{10}}\int_{\RR_t} \left|\nabla^2 v^\infty\right|^2\,\mathrm{d}\nu_{z;t} \mathrm{d}t&=\int_{\RR_{-s}}\left|\nabla v^\infty\right|^2\,\mathrm{d}\nu_{z;-s}-\int_{\RR_{-s/10}}\left|\nabla v^\infty\right|^2\,\mathrm{d}\nu_{z;-s/10} \notag \\
&=\sum_{i=k+1}^\infty a^2_i\lc s^{2\lambda_i-1}-\left(\frac{s}{10}\right)^{2\lambda_i-1}\rc \lambda_i. \label{eq:hessianextra001}
	\end{align}
Letting $s\nearrow 8$, the identity \eqref{eq:hessianextra001} holds for any $s \in [1/100, 8]$.

	\begin{claim}\label{coefficientnonzero}
There exists an $i_0 \ge k+1$ such that $a_{i_0} \ne 0$.
	\end{claim}

By our definition of $v^l$, for any $i \in \{1,\ldots,k\}$,
	\begin{align*}
	\int_{-8}^{-8/10}\int_{M_l} \la\nabla v^l,\nabla u_i^l\ra \,\mathrm{d}\nu^l_t \mathrm{d}t=0.
	\end{align*}	
	By the evolution equation $\diff{}{t}\int_{M_l} \la \nabla v^l,\nabla u_i^l\ra \,\mathrm{d}\nu^l_t=-2\int_{M_l} \la \nabla^2 v^l,\nabla^2 u_i^l\ra \,\mathrm{d}\nu^l_t$ and the fact that $\vec u^l$ is a $(k, l^{-2}, 1)$-splitting map, we obtain
	\begin{equation*}
		\left|\int_{M_l} \la \nabla v^l,\nabla u_i^l\ra \,\mathrm{d}\nu^l_{-1}\right|\leq C l^{-1}.
	\end{equation*}
Letting $l \to \infty$, we derive, by the same argument as in Claim \ref{takelimit1}, that
	\begin{equation*}
\int_{\RR_{-1}} \la \na v^\infty, \na  \psi_i \ra \,\mathrm{d}\nu_{z;-1}=0,
	\end{equation*}
which implies
	\begin{equation*}
a_i =\int_{\RR_{-1}} v^\infty  \psi_i\,\mathrm{d}\nu_{z;-1}=0
	\end{equation*}
for any $1 \le i \le k$. By our assumption \eqref{eq:hessianeq002} and Claim \ref{takelimit1} again, we conclude that the constants $a_i$ are not all zero for $i \ge k+1$, which completes the proof of Claim \ref{coefficientnonzero}.

By Claim \ref{coefficientnonzero}, it follows from \eqref{eq:hessianextra001} that
		\begin{align}\label{eq:hessianextra002}
\int_{-2}^{-\frac{2}{10}}\int_{\RR_t} \left|\nabla^2 v^\infty\right|^2\,\mathrm{d}\nu_{z;t} \mathrm{d}t=\sum_{i=k+1}^\infty a^2_i\lc 2^{2\lambda_i-1}-\left(\frac{2}{10}\right)^{2\lambda_i-1}\rc \lambda_i>0.
	\end{align}
Similarly, we have	
			\begin{align}\label{eq:hessianextra003}
\int_{-8}^{-\frac{8}{10}}\int_{\RR_t} \left|\nabla^2 v^\infty\right|^2\,\mathrm{d}\nu_{z;t}\mathrm{d}t=\sum_{i=k+1}^\infty a^2_i\lc 8^{2\lambda_i-1}-\left(\frac{8}{10}\right)^{2\lambda_i-1}\rc \lambda_i>0.
	\end{align}

Taking the quotient of \eqref{eq:hessianextra002} and \eqref{eq:hessianextra003}, we obtain
			\begin{align}\label{eq:hessianextra004}
\dfrac{\int_{-2}^{-\frac{2}{10}}\int_{\RR_t} \left|\nabla^2 v^\infty\right|^2\,\mathrm{d}\nu_{z;t}\mathrm{d}t}{\int_{-8}^{-\frac{8}{10}}\int_{\RR_t} \left|\nabla^2 v^\infty\right|^2\,\mathrm{d}\nu_{z;t}\mathrm{d}t}=\dfrac{\sum_{i=k+1}^\infty a^2_i\lc 2^{2\lambda_i-1}-\left(\frac{2}{10}\right)^{2\lambda_i-1}\rc \lambda_i}{\sum_{i=k+1}^\infty a^2_i\lc 8^{2\lambda_i-1}-\left(\frac{8}{10}\right)^{2\lambda_i-1}\rc \lambda_i} \le \lc\frac{1}{4}\rc^{2c_0},
	\end{align}	
where in the last inequality, we have used the fact from Claim \ref{claimeigen} that $2\lambda_i-1\geq 2c_0$ for all $i\geq k+1$.	
	
However, by taking the limit of \eqref{contradictioninequality1}, it follows from Claims \ref{takelimit1} and \ref{takelimit3} that
		\begin{equation*}
		\int_{-2}^{-\frac{2}{10}}\int_{\RR_t} \left|\nabla^2 v^\infty\right|^2\,\mathrm{d}\nu_{z;t} \mathrm{d}t\geq \int_{-8}^{-\frac{8}{10}}\int_{\RR_t} \left|\nabla^2 v^\infty\right|^2\,\mathrm{d}\nu_{z;t} \mathrm{d}t,
	\end{equation*}
	which contradicts \eqref{eq:hessianextra004}, and hence completes the proof of the theorem.
\end{proof}

\subsection{A covering lemma for independent points}

In this subsection, we discuss how to find independent points in applications.
\begin{lem}\label{existenceofindependentpointsafterbase}
There exists a constant $L=L(n, Y)>1$ such that the following holds.

Suppose $y_0^*=(y_0,t_0)\in\XX$ is $(\delta,r)$-selfsimilar. Let $S$ be the subset consisting of all $(\delta,r)$-selfsimilar points in $B^*(y_0^*,r)\bigcap M\times [t_0-\alpha^2 r^2/L^2, t_0+\alpha^2 r^2/L^2]$. Assume that there exist no $(k,\alpha,\delta, r)$-independent points in $S$ at $y_0^*$. If $\delta\leq\delta(n,Y,\alpha)$, then we can find $\{y_i^*\}_{1 \le i \le N} \subset S$ with $N\leq C(n,Y)\alpha^{1-k}$ and 
	\begin{align*}
		S\subset\bigcup_{i=0}^N B^*(y^*_i, \alpha r).
	\end{align*}
\end{lem}
\begin{proof}
	Without loss of generality, we assume $t_0=0$ and $r=1$, and set $\mathrm{d}\nu_t=\mathrm{d}\nu_{y_0^*;t}=(4\pi |t|)^{-\frac n 2} e^{-f} \,\mathrm{d}V_{g(t)}$.
	
	We choose the maximal $\alpha$-separated set $\{y_0^*, y_1^*,\ldots,y_N^*\} \subset S$ containing $y_0^*$, i.e., $d^*(y^*_i, y^*_j) \ge \alpha$ for all $i \ne j$. By Definition \ref{defnd*distance} and monotonicity, the following holds:
	\begin{equation*}
		d_{W_1}^{-\alpha^2/2}\left(\nu_{y_i^*;-\alpha^2/2},\nu_{y_j^*;-\alpha^2/2}\right)\geq \alpha.
	\end{equation*}
	
	We set $\beta=10 \alpha/L$. Then for a constant $D=D(n, Y)>1$ to be determined later, we may choose $L=L(n, Y)$ large so that
	\begin{align}\label{findind11a}
 d_{W_1}^{-\beta^2}\left(\nu_{y_i^*;-\beta^2},\nu_{y_j^*;-\beta^2}\right)\geq D\beta.
	\end{align}

	Set $\tau_i=\t(y_i^*)-t$, $ \mathrm{d}\nu_{y_i^*;t}=(4\pi\tau_i)^{-n/2}e^{-f_i}\,\mathrm{d}V_{g(t)}$, $F_i=\tau_i f_i$, $\mathcal{T}_i:= \tau_i \Ric+\nabla^2 F_i-g/2$, $w_i=\tau_i(2\Delta f_i-|\na f_i|^2+\scal)+f_i-n$ and $W_i=\WW_{y_i^*}(1)$ for $0 \le i \le N$. For $1 \le i \le N$, we define $u_i=F_i-F_0$. 
    
    \begin{claim}\label{findindclaim1}
		For any $i,j \in \{1,\ldots,N\}$, we can find a constant $q_{ij}$ such that
		\begin{equation*}
			\int_{-10}^{-\beta^2/2}\int_M \left|\la\nabla u_i,\nabla u_j\ra-q_{ij}\right|\,\mathrm{d}\nu_t \mathrm{d}t\leq C(n,Y, \beta) \delta^{1/4}.
		\end{equation*}
	\end{claim}
	
Since $y_i^*$ are $(\delta, 1)$-selfsimilar for $0 \le i \le N$, we can obtain, similar to the proof of \eqref{perpendicularofRichessianf1} by using Lemmas \ref{lem:imply1} and \ref{lem:nonexpand}, that
	\begin{align}\label{findind1}
		\int_{-10}^{-\beta^2/2}\int_M \left|\mathcal{T}_i\right|^2\,\mathrm{d}\nu_t \mathrm{d}t\leq C(n,Y, \beta)\delta^{1/2}.
	\end{align}

Taking the difference of \eqref{findind1} and using Theorem \ref{staticestimate} with $-1/30$ replaced by $-\beta^2/2$, it follows that
	\begin{align}\label{findind2}
		\int_{-10}^{-\beta^2/2}\int_M \left|\nabla^2 u_i\right|^2\,\mathrm{d}\nu_t \mathrm{d}t\leq C(n,Y, \beta)\delta^{1/2}.
	\end{align}
Moreover, it follows from Proposition \ref{integralbound}, Lemma \ref{lem:nonexpand} and \cite[Proposition A.2]{fang2025RFlimit} that
	\begin{align}\label{findind2x}
	\sup_{t \in [-10, -\beta^2/2]}	\int_M |\na u_i|^2\,\mathrm{d}\nu_t \leq C(n,Y, \beta).
	\end{align}

Similar to the proof of Lemma \ref{constructionofsplittingfunction1} by using \cite[Proposition A.2]{fang2025RFlimit}, for each $0 \le i \le N$, we obtain a smooth function $h_i$ on $M \times [-10, -\beta^2/2]$ such that
\begin{enumerate}[label=\textnormal{(\roman{*})}]
		\item $\displaystyle	\square h_i=-\frac{n}{2}$.
		\item $\displaystyle	\int_{-10}^{-\beta^2/2}\int_{M} \abs{\na^2(F_i-h_i)}^2\,\mathrm{d}\nu_{t} \mathrm{d}t\leq C(n, Y, \beta) \delta^{1/2}$.
		\item $\displaystyle	\sup_{t \in [-10, -\beta^2/2]}	\int_M |\na (F_i-h_i)|^2\,\mathrm{d}\nu_t \leq C(n,Y, \beta) \delta^{1/2}$.
	\end{enumerate}

We set $v_i:=h_i-h_0$ for $1 \le i \le N$. Then it follows from \eqref{findind2} and (ii) above that
	\begin{align}\label{findind2xx}
\int_{-10}^{-\beta^2/2}\int_{M} \abs{\na^2 v_i}^2\,\mathrm{d}\nu_{t} \mathrm{d}t\leq C(n, Y, \beta) \delta^{1/2}.
	\end{align}

A standard calculation yields
	\begin{align}\label{evolutioneq}
\diff{}{t}\int_M\la\nabla v_i,\nabla v_j\ra \,\mathrm{d}\nu_t =\int_M \la\nabla\square v_i,\nabla v_j\ra+\la\nabla v_i,\nabla \square v_j\ra-2\la\nabla^2 v_i,\nabla^2 v_j\ra\,\mathrm{d}\nu_t =\int_M -2\la\nabla^2 v_i,\nabla^2 v_j\ra\,\mathrm{d}\nu_t.
	\end{align}
	Plugging \eqref{findind2xx} into \eqref{evolutioneq}, we obtain for any $-10 \le s <t \le -\beta^2/2$, 
		\begin{align*}
\abs{\int_M\la\nabla v_i,\nabla v_j\ra \,\mathrm{d}\nu_t-\int_M\la\nabla v_i,\nabla v_j\ra \,\mathrm{d}\nu_s} \le C(n,Y, \beta) \delta^{1/2}.
	\end{align*}
	Therefore, it follows from (iii) above that we can find a constant $q_{ij}$ such that for any $t \in [-10, -\beta^2/2]$,
			\begin{align*}
\abs{\int_M\la\nabla u_i,\nabla u_j\ra \,\mathrm{d}\nu_t-q_{ij}} \le C(n,Y, \beta)\delta^{1/4}.
	\end{align*}
	
Thus, by the Poincar\'e inequality (see Theorem \ref{poincareinequ}), we have for any $t \in [-10, -\beta^2/2]$,
				\begin{align*}
\int_M\abs{\la\nabla u_i,\nabla u_j\ra-q_{ij}} \,\mathrm{d}\nu_t \le \sqrt{10 \pi} \int_M |\na u_i| |\na^2 u_j|+|\na u_j| |\na^2 u_i|\,\mathrm{d}\nu_t+C(n,Y, \beta)\delta^{1/4}.
	\end{align*}
	By integration, we obtain from Proposition \ref{integralbound} and \eqref{findind2} that
\begin{align*}
		&\int_{-10}^{-\beta^2/2} \int_M\abs{\la\nabla u_i,\nabla u_j\ra-q_{ij}} \,\mathrm{d}\nu_t \mathrm{d}t \\
		 \le& 100 \int_{-10}^{-\beta^2/2} \int_M |\na u_i| |\na^2 u_j|+|\na u_j| |\na^2 u_i|\,\mathrm{d}\nu_t \mathrm{d}t+C(n,Y, \beta)\delta^{1/4} \\
		 \le & 100 \lc \int_{-10}^{-\beta^2/2} \int_M |\na^2 u_j|^2\,\mathrm{d}\nu_t \mathrm{d}t \rc^{\frac 1 2} \lc \int_{-10}^{-\beta^2/2} \int_M |\na u_i|^2\,\mathrm{d}\nu_t \mathrm{d}t \rc^{\frac 1 2} \\
		 &+100 \lc \int_{-10}^{-\beta^2/2} \int_M |\na^2 u_i|^2\,\mathrm{d}\nu_t \mathrm{d}t \rc^{\frac 1 2} \lc \int_{-10}^{-\beta^2/2} \int_M |\na u_j|^2\,\mathrm{d}\nu_t \mathrm{d}t \rc^{\frac 1 2}+C(n,Y, \beta)\delta^{1/4} \le C(n,Y, \beta)\delta^{1/4},
	\end{align*}
	where we used \eqref{findind2x}. This proves Claim \ref{findindclaim1}.
	
	\begin{claim}\label{findindclaim2}
For any $1\leq i\neq j\leq N$,
		\begin{equation*}
			\int_{-10}^{-1/10}\int_M \left|\nabla (u_i-u_j)\right|^2\,\mathrm{d}\nu_t \mathrm{d}t\geq 100 \beta^2.
		\end{equation*}
	\end{claim}
	Fix $i\neq j$ and without loss of generality, assume $\tau_j\leq \tau_i$. By Claim \ref{findindclaim1}, we can find $q\in \mathbb{R}$ such that
	\begin{align}\label{findind11}
		\int_{-10}^{-\beta^2/2}\int_M \left||\nabla(u_i-u_j)|^2-q\right|\,\mathrm{d}\nu_t \mathrm{d}t\leq C(n,Y, \beta )\delta^{1/4}.
	\end{align}
		
	By \eqref{findind11} and \cite[Proposition 7.3]{bamler2020structure}, we can choose some $t^*\in [-\beta^2,-\beta^2/2]$ such that
	\begin{align*}
		\int_M \left||\nabla(\tau_i f_i-\tau_jf_j)|^2-q\right|\,\mathrm{d}\nu_{t^*}\leq& \Psi(\delta),\\
		\int_M \left|-\tau_i(|\nabla f_i|^2+\scal)+f_i-W_i \right|\,\mathrm{d}\nu_{y_i^*;t^*}\leq& \Psi(\delta),\\
		\int_M \left|-\tau_j(|\nabla f_j|^2+\scal)+f_j-W_j\right|\,\mathrm{d}\nu_{y_j^*;t^*}\leq& \Psi(\delta).
	\end{align*}
	Then choose $x_i^*=(x_i,t^*)$ and $x_j^*=(x_j,t^*)$ to be $H_n$-centers of $y_i^*$ and $y_j^*$, respectively. By \eqref{findind11a}, we have
		\begin{align}\label{findind11b}
d_{t^*}(x_i,x_j)\geq \frac{D}{2}\beta.
	\end{align}
		
	Note that in our setting, when evaluated at $t^*$,
			\begin{align}\label{findind11c}
\beta^2/3 \leq\tau_j\leq\tau_i\leq 2 \beta^2.
	\end{align}

	Set $B_i:= B_{t^*}(x_i,2\sqrt{H_n}\beta)$ and $S_i:= \{f_i\leq D\}\bigcap B_i$. It is clear from Proposition \ref{existenceHncenter} and \eqref{findind11b} that
			\begin{align*}
\nu_{y_i^*;t^*}(B_i)\geq \frac{1}{2} \quad \text{and} \quad \nu_{y_j^*;t^*}(B_i)\leq \frac{C(n)}{D^2}.
	\end{align*}	

On the other hand, by \cite[Theorem 8.1]{bamler2020entropy}, we have
	\begin{align*}
\nu_{y_i^*;t^*}(B_i\setminus S_i) \le \frac{1}{(4\pi \tau_i(t^*))^{\frac n 2}} e^{-D} C(n, Y) \beta^n \le C(n, Y) e^{-D}.
	\end{align*}
	Thus, we choose $D$ large so that 
				\begin{align*}
\nu_{y_i^*;t^*}(B_i\setminus S_i) \le \frac{1}{8}.
	\end{align*}	
By passing to a further subset of $S_i$, still denoted by $S_i$, we may assume on $S_i$,
		\begin{align*}
			\left||\nabla(\tau_i f_i-\tau_jf_j)|^2-q\right|\leq& \Psi(\delta),\\
			\left|-\tau_i(|\nabla f_i|^2+\scal)+f_i-W_i\right|\leq& \Psi(\delta),\\
			\left|-\tau_j(|\nabla f_j|^2+\scal)+f_j-W_j\right|\leq& \Psi(\delta),\\
			\nu_{y_i^*;t^*}(B_i\setminus S_i)+\nu_{y_j^*;t^*}(B_i\setminus S_i)\leq& \frac{1}{4}.
		\end{align*}
Thus, we compute on $S_i$ that if $\delta \le \delta(n, Y, \beta)$,
	\begin{align*}
		&\left|\tau_if_i-\tau_jf_j-(\tau_i^2-\tau_j^2) \scal \right| \\
		\leq& \left|-\tau_i^2(|\na f_i|^2+\scal)+\tau_if_i-\tau_iW_i\right|+\left|-\tau_j^2(|\na f_j|^2+\scal)+\tau_jf_j-\tau_jW_j\right|\\
		&+\left||\na (\tau_i f_i)|^2-|\na(\tau_j f_j)|^2\right|+|\tau_iW_i-\tau_jW_j|\\
	\leq& C(n, Y)\beta^2+\Psi(\delta)+\abs{\na(F_j-F_i)} \abs{\na(F_j-F_i)+2\na F_i}\\
		\leq & C(n, Y)\beta^2+\Psi(\delta)+q+C(n, Y) q^{\frac 1 2} \beta \leq q+C(n, Y) \beta^2.
	\end{align*}
In particular, this implies that on $S_i$
	\begin{align*}
\tau_j f_j\leq \tau_i f_i+(\tau_j^2-\tau_i^2) \scal+q+C(Y, D)\beta^2\leq \tau_i f_i+q+C(n, Y)\beta^2,
	\end{align*}	
	where in the last inequality, we have used \eqref{findind11c} and $\scal \geq -C(n)$. By \eqref{findind11c} again, we obtain 
	\begin{align*}
		f_j\leq 6 f_i +3\beta^{-2}q+C(n, Y) \leq f_i+5 D +3 \beta^{-2} q+C(n, Y) \le f_i + 3 \beta^{-2} q +C(n, Y).
	\end{align*}
Consequently, we have
	\begin{align*}
		\frac{1}{4}\leq \nu_{y_i^*;t^*}(S_i)=&\int_{S_i}(4\pi\tau_i)^{-n/2}e^{-f_i}\,\mathrm{d}V_{g(t^*)}\leq e^{3\beta^{-2} q+C(n, Y)}\int_{S_i}(4\pi\tau_j)^{-n/2}e^{-f_j}\,\mathrm{d}V_{g(t^*)}\\
		=&e^{3\beta^{-2} q+C(n, Y)}\nu_{y_j^*;t^*}(S_i)	\leq e^{3\beta^{-2} q+C(n,Y)}\frac{C(n)}{D^2},
	\end{align*}
which implies
	\begin{align*}
q\geq \frac{1}{3}\beta^2\lc2\log D-\log(4C(n))-C(n, Y)\rc\geq 101\beta^2,
	\end{align*}
provided that $D$ is sufficiently large. If $\delta \le \delta(n, Y, \beta)$, it follows from \eqref{findind11} that
	\begin{align}\label{findind12a}
\int_{-10}^{-1/10}\int_M|\nabla(u_i-u_j)|^2\,\mathrm{d}\nu_t \mathrm{d}t\geq 100 \beta^2,
	\end{align}
which completes the proof of Claim \ref{findindclaim2}.

	Combining Claims \ref{findindclaim1} and \ref{findindclaim2}, we finish the proof as follows. Consider the vector space $V_0$ spanned by $\{u_1,\ldots,u_N\}$, equipped with the inner product
\[(\cdot ,\cdot):= \aint_{-10}^{-1/10}\int_M \la\nabla \cdot,\nabla \cdot\ra \,\mathrm{d}\nu_t \mathrm{d}t.\]
	Then by Proposition \ref{integralbound} and \eqref{findind12a}, we know that
		\begin{align}\label{findind12bb}
\|u_i\|\leq C(n,Y),\quad \|u_i-u_j\|\geq 3\beta,\quad \forall i\neq j\in \{1,\ldots, N\}.
	\end{align}

Since we cannot find $(k,\alpha,\delta, 1)$-independent points from $\{y_i^*\}_{1 \le i \le N}$ at $y_0^*$, it follows from Definition \ref{defnindependentpoints} and an inductive argument that there exists a $(k-1)$-subspace $V_1\subset V_0$ such that
\[\{u_1,\ldots,u_N\}\subset B_{C(k)\alpha}(V_1),\]
	where $B_{C(k)\alpha}$ denotes the $C(k)\alpha$-neighborhood with respect to the inner product defined above. From this and \eqref{findind12bb}, it is easy to see $N\leq C(n,Y)\alpha^{1-k}$. 
	
In sum, the proof is complete.
\end{proof}

\section{Nondegeneration of almost splitting maps on Ricci flow limit spaces} \label{sec:limit}

In this section, we extend the results in previous sections to Ricci flow limit spaces and prove a nondegeneration result for almost splitting maps.

Throughout this section, suppose that $(Z,d_Z,\t)$ is a noncollapsed Ricci flow limit space obtained as a pointed Gromov--Hausdorff limit of a sequence $\XX^l \in \MM(n, Y, T)$. We set $\XX^l=\{M_l^n,(g_l(t))_{t\in \III^{++}}\}$. Moreover, let $\phi_l$ denote the diffeomorphisms given in Theorem \ref{thm:intro3}.

\subsection{Modified pointed entropy}

First, we generalize the pointed $\WW$-entropy in \eqref{defWentropy} of Definition \ref{defnentropy} to noncollapsed Ricci flow limit spaces.

\begin{defn}[Pointed $\WW$-entropy]
	For $z\in Z_{\III^{-}}$, we define the pointed $\WW$-entropy at $z$ by
		\begin{equation*}\index{$\WW_z(\tau)$}
			\WW_z(\tau):=\int_{\RR_{\t(z)-\tau}}\tau\left(2\Delta f_z-|\na f_z|^2+\scal_{g^Z}\right)+f_z-n\,\mathrm{d}\nu_{z;\t(z)-\tau},
		\end{equation*}
		for any $\tau>0$ such that $\t(z)-\tau \in \III^-$.
\end{defn}

Next, we define the following set, which consists of the time at which the integration by parts fails.

\begin{defn}\label{defndiscrete}
	For $z\in Z_{\III^{-}}$, we define $J^z \subset (0, \t(z)+0.98T)$\index{$J^z$} to be the set consisting of $\tau$ such that
			\begin{equation*}
\int_{\RR_{\t(z)-\tau}} \Delta f_z-|\na f_z|^2 \,\mathrm{d}\nu_{z;\t(z)-\tau} \ne 0.
		\end{equation*}
\end{defn}

Next, we investigate the continuity of $\WW$ under convergence.

\begin{prop}\label{prop:Wconv1}
Given $z \in Z_{\III^-}$, $J^z$ is a measure zero set. Moreover, for any sequence $z_l^* \in M_l \times \III$ converging to $z$ and any $\tau \in (0, \t(z)+0.98T) \setminus J^z$, we have
	\begin{align*}
		\lim_{l \to \infty} \WW_{z_l^*}\left(\t_l(z_l^*)-\t(z)+\tau\right) =\WW_{z} (\tau).
	\end{align*}
\end{prop}

\begin{proof}
For simplicity, we assume $\t_l(z_l^*)=\t(z)=0$, and the general case is similar. We set $w_{z_l^*}=\tau\lc 2\Delta f_{z_l^*}-|\na f_{z_l^*}|^2+\scal_{g_l}\rc +f_{z_l^*}-n$ and $w_z=\tau\left(2\Delta f_z-|\na f_z|^2+\scal_{g^Z} \right)+f_z-n$. Note that by Perelman's differential Harnack inequality and the smooth convergence in Theorem \ref{thm:intro3}, we have $w_{z_l^*} \le 0$ and $w_z \le 0$.

By smooth convergence again, we have
\begin{align}\label{Wconveq1}
	\limsup_{l \to \infty}\WW_{z_l^*}(\tau)=\limsup_{l \to \infty}\int_{M_l} w_{z_l^*}\, \mathrm{d}\nu_{z_l^*;-\tau}\leq \int _{\RR_{-\tau}}w_z\,\mathrm{d}\nu_{z;-\tau}.
\end{align}
Indeed, for any compact set $D \subset \RR_{-\tau}$, it follows from Theorem \ref{thm:intro3} that for the diffeomorphism $\phi_l$ therein, we have
\begin{align*}
\lim_{l \to \infty}\int_{\phi_l(D)} w_{z_l^*}\, \mathrm{d}\nu_{z_l^*;-\tau}= \int _{D}w_z\,\mathrm{d}\nu_{z;-\tau}.
\end{align*}
Thus, \eqref{Wconveq1} holds after exhausting $\RR_{-\tau}$ by such a compact set $D$.

On the other hand, since both $\scal_{g_l}$ and $\scal_{g^Z}$ are uniformly bounded below, we conclude as \eqref{Wconveq1} that
\begin{align*}
\liminf_{l \to \infty}\int_{M_l} \tau\lc |\na f_{z_l^*}|^2+\scal_{g_l}\rc\, \mathrm{d}\nu_{z_l^*;-\tau}\ge \int_{\RR_{-\tau}} \tau\lc |\na f_{z}|^2+\scal\rc \,\mathrm{d}\nu_{z;-\tau}.
\end{align*}
Combining this with the continuity of the Nash entropy (see \cite[Lemma 7.2]{fang2025RFlimit}), we conclude that
\begin{align}\label{Wconveq2}
 	\liminf_{l \to \infty}\WW_{z_l^*}(\tau)&=\liminf_{l\to\infty}\int_{M_l} \tau\lc |\na f_{z_l^*}|^2+\scal\rc +f_{z_l^*}-n\, \mathrm{d}\nu_{z_l^*;-\tau}\geq \int_{\RR_{-\tau}}\tau\left(|\na f_z|^2+\scal\right)+f-n\,\mathrm{d}\nu_{z;-\tau}.
 \end{align}
 By \eqref{Wconveq1} and \eqref{Wconveq2}, it holds that
 \begin{align}\label{Wconveq11}
 	\int_{\RR_{-\tau}}\tau\left(|\na f_z|^2+\scal\right)+f-n\,\mathrm{d}\nu_{z;-\tau}\leq \int _{\RR_{-\tau}} w_z\,\mathrm{d}\nu_{z;-\tau}.
 \end{align}

Now, we consider the cutoff functions $\eta_{r, A}$ based at $z$ from \cite[Proposition 8.20]{fang2025RFlimit}. For any $0<\tau_1 \le \tau_2 < \t(z)+0.98T$, we have
\begin{align}\label{Wconveq7}
 &\left|\int_{\tau_1}^{\tau_2}\int_{\RR_{-\tau}}\left(\tau\big(|\na f_z|^2+\scal\big)+f-n\right)(1-\eta_{r,A})\,\mathrm{d}\nu_{z;-\tau}\mathrm{d}\tau\right|\nonumber\\
 \leq& C(n)\lc \int_{\tau_1}^{\tau_2}\int_{\supp(\eta_{r,A})}\tau^2\left(|\na f_z|^4+\scal^2\right)+f^2+1\,\mathrm{d}\nu_{z;-\tau}\mathrm{d}\tau\rc^{1/2}\lc\int_{\tau_1}^{\tau_2}\int_{\supp(\eta_{r,A})}\left|1-\eta_{r,A}\right|^2\,\mathrm{d}\nu_{z;-\tau}\mathrm{d}\tau\rc^{1/2}\nonumber\\
 &+\left|\int_{\tau_1}^{\tau_2}\int_{\RR_{-\tau}\setminus \supp(\eta_{r,A})}\left(\tau\big(|\na f_z|^2+\scal\big)+f-n\right)(1-\eta_{r,A})\,\mathrm{d}\nu_{z;-\tau}\mathrm{d}\tau\right| \notag\\
 \leq &C(n,Y, A, \tau_2) r+C(n, Y, \tau_2) \Psi(A^{-1}),
 \end{align}
 where we used Proposition \ref{integralbound} by taking the limit, \cite[Proposition 8.20 (4)]{fang2025RFlimit} and \cite[Proposition 3.1]{FLloja05}. Similarly, we have
 \begin{align}\label{Wconveq8}
 	\left|\int_{\tau_1}^{\tau_2}\int_{\RR_{-\tau}}w_z(1-\eta_{r,A})\,\mathrm{d}\nu_{z;-\tau}\mathrm{d}\tau\right|\leq C(n,Y, A, \tau_2) r+C(n, Y,  \tau_2) \Psi(A^{-1}).
 \end{align}
Using integration by parts, it follows that
 \begin{align}\label{Wconveq9}
 	&\left|\int_{\tau_1}^{\tau_2}\int_{\RR_{-\tau}}\left(\tau\big(|\na f_z|^2+\scal\big)+f-n\right)\eta_{r,A}\,\mathrm{d}\nu_{z;-\tau}\mathrm{d}\tau-\int_{\tau_1}^{\tau_2}\int_{\RR_{-\tau}}w_z \eta_{r,A} \,\mathrm{d}\nu_{z;-\tau}\mathrm{d}\tau\right|\nonumber\\
 	\leq & 2 \tau_2 \int_{\tau_1}^{\tau_2}\int_{\RR_{-\tau}}|\na f_z||\na \eta_{r,A}|\,\mathrm{d}\nu_{z;-\tau}\mathrm{d}\tau\nonumber\\
 	\leq & C(n, Y, A) r^{-1} \tau_2 \lc \int_{\tau_1}^{\tau_2}\int_{\RR_{-\tau}}|\na f_z|^2\,\mathrm{d}\nu_{z;-\tau}\mathrm{d}\tau\rc^{1/2}\lc \int_{\tau_1}^{\tau_2}\int_{\{0<\eta_{r,A}<1\}} 1 \,\mathrm{d}\nu_{z;-\tau}\mathrm{d}\tau\rc^{1/2}\nonumber\\
 	\leq &C(n, Y, A,  \tau_2) r^{1/2}.
 \end{align}
 Combining \eqref{Wconveq7}, \eqref{Wconveq8} and \eqref{Wconveq9}, we first let $r \to 0$, followed by $A \to +\infty$, to obtain
 \begin{align}\label{Wconveq10}
 	\int_{\tau_1}^{\tau_2}\int_{\RR_{-\tau}}\tau\big(|\na f_z|^2+\scal\big)+f-n\,\mathrm{d}\nu_{z;-\tau}\mathrm{d}\tau=\int_{\tau_1}^{\tau_2}\int_{\RR_{-\tau}}w_z\,\mathrm{d}\nu_{z;-\tau}\mathrm{d}\tau.
 \end{align}
Thus, it follows from Definition \ref{defndiscrete}, \eqref{Wconveq11} and \eqref{Wconveq10} that $J^z$ has measure zero.

 Finally, by \eqref{Wconveq1} and \eqref{Wconveq2}, we conclude that for $\tau\in (0, \t(z)+0.98T) \setminus J^z$, 
 \begin{align*}
 \lim_{l \to \infty}\WW_{z_l^*}(\tau)=\int_{\RR_{-\tau}}\tau\left(|\na f_z|^2+\scal\right)+f-n\,\mathrm{d}\nu_{z;-\tau}= \int _{\RR_{-\tau}}w_z\,\mathrm{d}\nu_{z;-\tau}.
 \end{align*}
 This completes the proof.
\end{proof}

By Proposition \ref{prop:Wconv1}, we know that $\tau\mapsto \WW_z(\tau)$ is nonincreasing for $\tau \notin J^z$. For this reason, we make the following definition: 

\begin{defn}[Modified pointed $\widetilde \WW$-entropy]\label{defnWRFlimit1}
	For $z\in Z_{\III^{-}}$, the modified $\widetilde \WW$-entropy at $z$ is defined as
	\begin{align*}
		\widetilde\WW_z(\tau):=
		\begin{dcases}
		 \WW_z(\tau), \quad &\text{if} \quad \tau \notin J^z \\
		 \lim_{\tau_j \notin J^z \nearrow \tau} \WW_z(\tau_j), \quad &\text{if} \quad \tau \in J^z.
		\end{dcases}
	\end{align*}\index{$\widetilde\WW_z(\tau)$}
It is clear that $\widetilde\WW_z(\tau)$ is nonincreasing for all $\tau>0$, and $\WW_z(\tau)=\widetilde\WW_z(\tau)$ for $\tau \notin J^z$.
\end{defn}

\begin{lem}\label{lem:monolimit}
For $z\in Z_{\III^{-}}$ and $0<\tau_1\le \tau_2<\t(z)+0.98T$, we have
 \begin{align*}
\int_{\t(z)-\tau_2}^{\t(z)-\tau_1} \int_{\RR_t} 2\tau \abs{\Ric(g^Z)+\na^2 f_z-\frac{g^Z}{2\tau}}^2 \,\mathrm{d}\nu_{z;t}\mathrm{d}t \le \widetilde\WW_z(\tau_1)-\widetilde\WW_z(\tau_2),
 \end{align*}
where $\tau=\t(z)-t$.
\end{lem}

\begin{proof}
Take two sequences $\tau_1^j \nearrow \tau_1$ and $\tau_2^j \nearrow \tau_2$ such that neither $\tau_1^j$ nor $\tau_2^j$ belongs to $J^z$. By the monotonicity formula (see Proposition \ref{propNashentropy} (iii)), we have
 \begin{align*}
\int_{\t(z)-\tau^j_2}^{\t(z)-\tau^j_1} \int_{M_l} 2\tau_l \abs{\Ric(g_l)+\na^2 f_{z_l^*}-\frac{g_l}{2\tau_l}}^2 \,\mathrm{d}\nu_{z_l^*;t}\mathrm{d}t = \WW_{z_l^*}\left(\t_l(z_l^*)-\t(z)+\tau_1^j\right)-\WW_{z_l^*}\left(\t_l(z_l^*)-\t(z)+\tau_2^j\right),
 \end{align*}
where $\tau_l=\t_l(z_l^*)-t$, and $z_l^*$ is a sequence in $\XX^l$ converging to $z$. Letting $l \to \infty$, we have by the smooth convergence and Proposition \ref{prop:Wconv1} that
 \begin{align*}
\int_{\t(z)-\tau^j_2}^{\t(z)-\tau^j_1} \int_{\RR_t} 2\tau \abs{\Ric(g^Z)+\na^2 f_z-\frac{g^Z}{2\tau}}^2 \,\mathrm{d}\nu_{z;t} \mathrm{d}t \le \WW_{z}(\tau_1^j)-\WW_{z}(\tau_2^j).
 \end{align*}
Taking $j \to \infty$, the conclusion follows from the definition of $\widetilde \WW$.
\end{proof}

Next, we have the following definition that generalizes Definition \ref{defnselfsimilarity}.

\begin{defn}[$(\delta,r)$-selfsimilar]\label{defnselfsimilarity1}
	A point $z \in Z_{\III^-}$ is called \textbf{$(\delta,r)$-selfsimilar}\index{$(\delta,r)$-selfsimilar} if $\t(z)-\delta^{-1} r^2 \in \III^-$ and 
	\begin{align*}
	\widetilde 	\WW_{z}(\delta r^2)-\widetilde \WW_{z}(\delta^{-1}r^2) \leq \delta.
	\end{align*}
\end{defn}

\subsection{Sharp splitting maps and Hessian decay on Ricci flow limit spaces}

Next, we give the definition of strongly $(k,\alpha,\delta,r)$-independent points on $Z$, similar to Definition \ref{defnindependentpoints}:

\begin{defn}[Strongly $(k,\alpha,\delta,r)$-independent points]\label{defnindependentpointsRFlimit}\index{strongly $(k,\alpha,\delta,r)$-independent points}
Given constants $\alpha\in (0, 1)$, $\delta>0$, $r>0$, $k \in \{1, \ldots, n\}$ and a Ricci flow limit space $Z$ with a $(\bar \delta/2, r)$-selfsimilar point $z_0$, a set of points $\{z_i\}_{1\leq i\leq k}$ is called \textbf{strongly $(k,\alpha,\delta,r)$-independent} at $z_0$, if the following conditions hold.
	\begin{enumerate}[label=\textnormal{(\roman{*})}]	
		\item $\displaystyle \widetilde \WW_{z_i}(r^2/40)-\widetilde \WW_{z_i}(40 r^2) \le  \delta$ for all $i \in \{0,\ldots,k\}$.
		
		\item $\displaystyle d_Z (z_i, z_0) \le r/2$ for all $i \in \{1,\ldots,k\}$.

		\item $\displaystyle |\t(z_0)-\t(z_i)| \le \frac{r^2}{200}$ for all $i \in \{1,\ldots,k\}$.
		\item Set $f_i=f_{z_i}$, $\tau_i=\t(z_i)-\t$ and $F_i=\tau_if_i$. For $h_i=r^{-1}(F_i-F_0)$ and the symmetric matrix
		\begin{align*}
		A=(a_{ij}), \quad \text{where} \quad			a_{ij}:=\aint_{\t(z_0)-10r^2}^{\t(z_0)-r^2/10} \int_{\RR_t} \la\nabla h_i ,\nabla h_j\ra \,\mathrm{d}\nu_{z_0;t} \mathrm{d}t,
 		\end{align*}
the first eigenvalue $\lambda_1(A)$ of $A$ satisfies
 				\begin{align*}
			\lambda_1(A) \geq \alpha^2.
		\end{align*}
	\end{enumerate}
\end{defn}

With Definition \ref{defnindependentpointsRFlimit}, we can now give the definition of the entropy pinching in Ricci flow limit spaces similar to Definition \ref{defnentropypinching}. 

\begin{defn}[Strongly entropy pinching] 
	For $k\in \{1,2, \ldots,n\}$, $\alpha \in (0, 1)$, $\delta>0$ and $r>0$, the \textbf{strongly $(k,\alpha,\delta,r)$-entropy pinching at $z_0$} is defined as
	\begin{equation*}\index{$\mathfrak{S}_r^{k,\alpha,\delta}$}
		\mathfrak{S}_r^{k,\alpha,\delta}(z_0):= \inf_{\{z_i\}_{1 \le i \le k} \,\mathrm{are \,strongly}\, (k,\alpha,\delta,r)-\mathrm{independent}\ \mathrm{at}\ z_0}\sum_{i=0}^k \left(\widetilde \WW_{z_i}(r^2/40)-\widetilde \WW_{z_i}(40r^2)\right)^{\frac{1}{2}}.
	\end{equation*}
\end{defn}

Let $\{z_i\}_{1\leq i\leq k}$ be strongly $(k,\alpha,\delta,r)$-independent points at $z_0$ and assume that $x_{i,l}^*\in\XX^l$ converges to $z_i$ in the Gromov--Hausdorff sense. Set $\nu^l=\nu_{x_{0,l}^*}$, $t^l=\t_l(x_{0,l}^*)$, $\tau_i^l=\t_l(x_{i,l}^*)-t$, $f^l_i=f_{x_{i,l}^*}$, $F^l_i=\tau_i^l f^l_i$ and $h^l_i=r^{-1}(F^l_i-F^l_0)$.

By Theorem \ref{thm:intro3}, we have
\begin{align*}
	f^l_i \xrightarrow[l\to\infty]{\quad C_{\loc}^\infty(\RR)\quad}f_i.
\end{align*}
We define $A^l=(a_{ij}^l)$ as
\begin{align*}
	a_{ij}^l:=\aint_{t^l-10r^2}^{t^l-r^2/10}\int_{M_l}\la \na h^l_i, \na h^l_j\ra \,\mathrm{d}\nu^l_t \mathrm{d}t.
\end{align*}

\begin{lem}\label{eigenconvRF}
For any $i, j \in \{1, \ldots, k\}$,
  \begin{align*}
  	a_{ij}^l\xrightarrow[l\to\infty]{}a_{ij},\quad \lambda_1(A^l)\xrightarrow[l\to\infty]{}\lambda_1(A).
  \end{align*}
\end{lem}
\begin{proof}
Without loss of generality, we assume $\t(z_0)=\t_l(x_{0, l}^*)=0$ and $r=1$. According to Proposition \ref{integralbound} and \cite[Proposition A.2]{fang2025RFlimit}, we obtain that for any $1\leq i\leq k$,
\begin{align}\label{eigenconvRF1}
	\int_{-11}^{-1/11}\int_{M_l}\left|\na h^l_i\right|^4 \,\mathrm{d}\nu^l_t \mathrm{d}t \leq C(n,Y).
\end{align}
We consider the family of cutoff functions $\{\eta_{r, L}\}$ from \cite[Proposition 8.20]{fang2025RFlimit} based at $z_0$. Set $B^l_{r,L}:=\phi_l\lc \supp\big(\eta_{r,L}\big)\cap \RR_{[-10, -1/10]}\rc$, which is well defined for fixed $r$ and $L$, provided that $l$ is sufficiently large.

By the smooth convergence on $\RR$, it follows that
	\begin{align}\label{eigenconvRF2}
		\lim_{l\to\infty}\iint_{B^l_{r,L}} \la \na h^l_i, \na h^l_j \ra\eta_{r,L}\circ \phi_l^{-1}\,\mathrm{d}\nu^l_t \mathrm{d}t=\int_{\RR_{[-10, -1/10]}} \la \na h_i, \na h_j \ra\eta_{r,L}\, \mathrm{d}\nu_{z_0;t} \mathrm{d}t.
	\end{align}
	
	On the one hand, it follows from \cite[Theorem 1.5]{fang2025RFlimit} that once $r$ and $L$ are fixed, 
		\begin{align*}
B^*(x_{0,l}^*;L/3) \cap \phi_l\lc \supp\big(\eta_{r,L}\big)\rc  \subset B^l_{r,L}\subset B^*(x_{0,l}^*; 3 L)
	\end{align*}
	for sufficiently large $l$. Thus, it is clear that
		\begin{align*}
		\iint_{M_l \times [-10, -1/10] \setminus B_{r, L}^l} \abs{\la \na h^l_i, \na h^l_j \ra} \, \mathrm{d}\nu^l_t \mathrm{d}t \le & \iint_{M_l \times [-10, -1/10] \setminus B^*(x_{0,l}^*;L/3)}\abs{\la \na h^l_i, \na h^l_j \ra} \, \mathrm{d}\nu^l_t \mathrm{d}t \\
		+&\iint_{B^*(x_{0,l}^*;L/3)\cap M_l \times [-10, -1/10] \setminus \phi_l\lc \supp(\eta_{r,L})\rc}\abs{\la \na h^l_i, \na h^l_j \ra} \, \mathrm{d}\nu^l_t \mathrm{d}t.
	\end{align*}	
	
By Proposition \ref{existenceHncenter} and \eqref{eigenconvRF1}, we have
		\begin{align}\label{eigenconvRF1a}
&\iint_{M_l \times [-10, -1/10] \setminus B^*(x_{0,l}^*;L/3)}\abs{\la \na h^l_i, \na h^l_j \ra} \, \mathrm{d}\nu^l_t \mathrm{d}t \notag \\
\le &  \lc \iint_{M_l \times [-10, -1/10] \setminus B^*(x_{0,l}^*;L/3)}| \na h^l_i|^2 |\na h^l_j|^2 \, \mathrm{d}\nu^l_t \mathrm{d}t \rc^{\frac 1 2} \lc \iint_{M_l \times [-10, -1/10] \setminus B^*(x_{0,l}^*;L/3)}  1\, \mathrm{d}\nu^l_t \mathrm{d}t \rc^{\frac 1 2} \notag\\
 \le & C(n, Y) \Psi(L^{-1}).
	\end{align}	
By \cite[Proposition 8.20 (2)]{fang2025RFlimit}, we know that any $x \in  B^*(x_{0,l}^*;L/3) \setminus \phi_l\lc \supp(\eta_{r,L})\rc $ satisfies $r_{\Rm}(x) \le 3r$ for all sufficiently large $l$. Thus, by using \cite[Theorem 1.12 (b)]{fang2025RFlimit}, we obtain
		\begin{align}\label{eigenconvRF1b}
&\iint_{B^*(x_{0,l}^*;L/3)\cap M_l \times [-10, -1/10] \setminus \phi_l\lc \supp(\eta_{r,L})\rc}\abs{\la \na h^l_i, \na h^l_j \ra} \, \mathrm{d}\nu^l_t \mathrm{d}t \notag \\
\le & C(n, Y)\lc\iint_{B^*(x_{0,l}^*;L/3)\cap M_l \times [-10, -1/10] \setminus \phi_l\lc \supp(\eta_{r,L})\rc} 1 \, \mathrm{d}\nu^l_t \mathrm{d}t \rc^{\frac{1}{2}} \le C(n, L, Y) r.
	\end{align}	
Combining \eqref{eigenconvRF1a} and \eqref{eigenconvRF1b}, we obtain
		\begin{align}\label{eigenconvRF1c}
		\iint_{M_l \times [-10, -1/10] \setminus B_{r, L}^l}\abs{\la \na h^l_i, \na h^l_j \ra} \, \mathrm{d}\nu^l_t \mathrm{d}t \le C(n, Y)\Psi(L^{-1})+C(n, L, Y) r.
	\end{align}	

Similarly to \eqref{eigenconvRF1b}, for sufficiently large $l$, we have
	\begin{align}\label{eigenconvRF1d}
		&\left|\iint_{B_{r, L}^l}\abs{\la \na h^l_i, \na h^l_j \ra} \left(1-\eta_{r,L}\circ \phi_l^{-1}\right) \,\, \mathrm{d}\nu^l_t \mathrm{d}t\right|\nonumber\\
		\leq &  \lc\iint_{B_{r, L}^l}| \na h^l_i|^2 |\na h^l_j|^2 \, \mathrm{d}\nu^l_t \mathrm{d}t\rc^{\frac{1}{2}} \lc\iint_{B_{r, L}^l} \left(1-\eta_{r,L}\circ \phi_l^{-1}\right)^{2}\, \mathrm{d}\nu^l_t \mathrm{d}t\rc^{\frac{1}{2}} \le  C(n, L, Y) r,
	\end{align}
	where the last inequality holds by the same reason as in \eqref{eigenconvRF1b}.
	
	Thus, we can first choose a large $L$ and then a small $r$ so that all integrals \eqref{eigenconvRF1c} and \eqref{eigenconvRF1d} are as small as we want. Combining this fact with \eqref{eigenconvRF2}, we obtain $a_{ij}^l\xrightarrow[l\to\infty]{}a_{ij}$. The convergence of eigenvalues follows directly.
\end{proof}

As a corollary of Lemma \ref{eigenconvRF} by using Proposition \ref{prop:Wconv1}, we have
\begin{cor}\label{indeptsconv}
For any $\ep>0$, $\{x_{i,l}^*\}_{1\leq i\leq k}$ is a set of $(k,(1-\ep)\alpha,(1+\ep)\delta,r)$-independent points at $x_{0,l}^*$, provided that $l$ is sufficiently large.
\end{cor}

Next, we give the following characterization of the existence of a strongly $(k,\alpha,\delta,r)$-independent set at a point.

\begin{lem}\label{lem:indechar}
Suppose $z_0 \in Z_{\III^-}$ is $(\delta,r)$-selfsimilar and $(k, \delta, r)$-splitting. For any $\ep>0$, if $\delta \le \delta(n, Y, \ep)$, then there exists a strongly $ (k,1/4-\ep,\ep,r)$-independent set at $z_0$.
\end{lem}

\begin{proof}
Without loss of generality, we assume $r=1$ and $\t(z_0)=0$.

Suppose that the conclusion fails. Then there exists a sequence of noncollapsed Ricci flow limit spaces $Z^l$ over $\III_l$, obtained as the limit of a sequence of closed Ricci flows in $\MM(n, Y, T_l)$, where $\III_l=[-0.98T_l, 0]$. Moreover, there exists $z_{0,l} \in Z^l_0 $ that is $(l^{-2},1)$-selfsimilar and $(k, l^{-2}, 1)$-splitting. However, there exists no strongly $ (k,1/4-\ep,\ep,1)$-independent set at $z_{0,l}$ for some $\ep>0$.

By passing to a subsequence, we obtain the convergence
	\begin{align*}
		(Z^l, d_{Z^l}, z_{0,l},\t_l) \xrightarrow[l \to \infty]{\quad \hat C^\infty \quad} (Z, d_Z, z_0,\t),
	\end{align*}
where $(Z, d_Z, z_0,\t)$ is a Ricci shrinker space. Moreover, for the regular part $(\RR, \t, \partial_\t, g^Z)$ of $Z$, $\RR_t$ is connected for any $t<0$, and the Ricci shrinker equation is satisfied on $\RR_{(-\infty, 0)}$:
	\begin{align*}
\Ric(g^Z)+\na^2 f_z=\frac{g^Z}{2\tau},
	\end{align*}
where $\tau=-\t$. By our assumption, $\RR$ splits off an $\R^k$ and hence it follows from \cite[Proposition 8.4]{fang2025RFlimit} that there exists an isometry $\boldsymbol{\phi}^s$ for $s \in \R^k$ on $Z_{(-\infty,0]}$. Now, we set $z_i=\boldsymbol{\phi}^{e_i}(z_0)$, where $(e_1, \ldots, e_k)$ is the standard basis of $\R^k$. Then, by \cite[Proposition 8.6]{fang2025RFlimit}, we have
		\begin{align} \label{eq:ksplita}
		\aint_{-10}^{-1/10} \int_{\RR_t} \la\nabla h_i ,\nabla h_j\ra \,\mathrm{d}\nu_{z;t} \mathrm{d}t=\frac{\delta_{ij}}{4},
 		\end{align}
for $i,j \in \{1, \ldots, k\}$, where $h_i=\tau(f_{z_i}-f_{z_0})$.

We choose $z_{i,l} \in Z^l$ converging to $z_i$. On the one hand, it follows from Proposition \ref{prop:Wconv1} that $z_{i,l}$ is $(\ep,1)$-selfsimilar for any $0 \le i \le k$. On the other hand, by using Proposition \ref{integralbound}, it follows from the same argument as in the proof of Lemma \ref{eigenconvRF} and \eqref{eq:ksplita} that
		\begin{align*} 
		\aint_{-10}^{-1/10} \int_{\RR^l_t} \la\nabla h^l_i ,\nabla h^l_j\ra \,\mathrm{d}\nu_{z_{0,l};t} \mathrm{d}t \xrightarrow[l\to\infty]{} \frac{\delta_{ij}}{4},
 		\end{align*}
where $\RR^l$ is the regular part of $Z^l$, and $h^l_i=\tau(f_{z_{i,l}}-f_{z_{0,l}})$. Thus, it is clear that $\{z_{i,l}\}_{1\le i \le k}$ is a strongly $ (k,1/4-\ep,\ep,1)$-independent set at $z_{0,l}$, which contradicts our assumption.

Consequently, the proof is complete.
\end{proof}

By taking the limit of $(k,\ep,r)$-splitting maps at $x_0^*$ and taking the limit of $\WW$-entropy together with Definition \ref{defnWRFlimit1}, the following counterpart of Corollary \ref{corsharpsplitting} still holds on Ricci flow limit spaces. 

\begin{thm}[Sharp splitting map on Ricci flow limit spaces]\label{thmsharpsplittingRFlimit}
	Let $(Z,d_Z,\t)$ be a Ricci flow limit space with $z_0 \in Z_{\III^-}$ such that $z_0$ is $(\delta,r)$-selfsimilar and $(k, \delta, r)$-splitting. For any $\ep>0$ and $\alpha \in (0, 1/5)$, if $\ep \le \ep(n, Y, \alpha)$ and $\delta \le \delta(n, Y, \ep)$, then there exists a $(k,\ep',r)$-splitting map $\vec u=(u_1,\ldots, u_k):Z_{(\t(z_0)-10r^2,\t(z_0)]}\to \R^k$ at $z_0$, where
	\begin{align*}
\ep'= C(n,Y , \alpha)\mathfrak{S}_r^{k,\alpha,\ep}(z_0).
	\end{align*}
\end{thm}

\begin{proof}
Since $\delta \le \delta(n, Y, \ep)$, it follows from Lemma \ref{lem:indechar} that there exist points $\{z_i\}_{1 \le i \le k}$ which are strongly $(k, \alpha, \ep, r)$-independent points at $z_0$ such that
	\begin{align*}
\sum_{i=0}^k \left(\widetilde \WW_{z_i}(r^2/40)-\widetilde \WW_{z_i}(40r^2)\right)^{\frac{1}{2}} \le 2\mathfrak{S}_r^{k,\alpha, \ep}(z_0).
	\end{align*}
Assume that $x_{i,l}^* \in \XX^l$ converge to $z_i$ for $0 \le i \le k$. By Corollary \ref{indeptsconv}, we know that $\{x_{i,l}^*\}_{1 \le i \le k}$ are $(k, \alpha/2, 2 \ep, r)$-independent at $x_{0,l}^*$, for sufficiently large $l$. By Proposition \ref{constructionsplitting2}, if $\ep \le \ep(n, Y, \alpha)$, we can find a $(k,\ep_l,r)$-splitting map $\vec u^l=(u^l_1,\ldots, u^l_k):M_l \times [\t_l(x_{0,l}^*)-10r^2,\t_l(x_{0,l}^*)]\to \R^k$ at $x_{0,l}^*$, where
	\begin{align*}
\ep_l=C(n,Y, \alpha)\sum_{i=0}^k\lc\mathcal{W}_{x_{i,l}^*}(r^2/30)-\mathcal{W}_{x_{i,l}^*}(30r^2)\rc^{\frac{1}{2}}.
	\end{align*}
In particular, we have
	\begin{align} \label{eq:sharp1}
		\sum_{i=1}^k\int^{\t_l(x_{0,l}^*)-r^2/10}_{\t_l(x_{0,l}^*)-10r^2}\int_{M_l} \left|\nabla^2 u^l_i\right|^2 \,\mathrm{d}\nu_{x_{0,l}^*;t}\mathrm{d}t\leq C(n,Y, \alpha)\sum_{i=0}^k\lc\mathcal{W}_{x_{i,l}^*}(r^2/30)-\mathcal{W}_{x_{i,l}^*}(30r^2)\rc^{\frac{1}{2}}.
	\end{align}
Taking the limit for $l \to \infty$ in \eqref{eq:sharp1} and using Proposition \ref{prop:Wconv1} and Definition \ref{defnWRFlimit1}, we conclude that $\vec u^l$ converge to a $(k, \ep', r)$-splitting map $\vec u=(u_1, \ldots, u_k)$ at $z_0$, where
	\begin{align*}
\ep'=C(n,Y, \alpha)\sum_{i=0}^k\lc\widetilde{\mathcal{W}}_{z_i}(r^2/40)-\widetilde{\mathcal{W}}_{z_i}(40r^2)\rc^{\frac{1}{2}}=C(n,Y, \alpha)\mathfrak{S}_r^{k,\alpha, \ep}(z_0).
	\end{align*}
This completes the proof.	
\end{proof}

Next, we have the following definition.

\begin{defn}[Limiting heat flow]\label{defnlimitheatflow}\index{limiting heat flow}
Given a positive constant $L$, a map $h : Z_{(\t(z)-L, \t(z)]} \to \R$ is called a \textbf{limiting heat flow at $z \in Z_{\III^-}$} if $h \vert_{\RR_{(\t(z)-L, \t(z)]}}$ is a smooth limit of a sequence $h^l: M_l \times [\t_l(x_l^*)-L,\t_l(x_l^*)] \to \R$ such that $\square h^l=0$ and
	\begin{align*}
 \int_{M_l} |\na h^l|^2 \,\mathrm{d}\nu_{x_l^*;\t_l(x_l^*)-L} \le C
	\end{align*}
	for a sequence $x_l^* \in \XX^l$ converging to $z$ and a constant $C<\infty$. In particular, any component of a $(k, \ep, r)$-splitting map is a limiting heat flow.
\end{defn}

\begin{thm}[Hessian decay of limiting heat flows]\label{hessiandecayRFlimit}
Let $\eta>0$ be a fixed constant. Suppose that $z \in Z_{\III^-}$ is $(\delta,r)$-selfsimilar and is not $(k+1,\eta, r)$-splitting.
	
	Let $\vec u=(u_1,\ldots,u_k):Z_{(\t(z)-10r^2,\t(z))}\to\R^k$ be a $(k,\delta,r)$-splitting map at $z$, and $h: Z_{(\t(z)-10r^2,\t(z))}\to\R$ be a limiting heat flow. Define
	\begin{align*}
		v:= h-\sum_{i=1}^k b_iu_i, \quad \text{where} \quad b_i:= \aint_{\t(z)-8r^2}^{\t(z)-8r^2/10}\int_{\RR_t} \la \nabla h,\nabla u_i\ra \,\mathrm{d}\nu_{z;t} \mathrm{d}t.
	\end{align*}
	Then there exists a constant $\theta=\theta(n,Y,\eta)\in (0,1)$ such that the following holds: if $\delta\leq\delta(n,Y,\eta)$, then
	\begin{equation*}
		\int_{\t(z)-2r^2}^{\t(z)-2r^2/10}\int_{\RR_t}\left|\nabla^2 v\right|^2\,\mathrm{d}\nu_{z;t} \mathrm{d}t\leq \theta \int_{\t(z)-8r^2}^{\t(z)-8r^2/10}\int_{\RR_t}\left|\nabla^2 v\right|^2\,\mathrm{d}\nu_{z;t} \mathrm{d}t.
	\end{equation*}
\end{thm}

\begin{proof}
By Definition \ref{defnsplittingmap1}, $\vec{u}$ is obtained as the limit of $\vec u^l=(u^l_1,\ldots,u^l_k)$, which is a $(k,\delta,r)$-splitting map at $x_l^* \in \XX^l$, where $x_l^*$ converge to $z$.

Moreover, by Definition \ref{defnlimitheatflow}, there exist $h^l:M_l\times [\t_l(x_l^*)-10r^2,\t_l(x_l^*)]\to\R$ satisfying
\begin{align*}
	\square h^l=0,\quad \int_{M_l}\left|\na h^l\right|^2\,\mathrm{d}\nu_{x_l^*;\t_l(x_l^*)-10r^2}\leq C_0,\quad h^l\xrightarrow[l\to\infty]{\quad C_{\loc}^\infty(\RR)\quad}h.
\end{align*}

Define $v^l$ and $b^l_i$ by
\begin{align*}
	v^l:= h^l-\sum_{i=1}^k b_i^lu_i^l, \quad \text{where} \quad b_i^l:= \aint_{\t_l(x_l^*)-8r^2}^{\t_l(x_l^*)-8r^2/10}\int_{M_l}\la \nabla h^l,\nabla u_i^l\ra \,\mathrm{d}\nu_{x_l^*;t}\mathrm{d}t.
\end{align*}
Then we can argue as in the proof of Lemma \ref{eigenconvRF} to conclude that
	\begin{align*}
	b_i^l\xrightarrow[l\to\infty]{}b_i,\quad v^l\xrightarrow[l\to\infty]{\quad C_{\loc}^\infty(\RR)\quad}v.
\end{align*}
By a similar argument as in the proof of Claim \ref{takelimit1}, we have, for any $t\geq \t(z)-9r^2$, 
\begin{align*}
	\lim_{l\to\infty}\int_{M_l}\left|\na v^l\right|^2\,\mathrm{d}\nu^l_t=\int_{\RR_t}|\na v|^2\,\mathrm{d}\nu_{z;t}.
\end{align*}
In addition, the same proof leading to Claim \ref{takelimit3} yields that for any $\t(z)-9r^2\leq t_1<t_2< \t(z)$, 
	\begin{align*}
		\int_{\RR_{t_1}}|\na v|^2\,\mathrm{d}\nu_{z;t_1}-\int_{\RR_{t_2}}|\na v|^2\,\mathrm{d}\nu_{z;t_2}=2\int_{t_1}^{t_2}\int_{\RR_t}\left|\na ^2 v\right|^2\,\mathrm{d}\nu_{z;t} \mathrm{d}t.
	\end{align*}

Combining all the above convergence results, the conclusion follows from Theorem \ref{hessiandecay}.
\end{proof}

\subsection{Nondegeneration theorem}

In this subsection, we first investigate the behavior of an almost splitting map when restricted to a smaller scale.

\begin{thm}\label{geometrictransformation}
Given constants $\eta>0$, $\alpha \in (0, 1/5)$, $\ep>0$ and $\ep'>0$, suppose $z_0 \in Z_{\III^-}$ with $t_0:=\t(z_0)$ satisfies that for any $s \in [\bar r r,  r]$, $z_0$ is $(\delta,s)$-selfsimilar and $(k,\delta,s)$-splitting but not $(k+1,\eta,s)$-splitting.
	Let $\vec u=(u_1,\ldots, u_k):Z_{(t_0-10r^2,t_0]}\to\R^k$ be a $(k,\ep',r)$-splitting map at $z_0$. If $\ep \le \ep(n, Y, \alpha)$ and $\delta\leq\delta (n,Y,\eta , \alpha, \ep)$, then we can find $\beta=\beta(n,Y,\eta)>0$ such that for any $s\in [\bar r,1]$, the following Hessian estimate holds:
	\begin{equation}\label{estimategeometrictrans}
		\sum_{i=1}^k\int_{t_0-10s^2r^2}^{t_0-s^2r^2/10}\int_{\RR_t} \left|\nabla^2 u_i\right|^2\,\mathrm{d}\nu_{z_0;t}\mathrm{d}t\leq C(n,Y, \alpha)\lc\sum_{s\leq r_j=2^{-j}\leq 1}\left(\frac{s}{r_j}\right)^{\beta}\mathfrak{S}_{r_jr}^{k,\alpha, \ep}(z_0)+\ep' s^{\beta}\rc.
	\end{equation}
\end{thm}
\begin{proof}
Without loss of generality, we assume $t_0=0$ and $r=1$. 

By Theorem \ref{thmsharpsplittingRFlimit}, for any $j$ with $r_j \in [s, 1]$, we can find a $(k, \ep_j,r_j)$-splitting map $\vec v^j=(v_1^j,\ldots,v_k^j)$ based at $z_0$ which satisfies
	\begin{align*}
		\sum_{i=1}^k\int_{-10r_j^2}^{-r_j^2/10}\int_{\RR_t}\left|\nabla ^2 v_i^j\right|^2\,\mathrm{d}\nu_{z_0;t}\mathrm{d}t\leq \ep_j:=C(n,Y, \alpha)\mathfrak{S}^{k,\alpha,\ep}_{r_j}(z_0).
	\end{align*}

	For any $j$, we define $\vec u^j=(u^j_1,\ldots,u^j_k)$ by
	\begin{align*}
		u^j_{i}:=u_i-\sum_{l=1}^k b^j_{i,l} v^j_l,\quad \mathrm{where} \quad b^j_{i,l}:=\aint_{-8r_j^2}^{-8r_j^2/10}\int_{\RR_t}\la \na u_i,\na v^j_l\ra \,\mathrm{d}\nu_{z_0;t}\mathrm{d}t. 
	\end{align*}
	
	Since $\vec v^j$ and $\vec{u}$ are almost splitting maps, we have by \cite[Proposition 10.2]{fang2025RFlimit},
	\begin{align*}
|b^j_{i,l}|^2 \le C \lc \aint_{-8r_j^2}^{-8r_j^2/10}\int_{\RR_t}|\na u_i|^2 \,\mathrm{d}\nu_{z_0;t}\mathrm{d}t \rc \lc \aint_{-8r_j^2}^{-8r_j^2/10}\int_{\RR_t}|\na v_l^j|^2 \,\mathrm{d}\nu_{z_0;t}\mathrm{d}t \rc  \le C
	\end{align*}
where $C$ is a universal constant.
	
	Since $\vec v^j$ is a $(k, \ep_j,r_j)$-splitting map and $\delta\leq\delta (n,Y,\eta , \alpha, \ep)$, it follows from Theorem \ref{hessiandecayRFlimit}, we obtain for $i \in \{1, \ldots, k\}$,
	\begin{align*}
		\int_{-2r_j^2}^{-2r_j^2/10}\int_{\RR_t}\left|\nabla ^2 u_i^j\right|^2\,\mathrm{d}\nu_{z_0;t}\mathrm{d}t
		\leq \theta\int_{-8r_j^2}^{-8r_j^2/10}\int_{\RR_t}\left|\nabla ^2 u_i^j\right|^2\,\mathrm{d}\nu_{z_0;t}\mathrm{d}t.
	\end{align*}
	Then we have
	\begin{align}\label{iteration1}
		&\int_{-2r_j^2}^{-2r_j^2/10}\int_{\RR_t}\left|\nabla ^2 u_i\right|^2\,\mathrm{d}\nu_{z_0;t}\mathrm{d}t\nonumber\\
		\leq &\int_{-2r_j^2}^{-2r_j^2/10}\int_{\RR_t}\left|\nabla ^2 u_i^j\right|^2\,\mathrm{d}\nu_{z_0;t}\mathrm{d}t+C\sum_{l=1}^k \int_{-2r_j^2}^{-2r_j^2/10}\int_{\RR_t}\left|\nabla ^2 v_l^j\right|^2\,\mathrm{d}\nu_{z_0;t}\mathrm{d}t\nonumber\\
		\leq &\theta\int_{-8r_j^2}^{-8r_j^2/10}\int_{\RR_t}\left|\nabla ^2 u_i^j\right|^2\,\mathrm{d}\nu_{z_0;t}\mathrm{d}t+C\sum_{l=1}^k \int_{-2r_j^2}^{-2r_j^2/10}\int_{\RR_t}\left|\nabla ^2 v_l^j\right|^2\,\mathrm{d}\nu_{z_0;t}\mathrm{d}t\nonumber\\
		\leq &\theta\int_{-8r_j^2}^{-8r_j^2/10}\int_{\RR_t}\left|\nabla ^2 u_i\right|^2\,\mathrm{d}\nu_{z_0;t}\mathrm{d}t+C\lc\sum_{l=1}^k \int_{-8r_j^2}^{-8r_j^2/10}\int_{\RR_t}\left|\nabla ^2 v_l^j\right|^2\,\mathrm{d}\nu_{z_0;t}\mathrm{d}t+\sum_{l=1}^k \int_{-2r_j^2}^{-2r_j^2/10}\int_{\RR_t}\left|\nabla ^2 v_l^j\right|^2\,\mathrm{d}\nu_{z_0;t}\mathrm{d}t\rc\nonumber\\
		\leq &\theta\int_{-8r_j^2}^{-8r_j^2/10}\int_{\RR_t}\left|\nabla ^2 u_i\right|^2\,\mathrm{d}\nu_{z_0;t}\mathrm{d}t+C(n,Y, \alpha)\mathfrak{S}_{r_j}^{k,\alpha,\ep}(z_0).
	\end{align}
	
	We can rewrite \eqref{iteration1} as
	\begin{align}\label{iteration2}
		\sum_{i=1}^k\int_{-8r_{j+1}^2}^{-8r_{j+1}^2/10}\int_{\RR_t}\left|\nabla ^2 u_i\right|^2\,\mathrm{d}\nu_{z_0;t}\mathrm{d}t
		\leq \theta\sum_{i=1}^k\int_{-8r_j^2}^{-8r_j^2/10}\int_{\RR_t}\left|\nabla ^2 u_i\right|^2\,\mathrm{d}\nu_{z_0;t}\mathrm{d}t+C(n,Y, \alpha)\mathfrak{S}_{r_j}^{k,\alpha,\ep}(z_0).
	\end{align}
	Iterating \eqref{iteration2}, it follows that
	\begin{align}\label{iteration3}
		\sum_{i=1}^k\int_{-8r_j^2}^{-8r_j^2/10}\int_{\RR_t}\left|\nabla ^2 u_i\right|^2\,\mathrm{d}\nu_{z_0;t}\mathrm{d}t\leq \theta^j\sum_{i=1}^k \int_{-8r_0^2}^{-8r_0^2/10}\int_{\RR_t}\left|\nabla ^2 u_i\right|^2\,\mathrm{d}\nu_{z_0;t}\mathrm{d}t+C(n,Y, \alpha)\sum_{l=0}^{j-1} \theta^{j-1-l}\mathfrak{S}_{r_l}^{k,\alpha,\ep}(z_0).
	\end{align}
	Note that $r_0=1$ and thus
	\begin{align}\label{iteration4}
		\sum_{i=1}^k\int_{-8r_0^2}^{-8r_0^2/10}\int_{\RR_t}\left|\nabla ^2 u_i\right|^2\,\mathrm{d}\nu_{z_0;t}\mathrm{d}t\leq k \ep'.
	\end{align}  
	Choosing $\beta$ such that $2^{-\beta}=\theta$, we then obtain from \eqref{iteration3} and \eqref{iteration4} that for any $s \in [\bar r, 1]$,
	\begin{align*}
		\sum_{i=1}^k\int_{-10s^2}^{-s^2/10}\int_M\left|\nabla ^2 u_i\right|^2\,\mathrm{d}\nu_{z_0;t}\mathrm{d}t\leq C\ep' s^{\beta}+C(n,Y, \alpha)\sum_{s\leq r_j\leq 1} \lc \frac{s}{r_j} \rc^{\beta}\mathfrak{S}_{r_j}^{k,\alpha,\ep}(z_0).
	\end{align*}
	This gives \eqref{estimategeometrictrans} and thus finishes the proof.
\end{proof}

\begin{thm}[Nondegeneration for almost splitting maps]\label{nondegenrerationthm}
Given constants $\eta>0$, $\alpha \in (0, 1/5)$, $\ep>0$ and $\ep'>0$, suppose $z_0 \in Z_{\III^-}$ with $t_0:=\t(z_0)$ satisfies for any $s \in [\bar r r,  r]$, $z_0$ is $(\delta,s)$-selfsimilar and $(k,\delta,s)$-splitting but not $(k+1,\eta,s)$-splitting.

Let $\vec u=(u_1,\ldots, u_k):Z_{(t_0-10r^2,t_0]}\to\R^k$ be a $(k,\ep',r)$-splitting map at $z_0$. If $\ep \le \ep(n, Y, \alpha)$, $\delta\leq\delta (n,Y,\eta , \alpha, \ep)$, $\ep'\le \ep'(n, Y, \eta, \alpha)$ and
	\begin{equation*}
		\sum_{\bar r \leq r_j=2^{-j}\leq 1}\mathfrak{S}^{k,\alpha,\ep}_{r r_j}(z_0)<\ep',
	\end{equation*}
	then for all $s\in [\bar r,1]$, there exist a constant $C=C(n, Y, \eta ,\alpha)>0$ and a matrix $T_s$ with $\|T_s-\Id \|\leq C \ep' $ such that  the map $\vec u_s:=T_s\vec u:Z_{(t_0-
10s^2r^2,t_0]}\to\R^k$ is a $(k,C \ep',sr)$-splitting map at $z_0$.
\end{thm}
\begin{proof}
Without loss of generality, we assume $t_0=0$ and $r=1$. We only need to prove the theorem for $s=r_j=2^{-j}$ and the general case is similar. Also, we use the same notations as in Theorem \ref{geometrictransformation}. Then it follows that
	\begin{align}\label{nondege1}
		\sum_{a=1}^k\int_{-10r_j^2}^{-r_j^2/10}\int_{\RR_t} \left|\nabla^2 u_a\right|^2\,\mathrm{d}\nu_{z_0;t}\mathrm{d}t\leq C(n,Y,\alpha)\lc\sum_{i=0}^{j}2^{(i-j)\beta}\mathfrak{S}_{r_i}^{k,\alpha,\ep}(z_0)+2^{-j\beta} \ep' \rc.
	\end{align}
	
	\begin{claim}\label{nondege2}
		For any $t'\in [-10r_j^2,-r_j^2/10]$,
		\begin{align*}
			\sum_{1\leq a,b\leq k} \abs{\int_{\RR_{t'}}\la \na u_a,\na u_b\ra -\delta_{ab}\,\mathrm{d}\nu_{z_0;t'}}\leq C(n,Y, \eta, \alpha) \ep'.
		\end{align*}
	\end{claim}
	Note that for any $t\in [-10r_0^2,-r_0^2/10]$, it follows from \cite[Proposition 10.2]{fang2025RFlimit} that
	\begin{align}\label{nondege3}
		\sum_{1\leq a,b\leq k} \abs{\int_{\RR_t}\la \na u_a,\na u_b\ra -\delta_{ab}\,\mathrm{d}\nu_{z_0;t}}\leq C(n) \ep'.
	\end{align}
	Then by \eqref{nondege1} and the same argument in obtaining Claim \ref{takelimit3}, we have
	that for any $t'\in [-10r_i^2,-r_i^2/10]$ and $a,b \in \{1,\ldots,k\}$,
	\begin{align}\label{nondege4}
		&	\left|\int_{\RR_{t'}}\la \na u_a,\na u_b\ra -\delta_{ab}\,\mathrm{d}\nu_{z_0;t'}-\int_{\RR_{-10 r_i^2}}\la \na u_a,\na u_b\ra -\delta_{ab}\,\mathrm{d}\nu_{z_0;-10r_i^2}\right|\nonumber\\
		\leq & 2\int_{-10r_i^2}^{-r_i^2/10}\int_{\RR_t} \left|\na^2 u_a\right|\left|\na^2 u_b\right| \,\mathrm{d}\nu_{z_0;t}\mathrm{d}t\leq C(n,Y, \alpha)\lc\sum_{l=0}^{i}2^{(l-i)\beta}\mathfrak{S}_{r_l}^{k,\alpha,\ep}(z_0)+2^{-i\beta}\ep'\rc.
	\end{align}
	Iterating \eqref{nondege4} and using \eqref{nondege3}, it follows that for any $t'\in [-10r_j^2,-r_j^2/10]$ and $a,b \in \{1,\ldots,k\}$,
	\begin{align}\label{nondege7}
		\left|\int_{\RR_{t'}}\la \na u_a,\na u_b\ra -\delta_{ab}\,\mathrm{d}\nu_{z_0;t'}\right|\leq C(n) \ep'+C(n,Y,\alpha)\sum_{i=0}^j\lc\sum_{l=0}^{i}2^{(l-i)\beta}\mathfrak{S}_{r_l}^{k,\alpha,\ep}(z_0)+2^{-i\beta}\ep'\rc.
	\end{align}
	Note that $\beta=\beta(n,Y,\eta)>0$, thus we have 
	\begin{align}\label{nondege5}
		\sum_{i=0}^j2^{-i\beta}\ep' \leq C(n,Y,\eta)\ep'.
	\end{align}
	
	On the other hand, by assumption, it holds that
	\begin{align}\label{nondege6}
		\sum_{i=0}^j\sum_{l=0}^{i}2^{(l-i)\beta}\mathfrak{S}_{r_l}^{k,\alpha,\ep}(z_0)=&\sum_{l=0}^j\sum_{i=l}^j2^{(l-i)\beta} \mathfrak{S}_{r_l}^{k,\alpha,\ep}(z_0)	\leq C(n,Y,\eta)\sum_{l=0}^j \mathfrak{S}_{r_l}^{k,\alpha,\ep}(z_0)\leq C(n,Y,\eta)\ep'.
	\end{align}
	Plugging \eqref{nondege5} and \eqref{nondege6} into \eqref{nondege7}, we conclude that for any $t'\in [-10r_j^2,-r_j^2/10]$ and $a,b \in \{1,\ldots,k\}$,
	\begin{align*}
		\left|\int_{\RR_{t'}}\la \na u_a,\na u_b\ra -\delta_{ab}\,\mathrm{d}\nu_{z_0;t'}\right|\leq C(n,Y,\eta,\alpha)\ep'.
	\end{align*}
	This proves Claim \ref{nondege2}.

	By Claim \ref{nondege2}, we can find a $k\times k$ matrix $T_{r_j}$ with $\|T_{r_j}-\Id \|\leq C(n,Y,\eta,\alpha)\ep'$ such that  the map $\vec u_{r_j}:=T_{r_j}\vec u=(u_{r_j, 1}, \ldots, u_{r_j, k}):Z_{[-10 r_j^2,0]}\to\R^k$ satisfies
	\begin{align*}
		\int_{-10r_j^2}^{-r_j^2/10}\int_{\RR_t}\la \na u_{r_j, a},\na u_{r_j, b}\ra -\delta_{ab}\,\mathrm{d}\nu_{z_0;t}\mathrm{d}t=0.
	\end{align*}
Moreover, using \eqref{nondege1}, we have	
	\begin{align*}
		\sum_{a=1}^k\int_{-10r_j^2}^{-r_j^2/10}\int_{\RR_t} \left|\nabla^2 u_{r_j, a}\right|^2\,\mathrm{d}\nu_{z_0;t}\mathrm{d}t\leq C(n,Y,\eta,\alpha)\ep'.
	\end{align*}
	
	Thus, $\vec u_{r_j}$ is a $(k,C \ep',r_j)$-splitting map at $z_0$, where $C=C(n,Y,\eta,\alpha)$. This finishes the proof.
\end{proof}

We end this section with the following covering result, which is a direct consequence of Lemma \ref{existenceofindependentpointsafterbase} by taking the limit.

\begin{lem}\label{existenceofindependentpointsafterbaseRFlimit}
Suppose that $z_0 \in Z$ is $(\delta,r)$-selfsimilar. Let $S$ be the subset consisting of all $(\delta,r)$-selfsimilar points in $B_Z^*(z_0,r)\bigcap Z_{[\t(z_0)-\alpha^2 r^2/L^2, \t(z_0)+\alpha^2 r^2/L^2]}$, where $L=L(n, Y)>1$. Assume that there exists no strongly $(k,\alpha,\delta, r)$-independent points in $S$ at $z_0$. If $\delta\leq\delta(n,Y,\alpha)$, then we can find $\{z_i\}_{1 \le i \le N} \subset S$ with $N\leq C(n,Y)\alpha^{1-k}$ and 
	\begin{align*}
		S\subset\bigcup_{i=0}^N B_Z^*(z_i, \alpha r).
	\end{align*}
\end{lem}

\subsection{Hausdorff dimension and quantitative estimates of singular strata}

In this subsection, we consider $(k, \ep, r)$-symmetric points (see Definition \ref{defnalmostsymmetric}) in $Z$ and investigate the $k$-th stratum $\MS^k$.

First, we prove the following static estimate on $Z$.

\begin{lem} \label{lem:static3}
There exists a constant $C=C(n, Y)>0$ such that the following holds for any $\ep>0$ and $\beta>0$: if $x, y \in Z$ are $(\delta, r)$-selfsimilar, with $d_Z(x, y) \le r$ and $|\t(x)-\t(y)| \ge \beta r^2$, then both $x$ and $y$ are $(C\delta^{\frac 1 2} \beta^{-2}, r)$-static (see Definition \ref{defnstatic}).
\end{lem}

\begin{proof}
The result follows directly from Theorem \ref{staticestimate}, Proposition \ref{prop:Wconv1}, and a standard limiting argument.
\end{proof}

The following result follows immediately from Definition \ref{defnstatic}, Proposition \ref{equivalencesplitsymetric1}, and a standard limiting argument. For (ii), we also need \cite[Claim 22.7]{bamler2020structure}.

\begin{prop}\label{almostsymmetric}
For any $\ep>0$, if $\delta \le \delta(n, Y, \ep)$, then the following statements hold:
	\begin{enumerate}[label=\textnormal{(\roman{*})}]
		\item If $z \in Z$ is $(\delta,r)$-selfsimilar and there exists a $(k, \delta, r)$-splitting map at $z$, then $z$ is $(k, \ep, s)$-symmetric for any $s \in [\ep r, \ep^{-1} r]$.

	\item If $z \in Z$ is both $(\delta,r)$-selfsimilar and $(\delta, r)$-static, then $z$ is $(\ep, s)$-static for any $s \in [\ep r, \ep^{-1} r]$.
	\end{enumerate}
\end{prop}

Next, we prove the covering result.

\begin{lem}\label{lem:keycover}
Given $z_0 \in Z$, $\ep>0$, $\beta>0$ and $r>0$, let $S$ be the subset consisting of all $(\delta,r)$-selfsimilar points in $B^*(z_0,r)$ which are not $(k+1, \ep, s)$-symmetric for any $s \in [\ep^{1/2}\beta r, r]$. If $\delta \le \delta(n, Y, \ep, \beta)$, then we can find $\{z_i\}_{1 \le i \le N} \subset S$ with $N\leq C(n,Y) \beta^{-k}$ and 
	\begin{align*}
		S\subset\bigcup_{i=1}^N B^*(z_i,  \beta r).
	\end{align*}
\end{lem}

\begin{proof}
Without loss of generality, we assume $r=1$. Let $\{z_i\}_{1 \le i \le N} \subset S$ be a maximal $\beta$-separated set, i.e., $d_Z(z_i, z_j) \ge \beta$ for all $i \ne j$. Then it is clear that
	\begin{align*}
		S\subset\bigcup_{i=1}^N B^*(z_i, \beta).
	\end{align*}

Suppose, for contradiction, that the desired conclusion fails. Then we assume
	\begin{align} \label{eq:contranumber}
N \gg C(n,Y) \beta^{-k}.
	\end{align}

We now consider two cases to derive a contradiction.

\textbf{Case 1}: There exists $i \in [1, N]$ such that $|\t(z_i)-\t(z_j)| \le 3 \beta^2$ for any $j \ne i$.

Without loss of generality, we assume $i=1$. In this case, we claim that if $\delta \le \delta(n, Y,  \beta)$, there exist strongly $(k+1, 3L\beta, \delta, 1)$-independent points in $S$ at $z_1$, where $L=L(n, Y)$ is the constant from Lemma \ref{existenceofindependentpointsafterbaseRFlimit}. Otherwise, by Lemma \ref{existenceofindependentpointsafterbaseRFlimit} (applied with $r=1$ and $\alpha=3L\beta$), the set $S$ can be covered by at most $C(n, Y) \beta^{-k}$ balls of radius $3L \beta$, contradicting \eqref{eq:contranumber} if we further cover each ball with at most $C(n, Y)$ balls of radius $\beta$.

Given the existence of such independent points, Theorem \ref{thmsharpsplittingRFlimit} implies that there exists a $(k+1, \zeta, 1)$-splitting map at $z_1$, where $\zeta=C(n, Y, \beta)\delta^{1/2}$. Combining this fact with Proposition \ref{almostsymmetric} (i), we conclude that $(Z, d_Z, z_1,\t)$ is $\ep$-close (see Definition \ref{defn:close}) to a $(k+1)$-splitting Ricci shrinker space $(Z', d_{Z'}, z', \t')$ over $[-\ep^{-1}, \ep^{-1}]$ if $\delta \le \delta(n, Y, \beta, \ep)$, which is a contradiction.

\textbf{Case 2}: There is no $i \in [1, N]$ such that $|\t(z_i)-\t(z_j)| \le 3\beta^2$ for any $j \ne i$.

By the pigeonhole principle, there exists a subset $I \subset \{1, \ldots, N\}$ with cardinality $|I| \gg C(n,Y) \beta^{2-k}$ such that for any $i, j \in I$, we have $|\t(z_i)-\t(z_j)| \le \beta^2$. 

Without loss of generality, assume $I=\{1, \ldots, N'\}$. By our assumption, $N'<N$, and there exists $j >N'$ such that
	\begin{align*}
|\t(z_i)-\t(z_j)| \ge 2 \beta^2
	\end{align*}
for any $1 \le i \le N'$. Then, by Lemma \ref{lem:static3}, each $z_i$ and $z_j$ are $(C \delta^{1/2} \beta^{-4}, 1)$-static for $1 \le i \le N'$. Arguing as in Case 1, we can then find a $(k-1, \zeta, 2)$-splitting map at $z_1$, where $\zeta=C(n, Y, \beta)\delta^{1/2}$.

Thus, we conclude that $(Z, d_Z, z_1, \t)$ is $\ep$-close to a $(k-1)$-splitting static or quasi-static cone $(Z', d_{Z'}, z', \t')$, if $\delta \le \delta(n, Y, \beta, \ep)$. If $(Z', d_{Z'}, z', \t')$ is a static cone, one can derive the contradiction as in Case 1. Therefore, we may assume that $(Z', d_{Z'}, z', \t')$ is a quasi-static cone.

If $\t(z_1)<\t(z_j)$, since $d_Z(z_1, z_j) \le 2$ and $\t(z_1) \le \t(z_j)-2 \beta^2$, we conclude that $z_1$ is $(k+1, \ep, \ep^{1/2}\beta)$-symmetric by Definition \ref{defnalmostsymmetric}, provided that $\delta \le  \delta(n, Y,  \ep, \beta)$. Similarly, if $\t(z_j)<\t(z_1)$ , then $z_j$ is $(k+1, \ep, \ep^{1/2}\beta)$-symmetric.

In either case, we obtain a contradiction to \eqref{eq:contranumber}. Therefore, the original statement holds.
\end{proof}

By Lemma \ref{lem:keycover}, we prove the following covering lemma for the quantitative singular strata, whose proof is based on \cite{cheeger2013lower}; see also \cite[Proposition 11.2]{bamler2020structure}. Note that in \cite[Theorem 8.14]{fang2025RFlimit}, the conclusion has already been proved for $k=n-2$.

\begin{prop}\label{prop:quansingcover}
Given $k \in \{0, 1,\ldots, n-2\}$, $z_0 \in Z$, $\ep>0$ and $r>0$ with $\t(z_0)-2r^2 \in \III^-$, for any $\sigma \in (0, \ep)$, there exist $x_1, x_2, \ldots, x_N \in B^*(z_0, r)$ with $N \le C(n, Y, \ep) \sigma^{-k-\ep}$ and
	\begin{equation*}
		\MS^{\ep,k}_{\sigma r, \ep r} \bigcap B^* (z_0, r) \subset \bigcup_{j=1}^N B^* (x_j, \sigma r).
	\end{equation*}
\end{prop}

\begin{proof}
Without loss of generality, we assume $r=1$. In the following proof, we will choose $\beta=\beta(n, Y, \ep)$. Once $\beta$ is fixed, set $\delta=\delta(n, Y, \ep, \beta)$ so that Lemma \ref{lem:keycover} holds.

We first prove that there exists $x_1, x_2, \ldots, x_N \in B^*(z_0, 1)$ with $N \le C(n, Y, \ep) \sigma^{-k-\ep}$ such that
	\begin{equation} \label{eq:weakcover}
		\MS^{\ep,k}_{\ep^{1/2}\sigma , \ep } \bigcap B^* (z_0, 1) \subset \bigcup_{j=1}^{N} B^* (x_j, \sigma).
	\end{equation}
Once \eqref{eq:weakcover} is established, we immediately obtain the covering
	\begin{equation*}
		\MS^{\ep,k}_{\sigma , \ep } \bigcap B^* (z_0, 1) \subset \bigcup_{j=1}^{N} B^* (x_j, \ep^{-1/2}\sigma).
	\end{equation*}
By applying Proposition \ref{prop:volumebound}, each ball $B^* (x_j, \ep^{-1/2}\sigma)$ can be further covered by at most $C(n, Y, \ep)$ balls of radius $\sigma$. This completes the proof.

To prove \eqref{eq:weakcover}, it suffices to prove the conclusion for $\sigma=\beta^m \ep$, where $m\in\N$. First, note that for any $x \in B^*(z_0,r)$, the number of indices $j\in\N$ such that $x$ is not $(\delta,\beta^j \ep)$-selfsimilar is bounded by 
	\begin{equation*}
Q=Q(Y,\delta,\beta, \ep)=Q(n, Y, \beta, \ep).
	\end{equation*}

For any finite sequence $a_0,a_1,\ldots,a_{m-1}\in \{0,1\}$, define
	\begin{align*}
		E(a_0,a_1,\ldots,a_{m-1}):=\{x \in B^*(z_0,1) \mid x \ \mathrm{is}\ (\delta,\beta^j \ep)-\mathrm{selfsimilar\ if\ and\ only\ if}\ a_j=0\}.
	\end{align*}
If $a_0+a_1+\ldots +a_{m-1}>Q$, then $E(a_0,a_1,\ldots,a_{m-1})=\emptyset$. Thus, there are at most $m^Q$ such non-empty sets which together cover $B^*(z_0,1)$.
	
	\begin{claim}\label{cla:indcover}
		There exists a constant $C_1=C_1(n,Y)$ such that the following holds. 
		
For any $j\in \{0,1,\ldots,m-1\}$, any $x \in B^*(z_0,1)$, and any $a_0,a_1,\ldots,a_{m-1}\in \{0,1\}$, one can find $\{y_i\}_{1\leq i\leq N'}\subset \MS^{\ep,k}_{\ep^{1/2}\sigma, \ep} \bigcap E(a_0,a_1,\ldots,a_{m-1})\bigcap B^*(x,\beta^j  \ep)$ such that
		\begin{align*}
			\MS^{\ep,k}_{\ep^{1/2}\sigma, \ep}\bigcap E(a_0,a_1,\ldots,a_{m-1})\bigcap B^*(x,\beta^j \ep)\subset \bigcup_{i=1}^{N'} B^*(y_i,\beta^{j+1} \ep),
		\end{align*}
where $N'\leq C_1\beta^{-k}$ if $a_j=0$ and $N'\leq C_1 \beta^{-n-2}$ if $a_j=1$. 
	\end{claim}
	
	In fact, if $a_j=0$, applying Lemma \ref{lem:keycover} with $r=\beta^j \ep$ yields the desired covering. If $a_j=1$, the covering follows directly from Proposition \ref{prop:volumebound}. This proves the claim.
	
Note that $B^*(z_0, 1)$ can be covered by at most $C(n, Y, \ep)$ balls of radius $\ep$. Thus, by applying Claim \ref{cla:indcover} inductively, $\MS^{\ep,k}_{\ep^{1/2}\sigma, \ep}\bigcap E(a_0,a_1,\ldots,a_{m-1})$ can be covered by at most 
			\begin{align*}
C(n, Y, \ep)(C_1 \beta^{-n-2})^{1+\sum a_j}(C_1 \beta^{-k})^{m-1-\sum a_j}
		\end{align*}
balls of the form $B^*(y,\sigma)$ where $y \in B^*(z_0,1)$. Therefore, $\MS^{\ep,k}_{\ep^{1/2}\sigma, \ep}\bigcap B^*(z_0,1)$ can be covered by at most
	\begin{align*}
N\leq C(n, Y, \ep) m^Q (C_1 \beta^{-n-2})^{1+Q}(C_1 \beta^{-k})^{m-1} \leq C(n, Y, \beta, \ep) m^Q C_1^m \beta^{-mk}
		\end{align*}
	such balls. 
	
Now, we choose $\beta=\beta(n, Y, \ep)$ such that $\beta^{-\ep/2}\geq C_1$. Then
	\begin{align*}
N\leq C(n, Y, \ep) m^Q \beta^{-mk-m\ep/2} \leq C(n,Y, \ep) \beta^{-mk-m\ep} \le C(n,Y, \ep) \sigma^{-k-\ep},
		\end{align*}	
which completes the proof.
\end{proof}

Based on Proposition \ref{prop:quansingcover}, we obtain the following result by the same argument as in \cite[Corollaries 8.15, 8.16, 8.17]{fang2025RFlimit}.

\begin{cor}\label{cor:kthcover}
Given $z_0 \in Z$, $\ep>0$ and $r>0$ with $\t(z_0)-2 r^2 \in \III^-$, the following statements are true for any $k \in \{0, 1,\ldots, n-2\}$.
		\begin{enumerate}[label=\textnormal{(\alph{*})}]
			\item For any $\delta \in (0, \ep)$,
	\begin{align*}
\abs{B^*_{\delta r} \lc \MS^{\ep, k}_{\delta r,\ep r} \rc \cap B^* (z_0, r)} \le C(n,Y, \ep) \delta^{n+2-k-\ep} r^{n+2}.
	\end{align*}
			
		\item For any $\delta \in (0, \ep)$ and $t \in \R$,
	\begin{align*}
\abs{B^*_{\delta r} \lc \MS^{\ep, k}_{\delta r,\ep r} \rc \cap B^* (z_0, r) \cap Z_t}_t \le C(n,Y, \ep) \delta^{n-k-\ep} r^{n}.
	\end{align*}		
	
		\item The Hausdorff dimension satisfies
			\begin{align*}
		\dim_{\HHH} \mathcal S^k \le k.
	\end{align*}
				\end{enumerate}
\end{cor}

Note that a similar result was proved in \cite[Theorem 1.9 (a)]{bamler2020structure}. However, our definition of the stratum $\mathcal S^k$ differs from that in \cite{bamler2020structure}: in our setting, a point $z \in \mathcal S^k$ may admit a tangent flow that is $k$-splitting and quasi-static. For further discussion, we refer the reader to \cite[Section 8]{fang2025RFlimit}.

\section{Rectifiability of cylindrical singularities}\label{secrectifiablitycyl}

In this section, we prove the rectifiability of cylindrical singularities in a noncollapsed Ricci flow limit space. Throughout, we fix a noncollapsed Ricci flow limit space $(Z,d_Z,\t)$, which is obtained as a pointed Gromov--Hausdorff limit of a sequence in $\MM(n, Y, T)$. As before, we set 
\begin{align*}
\III:=[-0.98 T,0] \quad \text{and} \quad \III^-:=(-0.98 T,0].
	\end{align*}
Moreover, the regular part is given by a Ricci flow spacetime $(\RR, \t, \partial_\t, g^Z)$. For simplicity, we use $B^*(x,r)$ instead of $B_Z^*(x,r)$ for the metric balls in $Z$ with respect to $d_Z$.

\subsection{Discrete Lojasiewicz inequality for Ricci flow limit spaces}

In this section, we first prove the summability result of the $\widetilde \WW$-entropy.

\begin{prop}\label{sumWonRFlimit}
	For any $0<\ep\leq\ep(n,Y)$, $\alpha\in (1/4,1)$, and $\delta\leq\delta(n,Y,\ep,\alpha)$, the following holds. Suppose
	\begin{align*}
		\left|\widetilde \WW_{z}(s_1)-\widetilde \WW_{z}(s_2)\right|<\delta,
	\end{align*}
for $0<s_1<s_2$, and for any $s\in [s_1,s_2]$, $z$ is $(k,\delta,\sqrt{s})$-cylindrical (see Definition \ref{def:almost0}). Then
	\begin{align*}
		\sum_{s_1\leq r_j=2^{-j} \leq s_2}\left|\widetilde \WW_{z}(r_j)-\widetilde \WW_{z}(r_{j-1})\right|^{\alpha}<\ep.
	\end{align*}
\end{prop}
\begin{proof}
As in the proof of \cite[Corollary 6.28]{FLloja05}, it suffices to show that for any $j$ with $s_1\leq r_j\leq s_2$ and any $\theta\in (1/3,1)$, we have
	\begin{align}\label{sumW1}
		\left|\widetilde \WW_{z}(r_j)-\Theta_{n-k}\right|^{1+\theta}\leq C(n,Y)\lc \widetilde \WW_{z}(r_{j+1})-\widetilde \WW_{z}(r_{j-1})\rc,
	\end{align}
	provided that $\delta \le \delta(n, Y, \theta)$.
	
Assume that $\XX^l=\{M_l^n,(g_l(t))_{t\in \III_l}\}\in\MM(n,Y,T)$ with $ z_l^*\in\XX^l$ satisfy
		\begin{align*}
		\left(\XX^l,d_l^*, z_l^*,\t_l\right)\xrightarrow[l\to\infty]{\quad \hat C^\infty \quad}\left(Z,d_Z,z,\t\right).
	\end{align*} 
	Then, by our assumption, $z_l^*$ is $(k, 2\delta, \sqrt{r_j})$-cylindrical for any sufficiently large $l$. By \cite[Lemma 2.27]{FLloja05} and \cite[Theorem 6.26]{FLloja05}, we have
		\begin{align*}
		\left|\WW_{z_l^*}(s)-\Theta_{n-k}\right|^{1+\theta}\leq C(n,Y)\lc\WW_{z_l^*}(s/2)- \WW_{z_l^*}(2s)\rc.
	\end{align*}
	for any $s \in [r_j/2, 2r_j]$. Letting $l \to \infty$ and using Proposition \ref{prop:Wconv1}, we conclude that for any $\ep'>0$, there exists $s \in (r_j-\ep', r_j]$ with $s, s/2, 2s \notin J^z$ such that
		\begin{align*}
		\left|\WW_{z}(s)-\Theta_{n-k}\right|^{1+\theta}\leq C(n,Y)\lc\WW_{z}(s/2)- \WW_{z}(2s)\rc.
	\end{align*}	
Taking $\ep' \to 0$, we obtain from Definition \ref{defnWRFlimit1} that
			\begin{align*}
\left|\widetilde \WW_{z}(r_j)-\Theta_{n-k}\right|^{1+\theta}\leq C(n,Y)\lc \widetilde \WW_{z}(r_{j+1})-\widetilde \WW_{z}(r_{j-1})\rc.
	\end{align*}	
In other words, \eqref{sumW1} holds, which completes the proof.
\end{proof}

\subsection{Basic properties of almost cylindrical points}

We begin with the following lemma, which shows that the entire flow is rapidly clearing out after an almost cylindrical singularity.

\begin{lem}[Rapid clearing out]\label{clearout}
Suppose $z\in Z$ is $(k,\delta,r)$-cylindrical (see Definition \ref{def:almost0}). For any $\ep>0$, if $\delta\leq\delta(n,Y,\ep)$, we have 
	\begin{align*}
		\t^{-1}\big([\t(z)+\ep r^2, \infty)\big)\bigcap B^*(z,\ep^{-1}r)=\emptyset.
	\end{align*}
\end{lem}
\begin{proof}
Without loss of generality, we assume $\t(z)=0$ and $r=1$. Suppose that the conclusion fails. Then we can find a sequence of Ricci flow limit spaces $(Z^l, d_{Z^l}, \t_l)$ and $(k, l^{-2},1)$-cylindrical points $z_l\in Z^l$ such that
	\begin{align*}
	y_l \in	\t_l^{-1}\big([\ep , \infty)\big)\bigcap B^*(z_l,\ep^{-1}).
	\end{align*}
It follows from Definition \ref{def:almost0} that
	\begin{align*}
		(Z^l, d_{Z^l},z_l, \t_l)\xrightarrow[l\to\infty]{\quad \hat C^\infty \quad }(\bar{\mathcal C}^k, d_{\mathcal C}^*, p^*,\t).
	\end{align*}
	Moreover, $y_l\xrightarrow[l\to\infty]{}y_\infty\in \bar{\mathcal C}^k_{[\ep,\infty)} \cap B^*(p^*, 2\ep^{-1})$. However, it is clear that $\bar{\mathcal C}^k_{[\ep, \infty)}$ is empty, which gives the contradiction. This finishes the proof.
\end{proof}

Next, we prove

\begin{prop}\label{prop:almostconstant}
Suppose $z \in Z$ is $(k, \delta, r)$-cylindrical with respect to $\mathcal L_{z,r}$. For any $\ep>0$, if $\delta \le \delta(n, Y, \ep)$, then the following conclusions hold.
	\begin{enumerate}[label=\textnormal{(\roman{*})}]
		\item For any $\tau \in [\ep r^2, \ep^{-1} r^2]$,
	\begin{align*}
	\abs{\widetilde \WW_z(\tau)-\Theta_{n-k}}<\ep.
	\end{align*}
		\item Any point in $ B^*(z, \ep^{-1} r) \cap B^*_{\delta r}\lc \mathcal L_{z,r} \rc$ is $(k, \ep, s)$-cylindrical for any $s \in [\ep r, \ep^{-1} r]$.
		\item If $x \in B^*(z, \ep^{-1} r) \cap Z_{[\t(z)-r^2, \t(z)+r^2]}$ is $(k, \delta, s)$-cylindrical with respect to $\mathcal L_{x, s}$ for some $s \in [\ep r, \ep^{-1} r]$, then
\begin{align*}
d_{\mathrm{H}} \lc \mathcal L_{x, s} \cap B^*(x, \ep^{-1} s), \mathcal L_{z, r} \cap B^*(x, \ep^{-1} s)  \rc < \ep r,
\end{align*}
where $d_{\mathrm{H}}$ denotes the Hausdorff distance with respect to $d_Z$.
	\end{enumerate}	
\end{prop}

\begin{proof}
We only prove (i), as (ii) and (iii) can be proved similarly by a limiting argument.

Without loss of generality, we assume that $\t(z)=0$ and $r=1$. Suppose the conclusion (i) fails. Then we can find a sequence of Ricci flow limit spaces $(Z^l, d_{Z^l}, \t_l)$, which is obtained as the limit of a sequence in $\mathcal M(n, Y, T_l)$, and $z_l \in Z^l$ that is $(k, l^{-2}, 1)$-cylindrical, but there exists $s_l \in [\ep, \ep^{-1}]$ such that
\begin{align}\label{eq:almost001}
\abs{\widetilde \WW_{z_l}(s_l)- \Theta_{n-k}} \ge \ep.
\end{align}
By passing to a diagonal subsequence, we may assume that each $(Z^l, d_{Z^l}, \t_l)$ is given by a closed Ricci flow. By our assumption, we have
	\begin{align*}
		(Z^l, d_{Z^l},z_l, \t_l)\xrightarrow[l\to\infty]{\quad \hat C^\infty \quad }(\bar{\mathcal C}^k, d_{\mathcal C}^*, p^*,\t).
	\end{align*}
Then it follows from Proposition \ref{prop:Wconv1} that for almost all $\tau>0$,
	\begin{align*}
	\lim_{l \to \infty} \WW_{z_l}(\tau)=\WW_{p^*}(\tau)=\Theta_{n-k}.
	\end{align*}

Thus, it follows from Definition \ref{defnWRFlimit1} and the monotonicity that $\widetilde \WW_{z_l}(s_l) \to \Theta_{n-k}$, which contradicts \eqref{eq:almost001}. In sum, the proof is complete.
\end{proof}

Next, we show

\begin{prop}\label{prop:almostconstanta}
Suppose $z \in Z$ is $(k, \delta, r)$-cylindrical with respect to $\mathcal L_{z,r}$ and $\vec{u}=(u_1,\ldots, u_k)$ is a $(k, \delta, r)$-splitting map at $z$. For any $\ep>0$, if $\delta \le \delta(n, Y, \ep)$, then the set
\begin{align*}
V_{z, \ep^{-1} r}:=\vec{u} \lc \mathcal L_{z, r} \cap B^*(z, \ep^{-1} r) \rc \subset \R^k
\end{align*}
satisfies
\begin{align*}
d_{\mathrm{H}} \lc V_{z, \ep^{-1} r} \cap B(\vec{0}^k, \ep^{-1} r), B(\vec{0}^k, \ep^{-1} r)  \rc < \ep r,
\end{align*}
where $B(\vec{0}^k,\ep^{-1}r)$ denotes the ball in $\R^k$, and $d_{\mathrm{H}}$ denotes the Hausdorff distance with respect to the Euclidean metric.
\end{prop}

\begin{proof}
Without loss of generality, we assume that $\t(z)=0$ and $r=1$. Suppose the conclusion fails. Then we can find a sequence of Ricci flow limit spaces $(Z^l, d_{Z^l}, \t_l)$, which is obtained as the limit of a sequence in $\mathcal M(n, Y, T_l)$, and $z_l \in Z^l$ such that $z_l$ is $(k, l^{-2}, 1)$-cylindrical with respect to $\mathcal L_{z_l, 1}$. Moreover, there exist a $(k, l^{-2}, 1)$-splitting map $\vec{u}^l=(u^l_1, \ldots, u^l_k)$ at $z_l$.

By our assumption, 
	\begin{align*}
		(Z^l, d_{Z^l},z_l, \t_l)\xrightarrow[l\to\infty]{\quad \hat C^\infty \quad }(\bar{\mathcal C}^k, d_{\mathcal C}^*, p^*,\t).
	\end{align*}
Moreover, $\mathcal L_{z_l, 1}$ converge to $\bar{\mathcal C}^k_0$. Passing to a subsequence and composing with a fixed matrix in $\mathrm{O}(k)$ if necessary, it follows that the maps $\vec{u}^l$ converge to the standard coordinate functions $\vec{y} = (y_1, \ldots, y_k)$ on $\bar{\mathcal C}^k_{[-10, 0]}$. This convergence is evident on $\bar{\mathcal C}^k_{[-10, 0)}$, and the convergence at time slice $\bar{\mathcal C}^k_0$ follows from the reproduction formula and \cite[Theorem  5.20]{fang2025RFlimit}.

 Therefore, we conclude from Lemma \ref{lem:comparedis1} that
\begin{align*}
d_{\mathrm{H}} \lc V_{z_l, \ep^{-1}} \cap B(\vec{0}^k, \ep^{-1}), B(\vec{0}^k, \ep^{-1})  \rc \le \Psi(l^{-1}),
\end{align*}
which gives a contradiction.
 
In sum, the proof is complete.
\end{proof}

Next, we prove the converse conclusion regarding \cite[Lemma 2.27]{FLloja05}.

\begin{lem} \label{lem:smoothcyl2}
Suppose $z \in Z_{\III^-}$ is $(\delta, r)$-selfsimilar and there exists a diffeomorphism $\varphi$ from $\{\bar b:=2\sqrt{\bar f(-1)} \le \delta^{-1}\} \subset \bar M$ onto a subset of $\RR_{\t(z)-r^2}$ such that on $\{\bar b \le \delta^{-1}\}$,
\begin{align*}
\| r^{-2} \varphi^* g^Z_{\t(z)-r^2}-\bar g(-1)\|_{C^{[\delta^{-1}]}}+\| \varphi^* f_z(\t(z)-r^2)-\bar f(-1)\|_{C^{[\delta^{-1}]}} \le \delta,
\end{align*}
where $(\bar g(-1), \bar f(-1))$ comes from $\bar{\mathcal C}^k$ (see Subsection \ref{cacp}). For any $\ep>0$, if $\delta \le \delta(n, Y, \ep)$, then $z$ is $(k, \ep, s)$-cylindrical for any $s \in [\ep r, \ep^{-1} r]$.
\end{lem}

\begin{proof}
Without loss of generality, we assume $\t(z)=0$ and $r=1$. Suppose the conclusion fails. Then we can find a sequence of Ricci flow limit spaces $(Z^l, d_{Z^l}, \t_l)$, which is obtained as the limit of a sequence in $\mathcal M(n, Y, T_l)$ and $z_l \in Z^l$ such that $z_l$ is $(l^{-1}, 1)$-selfsimilar, and there exists a diffeomorphism $\varphi_l$ from $\{\bar b \le l\}$ onto a subset of the regular part $\RR^l_{-1}$ of $Z^l$ such that 
\begin{align} \label{eq:cylin1}
\| \varphi_l^* g^{Z^l}_{-1}-\bar g(-1)\|_{C^{l}}+\| \varphi_l^* f_{z_l}(-1)-\bar f(-1)\|_{C^{l}} \le l^{-1}.
\end{align}
However, $z_l$ is not $(k, \ep, s_l)$-cylindrical for some $s_l \in [\ep, \ep^{-1}]$.

By our assumption, we may assume
	\begin{align*}
		(Z^l, d_{Z^l},z_l, \t_l)\xrightarrow[l\to\infty]{\quad \hat C^\infty \quad }(Z, d_Z, z,\t),
	\end{align*}
for which $(Z, d_Z, z,\t)$ is a noncollapsed Ricci flow limit space over $(-\infty, 0]$. Moreover, as in the proof of Lemma \ref{lem:imply1}, we conclude that 
\begin{align*}
\Ric(g^Z) + \nabla^2 f_z = \frac{g^Z}{2|\t|}
\end{align*}
holds on $\RR_{(-\infty, 0)}$, and $\RR_t$ is connected for $t<0$. By \eqref{eq:cylin1} and the self-similarity, we conclude that
\begin{align*}
Z_{(-\infty, 0)}=\RR_{(-\infty, 0)}=\bar{\mathcal C}^k_{(-\infty, 0)}.
\end{align*}
Thus, it follows from \cite[Theorem 7.25]{fang2025RFlimit} that $Z_{(0,\infty)}$ is empty. In particular, we have $Z=\bar{\mathcal C}^k$. However, this implies that $z_l$ is $(k, \ep, s_l)$-cylindrical for sufficiently large $l$, contradicting our assumption.

In sum, the proof is complete.
\end{proof}

Next, we prove

\begin{prop} \label{prop:uniformcsmooth}
Let $\XX=\{M^n,(g(t))_{t \in \III^{++}}\} \in \mathcal M(n, Y, T)$. Suppose $z^* \in M \times \III$ is $(k, \delta, r)$-cylindrical and $ \WW_{z^*}(\delta \bar r^2)- \WW_{z^*}(\delta^{-1} r^2)<\delta$ for some $0<\bar r<r$.  For any $\ep>0$, if $\delta \le \delta(n, Y, \ep)$, then $z^*$ is $(k, \ep, s)$-cylindrical for any $s \in [\bar r, r]$.
\end{prop}

\begin{proof}
Without loss of generality, we assume $r=1$ and $\t(z^*)=0$.

For each $s \in [\bar r, 1]$, define $\mathbf{r}_A(s)$ as $\mathbf{r}_A$-radius (see \cite[Definition 5.1 (A)]{FLloja05}) for the weighted Riemannian manifold $(M, s^{-2}g(-s^2), f_{z^*}(-s^2))$.

By Lemma \ref{lem:smoothcyl2}, there exists a constant $L_1=L_1(n, Y, \ep)>1$ such that if $\mathbf{r}_A(s) \ge L_1$, then $z^*$ is $(k, \ep, s)$-cylindrical.

Furthermore, by \cite[Theorem 6.24]{FLloja05}, there exists $L_2=L_2(n, Y)>1$ such that if $\mathbf{r}_A(s) \ge L_2$, then
	\begin{align} \label{eq:uniformcy1}
\mathbf{r}_A(s) \ge \sqrt{-\log\lc \WW_{z^*}(s^2/2)- \WW_{z^*}(2s^2) \rc}.
	\end{align}

Now set $L=\max\{L_1, L_2\}$. Since $z^*$ is $(k, \delta, 1)$-cylindrical, \cite[Lemma 2.27]{FLloja05} implies that $\mathbf{r}_A(1) \ge 2L$, provided that $\delta$ is sufficiently small. 

Define $r'$ as the infimum in $[\bar r, 1]$ such that $\mathbf{r}_A(s) \ge 2L$ for any $s \in [r', 1]$. If $r'>\bar r$, then by continuity we must have $\mathbf{r}_A(r')=2L$. However, this contradicts \eqref{eq:uniformcy1} if $\delta \le \delta(n, Y, \ep)$. 

Thus, $\mathbf{r}_A(s) \ge 2L$ for any $s \in [\bar r, 1]$, and hence $z^*$ is $(k, \ep, s)$-cylindrical for all $s \in [\bar r, 1]$.
\end{proof}

Next, we have the following definition.

\begin{defn}
	We say $z\in Z$ is \textbf{uniformly $(k,\ep, r)$-cylindrical}\index{uniformly $(k,\ep, r)$-cylindrical} if $z$ is $(k,\ep,s)$-cylindrical for any $s\in (0,r]$.
\end{defn}

By taking the limit and using Propositions \ref{prop:Wconv1} and \ref{prop:uniformcsmooth}, we obtain the following result.

\begin{prop} \label{prop:uniformc}
Suppose $z \in Z$ is $(k, \delta, r)$-cylindrical and $\widetilde \WW_z(\delta\bar r^2)-\widetilde \WW_z(\delta^{-1} r^2)<\delta$ for some $0<\bar r<r$.  For any $\ep>0$, if $\delta \le \delta(n, Y, \ep)$, then $z$ is $(k, \ep, s)$-cylindrical for any $s \in [\bar r, r]$. In particular, if $z \in \MS^k_{\mathrm{c}}\setminus \MS^{k-1}_{\mathrm c}$ is $(k, \delta, r)$-cylindrical with $\delta \le \delta(n, Y, \ep)$, then $z$ is uniformly $(k,\ep, r)$-cylindrical.
\end{prop}

\subsection{Cylindrical neck regions}

We introduce the following notion of a cylindrical neck region, adapted from \cite[Definition 4.1]{fang2025volume} with an additional parameter $\cc$ that encodes finer quantitative information. Related neck-region notions for noncollapsed Ricci limit spaces appear in \cite[Definition 3.1]{jiang20212} and \cite[Definition 2.4]{Cheeger2018RectifiabilityOS}.

\begin{defn}[Cylindrical neck region]\label{defiofcylneckregion}
Given constants $\delta>0$, $\cc \in (0, 10^{-10 n })$, $r>0$ and $z \in Z_{\III^-}$ with $\t(z)-2\delta^{-1}r^2 \in \III^-$, we call a subset $\NNN \subset B^*(z, 2r)$ a \textbf{$(k,\delta, \cc, r)$-cylindrical neck region}\index{$(k,\delta,\cc,r)$-cylindrical neck region} if $\NNN=B^*(z,2r)\setminus B_{r_x}^*(\CCC)$, where $\CCC \subset B^*(z, 2r)$ is a nonempty closed subset with $r_x: \CCC \to \R_{+}:=[0, +\infty)$, satisfies:
	\begin{itemize}[leftmargin=*, label={}]
		\item \emph{(n1)} for any $x, y \in \CCC$, $d_Z(x, y) \ge \cc^2(r_x+r_y);$\index{$\cc$}
		\item \emph{(n2)} for all $x\in \CCC$,
		$$\widetilde \WW_{x}(\delta r^2_x)-\widetilde \WW_{x}(\delta^{-1} r^2)<\delta;$$
		\item \emph{(n3)} for each $x\in \CCC$ and $\cc^2 r_x\leq s\leq 2r$, $x$ is $(k,\delta,s)$-cylindrical with respect to $\mathcal L_{x,s};$
		\item \emph{(n4)} for each $x\in \CCC$ and $\cc^{-5} r_x\leq s\leq r-d_Z(x, z)/2$, we have $\mathcal{L}_{x,s} \cap B^*(x, s) \subset B^*_{\cc s}(\CCC)$ and $\CCC\bigcap B^*(x,s)\subset B^*_{\cc s}(\mathcal{L}_{x,s})$.
	\end{itemize}
Here, $\CCC$\index{$\CCC$} is called the \textbf{center} of the neck region, and $r_x$\index{$r_x$} is referred to as the \textbf{radius function}. We decompose $\CCC=\CCC_0\bigcup\CCC_{+}$\index{$\CCC_0$}\index{$\CCC_+$}, where $r_x>0$ on $\CCC_+$ and $r_x=0$ on $\CCC_0$. In addition, we use the notation
	\begin{align*}
	B^*_{r_x}(\CCC):=\CCC_0 \bigcup \bigcup_{x \in \CCC_+} B^*(x, r_x).
	\end{align*}
\end{defn}

It is clear from $(n1)$ above that $r_x:\CCC\to \R_{+}$ is a Lipschitz function with Lipschitz constant $\cc^{-2}$. Moreover, it follows from $(n2)$ and Definition \ref{defnselfsimilarity1} that any $x \in \CCC$ is $(\delta, s)$-selfsimilar for any $s \in [r_x, r]$. In addition, it is easy to check that the restriction of a cylindrical neck region to any small ball centered at points in $\CCC$ is still a cylindrical neck region.

The following result is immediate from Proposition \ref{prop:uniformc}.

\begin{lem}\label{uniformcyl}
	Given $\zeta>0$ and $z\in \MS^k_{\mathrm{c}}\setminus \MS^{k-1}_{\mathrm c}$, there exists a constant $r_z>0$ such that any point in $\lc\MS_{\mathrm c}^k\setminus \MS^{k-1}_{\mathrm c} \rc \bigcap B^*(z,r_z)$ is uniformly $(k,\zeta, r_z)$-cylindrical.
\end{lem}
\begin{proof}
By our assumption and Proposition \ref{prop:uniformc}, for any $\ep>0$, there exists a constant $r>0$ such that $z$ is uniformly $(k, \ep, r)$-cylindrical. By a limiting argument, given $\ep'>0$, if $\ep \le \ep(n, Y, \ep')$, then there exists a small constant $r' \in (0, r/2)$ such that any $w \in B^*(z, r')$ is also $(k, \ep', r/2)$-cylindrical. Then, it follows from Proposition \ref{prop:uniformc} again that any $w \in \lc\MS_{\mathrm c}^k\setminus \MS^{k-1}_{\mathrm c}\rc \bigcap B^*(z, r')$ is uniformly $(k, \zeta, r')$-cylindrical, provided that $\ep' \le \ep'(n, Y, \zeta)$. After choosing a small $\ep$, the proof is complete.
\end{proof}

\begin{prop}\label{cylneckdecomp}
Given $\delta>0$ and $\cc \in (0, 10^{-10 n})$, for any $x_0\in \MS_{\mathrm c}^k\setminus \MS^{k-1}_{\mathrm c}$, there exist a constant $r_0>0$ and $\CCC=\CCC_0\bigcup\CCC_+\subset B^*(x_0,2r_0)$ with $r_x:\CCC\to \R_{+}$ satisfying $r_x>0$ on $\CCC_+$ and $r_x=0$ on $\CCC_0$, such that the following hold:
	\begin{enumerate}[label=\textnormal{(\roman{*})}]
		\item $\NNN=B^*(x_0,2r_0)\setminus B^*_{r_x}(\CCC)$ is a $(k,\delta,\cc, r_0)$-cylindrical neck region.
		\item $\lc\MS_{\mathrm c}^k \setminus \MS^{k-1}_{\mathrm c}\rc \bigcap B^*(x_0,r_0)\subset\CCC_0$.
	\end{enumerate}
\end{prop}
\begin{proof}
In the proof, a ball $B^*(y,r)$ is called a \textbf{good ball} if there exists $z\in B^*(y,r)$ such that $z$ is $(k,\zeta,r)$-cylindrical, where $\zeta$ is a fixed constant to be determined. Otherwise, we call $B^*(y,r)$ a \textbf{bad ball}. Without loss of generality, we assume $\delta \ll \mathfrak c$. Moreover, throughout the proof, the constants $\zeta$ and $\ep$ will be chosen so that
	\begin{align*}
		\zeta\ll\ep\ll\delta.
	\end{align*}
	
	First, we determine $r_0$. Since $x_0\in \MS_{\mathrm c}^k\setminus \MS^{k-1}_{\mathrm c}$, it follows from Lemma \ref{uniformcyl} that we can find $r_0>0$ so that any point in $\lc\MS_{\mathrm c}^k \setminus \MS^{k-1}_{\mathrm c}\rc \bigcap B^*(x_0,2r_0)$ is uniformly $(k, \zeta, 2r_0)$-cylindrical with $\t(x_0)-5 \zeta^{-1} r_0^2 \in \III^-$. In particular, $B^*(x_0,r_0)$ is a good ball. 
	
We assume that $x_0$ itself is $(k,\zeta,r_0)$-cylindrical with respect to $\LL_{x_0,r_0}$. Then it follows from Proposition \ref{prop:almostconstant} (iii) that any $(k,\zeta, r_0)$-cylindrical point $z\in B^*(x_0,2r_0)$ satisfies 
	\begin{align*}
z\in B^*(x_0,2r_0)\bigcap B^*_{\ep r_0}(\LL_{x_0,r_0}).
	\end{align*}
In particular, we have
	\begin{align}\label{MSkinclustion1}
		\lc\MS_{\mathrm c}^k \setminus \MS^{k-1}_{\mathrm c}\rc\bigcap B^*(x_0,2 r_0)\subset B^*_{\ep r_0}\left(\LL_{x_0,r_0}\right) \bigcap B^*(x_0,2 r_0).
	\end{align}
	
By Proposition \ref{prop:almostconstant} (i) (ii), we can choose a small $\ep$ so that any point $z\in B^*_{\ep r_0}\left(\LL_{x_0,r_0}\right)\bigcap B^*(x_0,2r_0)$ is $(k, \delta^2, s)$-cylindrical for any $s \in [\delta r_0, \delta^{-1} r_0]$ and satisfies
	\begin{align}\label{neckdecom1}
		\left|\widetilde \WW_z(s^2)-\Th_{n-k}\right|\leq \delta^2.
	\end{align}
	
Set $\gamma=\cc^5$ and $L^1=\LL_{x_0,r_0} \bigcap B^*(x_0,2r_0)$. Then we choose a maximal $2 \cc^2 \gamma r_0$-separated set $\{y_\alpha\} \subset L^1$. In particular, $\left\{B^*(y_\alpha, \cc^2 \gamma r_0)\right\}$ are pairwise disjoint and
	\begin{align*}
L^1 \subset \bigcup_\alpha B^*(y_\alpha, 2 \cc^2	\gamma r_0).
	\end{align*}
	
	We can re-index those balls $B^*(y_\alpha, \gamma r_0)$ by good balls with index $g^1_i$ and bad balls with index $b^1_i$ and hence rewrite 
	\begin{align}\label{constcylneck1}
		B^*_{\ep r_0}\left(\LL_{x_0,r_0}\right)\bigcap B^*(x_0,2r_0)\subset \bigcup_{g^1_i} B^*(y_{g^1_i}, \gamma r_0)\bigcup\bigcup _{b^1_i}B^*(y_{b^1_i}, \gamma r_0).
	\end{align}
	
Then we set
	\begin{align*}
		\NNN^1:=B^*(x_0,2r_0)\setminus \lc\bigcup_{g^1_i} B^*(y_{g^1_i}, \gamma r_0)\bigcup\bigcup _{b^1_i}B^*(y_{b^1_i}, \gamma r_0)\rc.
	\end{align*}
From our construction and \eqref{neckdecom1}, we know that $\NNN^1$ is a $(k,\delta,\cc, r_0)$-cylindrical neck region with center $\bigcup_i\{y_{g^1_i}, y_{b^1_i}\}$ and radius function $r_x \equiv \gamma r_0$ for $x \in \bigcup_i\{y_{g^1_i}, y_{b^1_i}\}$. Indeed, condition $(n1)$ follows from our construction. $(n2)$ follows from \eqref{neckdecom1}. Moreover, $(n3)$ and $(n4)$ follow from our construction and Proposition \ref{prop:almostconstant}.

Moreover, by the definition of bad balls, there exists no $(k, \zeta, \gamma r_0)$-cylindrical point in $\bigcup_{b_i^1} B^*(y_{b_i^1}, \gamma r_0)$. Thus, it follows from \eqref{MSkinclustion1} and \eqref{constcylneck1} that
	\begin{align*}
		\lc\MS_{\mathrm c}^k \setminus \MS^{k-1}_{\mathrm c}\rc\bigcap B^*(x_0,r_0)\subset \bigcup_{g^1_i}B^*(y_{g_i^1}, \gamma r_0).
	\end{align*}

Next, we use induction to construct a sequence of $(k, \delta, \cc, r_0)$-cylindrical neck regions $\NNN^l$ iteratively.

Suppose that for some $l \ge 1$, we have constructed a $(k, \delta, \cc, r_0)$-cylindrical neck region
	\begin{align}\label{constcylneck2b}
		\NNN^l:= B^*(x_0, 2r_0)\setminus \lc\bigcup_{g^l_i}B^*(y_{g_i^l}, \gamma^l r_0)\bigcup\bigcup_{1\leq j\leq l}\bigcup_{b_i^j}B^*(y_{b_i^j},\gamma^j r_0)\rc
	\end{align}
such that 
\begin{itemize}
\item each $B^*(y_{g_i^l}, \gamma^l r_0)$ is a good ball, while each $B^*(y_{b_i^j},\gamma^j r_0)$ is a bad ball.

\item each $y_{g_i^l}$ is $(k, \delta^2, s)$-cylindrical for any $s \in [\delta \gamma^l r_0, \delta^{-1} r_0]$, and satisfies
	\begin{align*}
\left|\widetilde \WW_{y_{g_i^l}}(s^2)-\Theta_{n-k}\right|\leq \delta^2;
	\end{align*}

\item each $y_{b_i^j}$ (for $1 \le j \le l$) is $(k, \delta^2, s)$-cylindrical for any $s \in [\delta \gamma^j r_0, \delta^{-1} r_0]$, and satisfies
	\begin{align*}
\left|\widetilde \WW_{y_{b_i^j}}(s^2)-\Theta_{n-k}\right|\leq \delta^2;
	\end{align*}
	
\item we have
	\begin{align} \label{neckdecom1a1}
		\lc\MS_{\mathrm c}^k \setminus \MS^{k-1}_{\mathrm c}\rc \bigcap B^*(x_0,r_0)\subset \bigcup_{g^l_i}B^*(y_{g_i^l}, \gamma^l r_0).
	\end{align}
\end{itemize}

For each good ball $B^*(y_{g_i^l}, \gamma^l r_0)$, we can find a $(k,\zeta,\gamma^l r_0)$-cylindrical point $y_{g_i^l}' \in B^*(y_{g_i^l}, \gamma^l r_0)$ with respect to $\LL_{y_{g_i^l}',\gamma^l r_0}$. Then we have
	\begin{align*}
		\lc\MS_{\mathrm c}^k \setminus \MS^{k-1}_{\mathrm c}\rc \bigcap B^*(y_{g_i^l}, \gamma^l r_0)\subset B^*_{\ep\gamma^l r_0}\left(\LL_{y_{g_i^l}',\gamma^l r_0}\right).
	\end{align*}
By the same argument as before, we conclude that any point $z \in B^*_{\ep \gamma^l r_0}\lc\LL_{y_{g_i^l}',\gamma^l r_0} \rc\bigcap B^*(y_{g_i^l}',2\gamma^l r_0)$ is $(k, \delta^2, s)$-cylindrical for any $s \in [\delta \gamma^l r_0, \delta^{-1} \gamma^l r_0]$ and satisfies
	\begin{align}\label{neckdecom1a}
		\left|\widetilde \WW_z(s^2)-\Th_{n-k}\right|\leq \delta^2.
	\end{align}	
	
We set 
	\begin{align*}
L^{l+1}:=\bigcup_{g_i^l} \lc \LL_{y_{g_i^l}',\gamma^l r_0} \bigcap B^*(y_{g_i^l}',2\gamma^l r_0) \rc \bigcap B^*(x_0, 2 r_0) \setminus \lc \bigcup_{b_i^l} B^*(y_{b_i^l}, 2 \cc^2 \gamma^l r_0) \rc.
	\end{align*}		
Then we choose a maximal $2 \cc^2 \gamma^{l+1} r_0$-separated set $\{x_\alpha\} \subset L^{l+1}$. In particular, 
	\begin{align*}
		L^{l+1} \subset \bigcup_\alpha B^*(x_\alpha, 2 \cc^2 \gamma^{l+1} r_0).
	\end{align*}	
From our construction, $\left\{B^*(x_\alpha, \cc^2 \gamma^{l+1} r_0), B^*(y_{b_i^j}, \cc^2 \gamma^j r_0) \right\}_{\alpha, i, 1\le j \le l}$ are pairwise disjoint. 
	
After re-indexing by good balls and bad balls, we define
	\begin{align*}
		\NNN^{l+1}:= B^*(x_0, 2r_0)\setminus \lc\bigcup_{g^{l+1}_i}B^*(y_{g_i^{l+1}}, \gamma^{l+1} r_0)\bigcup\bigcup_{1\leq j \leq l+1}\bigcup_{b_i^j}B^*(y_{b_i^j},\gamma^j r_0)\rc.
	\end{align*}
	
	For $z=y_{g_i^{l+1}}$ or $y_{b_i^{l+1}}$, we claim that $z$ is $(k, \delta^2, s)$-cylindrical for any $s \in [\delta \gamma^{l} r_0, \delta^{-1} r_0]$ and satisfies
		\begin{align*}
		\left|\widetilde \WW_{z}(s^2)-\Th_{n-k}\right|\leq \delta^2.
	\end{align*}
Indeed, if $s \in [\delta \gamma^{l} r_0, \delta^{-1} \gamma^{l}r_0]$, this follows immediately from our construction and \eqref{neckdecom1a}. In general, if $s \in [\delta \gamma^{j} r_0, \delta^{-1} \gamma^{j}r_0]$ for $1 \le j \le l-1$, it follows from our construction that there exists a good ball $B^*(y_{g_i^{j}}, \gamma^j r_0)$ so that $d_Z(z, y_{g_i^{j}}) \le 2\gamma^j r_0$ and
		\begin{align*}
d_Z(z, \LL_{y'_{g_i^{j}}, \gamma^j r_0}) \le \sum_{m=j}^l \ep \gamma^m r_0 \le 2 \ep \gamma^j r_0,
	\end{align*}
	where $y'_{g_i^j}$ is a $(k, \zeta, \gamma^j r_0)$-cylindrical point in $B^*(y_{g_i^{j}}, \gamma^j r_0)$. Thus, $z$ is $(k, \delta^2, s)$-cylindrical with $\left|\widetilde \WW_{z}(s^2)-\Th_{n-k}\right|\leq \delta^2$ by Proposition \ref{prop:almostconstant} (i)(ii).
	
As before, $\NNN^{l+1}$ is a $(k, \delta, \cc, r_0)$-cylindrical neck region with a center $\CCC^{l+1}$ consisting of the centers of the above balls and the radius function chosen to be the corresponding radii of these balls. 

Here, $(n1)$ follows by construction, while $(n2)$ and $(n3)$ were established earlier. The second assertion of $(n4)$ follows from our construction together with Proposition \ref{prop:almostconstant} (iii). Thus, the only remaining point is to verify the first assertion of $(n4)$. Specifically, we need to show that for $x = g_{i}^{l+1}$ or $b_{i}^{l+1}$ and
		\begin{align*}
s \in [\cc^{-5} \gamma^{l+1} r_0, r_0-d_Z(x, x_0)/2]=[\gamma^l r_0, r_0-d_Z(x, x_0)/2],
	\end{align*}
the inclusion
		\begin{align*}
\LL_{x, s} \cap B^*(x, s) \subset B^*_{\cc s}(\CCC^{l+1})
	\end{align*}
holds. Without loss of generality, we may assume $s \in[ \gamma^l r_0, \gamma^{l-1} r_0]$. For the general case $s \in[ \gamma^j r_0, \gamma^{j-1} r_0]$ for $1 \le j \le l-1$, the same reasoning applies after selecting the comparison point from $\{y_{g_i^{l+1}}, y_{b_i^m}\}_{j-1 \le m \le l+1}$.

For any $w \in \LL_{x, s} \cap B^*(x, s)$, one of the following alternatives occurs:

\begin{itemize}
\item There exists another point $x_{i'}= y_{g_{i'}^{l+1}}$ or $y_{b_{i'}^{l+1}}$ such that 
	\begin{align*}
d_Z(w, x_{i'}) \le 4\mathfrak c^2 \gamma^{l+1} r_0 < \cc s;
	\end{align*}	
\item or there exists a point $y_{b_{i'}^l}$ such that
	\begin{align*}
d_Z(w, y_{b_{i'}^l}) \le 2 \cc^2 \gamma^l r_0 <\cc s.
	\end{align*}	
\end{itemize}

Thus, in both cases, the first statement in $(n4)$ is verified.

Now, it is clear from \eqref{neckdecom1a1} that
	\begin{align}\label{constcylneck2}
		\lc\MS_{\mathrm c}^k \setminus \MS^{k-1}_{\mathrm c}\rc \bigcap B^*(x_0, r_0)\subset \bigcup_{g^{l+1}_i}B^*(y_{g_i^{l+1}}, \gamma^{l+1} r_0).
	\end{align}

We set $\mathcal G^l:=\bigcup_{g^l_i} \{y_{g_i^l}\}$. Then it is straightforward from our construction that $\mathcal G^{l+1} \subset B^*_{2\gamma^l r_0}(\mathcal G^l)$. Then we define 
	\begin{align}\label{constcylneck2a}
\CCC_0:=\lim_{l \to \infty} \mathcal G^l,
	\end{align}
where the limit is the Hausdorff limit with respect to $d_Z$.

Letting $l\to\infty$, we obtain
	\begin{align*}
		\NNN:= B^*(x_0, 2 r_0)\setminus \lc \CCC_0\bigcup\bigcup_{1\leq j < \infty}\bigcup_{b_i^j}B^*(y_{b_i^j},\gamma^j r_0)\rc,
	\end{align*}
	is a $(k, \delta, \cc, r_0)$-cylindrical neck region with center $\CCC:=\CCC_0\bigcup \CCC_+$, where $\CCC_+=\bigcup_{b_i^j}\{y_{b^j_i}\}$, and radius function chosen to be the corresponding radii of these balls. By \eqref{constcylneck2}, we have
	\begin{align*}
		\lc\MS_{\mathrm c}^k \setminus \MS^{k-1}_{\mathrm c}\rc \bigcap B^*(x_0,r_0)\subset \CCC_0.
	\end{align*}
	This finishes the proof.
\end{proof}

Next, we define the packing measure associated with a cylindrical neck region as follows:

\begin{defn}[Packing measure]\label{defiofpackingmeasure2}\index{packing measure}
	For a $(k,\delta,\cc, r)$-cylindrical neck region $\NNN=B^*(z,2r)\setminus B_{r_x}^*(\CCC)$, we define the \textbf{packing measure} $\mu$ as
	\begin{equation*}
		\mu:=\sum_{x\in \CCC^+}r_x^{k}\delta_x+\HHH^{k}|_{\CCC_0},
	\end{equation*}
	where $\HHH^{k}$ is the $k$-dimensional Hausdorff measure with respect to $d_Z$.
\end{defn}

For the remainder of this subsection, we fix a $(k,\delta, \cc, r)$-cylindrical neck region $\NNN=B^*(z,2r)\setminus B_{r_x}^*(\CCC)$. First, we prove

\begin{lem}\label{timeestimate}
	Let $x,y\in\CCC$ with $s=d_Z(x,y)$. For any $\ep>0$, if $\delta\leq\delta(n,Y,\ep)$, then 
	$$|\t(x)-\t(y)|\leq \ep s^2.$$
\end{lem}
\begin{proof}
By $(n1)$ in Definition \ref{defiofcylneckregion}, we have
	$$s=d_Z(x,y)\geq \cc^2(r_x+r_y)\geq \cc^2\max\{r_x,r_y\}.$$
By $(n3)$ in Definition \ref{defiofcylneckregion}, both $x$ and $y$ are $(k,\delta,s)$-cylindrical. If $\delta\leq \delta(n,Y,\ep)$, then by the rapid clearing out Lemma \ref{clearout}, 
	\begin{align*}
		\t^{-1}\big([\t(x)+\ep s^2, \infty) \big)\bigcap B^*(x,\ep^{-1}s)=\emptyset.
	\end{align*} 
	Since $y\in B^*(x,\ep^{-1}s) \cap Z_{[\t(x)-s^2, \t(x)+s^2]}$, we obtain $\t(y)\leq \t(x)+\ep s^2$. Similarly, we have $\t(x)\leq \t(y)+\ep s^2$. This finishes the proof.
\end{proof}

Since $\CCC$ is nonempty and every point in $\CCC$ is both $(k, \delta, r)$-cylindrical and $(\delta, r)$-selfsimilar, it follows from Proposition \ref{equivalencesplitsymetric1} (ii) and \cite[Proposition 10.7]{fang2025RFlimit} that the following holds.

\begin{lem}\label{lem:scale1}
For any $\ep>0$, if $\delta \le \delta(n, Y, \ep)$, then there exists a limiting heat flow $\vec{u}=(u_1,\ldots,u_k)$, defined on $Z_{[\t(z)-100\ep^{-2} r^2, \t(z)]}$, such that for any $x \in \CCC$ and $s \in [\ep r, \ep^{-1} r]$, we can find a matrix $T'_{x, s}$ with $\|T'_{x, s}-\mathrm{Id}\| \le \ep$ such that $\vec{u}'_{x, s}=T'_{x, s}\lc \vec{u}-\vec{u}(x) \rc$ is a $(k,\epsilon,s)$-splitting map at $x$.
\end{lem}

For the rest of this subsection, we use $\vec{u}$ to denote the map constructed in Lemma \ref{lem:scale1}. Next, we prove the Lipschitz property of $\vec{u}$.

\begin{lem}\label{lem:lip1}
There exists a constant $C_{\lip}=C_{\lip}(n)>0$ such that if $\delta \le \delta(n, Y)$, then 
	\begin{align*}
|\vec{u}(x)-\vec{u}(y)| \le C_{\lip} d_Z(x, y), \quad \forall x, y \in \CCC.
	\end{align*} 
\end{lem}

\begin{proof}
Passing to the limit in \cite[Proposition 10.4]{fang2025RFlimit}, we may choose $s=s(n)$ sufficiently large so that
	\begin{align*}
|\vec{u}'_{x, s}(y)| \le C(n) d_Z(x, y), \quad \forall x, y \in \CCC.
	\end{align*} 
By the definition of $\vec{u}_{x, s}$ and Lemma \ref{lem:scale1}, it follows that
	\begin{align*}
|\vec{u}(x)-\vec{u}(y)| \le C(n) d_Z(x, y),
	\end{align*} 
	which completes the proof.
\end{proof}

We consider the following extra assumption:
\begin{itemize}
\item[$(*_D)$] For any $x\in\CCC$ and $r_x\leq s\leq r-d_Z(x, z)/2$, we have
	\begin{align*}
D^{-1}s^k\leq \mu(B^*(x,s))\leq D s^k.
	\end{align*} 
\end{itemize}\index{$(*_D)$}

\begin{prop}\label{choosenondegeneratingpoints}
	For any positive constants $D$ and $\ep$, if $(*_D)$ holds and $\delta\leq \delta(n,Y, \cc, D,\epsilon)$, then we can find $\CCC_{\epsilon}\subset\CCC\cap B^*(z,15r/8) $ such that the following hold.
	\begin{enumerate}[label=\textnormal{(\roman{*})}]
		\item $\mu\left(\CCC_{\epsilon}\right)\geq (1-\epsilon)\mu\left(\CCC\cap B^*(z,15r/8)\right)$.
		
		\item For any $x\in\CCC_{\epsilon}$, we have
		\begin{equation*}
			\sum_{r_x\leq r_i=2^{-i} \leq 2^{-5} r}\aint_{B^*(x,r_i)}\left|\widetilde \WW_y(r_i^2/40)-\widetilde \WW_y(40 r_i^2)\right|^{1/2}\,\mathrm{d}\mu(y)\leq\epsilon.
		\end{equation*}
		
				\item For any $x\in \CCC_{\epsilon}$ and $s \in [r_x, r-d_Z(x, z)/2]$, there exists a matrix $T_{x, s}$
 with $\|T_{x, s}-\mathrm{Id} \| \le \ep$ such that $\vec{u}_{x, s}=T_{x, s}\lc \vec{u}-\vec{u}(x) \rc$ is a $(k,\epsilon,s)$-splitting map at $x$. Here, $\vec{u}$ is the map obtained in Lemma \ref{lem:scale1}.
		
		\item $\vec{u}$ is bi-Lipschitz on $\CCC_{\epsilon}$. More precisely, there exists a constant $c_{\lip}=c_{\lip}(n, Y)>0$ such that for any $x, y \in \CCC_{\epsilon}$,
		\begin{equation*}
			c_{\lip} d_Z(x,y)\leq |\vec{u}(x)-\vec{u}(y)|\leq C_{\lip} d_Z(x,y),
		\end{equation*}
		where $C_{\lip} $ is the same constant in Lemma \ref{lem:lip1}.
	\end{enumerate}
\end{prop} 
\begin{proof}
For simplicity, we assume $r=1$. First, we prove
	\begin{claim}\label{ahlclaim1}
		For $\delta'>0$, if $\delta\leq\delta(n,Y, \cc, D, \delta')$, then 
		\begin{equation*}
			\aint_{B^*(z,15/8)}\sum_{r_x\leq r_i\leq 2^{-5}}\aint_{B^*(x,r_i)}\left|\widetilde \WW_y(r_i^2/40)-\widetilde \WW_y(40 r_i^2)\right|^{1/2}\,\mathrm{d}\mu(y)\, \mathrm{d}\mu(x)\leq\delta'.
		\end{equation*}
	\end{claim}
	
We define $W_x(s)=\abs{\widetilde \WW_x(s/40)-\widetilde \WW_x(40 s)}^{\frac 1 2}$ for $s \ge  r_x^2/C_0^2$ and $W_x(s)=0$ for $s <  r_x^2/C^2_0$, where $C_0=\cc^{-2}+1$. Then it is clear that
	\begin{align}\label{ahl0a}
		\sum_{r_x\leq r_i\leq 2^{-5}}\aint_{B^*(x,r_i)}\left|\widetilde \WW_y(r_i^2/40)-\widetilde \WW_y(40 r_i^2)\right|^{1/2}\,\mathrm{d}\mu(y) \le \sum_{r_i\leq 2^{-5}}\aint_{B^*(x,r_i)} W_y(r_i^2)\,\mathrm{d}\mu(y).
	\end{align}
	Indeed, for any $y \in B^*(x, r_i)$, if $r_i \ge r_x$, it follows from the $\cc^{-2}$-Lipschitz property of the radius function that 
\begin{equation*}
r_y \le r_x+\cc^{-2}d_Z(x, y) \le C_0 r_i.
		\end{equation*}	
From this, it is clear that \eqref{ahl0a} holds. Moreover, if $r_i \le \frac{r_x}{2C_0}$, we have for any $y \in B^*(x, r_i)$,
\begin{equation*}
r_y \ge r_x-\cc^{-2}d_Z(x, y) > r_x/2 \ge C_0 r_i,
		\end{equation*}
which implies $W_y(r_i^2)=0$. From this fact, we obtain
	\begin{align}\label{ahl0aex001}
\sum_{r_i \le \frac{r_x}{2C_0}}\aint_{B^*(x,r_i)} W_y(r_i^2)\,\mathrm{d}\mu(y)=0.
	\end{align}

Then, we compute
	\begin{align}\label{ahl1}
		&\aint_{B^*(z,15/8)}\sum_{r_x\leq r_i\leq 2^{-5}}\aint_{B^*(x,r_i)}\left|\widetilde \WW_y(r_i^2/40)-\widetilde \WW_y(40 r_i^2)\right|^{1/2}\,\mathrm{d}\mu(y)\, \mathrm{d}\mu(x)\nonumber\\
		\leq& \aint_{B^*(z,15/8)}       \sum_{r_i\leq 2^{-5}}\aint_{B^*(x,r_i)}W_y(r_i^2)\,\mathrm{d}\mu(y)\, \mathrm{d}\mu(x)\nonumber\\
	=& \aint_{B^*(z,15/8)}      \lc \sum_{r_x \le r_i\leq 2^{-5}}\aint_{B^*(x,r_i)}W_y(r_i^2)\,\mathrm{d}\mu(y)\,+\sum_{\frac{r_x}{2C_0} \le r_i <r_x} \aint_{B^*(x,r_i)}W_y(r_i^2)\,\mathrm{d}\mu(y) \rc \, \mathrm{d}\mu(x) \nonumber\\	
		=& \aint_{B^*(z,15/8)}  \sum_{r_x \le r_i\leq 2^{-5}} \frac{1}{\mu(B^*(x,r_i))} \int_{B^*(x,r_i)}W_y(r_i^2)\,\mathrm{d}\mu(y)\,\mathrm{d}\mu(x)+ C(\cc) \delta^{\frac 1 2} \nonumber\\	
		\leq & D \aint_{B^*(z,15/8)}  \sum_{r_i\leq 2^{-5}} r_i^{-k} \int_{B^*(x,r_i)}W_y(r_i^2)\,\mathrm{d}\mu(y)\,\mathrm{d}\mu(x)+ C(\cc) \delta^{\frac 1 2} \nonumber\\	
		\leq & D^2\int_{B^*(z,31/16)}\int_{B^*(z,31/16)}\sum_{r_i\leq 2^{-5}}r_i^{-k}\chi_{\{d_Z(x,y)\leq r_i\}}W_y(r_i^2)\, \mathrm{d}\mu(y)\,\mathrm{d}\mu(x) + C(\cc) \delta^{\frac 1 2}\nonumber\\
		= & D^2\int_{B^*(z,31/16)}\int_{B^*(z,31/16)} \sum_{r_i\leq 2^{-5}}r_i^{-k}\chi_{\{d_Z(x,y)\leq r_i\}}W_y(r_i^2)\, \mathrm{d}\mu(x)\,\mathrm{d}\mu(y) + C(\cc) \delta^{\frac 1 2} \nonumber\\
		\leq & D^2\int_{B^*(z,31/16)}\sum_{r_i\leq 2^{-5}} r_i^{-k}\mu(B^*(y,r_i))W_y(r_i^2)\, \mathrm{d}\mu(y) + C(\cc) \delta^{\frac 1 2}\nonumber\\
		\leq & D^3 \int_{B^*(z,31/16)} \sum_{r_i\leq 2^{-5}}W_y(r_i^2)\, \mathrm{d}\mu(y)+ C(\cc) \delta^{\frac 1 2}\nonumber\\
		\leq& D^3 \int_{B^*(z,31/16)}\sum_{ C_0^{-1} r_y\leq r_i\leq 2^{-5}}\left|\widetilde \WW_y(r_i^2/40)-\widetilde \WW_y(40 r_i^2)\right|^{1/2}\, \mathrm{d}\mu(y)+ C(\cc) \delta^{\frac 1 2},
	\end{align}
	where we used \eqref{ahl0aex001} and $(*_D)$. Note that in the seventh line, we also used Fubini's theorem. From Proposition \ref{sumWonRFlimit} and $(n2)$ of Definition \ref{defiofcylneckregion}, we conclude that for any $y\in\CCC$,
	\begin{align}\label{ahlforsxx1}
		\sum_{ C_0^{-1} r_y\leq r_i\leq 2^{-5}}|\widetilde \WW_y(r_i^2/40)-\widetilde \WW_y(40 r_i^2)|^{1/2}\leq\Psi(\delta).
	\end{align}
	Combining with \eqref{ahl1}, this finishes the proof of Claim \ref{ahlclaim1}.
	
	\begin{claim}\label{ahlclaim2}
		For $\epsilon>0$, if $\delta''\leq\delta''(n,Y,D,\epsilon)$ and for some $x\in \CCC\bigcap B^*(z,15/8)$ with
		\begin{equation}\label{ahlfors1}
			\sum_{r_x\leq r_i\leq 2^{-5}}\aint_{B^*(x,r_i)}\left|\widetilde \WW_y(r_i^2/40)-\widetilde \WW_y(40 r_i^2)\right|^{1/2}\,\mathrm{d}\mu(y)\leq\delta'',
		\end{equation}
		then for any $r_x\leq s \leq 2^{-5}$, there exists a matrix $T_{x, s}$
 with $\|T_{x, s}-\mathrm{Id}\| \le \ep$ such that $\vec{u}_{x, s}=T_{x, s}\lc \vec{u}-\vec{u}(x) \rc$ is a $(k,\epsilon,s)$-splitting map at $x$.
	\end{claim}
	
	We set $\alpha=c_1D^{-2}$, where $c_1=c_1(n, Y)>0$ is a constant to be determined. For any $r_i=2^{-i}$ with $\alpha^{-1}r_x\leq r_i  \le 2^{-5}$, we can find a subset $\CCC_{r_i,x}\subset\CCC \cap B^*(x, r_i)$ such that 
		\begin{align*}
\mu(\CCC_{r_i,x}) \ge \frac{1}{2}\mu \lc \CCC \cap B^*(x, r_i)\rc \ge  \frac{r_i^k}{2 D}
	\end{align*}
	and for any $w\in \CCC_{r_i,x}$, we have
	\begin{align}\label{ahlfors1a}
		\left|\widetilde \WW_w(r_i^2/40)-\widetilde \WW_w(40 r_i^2)\right|^{1/2}\leq 2\aint_{B^*(x,r_i)}\left|\widetilde \WW_y(r_i^2/40)-\widetilde \WW_y(40 r_i^2)\right|^{1/2}\,\mathrm{d}\mu(y).
	\end{align}

From Lemma \ref{timeestimate}, we know that for $w\in \CCC_{r_i,x}$, $|\t(w)-\t(x)|\leq \Psi(\delta)r_i^2$. By Definition \ref{defnselfsimilarity}, any point $w\in\CCC_{r_i,x}$ is $(\delta, s)$-selfsimilar for $r_w \leq s \leq 1$.

We claim that there must exist strongly $(k, \alpha, \delta, r_i)$-independent points in $\CCC_{r_i,x}$ at $x$. Indeed, suppose otherwise. Then by Lemma \ref{existenceofindependentpointsafterbaseRFlimit}, we can find $x_i\in\CCC_{r_i,x}$ for $1\le i \le N$ with $N\leq C(n,Y)\alpha^{1-k}$ and
	\begin{align*}
		\CCC_{r_i,x}\subset \bigcup_{i=1}^N B^*(x_i, \alpha r_i).
	\end{align*} 
	Then by assumption $(*_D)$, we have
	\begin{align*}
		\frac{r_i^k}{2D}\leq\mu(\CCC_{r_i,x})\leq \sum_i^N\mu \big(B^*(x_i,\alpha r_i)\big)\leq N D(\alpha r_i)^k\leq C(n,Y)D \alpha r_i^k.
	\end{align*}
	By the definition of $\alpha$, we obtain a contradiction if $c_1$ is chosen to be small.

Let $\{y_j\}_{1 \le j \le k} \subset \CCC_{r_i,x}$ be strongly $(k,\alpha,\delta,r_i)$-independent points at $x$. Then by \eqref{ahlfors1a}, we have
	\begin{align*}
		\left|\widetilde \WW_{y_j}(r_i^2/40)-\widetilde \WW_{y_j}(40 r_i^2)\right|^{1/2}\leq 2\aint_{B^*(x,r_i)}\left|\widetilde \WW_y(r_i^2/40)-\widetilde \WW_y(40 r_i^2)\right|^{1/2}\, \mathrm{d}\mu(y),
	\end{align*}
	which implies
	\begin{align*}
		\mathfrak{S}_{r_i}^{k,\alpha,\delta}(x)\leq 2k \aint_{B^*(x,r_i)}\abs{\widetilde \WW_y(r_i^2/40)-\widetilde \WW_y(40 r_i^2)}^{1/2}\,\mathrm{d}\mu(y)+\abs{\widetilde \WW_x(r_i^2/40)-\widetilde \WW_x(40 r_i^2)}^{1/2}.
	\end{align*}
	Combining this with \eqref{ahlforsxx1} and \eqref{ahlfors1}, we have
		\begin{align*}
\sum_{(2\alpha)^{-1}r_x \leq r_i\leq 2^{-5}}	\mathfrak{S}_{r_i}^{k,\alpha,\delta}(x) \leq C(n) \delta''.
	\end{align*}
	Thus, Claim \ref{ahlclaim2} follows from Proposition \ref{equivalencesplitsymetric1} and Theorem \ref{nondegenrerationthm}.

Now, we define $\CCC_\ep$ as:
	\begin{align*}
		\CCC_\ep:=\left\{x\in \CCC\cap B^*(z,15 /8)\big|\sum_{r_x\leq r_i\leq 2^{-5}}\aint_{B^*(x,r_i)}\left|\widetilde \WW_y(r_i^2/40)-\widetilde \WW_y(40 r_i^2)\right|^{1/2}\,\mathrm{d}\mu(y)<\delta''\right\}.
	\end{align*}
If $\delta'' \ll \ep$ and $\delta' \ll \delta''$, then statement (ii) holds. Moreover, by Claim \ref{ahlclaim1}, we obtain
		\begin{align*}
\mu\left(\CCC_{\epsilon}\right)\geq (1-\epsilon)\mu\left(\CCC\cap B^*(z,15/8)\right).
	\end{align*}
which implies statement (i). Furthermore, statement (iii) follows from Claim \ref{ahlclaim2}.

Next, we prove the bi-Lipschitz property of $\vec{u}$ on the set $\CCC_{\ep}$. Note that the upper bound has already been proved by Lemma \ref{lem:lip1}, so we focus on the lower bound.

Given any $x, y \in \CCC_{\ep}$, define $s := d_Z(x, y) \geq \cc^{2} \max\{r_x, r_y\}$. By $(n3)$ of Definition \ref{defiofcylneckregion} and Lemma \ref{timeestimate}, we have $|\t(x) - \t(y)| \leq \Psi(\delta) s^2$. Moreover, $x$ is $(k, \delta, s)$-cylindrical. Also, by Claim \ref{ahlclaim2}, the map $\vec{u}_{x, s}=T_{x, s}\lc \vec{u}-\vec{u}(x) \rc$ is a $(k,\Psi(\delta),s)$-splitting map at $x$.

By Definition \ref{def:almost0} and Proposition \ref{prop:almostconstant} (iii), $\vec{u}_{x,s}$ is close to the standard coordinate functions on $\bar{\mathcal{C}}^k$, and $y$ is close to a point in $\bar{\mathcal{C}}^k_0$, provided that $\delta$ is sufficiently small. Applying a limiting argument and Lemma \ref{lem:comparedis1}, we obtain
		\begin{align*}
s \le C(n, Y) |\vec{u}_{x, s}(y)|.
	\end{align*}
Using the definition of $\vec{u}_{x,s}$, this implies
		\begin{align*}
s \le C(n, Y) |\vec{u}(x)-\vec{u}(y)|.
	\end{align*}

This establishes the bi-Lipschitz property of $\vec{u}$ on $\CCC_{\ep}$ and completes the proof of statement (iv), thereby concluding the proof of the proposition.
\end{proof}

To establish the Ahlfors regularity of the measure $\mu$, we first prove

\begin{lem} \label{lem:inductahlfors}
Under the assumption $(*_D)$, if $\cc \le \cc(n)$ and $\delta \le \delta(n, Y, \cc, D)$, then for any $x \in \CCC$ and $r_x\leq s\leq r-d_Z(x, z)/2$, the following holds:
		\begin{equation*}
		 C^{-1}_0(n,Y, \cc) s^{k}\leq\mu(B^*(x,s))\leq C_0(n, Y, \cc) s^{k}.
		\end{equation*}
\end{lem}

\begin{proof}
Note that by Definition \ref{defiofcylneckregion}, $\NNN'=\NNN \cap B^*(x, 2s)$ is also a $(k, \delta,\cc, s)$-neck region with center $\CCC'=\CCC \cap B^*(x, 2s)$ and the same radius function. 

We first prove the upper bound. By Proposition \ref{choosenondegeneratingpoints}, there exists $\CCC'' \subset \CCC' \cap B^*(x, 15 s/8)$ so that
		\begin{align}\label{volupp1}
			\mu\left(\CCC''\right)\geq \mu\left(\CCC'\bigcap B^*(x,15s/8)\right)/2.
		\end{align}
We write $\CCC''=\CCC''_+\bigcup \CCC''_0$, where $\CCC''_+=\CCC''\bigcap\CCC_+$ and $\CCC''_0=\CCC''\bigcap \CCC_0$.

By Definition \ref{defiofcylneckregion} and Proposition \ref{choosenondegeneratingpoints} (iv), $\left\{B\lc\vec{u}(y), c_{\lip} \cc^2 r_y \rc \right\}_{y\in\CCC''}$ are mutually disjoint. Moreover, by Lemma \ref{lem:lip1}, we have
		\begin{align*}
|\vec{u}(x)-\vec{u}(y)| \le C_{\lip} d_Z(x, y) \le C(n) s
	\end{align*}
for any $y\in \CCC''$. Thus, we estimate
		\begin{align*}
			\mu(\CCC'')&=\sum_{y\in \CCC''_+}r_y^{k}+\HHH^k\left(\CCC''_0\right)\\
			&\leq C(n,Y) \cc^{-2k} \sum_{y\in\CCC''_+} \abs{B(\vec{u}(y), c_{\lip} \cc^2 r_y)}+C(n,Y) \abs{\vec{u}(\CCC''_0)} \leq C(n,Y, \cc) s^k,
		\end{align*}
where we use $|\cdot|$ to denote the volume in $\R^k$. Combining with \eqref{volupp1}, we obtain
		\begin{align*}
\mu\left(B^*(x,s)\right)\leq\mu\left(B^*(x,15s/8)\bigcap\CCC\right)\leq 2\mu\left(\CCC''\right)\leq C(n,Y, \cc) s^k.
		\end{align*}

Next, we prove the lower bound. By Lemma \ref{lem:scale1} and Theorem \ref{thm:existtransRFL}, for a small constant $\ep>0$ to be determined later, if $\delta \le \delta(n, Y, \ep)$, then for every point $y\in\CCC\bigcap B^*(x,s)$ and every scale $s' \in [\cc^2 r_y, 2s]$, there exists a lower triangular $k \times k$ matrix $A_{y, s'}$ such that $A_{y, s'}( \vec{u}-\vec{u}(y))$ is a $(k, \ep, s')$-splitting map at $y$. Without loss of generality, we assume $\vec u(x)=\vec{0}^k$, $A_{x, s}=\mathrm{Id}$ and $\|A_{y, s}-\mathrm{Id}\| \le \ep$ for any $y\in\CCC\bigcap B^*(x,s)$. Moreover, $A_{y, s'}$ satisfies conclusions (ii) and (iii) in Theorem \ref{thm:existtransRFL}. 

Note that since $y$ is $(k, \delta, s')$-cylindrical for any $s' \in [\cc^2 r_y, s]$, it follows that $y$ is not $(k+1, \eta, s')$-splitting for some constant $\eta=\eta(n, Y)$. Hence, the conditions of Theorem \ref{thm:existtransRFL} are readily verified.

Next, we claim:
	\begin{equation}\label{lowerboundahlforsregularity}
		B(\vec{0}^k,s/2)\subset\bigcup_{y\in\CCC\bigcap B^*(x,s)}\left\{A_{y,\cc^{-5} r_y}^{-1}\lc \bar{B}(\vec u(y),\cc^{-5} r_y)\rc\right\},
	\end{equation}
	where $A_{y,s}^{-1}\lc \bar{B}(\vec u(y),s)\rc:=\vec u(y)+A_{y,s}^{-1}\bar{B}(\vec 0^k, s)$.

Once \eqref{lowerboundahlforsregularity} is proved, it follows that
	\begin{align*}
		0<C^{-1}(n,Y) s^k\leq \left|B(\vec{0}^k,s/2)\right|&\leq C(k)\cc^{-5k} \sum_{y\in \CCC_+\bigcap B^*(x,s)} \det\lc A_{y,\cc^{-5} r_y}^{-1} \rc r_y^k+\abs{\vec{u}\big(\CCC_0\bigcap B^*(x,s)\big)}\\
		&\leq C(n) \cc^{-5k} \mu \left(B^*(x,s)\right),
	\end{align*}
	where we used the fact that $\det(A_{y, \cc^{-5} r_y}) \ge 2^{-k}$ and $\vec{u}$ is $C(n)$-Lipschitz.

	To prove \eqref{lowerboundahlforsregularity}, we argue by contradiction. Assume that there exists $z'\in B(\vec 0^k,s/2)$ not covered by $\bigcup_{y\in\CCC\bigcap B^*(x,s)}\left\{A_{y,\cc^{-5} r_y}^{-1}\lc \bar{B}(\vec u(y),\cc^{-5} r_y)\rc\right\}$. Then for any $y\in\CCC\bigcap B^*(x,s)$, set $s_y$ to be the minimal number $s' \geq \cc^{-5}  r_y$ such that $z' \in A_{y,s'}^{-1}\lc \bar{B}(\vec u(y),s')\rc$. Then define $\bar s=\min_{y\in\CCC\bigcap B^*(x,s)}s_y$. 
	
	We fix a small parameter $\ep'>0$ to be determined later. Note that by Proposition \ref{prop:almostconstanta}, if $\ep \le \ep(n, Y, \ep')$, then we can find $x'\in \LL_{x,s}\bigcap B^*(x,3s/4)$ such that $|\vec u(x')-z'|\leq \ep' s$. Then using $(n4)$ in Definition \ref{defiofcylneckregion} and $C(n)$-Lipschitz property of $\vec u$, we can find $x''\in \CCC\bigcap B^*(x,s)$ such that 
			\begin{align} \label{lowerextra001}
|\vec u(x'')-z'|\leq \ep' s+C(n) \cc s.
	\end{align}
Since $\vec u$ and $A_{x'', s}(\vec u-\vec u(x''))$ are both $(k, \ep, s)$-splitting maps at $x$ and $x''$, respectively, and $|\t(x'')-\t(x)| \le \Psi(\delta) s^2$ by Lemma \ref{timeestimate}, it follows from the argument of Lemma \ref{lem:compmatrix} that
		\begin{align*}
\rVert A_{x'',s}-\mathrm{Id}\rVert \le \ep',
	\end{align*}
Combining this with \eqref{lowerextra001}, we obtain
		\begin{align*}
\abs{A_{x'', s} \lc \vec{u}(x'')-z' \rc} \le 2 \ep' s+2C(n)\cc s.
	\end{align*}
If we choose $\cc \le \cc(n)$ and $\ep'$ to be sufficiently small, then we have
			\begin{align} \label{lowerextra002}
\bar s\leq s_{x''}\leq 2(\ep' s+C(n) \cc s)< \frac{s}{100}.
	\end{align}
	
	 Choose $\bar y\in\CCC\bigcap B^*(x,s)$ such that $\bar s= s_{\bar y}$. Then we must have $\bar s> \cc^{-5} r_{\bar y}$ and $z' \in A_{\bar y,2\bar s}^{-1}\lc \bar{B}(\vec u(\bar y),2\bar s)\rc$. Thus,
		\begin{align*}
A_{\bar y,2\bar s}z'\in \bar{B} \lc A_{\bar y,2\bar s}\vec u(\bar y), 2\bar s \rc.
	\end{align*}

	Since $A_{\bar y,2\bar s}(\vec{u}-\vec{u}(\bar y))$ is a $(k,\ep,2\bar s)$-splitting map at $\bar y$ and $\bar y$ is $(k, \ep, 2\bar s)$-cylindrical, it follows from Proposition \ref{prop:almostconstanta} that there exists $y' \in \mathcal L_{\bar y, 2\bar s} \cap B^*(\bar y, 2\bar s)$ such that
			\begin{align}\label{eq:lowerextra1}
\abs{A_{\bar y,2\bar s} \lc \vec{u}(y')-z' \rc} \le \ep' \bar s,
	\end{align}
	provided that $\ep \le \ep(n, Y, \ep')$.
	
In particular, by Theorem \ref{thm:existtransRFL} (ii), this implies that 
			\begin{align} \label{lowerextra003}
\abs{\vec u(y')-z'} \le \ep' \lc \frac{s}{\bar s} \rc^{C(n) \ep} \bar s \leq \frac{s}{100},
	\end{align}
where we used \eqref{lowerextra002}, provided that $\ep \le \ep(n, Y)$ is small.

Since $\vec u$ is a $(k, \ep, s)$-splitting map at $x$, it follows from \eqref{lowerextra003} and a limiting argument that $y'\in B^*(x,3s/4)$. Moreover, from $(n4)$ in Definition \ref{defiofcylneckregion}, we can find $y'' \in \CCC\bigcap B^*(\bar y, 2\bar s)$ so that $d_Z(y', y'') \le 2 \cc \bar s$, which gives that $y''\in B^*(x,(3/4+\cc)s)\subset B^*(x,s)$. 

Since $A_{\bar y,2\bar s}(\vec{u}-\vec{u}(\bar y))$ is $C_{\lip}$-Lipschitz by Lemma \ref{lem:lip1}, we obtain that 
			\begin{align}\label{eq:lowerextra2}
\abs{A_{\bar y,2\bar s} \lc \vec{u}(y')-\vec{u}(y'') \rc} \le 2C(n) \cc \bar s.
	\end{align}

Combining \eqref{eq:lowerextra1} and \eqref{eq:lowerextra2}, it implies that
	\begin{align}\label{equ:lowA3}
\abs{A_{\bar y,2\bar s} \lc \vec{u}(y'')-z' \rc} \le (2C(n) \cc+\ep') \bar s.
	\end{align}

Since $A_{\bar y, 2\bar s}(\vec{u}-\vec{u}(\bar y))$ and $A_{y'', 2\bar s}(\vec{u}-\vec{u}(y''))$ are both $(k, \ep, 2\bar s)$-splitting map at $\bar y$ and $y''$, respectively, and $|\t(y'')-\t(\bar y)| \le \Psi(\delta) \bar s^2$ by Lemma \ref{timeestimate}, it follows from the argument of Lemma \ref{lem:compmatrix} that
		\begin{align*}
		 \rVert A_{\bar y, 2\bar s}\circ A_{y'', 2\bar s}^{-1}-\mathrm{Id}\rVert\leq \ep'.
	\end{align*}
Combining this with \eqref{equ:lowA3}, it follows that
	\begin{align}\label{equ:lowA4}
\abs{A_{y'',2\bar s} \lc \vec{u}(y'')-z' \rc} \le 2(2C(n) \cc+\ep') \bar s.
	\end{align}
On the other hand, it follows from Lemma \ref{lem:compmatrix} (or equation \eqref{equ:compmatrix}), we have
	\begin{align*}
		 \rVert  A_{y'', 2\bar s}^{-1} \circ A_{y'', \bar s/2} -\mathrm{Id}\rVert\leq C(n) \ep,
	\end{align*}
which, when combined with \eqref{equ:lowA4}, implies
	\begin{align}\label{equ:lowA6}
\abs{A_{y'',\bar s/2} \lc \vec{u}(y'')-z' \rc} \le 3(2C(n) \cc+\ep') \bar s.
	\end{align}
	
Thus, if $\cc \le \cc(n)$ and $\ep \le \ep(n, Y)$, we would have $3(2C(n) \cc+\ep')<1/2$. However, \eqref{equ:lowA6} then implies that $z'\in A_{y'',\bar s/2}^{-1}\lc B(\vec u(y''),\bar s/2)\rc$, which contradicts the definition of $\bar s$.

In sum, we have proved \eqref{lowerboundahlforsregularity}, thereby completing the proof of the theorem.
\end{proof}

Next, we prove the Ahlfors regularity of $\mu$ under the extra assumption that $\inf_{x\in\CCC}r_x>0$.

\begin{lem}\label{lem:inductahlfors2}
	Suppose $\inf_{x\in\CCC}r_x>0$. If $\cc \le \cc(n)$ and $\delta \le \delta(n, Y, \cc)$, then for any $x\in \CCC$ and $r_x\leq s\leq r-d_Z(x, z)/2$, the following holds:
		\begin{equation*}
			 C^{-1}(n,Y,\cc)s^{k}\leq\mu\big(B^*(x,s)\big)\leq C(n,Y, \cc) s^{k}.
		\end{equation*}
\end{lem}

\begin{proof}
We say that the property $(A_j)$ holds with a constant $D$ if any $(k, \delta, \cc, r)$-cylindrical neck region $\NNN$ with $\delta \le \delta(n, Y, \cc)$ and $\inf_{y\in\CCC} r_y\geq 2^{-j} r$ satisfies $(*_D)$.

First, note that by Proposition \ref{prop:volumebound}, there are at most $L_1=L_1(n, Y, \cc)$ mutually disjoint balls of size $\cc^2 2^{-5} r$ contained in a ball of size $2r$. Thus, it follows from Definition \ref{defiofpackingmeasure2} that the property $(A_5)$ holds with a constant $D_1=D_1(n, Y, \cc)$.

We set $D_0:= C_0(n, Y, \cc)$, where $C_0$ is the constant from Lemma \ref{lem:inductahlfors}. Then, we define $D_2=D_0+D_1$. We set $D_3=100^k \cc^{-2k} D_2 L_1$ and require $\delta \le \delta(n, Y, \cc, D_3)$ as in Lemma \ref{lem:inductahlfors}.

Assume that the property $(A_{j_0})$ holds with $D_2$ with $j_0 \ge 2$. Given a $(k,\delta, \cc, r)$-cylindrical neck region $\NNN \subset B^*(z, 2r)$ with $\inf_{y\in\CCC} r_y\geq 2^{-j_0-1} r$, we consider $x \in \CCC$ and $r_x\leq s\leq r-d_Z(x, z)/2$.

If $s \le r/2$, then $\inf_{y\in\CCC} r_y\geq 2^{-j_0-1} r \ge 2^{-j_0}s$. By our induction assumption, we have
	\begin{align}\label{inductionstep2}
		D_2^{-1} s^k \leq \mu\left(B^*(x,s)\right)\leq D_2 s^k.
	\end{align}

If $s >r/2$, then $B^*(x,s)\subset B^*(z,3r/2)$. If $r_x \le 2^{-4}r$, then by the inductive assumption,
	\begin{align*}
\mu \left(B^*(x, s)\right) \ge \mu \left(B^*(x,2^{-4} r)\right)  \ge 16^{-k} D_2^{-1} r^k \ge 16^{-k}  D_2^{-1} s^k.
	\end{align*}
If $r_x>2^{-4}r$, then by Definition \ref{defiofpackingmeasure2},
	\begin{align*}
\mu \left(B^*(x, s)\right) \ge r_x^k \ge 16^{-k} r^k \ge 16^{-k} s^k.
	\end{align*}
Thus, in either case, 
	\begin{align}\label{inductionstep2a}
\mu \left(B^*(x, s)\right) \ge 16^{-k}  D_2^{-1} s^k.
	\end{align}
	
On the other hand, let $\CCC':=\{y_i\}_{1 \le i \le N}$ to be the subset of $B^*(z,3r/2) \bigcap \CCC$ such that $r_{y_i} \ge 2^{-4}r$. Clearly, $N \le L_1$. Moreover, $r_{y_i} \le r_x+\cc^{-2} d_Z(x, y_i) \le (1+3\cc^{-2})r$.

Next, choose a maximal $2^{-4}r$-separated set $\{x_j\}$ of $\lc B^*(z,3r/2) \bigcap \CCC\rc \setminus \CCC'$. For each $j$, the inductive assumption gives
	\begin{align*}
\mu \left(B^*(x_j,2^{-4} r)\right)\leq 16^{-k} D_2 r^k.
	\end{align*}
Therefore, by Definition \ref{defiofpackingmeasure2},
	\begin{align}\label{inductionstep2b}
\mu \left(B^*(x, s)\right) \le L_1 \lc (1+3\cc^{-2})^k+16^{-k} D_2 \rc r^k \le L_1 2^k \lc (1+3\cc^{-2})^k+16^{-k} D_2 \rc s^k.
	\end{align}

Combining \eqref{inductionstep2a} and \eqref{inductionstep2b}, we conclude that the property $(A_{j_0+1})$ holds with the constant $D_3$. Hence, we may apply Lemma \ref{lem:inductahlfors} to deduce that for any $x \in \CCC$ and $r_x\leq s\leq r-d_Z(x, z)/2$, we have
	\begin{align*}
D_0^{-1} s^k \le \mu\left(B^*(x,s)\right) \le D_0 s^k.
	\end{align*}
Since $D_0 \le D_2$, it follows that $(A_{j_0+1})$ holds with constant $D_2$. By induction, we conclude that for all $j \ge 5$, the property $(A_j)$ holds with constant $D_2$. This completes the proof.
\end{proof}

Next, we prove the Ahlfors regularity for the $(k,\delta, \cc,r)$-cylindrical neck region constructed in Proposition \ref{cylneckdecomp}.

\begin{thm}[Ahlfors regularity]\label{ahlforsregucyl1}
Given $\ep>0$, if $\cc \le \cc(n)$ and $\delta \le \delta(n, Y, \cc,\ep)$, then, for the $(k,\delta, \cc,r_0)$-cylindrical neck region $\NNN=B^*(x_0 ,2r_0)\setminus B_{r_x}^*(\CCC)$ constructed in Proposition \ref{cylneckdecomp}, for any $x\in\CCC$ and $r_x\leq s\leq r_0-d_Z(x, x_0)/2$, we have
	\begin{align} \label{eq:extraaf}
		 C^{-1}(n,Y, \cc) s^k\leq\mu(B^*(x,s))\leq C(n,Y, \cc)s^k.
	\end{align} 
Moreover, we can find a countable collection of $\HHH^k$-measurable subsets $ E_i\subset \CCC_0$ such that $\HHH^k(\CCC_0\setminus\bigcup_i E_i)=0$ and for each $i$, there exists a bi-Lipschitz map $u_i:E_i\to \R^k$, where $\R^k$ is equipped with the standard Euclidean distance, such that
	\begin{align*}
\sqrt{|\t(x)-\t(y)|}\leq\ep d_Z(x,y),\quad \forall x, y \in E_i.
	\end{align*} 
	\end{thm}

\begin{proof}
From the construction in the proof of Proposition \ref{cylneckdecomp}, there exist a sequence of $(k,\delta, \cc,  r_0)$-cylindrical neck region $\NNN^l$ (see \eqref{constcylneck2b}) with center $\CCC^l$ and radius function $r_x^l$ such that
		\begin{enumerate}[label=\textnormal{(\roman{*})}]
			\item $r^l_{x}\to r_x$ uniformly, 
			
			\item $\NNN^l$ converge to $\NNN$ in the Hausdorff distance, and
			
			\item $\inf_{x\in\CCC^l}r^l_{x}\geq \gamma^l r_0$,
		\end{enumerate}
 where $\gamma=\cc^5$. Moreover, it follows from \eqref{constcylneck2a} that for any $l \ge 1$,
 	\begin{align}\label{covering1a}
\CCC_0 \subset B^*_{2\gamma^l r_0}(\mathcal G^l).
	\end{align}

Let $\mu^l$ be the packing measure with respect to $\NNN^l$. We assume that $\mu^{\infty}$ is the limit of $\mu^l$ in the weak sense. It is clear that $\mu^{\infty}=\mu$ on $\CCC_+$.

Set $r_l:=2 \gamma^l r_0$. For any $x \in \CCC_0$ and $s \in (0, r_0-d(x_0, x)/2)$, we have
	\begin{align}\label{covering1}
		r_l^{k-n-2}\abs{B^*_{r_l} \lc \CCC_0\cap B^*(x, s)\rc} \leq C(n,Y, \cc)\mu^l \lc B^*_{2r_l}\lc B^*(x, s) \cap \CCC_0 \rc\rc
	\end{align}
provided that $r_l<s$. To see this, let $\{B^*(x^l_i, r_l),\, x^l_i \in \mathcal G^l \cap B^*_{r_l}\lc B^*(x, s) \cap \CCC_0 \rc \}_{1 \le i \le N}$ be a covering of $B^*(x, s) \cap \CCC_0$. Note that by \eqref{covering1a}, this covering is possible. 

By Lemma \ref{lem:inductahlfors2}, $\mu^l(B^*(x^l_i, r_l)) \ge c(n, Y, \cc) r_l^k>0$, since the radius function is $r_l$ at $x_i^l$. Since $\{B^*(x_i^l, \cc^2 r_l/2)\}$ are disjoint, it follows from Proposition \ref{prop:volumebound} that any point belongs to at most $C(n, Y, \cc)$ balls in $\{B^*(x_i^l, r_l/2)\}$. Consequently, we conclude that the cardinality $N$ satisfies
	\begin{align*}
N \le C(n, Y, \cc) r_l^{-k} \mu^l \lc B^*_{2r_l}\lc B^*(x, s) \cap \CCC_0 \rc \rc.
	\end{align*}
By Proposition \ref{prop:volumebound}, we obtain
	\begin{align*}
		\abs{B^*_{r_l} \lc \CCC_0\cap B^*(x, s)\rc} \le N \max_i \abs{B^*(x^l_i, 2 r_l)} \le C(n,Y, \cc) \mu^l\left( B^*_{2r_l}\lc B^*(x, s) \cap \CCC_0 \rc\right) r_l^{n+2-k},
	\end{align*}
which yields \eqref{covering1}.

Taking $l\to\infty$, we get the following estimate for Minkowski content $\MMM^k$ from \eqref{covering1}(cf. \cite[Definition 8.9]{fang2025RFlimit}):
	\begin{align} \label{eq:minkvo}
 \MMM^k\left(B^*(x,s)\cap\CCC_0\right)\leq C(n,Y, \cc)\mu^{\infty}\left(\overline{B^*(x, s)}\cap\CCC_0\right) \le C(n,Y, \cc)\mu^{\infty}\left(B^*(x, 2s)\cap\CCC_0\right).
	\end{align}

From this, we have
	\begin{align*}
		\mu\left(B^*(x,s)\cap\CCC_0\right)=\HHH^k\left(B^*(x,s)\cap\CCC_0\right)\leq C(k)\MMM^k \left(B^*(x,s)\cap\CCC_0\right)\leq C(n,Y, \cc)\mu^{\infty}\left(B^*(x,2s)\cap\CCC_0 \right).
	\end{align*}	
Consequently, by Lemma \ref{lem:inductahlfors2}, we have proved that for any $x \in \CCC_0$ and $s>0$,
	\begin{align}\label{covering2aa}
 \mu\lc B^*(x,s) \rc \le C(n, Y, \cc) s^k.
	\end{align}
Since $\mu=\mu^\infty$ on $\CCC_+$, this proves the upper bound for the Ahlfors regularity of $\mu$.

Next, we prove that $\CCC_0$ is $k$-rectifiable with respect to $d_Z$. Fix $\ep>0$, $x \in \CCC_0$, $s \le r_0-d_Z(x, x_0)/2$ and $\delta\leq\delta(n,Y,\ep)$. We define the neck region $\NNN^l_{x,s}:=B^*(x, 2s) \setminus B^*_{r_y^l}(\CCC^l_{x, s})$ with center $\CCC^l_{x, s}:=\CCC^l \cap B^*(x, 2s)$. Note that by Lemma \ref{lem:inductahlfors2}, the corresponding packing measure, denoted by $\mu^l_{x,s}$, satisfies the Ahlfors regularity. Thus, by Proposition \ref{choosenondegeneratingpoints}, we can find $(k,\delta,s)$-splitting map $\vec{u}^l_{x,s}$ and $\CCC^l_{x,s,\ep}\subset \CCC^l_{x,s}$ such that $\vec{u}^l_{x,s}:\CCC^l_{x,s,\ep}\to \R^k$ is $C(n, Y)$-bi-Lipschitz and 
	\begin{align}\label{appro1}
		\mu^l_{x,s}\left(B^*(x,s)\setminus\CCC^l_{x,s,\ep}\right)<\ep s^k.
	\end{align}
Next, we set
	\begin{align*}
\mu_{x,s}:=\lim_{l\to\infty}\mu^l_{x,s}, \quad \CCC_{x,s,\ep}:=\lim_{l\to\infty}\CCC^l_{x,s,\ep}, \quad \vec{u}_{x,s}:=\lim_{l\to\infty} \vec{u}^l_{x,s}.
	\end{align*}
Here, the first limit refers to convergence in the weak sense, while the second refers to Hausdorff convergence. From \eqref{covering2aa} and \eqref{appro1}, we have
	\begin{align}\label{appro2}
		\mu\left(B^*(x,s)\setminus \CCC_{x,s,\ep}\right)\leq C(n,Y, \cc)\mu_{x,s}\left(B^*(x,s)\setminus \CCC_{x,s,\ep}\right)\leq C(n,Y, \cc)\ep s^k.
	\end{align}
Moreover, $\vec{u}_{x,s}$ is a $C(n, Y)$-bi-Lipschitz map on $\CCC_{x,s,\ep}$. Set $E_{x,s,\ep}:=\CCC_0\bigcap \CCC_{x,s,\ep}$. Since the Ricci flow limit space $Z$ is separable (see Theorem \ref{thm:intro1}), we can choose a countable dense subset $\{x_i\}$ of $\CCC_0$ and define
	\begin{align*}
		E=\bigcup_i\bigcup_{s\in\mathbb Q\cap(0,r_0-d_Z(x_0,x_i)/2)}E_{x_i,s,\ep}.
	\end{align*}
By \eqref{appro2} and a standard geometric measure theory argument, we obtain that $\HHH^k(\CCC_0\setminus E)=0$.  Indeed, if $\HHH^k(\CCC_0\setminus E)>0$, then there exists a $k$-density point $y \in \CCC_0\setminus E$ (see \cite[Theorem 3.6]{leon83}). In other words, we have
	\begin{align*}
\limsup_{s \to 0} \frac{\HHH^k \lc B^*(y, s) \cap (\CCC_0\setminus E) \rc}{s^k} \ge c(k)>0.
	\end{align*}
In particular, we can find a small $s \in (0, r_0-d_Z(x_0, y))$ such that $\HHH^k \lc B^*(y, s) \cap (\CCC_0\setminus E) \rc \ge c(k) s^k/2$. Then, there exist $x_i$ close to $y$ and $s' \in \Q\bigcap \lc 0, r_0-d_Z(x_0, x_i)/2 \rc$ close to $s$ such that
	\begin{align*}
\HHH^k \lc B^*(x_i, s') \cap (\CCC_0\setminus E) \rc \ge c(k) (s')^k/2,
	\end{align*}
which contradicts \eqref{appro2}. Consequently, we have $\HHH^k(\CCC_0\setminus E)=0$ and hence $\CCC_0$ is rectifiable.

Let $\nu:=\mu^\infty\llcorner \CCC_0$.  By passing the uniform upper Ahlfors estimate in Lemma \ref{lem:inductahlfors2} for the measures $\mu^l$ to the weak limit, we have
\[
\nu \lc B^*(x, \rho) \rc\le C(n,Y, \cc)\rho^k
\qquad\text{for every }x\in \CCC_0,\ \rho>0.
\]
Consequently, by the Hausdorff covering definition, for every
Borel set $E\subset \CCC_0$,
\[
\nu(E)\le C(n,Y, \cc)\HHH^k(E).
\]
Since $\mu \llcorner \CCC_0=\HHH^k\llcorner \CCC_0$, we conclude that
\[
\mu^\infty \lc B^*(y, r) \cap \CCC_0 \rc
\le C(n,Y, \cc)\mu \lc B^*(y, r) \cap \CCC_0 \rc.
\]
Together with $\mu^\infty=\mu$ on $\CCC_+$, this proves the lower bound for the Ahlfors regularity of $\mu$ by Lemma \ref{lem:inductahlfors2}, and hence proves \eqref{eq:extraaf}.

With \eqref{eq:extraaf} established, the last statement of the theorem follows directly from Lemma \ref{clearout} and Proposition \ref{choosenondegeneratingpoints}. This completes the proof.
\end{proof}

We now prove the main theorem of the subsection.

\begin{thm} \label{recticylsing2}
For any $k \in \{0,1, \ldots, n-2\}$, the set $\MS^k_{\mathrm{c}}$ is horizontally parabolic $k$-rectifiable with respect to $d_Z$.
\end{thm}

\begin{proof} 
The case $k \in \{1, \ldots, n-2\}$ follows directly from Corollary \ref{cor:kthcover} (c), Proposition \ref{cylneckdecomp}, Theorem \ref{ahlforsregucyl1}, and a standard countable covering argument. The case $k=0$ is simpler: any point $z \in \MS^0_{\mathrm{c}}$ must be isolated, since every tangent flow at $z$ is the standard sphere $S^n$.
\end{proof}

We end this subsection with the following result.

\begin{cor}\label{conntimeslicea}
Each connected component of $\MS_{\mathrm c}^2$ is contained in a time slice.
\end{cor}
\begin{proof}
It follows from Proposition \ref{cylneckdecomp} and a standard covering argument that, for some sufficiently small $\delta>0$, 
	\begin{align*}
		\lc \MS_{\mathrm c}^2\setminus \MS_{\mathrm c}^1\rc \subset \bigcup_i \lc \CCC_{i, 0} \bigcap B^*(x_i, r_i) \rc,
	\end{align*}
where $\NNN_i=B^*(x_i, 2r_i) \setminus B^*_{r_x}(\CCC_i)$ is a $(2, \delta, \cc,r_i)$-neck region with $x_i \in \MS_{c}^2$. Moreover, by the definition of the cylindrical neck region, we have $\CCC_{i, 0} \bigcap B^*(x_i, 2r_i) \subset \MS_c^2$.

We claim that
	\begin{align}\label{eq:connected002}
\HHH^1\lc \t\big(\MS_{\mathrm c}^2\setminus \MS_{\mathrm c}^1\big) \rc=0,
	\end{align}
where $\HHH^1$ denotes the $1$-dimensional Hausdorff measure on $\R$. To prove \eqref{eq:connected002}, it suffices to show that for each $i$:
	\begin{align}\label{eq:connected003}
\HHH^1\lc \t \lc \CCC_{i, 0} \bigcap B^*(x_i, r_i) \rc \rc=0.
	\end{align}
Let $S:=\CCC_{i, 0} \bigcap \overline{B^*(x_i, r_i)}$, which is a closed subset of $\MS^2_c$. By Theorem \ref{ahlforsregucyl1}, we have
	\begin{align*}
\HHH^2(S) \le \HHH^2 \lc \CCC_{i, 0} \bigcap B^*(x_i, 3r_i/2) \rc \le C(n, Y) r_i^2<\infty.
	\end{align*}
Since $S$ is compact, Lemma \ref{uniformcyl} implies that for any $\ep>0$, there exists $r_\ep>0$ such that each $x \in S$ is uniformly $(2, \ep, r_{\ep})$-cylindrical. In particular, by Lemma \ref{clearout}, we have
	\begin{align}\label{eq:connected005}
|\t(x)-\t(y)| \le \Psi(\ep) d_Z^2(x, y)
	\end{align}
for any $x, y \in S$ with $d_Z(x, y) \le r_\ep$. 

By the definition of Hausdorff measure, there exists a countable cover $\{B^*(y_j, s_j)\}$ of $S$ with $s_j<r_{\ep}/2$ such that
	\begin{align}\label{eq:connected006}
\sum_j s_j^2 \le C(n, Y) r_i^2.
	\end{align}
Then, from \eqref{eq:connected005}, for any $x , y \in B^*(y_j, s_j)$, we obtain
	\begin{align*}
|\t(x)-\t(y)| \le \Psi(\ep) d_Z^2(x, y) \le \Psi(\ep) s_j^2.
	\end{align*}
Thus, it implies
	\begin{align*}
\HHH^1\lc \t(B^*(y_j, s_j) \bigcap S) \rc \le \Psi(\ep) s_j^2
	\end{align*}
and summing over all $j$, it follows from \eqref{eq:connected006} that
	\begin{align*}
\HHH^1\lc \t(S) \rc \le \Psi(\ep) r_i^2.
	\end{align*}
Since $\ep$ is arbitrary, we conclude that $\HHH^1\lc \t(S) \rc=0$, which proves \eqref{eq:connected003}, and hence \eqref{eq:connected002}.

On the other hand, it follows from Corollary \ref{cor:kthcover} (c) that $\HHH^2 ( \MS_{\mathrm c}^1 )=0$. Since the time-function $\t$ is $2$-H\"older, a similar argument yields $\HHH^1 \lc \t(\MS_{\mathrm c}^1) \rc=0$. Combining this with \eqref{eq:connected002}, we obtain
	\begin{align}\label{eq:connected002a}
\HHH^1\lc \t\big(\MS_{\mathrm c}^2\big) \rc=0.
	\end{align}
It now follows from \eqref{eq:connected002a} that the image under $\t$ of any connected component of $\MS_{\mathrm c}^2$ must be a single point. This completes the proof.
\end{proof}

\subsection{Rectifiability of quotient cylindrical singularities}\label{subsec:quo}

In this subsection, we consider the quotient cylindrical singularities.

As before, we set 
\begin{align*}
\mathcal C^{m}_{-1}=(\bar M,\bar g,\bar f)=\left(\R^{m}\times S^{n-m}, g_E \times g_{S^{n-m}}, \frac{|\vec{x}|^2}{4}+\frac{n-m}{2}+\Theta_{n-m} \right).
\end{align*}

For any finite group $\Gamma \leqslant \mathrm{Iso}(\mathcal C^{m}_{-1})$ acting freely on $\bar M$, we set $\pi: \bar M \to \bar M/\Gamma$ to be the natural quotient map and define
\begin{align} \label{eq:quoc}\index{$\mathcal C^{m}(\Gamma)$}
\mathcal C^{m}(\Gamma)_{-1}=(\bar M_{\Gamma}, \bar g_{\Gamma},\bar f_{\Gamma}),
\end{align}
where $\bar M_{\Gamma}=\bar M/\Gamma$, $\bar g_{\Gamma}$ is the quotient metric of $\bar g$, and $\bar f_{\Gamma}$ is defined so that
\begin{align*} 
\pi^* \bar f_{\Gamma}=\bar f-\log |\Gamma|.
\end{align*}

Also, we define
\begin{align*}\index{$\Theta_{n-m}(\Gamma)$}
\Theta_{n-m}(\Gamma):=\Theta_{n-m}-\log |\Gamma|.
\end{align*}

Let $\mathcal C^m_k(\Gamma)_{-1}$ be a quotient cylinder as in \eqref{eq:quoc} so that $\mathcal C^{m}(\Gamma)_{-1}$ splits off an $\R^k$ exactly. We define $\mathcal C^m_k(\Gamma)=(\bar M_\Gamma, (\bar g_{\Gamma}(t))_{t<0}, (\bar f_{\Gamma}(t))_{t<0})$ to be the associated Ricci flow, where $t=0$ corresponds to the singular time, and the potential function is given by
\begin{align*}
\pi^* \bar f_{\Gamma}(t)=\frac{|\vec{x}|^2}{4|t|}+\frac{n-m}{2}+\Theta_{n-m}(\Gamma).
\end{align*}

Let $\bar{\mathcal C}^m_k(\Gamma)$\index{$\bar{\mathcal C}^m_k(\Gamma)$} be the completion of $\mathcal C^m_k(\Gamma)$, equipped with the spacetime distance $d_{\mathcal C}^*$ induced by the variational formula in Definition \ref{defnd*distance}. We also define the base point $p^*$ as the limit of $(\bar p, t)$ as $t \nearrow 0$ with respect to $d_{\mathcal C}^*$, where $\bar p \in \mathcal C^m_k(\Gamma)_{-1}$ is a minimum point of $\bar f_{\Gamma}(-1)$. For any $t<0$, we have
	\begin{align*}
\nu_{p^*;t}=(4\pi |t|)^{-\frac n 2} e^{-\bar f_{\Gamma}(t)} \,\mathrm{d}V_{\bar g_{\Gamma}(t)}.
	\end{align*}

Next, we consider a general noncollapsed Ricci flow limit space $(Z, d_Z, \t)$ over $\III$, obtained as the limit of a sequence in $\mathcal M(n,Y, T)$. Then, we have the following definition similar to Definition \ref{def:ccc}.

\begin{defn}
A point $z \in Z_{\III^-}$ is called \textbf{a quotient cylindrical point regarding $\bar{\mathcal C}^m_k(\Gamma)$} if a tangent flow at $z$ is $\bar{\mathcal C}^m_k(\Gamma)$ for some $k$, $m$ and $\Gamma$.
\end{defn}

It follows from \cite[Theorem 8.11]{FLloja05} that at any quotient cylindrical point, the tangent flow is unique. Thus, we have the following definition similar to Definition \ref{def:cylindr}.

\begin{defn}\label{def:cylindrq}\index{$\MS_{\mathrm{qc}}^k(m, \Gamma)$}
Let $\MS_{\mathrm{qc}}^k(m, \Gamma)$ be the subset of $Z_{\III^-}$ consisting of all singular points at which one tangent flow is isometric to $\bar{\mathcal C}^m_k(\Gamma)$. Moreover, we define
	\begin{align*}
\MS_{\mathrm{qc}}^k=\bigcup_{m, \Gamma, 0 \le j \le k} \MS_{\mathrm{qc}}^j(m, \Gamma).
	\end{align*}\index{$\MS_{\mathrm{qc}}^k$}
It is clear that $\MS_{{\mathrm{c}}}^k \subset \MS_{\mathrm{qc}}^k \subset \MS^k$. 
\end{defn}

Next, we consider the quantitative version.

\begin{defn}
A point $z \in Z_{\III^-}$ is called \textbf{$(k,\ep,r)$-quotient cylindrical regarding $\bar{\mathcal C}^m_k(\Gamma)$}\index{$(k,\ep,r)$-quotient cylindrical regarding $\bar{\mathcal C}^m_k(\Gamma)$} if $\t(z)-\ep^{-1} r^2 \in \III^-$ and
	  \begin{align*}
(Z, r^{-1} d_Z, z, r^{-2}(\t-\t(z))) \quad \text{is $\ep$-close to} \quad (\bar{\mathcal C}^m_k(\Gamma) ,d^*_{\mathcal C}, p^*,\t) \quad \text{over} \quad [-\ep^{-1}, \ep^{-1}].
  \end{align*} 
Let $\tilde \phi$ be an $\ep$-map from Definition \ref{defn:close}, which is defined from $B^*(p^*,\ep^{-1}) \cap \bar{\mathcal C}^m_k(\Gamma)_{[-\ep^{-1}, \ep^{-1}]}$ to $Z_{[\t(z)-\ep^{-1} r^2, \t(z)+\ep^{-1} r^2]}$, where $B^*(p^*,\ep^{-1})$ is the metric ball in $\bar{\mathcal C}^m_k(\Gamma)$ with respect to $d_{\mathcal C}^*$. Then, we define\emph{:}
\begin{align*}
\LL_{z,r}:=\tilde \phi \lc B^*(p^*,\ep^{-1}) \cap \mathrm{spine}(\bar{\mathcal C}^m_k(\Gamma)) \rc,\index{$\LL_{z,r}$}
\end{align*}
	and say that $z$ is $(k,\ep,r)$-quotient cylindrical regarding $\bar{\mathcal C}^m_k(\Gamma)$ \textbf{with respect to $\LL_{z,r}$}. In general, we say that $z$ is $(k,\ep,r)$-quotient cylindrical if it is $(k,\ep,r)$-quotient cylindrical regarding some $\bar{\mathcal C}^m_k(\Gamma)$. 
\end{defn}

By using the Lojasiewicz inequality in \cite[Theorem 8.9]{FLloja05}, we can argue in the same way as Proposition \ref{sumWonRFlimit} to prove the following summability result.

\begin{prop}
	For any $0<\ep\leq\ep(n,Y)$, $\alpha\in (1/4,1)$, and $\delta\leq\delta(n,Y,\ep,\alpha)$, the following holds. Suppose
	\begin{align*}
		\left|\widetilde \WW_{z}(s_1)-\widetilde \WW_{z}(s_2)\right|<\delta,
	\end{align*}
for $0<s_1<s_2$, and for any $s\in [s_1,s_2]$, $z$ is $(k,\delta,\sqrt{s})$-quotient cylindrical regarding $\bar{\mathcal C}^m_k(\Gamma)$. Then
	\begin{align*}
		\sum_{s_1\leq r_j=2^{-j} \leq s_2}\left|\widetilde \WW_{z}(r_j)-\widetilde \WW_{z}(r_{j-1})\right|^{\alpha}<\ep.
	\end{align*}
\end{prop}

Similar to Proposition \ref{prop:uniformc}, we also have

\begin{prop} \label{prop:uniformcq}
Suppose $z \in Z$ is $(k,\delta,r)$-quotient cylindrical regarding $\bar{\mathcal C}^m_k(\Gamma)$ and $\widetilde \WW_z(\delta \bar r^2)-\widetilde \WW_z(\delta^{-1} r^2)<\delta$ for some $0<\bar r<r$.  For any $\ep>0$, if $\delta \le \delta(n, Y, \ep)$, then $z$ is $(k,\ep,s)$-quotient cylindrical regarding $\bar{\mathcal C}^m_k(\Gamma)$ for any $s \in [\bar r, r]$. In particular, if $z \in \MS^k_{\mathrm{qc}}\setminus \MS^{k-1}_{\mathrm{qc}}$ is $(k,\delta,r)$-quotient cylindrical regarding $\bar{\mathcal C}^m_k(\Gamma)$ and $\delta \le \delta(n, Y, \ep)$, then $z$ is \textbf{uniformly $(k,\ep,r)$-quotient cylindrical regarding $\bar{\mathcal C}^m_k(\Gamma)$}, meaning that $z$ is $(k,\ep,s)$-quotient cylindrical regarding $\bar{\mathcal C}^m_k(\Gamma)$ for any $s \in (0, r]$.
\end{prop}

Next, we define the quotient cylindrical neck region and its packing measure as in Definitions \ref{defiofcylneckregion} and \ref{defiofpackingmeasure2}.

\begin{defn}[Quotient cylindrical neck region]\label{defiofcylneckregionquo}
Given $\delta>0$, $\cc \in (0, 10^{-10n})$, $r>0$ and $z \in Z_{\III^-}$ with $\t(z)-2\delta^{-1}r^2 \in \III^-$, we call a subset $\NNN \subset B^*(z, 2r)$ a \textbf{$(k,\delta, \cc,r)$-quotient cylindrical neck region regarding $\bar{\mathcal C}^m_k(\Gamma)$}\index{$(k,\delta, \cc,r)$-quotient cylindrical neck region} if $\NNN=B^*(z,2r)\setminus B_{r_x}^*(\CCC)$, where $\CCC \subset B^*(z, 2r)$ is a nonempty closed subset with $r_x: \CCC \to \R_{+}$\index{$\CCC$}\index{$\CCC_0$}\index{$\CCC_+$}\index{$r_x$}, satisfies:
	\begin{itemize}[leftmargin=*, label={}]
		\item \emph{(n1)} for any $x, y \in \CCC$, $d_Z(x, y) \ge \cc^2(r_x+r_y);$
		\item \emph{(n2)} for all $x\in \CCC$,
		$$\widetilde \WW_{x}(\delta r^2_x)-\widetilde \WW_{x}(\delta^{-1} r^2)<\delta;$$
		\item \emph{(n3)} for each $x\in \CCC$ and $\cc^2 r_x\leq s\leq 2 r$, $x$ is $(k,\delta,s)$-quotient cylindrical regarding $\bar{\mathcal C}^m_k(\Gamma)$ with respect to $\mathcal L_{x,s};$
		\item \emph{(n4)} for each $x\in \CCC$ and $\cc^{-5} r_x\leq s\leq r-d_Z(x, z)/2$, we have $\mathcal{L}_{x,s} \cap B^*(x, s) \subset B^*_{\cc s}(\CCC)$ and $\CCC\bigcap B^*(x,s)\subset B^*_{\cc s}(\mathcal{L}_{x,s})$.
	\end{itemize}
Furthermore, we define the \textbf{packing measure}\index{packing measure} $\mu$ as
	\begin{equation*}
		\mu:=\sum_{x\in \CCC^+}r_x^{k}\delta_x+\HHH^{k}|_{\CCC_0}.
	\end{equation*}
\end{defn}

Analogous to Proposition \ref{cylneckdecomp}, we obtain the following construction of the neck region. Note that in this setting, one must consider the entropy $\Theta_{n-m}(\Gamma)$ instead.

\begin{prop}\label{cylneckdecompquo}
Given $\delta>0$ and $\cc \in (0, 10^{-10n})$ and $x_0\in \MS_{\mathrm{qc}}^k(m, \Gamma)$, we can find a constant $r_0>0$ and $\CCC=\CCC_0\bigcup\CCC_+\subset B^*(x_0,2r_0)$ with $r_x:\CCC\to\R^+$ which satisfies $r_x>0$ on $\CCC_+$, $r_x=0$ on $\CCC_0$, such that the following hold:
	\begin{enumerate}[label=\textnormal{(\roman{*})}]
		\item $\NNN=B^*(x_0,2r_0)\setminus B^*_{r_x}(\CCC)$ is a $(k,\delta, \cc,r_0)$-quotient cylindrical neck region regarding $\bar{\mathcal C}^m_k(\Gamma)$.
		\item $\MS_{\mathrm{qc}}^k(m, \Gamma) \bigcap B^*(x_0,r_0)\subset\CCC_0$.
	\end{enumerate}
\end{prop}

Now, one can proceed as in the previous subsection to establish the following result. The only difference lies in the proof of the analog of Proposition \ref{choosenondegeneratingpoints} (iv), where a version of Lemma \ref{lem:comparedis1} is required for $\bar{\mathcal C}^m_k(\Gamma)$ and points in the spine of $\bar{\mathcal C}^m_k(\Gamma)$. This can, however, be readily obtained from \cite[Lemma 8.5]{fang2025RFlimit}.

\begin{prop}[Ahlfors regularity]\label{ahlforsregucyl1quo}
Given $\ep>0$, if $\cc \le \cc(n)$ and $\delta \le \delta(n, Y, \cc,\ep)$, then for the $(k,\delta,\cc,r_0)$-quotient cylindrical neck region $\NNN=B^*(z,2r_0)\setminus B_{r_x}^*(\CCC)$ constructed in Proposition \ref{cylneckdecompquo}, there exists a constant $C=C(n,Y, \cc)>1$ such that for any $x\in\CCC$ and $r_x\leq s\leq r_0-d_Z(x, z)/2$,
	\begin{align*}
	 C^{-1} s^k\leq\mu(B^*(x,s))\leq C s^k.
	\end{align*} 
Moreover, we can find a countable collection of $\HHH^k$-measurable subsets $ E_i\subset \CCC_0$ such that $\HHH^k(\CCC_0\setminus\bigcup_i E_i)=0$ and for each $i$, there exists a bi-Lipschitz map $u_i:E_i\to \R^k$, where $\R^k$ is equipped with the standard Euclidean distance, such that
	\begin{align*}
\sqrt{|\t(x)-\t(y)|}\leq\ep d_Z(x,y),\quad \forall x, y \in E_i.
	\end{align*} 
\end{prop}

Since $Z$ has entropy bounded below by $-Y$, there are at most $C(n, Y)$ conjugacy classes for the group $\Gamma$. Therefore, by combining Propositions \ref{cylneckdecompquo} and \ref{ahlforsregucyl1quo}, we obtain the following:
\begin{thm}\label{thmrectquotient}
For any $k \in \{0,1, \ldots, n-2\}$, $\MS_{\mathrm{qc}}^k$ is horizontally parabolic $k$-rectifiable with respect to $d_Z$.
\end{thm}

Similar to Corollary \ref{conntimeslicea}, we also have

\begin{cor}\label{conntimesliceaq}
Each connected component of $\MS_{\mathrm{qc}}^2$ is contained in a time slice.
\end{cor}

\section{Analysis of singularities in dimension four}\label{sec:nd}

In this section, we prove a neck decomposition theorem for noncollapsed Ricci flow limit spaces in dimension four. As an application, we establish the rectifiability and derive a sharp volume estimate for the singular set. Also, we derive the $L^1$-curvature bounds for four-dimensional closed Ricci flows.

Throughout this section, we fix a noncollapsed Ricci flow limit space $(Z,d_Z,\t)$ , which is obtained as a pointed Gromov--Hausdorff limit of a sequence in $\MM(4, Y, T)$. As before, we set 
\begin{align*}
\III:=[-0.98 T,0] \quad \text{and} \quad \III^-:=(-0.98 T,0].
	\end{align*}
Moreover, the regular part is given by a Ricci flow spacetime $(\RR, \t, \partial_\t, g^Z)$. For simplicity, we use $B^*(x,r)$ instead of $B_Z^*(x,r)$ for the metric balls in $Z$ with respect to $d_Z$.

\subsection{Overview of four-dimensional singularities}

Recall that the singular set $\MS \subset Z_{\III^-}$ has the following stratification:
\begin{equation*}
	\mathcal S^0 \subset \mathcal S^1 \subset \mathcal S^{2}=\mathcal S.
\end{equation*}

Given a point $z \in \MS$, we consider the following cases.

\textbf{Case A}: $z \in \MS^2 \setminus \MS^1$.

In this case, there exists a tangent flow $(Z', d_{Z'}, z', \t')$ at $z$, which is $2$-symmetric, but not $3$-symmetric. Let $(\RR', \t', \partial_{\t'}, g^{Z'}_t)$ denote its regular part. Then $\RR'_{(-\infty, 0)}$ has vanishing Ricci curvature or splits off an $\R^2$. By the classification of all $3$-dimensional Ricci shrinkers (see \cite{hamilton1993formations, perelman2002entropy, naber2010noncompact, ni2008classification, CCZ08}), we conclude that
	\begin{align*}
\lc \RR'_{-1}, g^{Z'}_{-1} \rc= S^2 \times \R^2,\,\, \RP^2 \times \R^2 \text{ or } \R^4/\Gamma.
	\end{align*}
Thus, $(Z', d_{Z'}, z', \t')$ is isometric to $\bar{\mathcal C}^2$, $\bar{\mathcal C}^2_2(\Z_2)$ or $\R^4/\Gamma\times \R$. In the first two cases, the tangent flow at $z$ is unique. We therefore decompose
\begin{align*}
\MS^2 \setminus \MS^1 = \lc \MS^2_{\mathrm{qc}}\setminus\MS^1_{\mathrm {qc}}\rc \sqcup \MS^2_{\mathrm{F}},
\end{align*}
where $\MS^2_{\mathrm{F}} \subset \MS^2$ denotes the set of points for which at least one tangent flow is isometric to $\R^4/\Gamma \times \R$.

\textbf{Case B}: $z \in \MS^1 \setminus \MS^0$.

In this case, there exists a tangent flow $(Z', d_{Z'}, z', \t')$ at $z$, which is $1$-symmetric, but not $2$-symmetric. Arguing as in Case A, we obtain
	\begin{align*}
\lc \RR'_{-1}, g^{Z'}_{-1} \rc= S^3/\Gamma \times \R\text{ or } (S^2 \times_{\Z_2} \R) \times \R.
	\end{align*}
Hence, $(Z', d_{Z'}, z', \t')$ is isometric to $\bar{\mathcal C}^1_1(\Gamma)$ or $\bar{\mathcal C}^2_1(\Z_2)$. In particular, any tangent flow at $z$ is unique, and we conclude that
\begin{align*}
\MS^1 \setminus \MS^0 = \MS^1_{\mathrm{qc}}\setminus \MS^0_{\mathrm{qc}}.
\end{align*}

\textbf{Case C}: $z \in \MS^0$.

In this case, any tangent flow $(Z', d_{Z'}, z', \t')$ at $z$ is either a quasi-static cone, so that $Z'_{(-\infty, t_a]}$ is isometric to $\R^4/\Gamma \times (-\infty, t_a]$ for some arrival time $t_a \in [0, \infty)$, or else its regular part $\lc \RR'_{-1}, g^{Z'}_{-1} \rc$ has nonvanishing Ricci curvature and does not split off any line.

We claim that in the former case, one must have $t_a=0$. Indeed, suppose $t_a>0$ and $(Z', d_{Z'}, z', \t')$ arises as the $\hat C^\infty$-limit of a sequence 
\begin{align*}
\lc Z, r_i^{-1} d_Z, z, r_i^{-2}(\t-\t(z)) \rc, \quad r_i \to 0.
\end{align*}
Then we can choose another sequence $s_i \to 0$ so that
\begin{align*}
\lc Z, (r_i s_i)^{-1} d_Z, z, (r_i s_i)^{-2}(\t-\t(z)) \rc
\end{align*}
converges to $\R^4/\Gamma \times \R$, contradicting the definition of $\MS^0$.

Next, we prove

\begin{prop} \label{prop:countable}
$\MS^0$ is a countable set.
\end{prop}

\begin{proof}
Recall that we have
\begin{align*}
\MS^{0}=\bigcup_{\ep \in \mathbb Q \cap (0, 1)} \bigcap_{0<r<\ep } \MS^{\ep,0}_{r, \ep}.
\end{align*}
To finish the proof, it suffices to show that $S_\ep:=\bigcap_{0<r<\ep } \MS^{\ep,0}_{r, \ep}$ is countable for any $\ep>0$.

For any $z\in Z$, we define the entropy at $z$ by
\begin{align*}\index{$\NN_z(0)$}
	\NN_z(0):=\lim_{\tau\searrow 0}\NN_z(\tau).
\end{align*}
This limit exists because $\NN_z(\tau)$ is nonpositive and nonincreasing with respect to $\tau$. 

\begin{claim}\label{claim:local}
For any $z \in S_\ep$, there exists a constant $\delta>0$ such that there is no point $x \in S_\ep \bigcap B^*(z, \delta)$ satisfying $\abs{\NN_{x}(0)-\NN_z(0)}<\delta$.
	\end{claim}

Suppose, for contradiction, that the claim fails. Then there exists a sequence $x_j \in S_\ep$ such that $x_j \to z$ and $\NN_{x_j}(0) \to \NN_z(0)$. Define $r_j:=d_Z(x_j,z)$, and let $(Z',d_{Z'},z',\t')$ be a subsequential limit of $\lc Z,r_j^{-1}d_Z,z,r_j^{-2}(\t-\t(z)) \rc$. It is clear from \cite[Lemma 7.2]{fang2025RFlimit} that $(Z',d_{Z'},z',\t')$ is a Ricci shrinker space, and
\begin{align}\label{equ:entrtang}
\NN_{z'}(\tau)=\NN_z(0), \quad \text{for all} \quad \tau>0.
\end{align}

We assume $x_j$ converge to $x' \in Z'$ so that $d_{Z'}(z', x')=1$. For any $\tau>0$, it follows from \cite[Lemma 7.2]{fang2025RFlimit} again and our assumption that
\begin{align*}
\NN_{x'}(\tau) =\lim_{j \to \infty} \NN_{x_j}(\tau r_j^2) \le \lim_{j \to \infty} \NN_{x_j}(0)=\NN_z(0).
\end{align*}
Therefore, by \cite[Lemma D.3]{fang2025RFlimit} and \eqref{equ:entrtang}, we deduce that $x'$ lies on the spine of $(Z', d_{Z'}, z', \t')$. By \cite[Lemma D.5 and Proposition D.8]{fang2025RFlimit}, $(Z',d_{Z'},z',\t')$ is $1$-splitting, a static cone, or a quasi-static cone. Since $z \in S_\ep$, it must be a quasi-static cone. In particular, $\t'(x') \ne 0$.

If $\t'(x') > 0$, then for $s_j:=\sqrt{|\t'(x')|} r_j \ep$ and large $j$,
\begin{align*}
\lc Z, s_j^{-1} d_Z, z, s_j^{-2}(\t-\t(z)) \rc
\end{align*}
is $\ep$-close to $\R^4/\Gamma \times \R$, contradicting the assumption that $z \in S_\ep$. Similarly, if $\t'(x') < 0$, then for large $j$ and $s_j$ as above,
\begin{align*}
\lc Z, s_j^{-1} d_Z, x_j, s_j^{-2}(\t-\t(x_j)) \rc
\end{align*}
is $\ep$-close to $\R^4/\Gamma \times \R$, again contradicting $x_j \in S_\ep$. Hence, Claim \ref{claim:local} follows.
	
Now, define a map $I:S_\ep \to Z\times \R$ by
\begin{align*} 
I(z)=\lc z,\NN_z(0) \rc.
\end{align*}
It is clear that $I$ is injective. By Claim \ref{claim:local}, for each $z_0 \in S_\ep$, there exists $\delta_0>0$ such that 
	\begin{align*} 
\mathrm{image}(I) \bigcap B_P\lc I(z_0) , \delta_0 \rc=\{I(z_0)\},
\end{align*}
	where $B_P$ denotes the metric ball in $Z\times \R$ with respect to the product metric $d_Z \times d_E$. Since $Z$ is separable by \cite[Theorem 1.3]{fang2025RFlimit}, so is $Z \times \R$. Hence, $\mathrm{image}(I)$ is countable, and therefore $S_\ep$ is also countable.
	
This completes the proof.
\end{proof}

\begin{rem}
By the same argument, the conclusion that $\MS^0$ is countable holds for noncollapsed Ricci flow limit spaces in all dimensions.
\end{rem}

\subsection{Flat neck regions}

In this subsection, we consider the model space:
\begin{align*}
\mathcal F^a(\Gamma)=\R^4/\Gamma \times (-\infty, a],
\end{align*}
where $a \in [0, \infty]$ is a constant, and $\Gamma \leqslant \mathrm{O}(4)$ is a nontrivial finite group acting freely on $S^3$. Since $Z$ has entropy bounded below by $-Y$, there are at most $C(Y)$ such $\Gamma$, up to conjugacy, in $\mathrm{O}(4)$. Note that $\mathcal F^a(\Gamma)$ is equipped with the spacetime distance $d^*_{\mathcal F}$ induced by the variational definition of the spacetime distance. We use $B^*_{\mathcal F^a(\Gamma)}(\cdot,\cdot)$ to denote the $d^*$-balls in $\mathcal F^a(\Gamma)$. We also define the base point $p^*=([\vec 0^4], 0)$, where $[\vec 0^4]$ is the image of $\vec 0^4 \in \R^4$ to $\R^4/\Gamma$.

Next, we consider the quantitative closeness.

\begin{defn}
A point $z \in Z_{\III^-}$ is called \textbf{$(\ep, r)$-close}\index{$(\ep,r)$-close} to $\mathcal F^a(\Gamma)$ if $\t(z)-\ep^{-1} r^2 \in \III^-$ and
	  \begin{align*}
(Z, r^{-1} d_Z, z, r^{-2}(\t-\t(z))) \quad \text{is $\ep$-close to} \quad \lc \mathcal F^a(\Gamma),d^*_{\mathcal F}, p^*,\t \rc \quad \text{over} \quad [-\ep^{-1}, \min\{\ep^{-1}, a\}].
  \end{align*} 
 Let $\tilde \phi$ be an $\ep$-map from Definition \ref{defn:close}, which is from $B_{\mathcal F^a(\Gamma)}^*(p^*,\ep^{-1}) \bigcap \mathcal F^a(\Gamma)_{[-\ep^{-1}, \min\{\ep^{-1}, a\}]}$ to $Z_{[\t(z)-\ep^{-1} r^2, \t(z)+\min\{\ep^{-1}, a\}r^2]}$. Then, we define
\begin{align*}
\LL_{z,r}:=\tilde \phi \lc  [\vec 0^4] \times [-\ep^{-1}, \min\{\ep^{-1}, a\}]\rc,
\end{align*}
and we say that $z$ is $(\ep, r)$-close to $\mathcal F^a(\Gamma)$ \textbf{with respect to $\LL_{z,r}$}.
\end{defn}

Parallel to cylindrical or quotient cylindrical neck regions, we introduce the notion of \textbf{flat neck regions}, which, together with quotient cylindrical neck regions, plays a central role in the decomposition of Ricci flow limit spaces in dimension $4$.

\begin{defn}[Flat neck region]\label{defnneckgeneral}
Given constants $\delta>0$, $\cc \in (0, 10^{-40})$, $r>0$ and $z \in Z_{\III^-}$ with $\t(z)-2\delta^{-1}r^2 \in \III^-$, we call a subset $\NNN \subset B^*(z, 2r)$ a \textbf{$(\delta, \cc, r)$-flat neck region regarding $\mathcal F^0(\Gamma)$}\index{$(\delta, \cc, r)$-flat neck region} if $\NNN=B^*(z,2r)\setminus B_{r_x}^*(\CCC)$, where $\CCC \subset B^*(z, 2r)$ is a nonempty closed subset with $r_x: \CCC \to \R_{+}$, satisfies\emph{:}
	\begin{itemize}[leftmargin=*, label={}]
		\item \emph{(n1)} for any $x, y \in \CCC$, $d_Z(x, y) \ge \cc^2(r_x+r_y);$
		
		\item \emph{(n2)} for all $x\in \CCC$,
		\begin{align*}
\widetilde \WW_{x}(\delta r^2_x)-\widetilde \WW_{x}(\delta^{-1} r^2)<\delta;
	\end{align*}	
		\item \emph{(n3)} for each $x\in \CCC$ and $\cc^2 r_x\leq s\leq 2 r$, $x$ is $(\delta, s)$-close to $\mathcal F^0(\Gamma)$ with respect to $\LL_{x,s};$

\item \emph{(n4)} for each $x\in \CCC$ and $\cc^{-5} r_x\leq s\leq r-d_Z(x, z)/2$, we have $\mathcal{L}_{x,s} \bigcap B^{*,-}(x, s) \subset B^*_{\cc s}(\CCC)$ and $\CCC\bigcap B^{*,-}(x,s)\subset B^*_{\cc s}(\mathcal{L}_{x,s})$, where
\begin{align*}
B^{*,-}(x, s):=B^{*}(x, s) \bigcap \{y \in Z \mid \t(y) \le \t(x)\}.
\end{align*}
\end{itemize}
	As before, $\CCC$ is called the \textbf{center} of the flat neck region, and $r_x$ is referred to as the \textbf{radius function}. We decompose $\CCC=\CCC_0\bigcup\CCC_{+}$, where $r_x>0$ on $\CCC_+$ and $r_x=0$ on $\CCC_0$. In addition, the corresponding \textbf{packing measure} is defined as
\begin{align*}\index{packing measure}
	\mu:=\sum_{x\in \CCC_{+}}r_x^{2}\delta_x+\HHH^{2}|_{\CCC_0}.
\end{align*}
\end{defn}

Next, we prove

\begin{lem} \label{lem:bilitpsta}
	Let $\NNN \subset B^*(z, 2r)$ be a $(\delta, \cc, r)$-flat neck region regarding $\mathcal F^0(\Gamma)$. If $\delta \le \delta(Y)$, then for any $x \in \CCC$ and $s \in (0, 2r]$,
		\begin{align*}
0<c_0(Y) d_Z^2(x, y)\leq |\t(x)-\t(y)|\leq d_Z^2(x,y)
	\end{align*} 
	for any $y \in B^{*,-}(x, s) \cap \CCC$. In particular, $\CCC_0$ is vertically parabolic $2$-rectifiable.
\end{lem}

\begin{proof}
The second inequality is obvious, so we only prove the first one. 

We set $s_0=d_Z(x, y)$. Then it is clear that $s \ge s_0  \ge \cc^2 (r_x+r_y)$. Then by $(n3)$ of Definition \ref{defnneckgeneral}, we know that $x$ is $(\delta, s_0)$-close to $\mathcal F^0(\Gamma)$ with respect to $\mathcal L_{x, s_0}$.

Since our model space is the static cone $\R^4/\Gamma \times \R_-$ and $\t(x) \ge \t(y)$, it follows that for any $\ep>0$, if $\delta \le \delta(Y, \ep)$ and $d_Z(y, \mathcal L_{x, s_0}) \ge \ep s_0$, then $y$ is $(3, \ep, \ep^2 s_0)$-symmetric. However, this would imply—by condition $(n2)$ of Definition \ref{defnneckgeneral}---that $y$ is $(3, 2\ep, s_0)$-symmetric, leading to a contradiction.

Thus, if $\delta \le \delta(Y, \ep)$, we can find $y' \in \mathcal L_{x, s_0
}$ so that
	\begin{align*} 
d_Z(y, y') \le \ep s_0,
	\end{align*}
which implies	
		\begin{align*}
|\t(y)-\t(y')| \le \ep^2 s^2_0.
	\end{align*}
	
On the other hand, it follows from \cite[Lemma 7.24]{fang2025RFlimit} that if $\delta$ is sufficiently small,
		\begin{align*}
|\t(x)-\t(y')|  \ge c(Y) d^2_Z(x, y') >0.
	\end{align*}

Combining the above inequalities, we obtain
		\begin{align*}
|\t(x)-\t(y)| \ge |\t(x)-\t(y')|-|\t(y)-\t(y')| \ge c(Y) (1-\ep)^2 s_0^2-\ep^2 s^2_0 \ge c(Y) s_0^2/2,
	\end{align*} 
provided that $\ep \le \ep(Y)$ is small. 

The last conclusion follows from a standard covering argument, which completes the proof.
\end{proof}

In the following, we will always assume $\delta \le \delta(Y)$ so that Lemma \ref{lem:bilitpsta} holds.

\begin{prop}[Ahlfors regularity---flat case]\label{ahlforsregforstaticneck}
	Let $\NNN \subset B^*(z, 2r)$ be a $(\delta, \cc, r)$-flat neck region. If $\cc\leq \cc(Y)$, then for any $x\in\CCC$ and $r_x\leq s\leq r-d_Z(x, z)/2$,
	\begin{align}\label{eq:bilipextra001}
C^{-1}(Y, \cc) s^2 \leq  \mu(B^{*,-}(x,s))\leq\mu(B^*(x,s))\leq C(Y, \cc) s^2.
	\end{align} 
\end{prop}

\begin{proof}
Note that by Definition \ref{defnneckgeneral}, $\NNN'=\NNN \cap B^*(x, 2s)$ is also a $(\delta, \cc, s)$-neck region with center $\CCC'=\CCC \cap B^*(x, 2s)$ and the same radius function. Thus, we only prove \eqref{eq:bilipextra001} for the case $x=z \in \CCC$ and $s=r$.

We choose a maximal $r/8$-separated subset $\{x_i\}_{1 \le i \le N}$ of $B^*(z,r) \cap \CCC$. Then $\{B^*(x_i, r/16)\}$ are pairwise disjoint. By Proposition \ref{prop:volumebound}, $N \le C(Y)$.
 
On $\R$, we set the parabolic ball $P(a, s'):=\{b \in \R \mid d_P(b, a):=\sqrt{|b-a|} \le s'\}$. For each ball $B^*(x_i, r/8)$, Lemma \ref{lem:bilitpsta} applies. Thus, we have
		\begin{align*}
			\mu(B^*(x_i, r/8))&=\sum_{y\in B^*(x_i, r/8) \cap \CCC_+}r_y^{2}+\HHH^2 \lc B^*(x_i, r/8) \cap \CCC_0 \rc\\
			&\leq C(Y) \cc^{-2}\sum_{y\in B^*(x_i, r/8) \cap \CCC_+} \HHH_P^2 \lc P(\t(y), c(Y) \cc^2 r_y) \rc+C(Y) \HHH_P^2 \lc\t\lc B^*(x_i, r/8) \cap \CCC_0 \rc \rc\leq C(Y, \cc) r^2,
		\end{align*}
where we choose a small $c(Y)$ so that $\{P(\t(y), c(Y) \cc^2 r_y)\}_{y \in B^*(x_i, r/8) \cap \CCC} \subset \R$ are mutually disjoint. Moreover, we use $\HHH_P^2$ to denote the $2$-Hausdorff measure on $\R$ with respect to the parabolic distance $d_P$. Thus, we obtain $\mu(B^*(z, r)) \le C(Y, \cc) r^2$.

On the other hand, we set $P^-(a, s'):=\{b \in P(a, s') \mid  b \le a\}$ and claim that
		\begin{align} \label{eq:timelower}
P^-(\t(z), c_1(Y) r) \subset \bigcup_{y \in B^{*, -}(z, r) \cap \CCC} \overline{P^-(\t(y),  \cc^{-5} r_y)},
		\end{align}
		where $c_1:=\sqrt{c_0}/100$, and $c_0=c_0(Y)$ is the constant from Lemma \ref{lem:bilitpsta}.
		
Once \eqref{eq:timelower} is established, it follows that
	\begin{align*}
		0&< C^{-1}(Y) r^2 \leq \HHH_P^2\lc P^-(\t(z), c_1 r) \rc \\
		&\leq \sum_{y \in B^{*, -}(z, r) \cap \CCC_+} \HHH_P^2\lc \overline{P^-(\t(y), \cc^{-5} r_y)} \rc+\HHH_P^2\lc \t \lc \CCC_0 \bigcap B^{*, -}(z, r) \rc \rc \\
		& \leq \cc^{-10} \sum_{y \in B^{*, -}(z, r) \cap \CCC_+} r_y^2+\HHH^2 \lc \CCC_0 \bigcap B^{*,-}(z, r) \rc \le \cc^{-10} \mu \lc B^{*, -}(z, r) \rc.
	\end{align*}
	
Assume now, for contradiction, that \eqref{eq:timelower} fails. Then there exists 
	\begin{align*}
a \in P^-(\t(z), c_1(Y) r) \setminus \bigcup_{y \in B^{*, -}(z, r) \cap \CCC} \overline{P^-(\t(y), \cc^{-5} r_y)}.
	\end{align*}
For any $x \in \CCC \bigcap B^{*,-}(z, r) \bigcap Z_{[a, \t(z)]}$, define $s_x=\sqrt{\t(x)-a}$ and set $\bar s:=\min_{x \in \CCC \bigcap B^*(z, r) \bigcap Z_{[a, \t(z)]}} s_x>0$. By our choice of $c_1$ and Lemma \ref{lem:bilitpsta}, we can find $z_0 \in \CCC \bigcap B^*(z, r/4) \bigcap Z_{[a, \t(z)]}$ with $\bar s=s_{z_0}$. In particular, we have $a+\bar s^2=\t(z_0)\leq\t(z)$.

Since $\bar s \ge \cc^{-5} r_{z_0}$, we consider $\LL_{z_0, \bar s}$. Given that our model space is $\R^4/\Gamma \times \R_-$, we can find $z_1 \in B^{*,-}(z_0, \bar s) \cap \LL_{z_0, \bar s}$ so that $0< \t(z_1)-a< (1-c(Y))\bar s^2$. By $(n4)$ of Definition \ref{defnneckgeneral}, there exists $z_2 \in \CCC$ so that $d_Z(z_1, z_2) \le \cc \bar s$. If $\cc<\cc(Y)$, we have $z_2 \in B^{*,-}(z, r) \cap \CCC$. Moreover, 
	\begin{align*}
\abs{\t(z_2)-\t(z_1)} \le d^2_Z(z_1, z_2) \le  \cc^2 \bar s^2,
	\end{align*}
which implies that $0<\t(z_2)-a<(1-c(Y)+\cc^2)\bar s^2<\bar s^2$---a contradiction.

Hence, we conclude that claim \eqref{eq:timelower} must hold, completing the proof.
\end{proof}

\subsection{Neck decomposition theorem}

We begin this subsection by proving the following result:

\begin{prop}\label{prop:almostconstant1}
Suppose $z \in Z$ is $(\delta, r)$-close to $\mathcal F^a(\Gamma)$ with respect to $\mathcal L_{z,r}$. Then the following conclusions hold.
	\begin{enumerate}[label=\textnormal{(\roman{*})}]
		\item If $\delta \le \delta(Y, \ep)$, there exists a constant $W$ such that for any $x \in \mathcal L_{z, r} \cap B^*(z, \ep^{-1} r)$ and $\tau \in [\ep r^2, \ep^{-1} r^2]$,
	\begin{align*}
	\abs{\widetilde \WW_x(\tau)-W}<\ep.
	\end{align*}
		\item If $\delta \le \delta( Y, \ep, \eta)$, then any point in $B^*_{\delta r}(\mathcal L_{z, r}) \cap B^*(z, \ep^{-1} r)$ is not $(3, \eta, s)$-symmetric for any $s \in [\ep r, \ep^{-1} r]$.
		
		\item If $\delta \le \delta(Y, \ep)$ and $x \in B^*_{\delta r}(\mathcal L_{z, r}) \cap B^*(z, \ep^{-1} r)$ is $(\delta, s)$-close to $\mathcal F^0(\Gamma')$ for some $s \in [\ep r, \ep^{-1} r]$, then $\Gamma'=\Gamma$ and
	\begin{align*}
d_{\mathrm{H}} \lc \mathcal L_{x, s} \cap B^*(x, \ep^{-1} s), \mathcal L_{z, r} \cap B^{*,-}(x, \ep^{-1} s)  \rc < \ep s.
\end{align*}	
	\end{enumerate}	
\end{prop}

\begin{proof}
We only prove (i), as (ii) and (iii) can be proved similarly by a limiting argument.

Without loss of generality, we assume that $\t(z)=0$ and $r=1$. Suppose the conclusion (i) fails. Then we can find a sequence of Ricci flow limit spaces $(Z^l, d_{Z^l}, \t_l)$, which is obtained as the limit of a sequence in $\mathcal M(4, Y, T_l)$, and $z_l \in Z^l$ such that $z_l$ is $(l^{-2}, 1)$-close to $\mathcal F^{a_l}(\Gamma_l)$ with respect to $\mathcal L_{z_l, 1}$, but the conclusion fails.

After passing to a diagonal subsequence, we may assume that $a_l \to a$, each $(Z^l, d_{Z^l}, \t_l)$ is given by a closed Ricci flow, and $\Gamma_l=\Gamma$ for all $l$. By our assumption, we have
	\begin{align*}
		(Z^l, d_{Z^l},z_l, \t_l)\xrightarrow[l\to\infty]{\quad \hat C^\infty \quad } \lc \mathcal F^a(\Gamma),d^*_{\mathcal F}, p^*,\t \rc.
	\end{align*}
Then, it is clear that $\mathcal L_{z_l, 1}$ converge to the spine $ [\vec 0^4] \times  (-\infty, a]$ on which $\widetilde \WW$ is constant.
	
Thus, we obtain a contradiction by Proposition \ref{prop:Wconv1} and monotonicity if $l$ is sufficiently large.
\end{proof}

The main result of this subsection is the following neck decomposition theorem in dimension $4$.

\begin{thm}[Neck decomposition theorem]\label{neckdecomgeneral}
For any constants $\delta>0$, $\eta>0$ and $k \in \{1, 2\}$, if $\zeta \le \zeta(Y, \delta, \eta)$, then the following holds.

Given $z_0 \in Z$ with $\t(z_0)-2 \zeta^{-2} r_0^2 \in \III^-$, we have the decomposition\emph{:}
	\begin{align*}
		&B^*(z_0,r_0)\subset \bigcup_a\big(\NNN'_a\bigcap B^*(x_a,r_a)\big)\bigcup\bigcup_b B^*(x_b,r_b)\bigcup S^{k,\delta,\eta},\\
		&S^{k,\delta,\eta}\subset \bigcup_a\big(\CCC_{0,a}\bigcap B^*(x_a,r_a)\big)\bigcup\tilde{S}^{k,\delta,\eta},
	\end{align*}
with the following properties\emph{:}
	\begin{enumerate}[label=\textnormal{(\alph{*})}]
		\item For each $a$, $\NNN_a=B^*(x_a,2r_a)\setminus B^*_{r_x}(\CCC_a)$ is either a $(k, \delta, \cc, r_a)$-quotient cylindrical neck region or a $(\delta, \cc, r_a)$-flat neck region, where $\cc=\cc(Y)$ if $k=2$ and $\cc=\cc(Y, \eta)$ if $k=1$. In the former case, we set $\NNN_a'=\NNN_a;$ in the latter, $\NNN_a'$ denotes the modified $\delta$-region associated with $\NNN_a$ (see Definition \ref{def:modi}).
		
		\item For each $b$, there exists a point in $B^*(x_b,2 r_b)$ which is $(k+1,\eta,r_b)$-symmetric.
		
		\item The following content estimates hold\emph{:}
		\begin{align*}
		\begin{dcases}
\sum_a r_a^2+\sum_b r_b^2+\HHH^2(S^{2,\delta,\eta})\leq C(Y) r_0^2 \quad &\text{if} \quad k=2,\\
\sum_a r_a+\sum_b r_b+\HHH^1(S^{1,\delta,\eta})\leq C(Y, \eta) r_0 \quad &\text{if} \quad k=1.
\end{dcases}
	\end{align*}	
	
		\item $\HHH^k(\tilde{S}^{k,\delta,\eta})=0$. 
	\end{enumerate}
\end{thm}

Throughout the proof, we fix $z_0 \in Z$ and $r_0>0$ such that $\t(z_0)-2 \zeta^{-2} r_0^2 \in \III^-$. We will only consider balls $B^*(x, r)$ with $x \in B^*(z_0, 4 r_0)$ and $r \le r_0-d_Z(x, z_0)/4$.

\begin{defn}[Pinching set]
We define 
\begin{align}\label{infNash}
	\bar W:=\inf _{z\in B^*(z_0,4 r_0)}\widetilde \WW_z(\zeta^{-2}r_0^2).
\end{align}
and introduce the pinching set
	\begin{equation*}\index{$\PP_{\zeta,r}$}
		\PP_{\zeta,r}:= \{y\in B^*(z_0,4 r_0) \mid \widetilde \WW_y(\zeta^2 r^2)\leq \bar W+\zeta\}.
	\end{equation*}
For any $x \in B^*(z_0,4 r_0)$ with $r \le r_0-d_Z(x, z_0)/4$, we define the localized pinching set:
	\begin{equation*}\index{$\PP_{\zeta,r}(x)$}
		\PP_{\zeta,r}(x):=\PP_{\zeta,r} \bigcap B^*(x, 2.5 r),
	\end{equation*}
and set
		\begin{equation}\label{eq:pinchtime}\index{$t_p(x, r)$}
		t_p(x, r):= \sup_{y \in \PP_{\zeta,r}(x)} \t(y).
	\end{equation}
\end{defn}

\begin{defn}[Different types of balls]
	For any constants $\delta>0$, $\eta>0$, $\cc \in (0, 10^{-40})$ and $\beta>0$, we define the following different types of balls $B^*(x, r)$ with $x \in B^*(z_0, 4 r_0)$ and $r \le r_0-d_Z(x, z_0)/4$.
	\begin{enumerate}[label=\textnormal{(\alph{*})}]
		\item A ball $B^*(x_a,r_a)$ is called an $a$-ball if $\NNN_a=B^*(x_a,2r_a)\setminus B^*_{r_x}(\CCC_a)$ is either a $(k,\delta, \cc, r_a)$-quotient cylindrical neck region or a $(\delta, \cc, r_a)$-flat neck region.
		
		\item A ball $B^*(x_b,r_b)$ is called a $b$-ball if there exists a $(k+1,\eta, r_b)$-symmetric point in $B^*(x_b, 2 r_b)$. 
		
		\item A ball $B^*(x_c,r_c)$ is called a $c$-ball if it is not a $b$-ball and satisfies $\mathcal{P}_{\zeta,r_c}(x_c) \cap B^*(x_c, r_c) \ne \emptyset$ and
			\begin{align}\label{eq:contentextra1}
\abs{B_{3 \beta r_c}^*\lc \mathcal{P}_{\zeta,r_c}(x_c) \rc} > D \beta^{7-k}r_c^{6},
	\end{align}		
	where $D=D(Y)$ is a positive constant to be determined in Lemma \ref{contentconesplitting}.
	
	\item A ball $B^*(x_d,r_d)$ is called a $d$-ball if it is not a $b$-ball and satisfies $\mathcal{P}_{\zeta,r_d}(x_d) \cap B^*(x_d, r_d) \ne \emptyset$ and
					\begin{align}\label{eq:extravol2}
0<\abs{B_{3 \beta r_d}^*\lc \mathcal{P}_{\zeta,r_d}(x_d) \rc} \le  D \beta^{7-k}r_d^{6}.
	\end{align}		
	
		\item A ball $B^*(x_e,r_e)$ is called an $e$-ball if $\mathcal{P}_{\zeta,r_e}(x_e) \cap B^*(x_e, r_e) = \emptyset$.
	\end{enumerate}
\end{defn}

By our definition, every ball $B^*(x, r)$ with $x \in B^*(z_0, 4 r_0)$ and $r \le r_0-d_Z(x, z_0)/4$ must belong to one of the categories: $b$-ball, $c$-ball, $d$-ball, or $e$-ball. However, it is possible that $B^*(x, r)$ is both a $b$-ball and an $e$-ball.

For a $c$-ball, we first prove

\begin{lem} \label{contentconesplitting}
Given $\eta>0$, $\beta>0$, $k \in \{1, 2\}$ and $\ep>0$, if $\zeta \le \zeta(Y, \eta, \ep, \beta)$ and $B^*(x, r)$ is a $c$-ball, then any point $y_0 \in \mathcal{P}_{\zeta,r}(x)$ is either $(k,\ep,r)$-quotient cylindrical or $(\ep, r)$-close to $\mathcal F^a(\Gamma)$ for some $a \ge 0$ and $\Gamma \leqslant \mathrm{O}(4)$.
\end{lem}

\begin{proof}
Let $S:=\{y_i\}_{0 \le i \le N}$ be a maximal $\beta r$-separated subset of $\mathcal{P}_{\zeta,r}(x)$ containing $y_0$. Since $B^*(x, r)$ is not a $b$-ball, arguing as in the proof of Lemma \ref{lem:keycover}, we conclude that if $N \ge C(Y) \beta^{1-k}$, then $y_0$ is either $(k, \ep, r)$-symmetric---which must in turn be $(k, \ep, r)$-quotient cylindrical---or $y_0$ is $(\ep, r)$-close to $\mathcal F^a(\Gamma)$, provided that $\zeta \le \zeta( Y, \eta, \ep, \beta)$. 

If instead $N \le C(Y) \beta^{1-k}$, then $\mathcal{P}_{\zeta,r}(x)$ can be covered by at most $C(Y) \beta^{1-k}$ balls of radius $\beta r$. Combining this with Proposition \ref{prop:volumebound}, we obtain
		\begin{align*}
\abs{B_{3 \beta r}^*\lc \mathcal{P}_{\zeta,r}(x) \rc} \le D(Y) \beta^{7-k} r^{6}.
	\end{align*}	
By \eqref{eq:contentextra1}, this leads to a contradiction.

Therefore, the proof is complete.
\end{proof}

In the setting of Lemma \ref{contentconesplitting}, if there exists a point $y \in \mathcal{P}_{\zeta,r}(x)$ that is $(k, \ep, r)$-quotient cylindrical regarding $\bar{\mathcal C}^m_k(\Gamma)$, then, by rigidity, any other point $y' \in \mathcal{P}_{\zeta,r}(x) $ is also $(k, \ep, r)$-quotient cylindrical regarding $\bar{\mathcal C}^m_k(\Gamma)$. 

On the other hand, if one point $y \in \mathcal{P}_{\zeta,r}(x) $ is $(\ep, r)$-close to $\mathcal F^a(\Gamma)$, then any other point $y' \in \mathcal{P}_{\zeta,r}(x)$ is $(\ep, r)$-close to $\mathcal F^{a'}(\Gamma)$. In this case, we define the following time-function:
		\begin{align} \label{eq:statictime}\index{$t_s(x, r)$}
t_s(x, r):=\sup \left\{ ar^2+\t(y) \mid y \in \mathcal{P}_{\zeta,r}(x) \cap B^*(x, r) \text{ is $(\ep, r)$-close to } \mathcal F^a(\Gamma) \right\}.
	\end{align}	
It is clear that if $\zeta\le \zeta(Y, \eta, \ep)$, we always have
		\begin{align} \label{eq:timea1}
t_s(x, r) \ge t_p(x, r),
	\end{align}	
where $t_p(x, r)$ is the function in \eqref{eq:pinchtime}.

By Lemma \ref{contentconesplitting} and the subsequent discussion, each $c$-ball is modeled either on a quotient cylinder or on a flat cone. We refer to the former as a \textbf{type-I} $c$-ball and to the latter as a \textbf{type-II} $c$-ball.

We now further decompose $c$-balls into neck regions according to their types. For type-I $c$-balls, the decomposition is similar to that in Proposition \ref{cylneckdecomp}.

\begin{prop}[Decomposition of type-I $c$-balls]\label{cballdecomposition1}
	For $z\in B^*(z_0, 4 r_0)$ and $ s \le r_0-d_Z(z, z_0)/4$, let $B^*(z,s)$ be a type-I $c$-ball. For any constants $\delta>0$, $\eta>0$, $\cc \in (0, 10^{-40})$ and $\beta>0$, if $\zeta \le \zeta( Y, \delta, \eta, \cc, \beta)$, then we have the decomposition
	\begin{equation*}
		B^*(z,2s)\subset \left(\CCC_0\bigcup\NNN\right)\bigcup\bigcup_b B^*(x_b,r_b)\bigcup \bigcup_d B^*(x_d,r_d)\bigcup \bigcup_e B^*(x_e,r_e)
	\end{equation*}
	satisfying
	\begin{enumerate}[label=\textnormal{(\roman{*})}]
		\item for each $b$, $B^*(x_b,r_b)$ is a $b$-ball with $r_b \le \cc^5 s;$
				\item for each $d$, $B^*(x_d,r_d)$ is a $d$-ball with $r_d \le \cc^5 s;$
		\item for each $e$, $B^*(x_e,r_e)$ is an $e$-ball with $r_e \le \cc^5 s;$
		\item $\displaystyle \NNN:=B^*(z,2s)\setminus \lc\CCC_0\bigcup\bigcup_b B^*(x_b,r_b)\bigcup \bigcup_d B^*(x_d,r_d)\bigcup \bigcup_e B^*(x_e,r_e)\rc$ is a $(k,\delta, \cc, s)$-quotient neck region regarding some $\bar{\mathcal C}^m_k(\Gamma);$
		\item the following content estimate holds:
						\begin{align*}
\HHH^{k} \lc \CCC_0 \bigcap B^*(z,3s/2) \rc+\sum_{x_b \in B^*(z,3s/2)} r_b^{k}+\sum_{x_d \in B^*(z,3s/2)}r_d^{k}+\sum_{x_e \in B^*(z,3s/2)} r_e^{k}\leq C(Y, \cc) s^k.
	\end{align*}				
		\end{enumerate} 		
\end{prop}
\begin{proof}
Without loss of generality, we assume $s=1$. In the proof, we choose a small parameter $\ep \ll \delta$ and assume that Lemma \ref{contentconesplitting} holds for constants $\ep$, $\eta$, $\beta$ and $\zeta$ by default. Moreover, we may assume $\delta \ll \cc$ and define $\gamma:=\cc^{5}$.

Since $B^*(z,1)$ is a type-I $c$-ball, it follows from Lemma \ref{contentconesplitting} that there exists $q\in \PP_{\zeta,1}(z) \cap B^*(z, 1)$ such that $q$ is $(k,\ep,1)$-quotient cylindrical regarding $\bar{\mathcal C}^m_k(\Gamma)$ with respect to $\mathcal{L}_{q,1}$. We define
\begin{align*}
L^1:=\mathcal{L}_{q,1} \bigcap B^*(z, 2).
	\end{align*}	
	
Choose a maximal $2\cc^2 \gamma$-separated set $\{x_{f^1}\} \subset L^1$. In particular, $\{B^*(x_{f^1}, \cc^2 \gamma)\}$ are pairwise disjoint and
\begin{align*}
L^1\subset \bigcup_{f^1}B^*(x_{f^1}, 2 \cc^2\gamma).
	\end{align*}	

 According to the type of $B^*(x_{f^1},\gamma)$, we write
\begin{align*}
L^1 \subset \bigcup_{b^1} B^*(x_{b^1}, \gamma)\bigcup\bigcup_{c^1} B^*(x_{c^1}, \gamma)\bigcup\bigcup_{d^1} B^*(x_{d^1}, \gamma)\bigcup\bigcup_{e^1} B^*(x_{e^1}, \gamma).
	\end{align*}		

We set 
\begin{align*}
\NNN^1:= B^*(z,2)\setminus \lc\bigcup_{b^1} B^*(x_{b^1}, \gamma)\bigcup\bigcup_{c^1} B^*(x_{c^1}, \gamma)\bigcup\bigcup_{d^1} B^*(x_{d^1}, \gamma)\bigcup\bigcup_{e^1} B^*(x_{e^1}, \gamma)\rc.
	\end{align*}	
Moreover, we define $\CCC^1:= \{x_{b^1},x_{c^1},x_{d^1},x_{e^1}\}$ with radius $r_x:\equiv \gamma$.

	\begin{claim}\label{checknecks}
If $\ep \le \ep(Y, \delta)$ and $\zeta \le \zeta(Y, \delta, \eta, \ep, \cc, \beta)$, then $\NNN^1$ is a $(k,\delta, \cc, 1)$-quotient cylindrical neck region regarding $\bar{\mathcal C}^m_k(\Gamma)$.
	\end{claim}

Indeed, property $(n1)$ follows from the construction. For $(n2)$, since $q \in \PP_{\zeta,1}(z)$, we have
\begin{align*}
		\left|\widetilde \WW_q(\tau)-\bar W\right|\leq \delta^4,\quad \forall \tau \in \left[(\delta\gamma)^2,\delta^{-2}\right],
	\end{align*}
if $\zeta$ is sufficiently small. Moreover, it follows from the generalization of Proposition \ref{prop:almostconstant} in the quotient cylindrical case that if $\ep$ is small, then for any $x \in \mathcal{L}_{q,1} \bigcap B^*(z, 2)$,
	\begin{align*}
		\left|\widetilde \WW_x(\tau)-\widetilde \WW_q(\tau)\right|\leq \delta^4,\quad \forall \tau \in \left[(\delta\gamma)^2,\delta^{-2}\right].
	\end{align*}
Combining the above two inequalities, we have for any $x \in \mathcal{L}_{q,1} \bigcap B^*(z, 2)$,
	\begin{align} \label{eq:extrac11}
		\left|\widetilde \WW_x(\tau)-\bar W\right|\leq \delta^3, \quad \forall \tau \in \left[(\delta\gamma)^2,\delta^{-2}\right],
	\end{align}
	which establishes $(n2)$. Moreover, $(n3)$ and $(n4)$ follow from our construction. This completes the proof of Claim \ref{checknecks}.

Each $c$-ball $B^*(x_{c^1},\gamma)$ must be of type-I, since otherwise the model space at scale $\gamma$ is a flat cone, which contradicts \eqref{eq:extrac11} by Proposition \ref{prop:uniformcq}. Thus, we can repeat the above process. Specifically, for each $c$-ball $B^*(x_{c^1},\gamma)$, we can find $q_{c^1}\in \PP_{\zeta, \gamma}(x_{c^1}) \cap B^*(x_{c^1},\gamma)$ which is $(k,\ep, \gamma)$-quotient cylindrical regarding $\bar{\mathcal C}^m_k(\Gamma)$ with respect to $\LL_{q_{c^1}, \gamma}$. Notice that the model space is the same $\bar{\mathcal C}^m_k(\Gamma)$ by Proposition \ref{prop:uniformcq} and \eqref{eq:extrac11}.

We define 
	\begin{align*}
L^2:=B^*(z, 2) \bigcap \bigcup_{c^1} \lc \LL_{q_{c^1}, \gamma} \bigcap B^*(x_{c^1}, 2\gamma) \rc \setminus\lc\bigcup_{b^1} B^*(x_{b^1}, 2 \cc^2 \gamma)\bigcup\bigcup_{d^1} B^*(x_{d^1}, 2 \cc^2 \gamma)\bigcup\bigcup_{e^1} B^*(x_{e^1}, 2 \cc^2 \gamma)\rc.
	\end{align*}

Choose a maximal $2\cc^2 \gamma^2$-separated set $\{x_{f^2}\} \subset L^2$. After doing this process for each $c$-ball $B^*(x_{c^1},\gamma)$, we can re-index the above balls by $b,c,d$-balls and define
	\begin{align*}
		\NNN^2:= B^*(z,2)\setminus\lc\bigcup_{c^2} B^*(x_{c^2}, \gamma^2)\bigcup\bigcup_{1\leq j\leq 2} \lc \bigcup_{b^j} B^*(x_{b^j}, \gamma^j)\bigcup\bigcup_{d^j} B^*(x_{d^j}, \gamma^j)\bigcup\bigcup_{e^j}B^*(x_{e^j}, \gamma^j) \rc\rc.
	\end{align*}
Moreover, we define $\CCC^2:= \{x_{b^1}, x_{b^2},x_{d^1},x_{d^2},x_{e^1},x_{e^2}, x_{c^2}\}$, with the radius function chosen to be the corresponding radii of the associated balls. Then, one can check as in Claim \ref{checknecks} and the proof of Proposition \ref{cylneckdecomp} that $\NNN^2$ is a $(k,\delta, \cc, 1)$-quotient cylindrical neck region regarding $\bar{\mathcal C}^m_k(\Gamma)$.
	
	Repeating the above decomposition for $l$ steps, we obtain
	\begin{align*}
		\NNN^l:= B^*(z,2)\setminus\lc\bigcup
		_{c^l} B^*(x_{c^l}, \gamma^l)\bigcup\bigcup_{1\leq j\leq l} \lc \bigcup_{b^j} B^*(x_{b^j}, \gamma^j)\bigcup\bigcup_{d^j} B^*(x_{d^j}, \gamma^j)\bigcup\bigcup_{e^j}B^*(x_{e^j}, \gamma^j) \rc \rc.
	\end{align*}
	It follows as before that $\NNN^l$ is a $(k,\delta, \cc, 1)$-quotient cylindrical neck region regarding $\bar{\mathcal C}^m_k(\Gamma)$.
	
	Set $\mathcal G^l:= \bigcup_{c^l}\{x_{c^l}\}$. It is clear from our construction that $\mathcal G^{l+1} \subset B^*_{2\gamma^l}(\mathcal G^l)$. Then we denote the Hausdorff limit of $\mathcal G^l$ by $\CCC_0$ and define
	\begin{align*}
		\NNN:= B^*(z,2)\setminus\lc\CCC_0\bigcup\bigcup_{1\leq j<\infty}\lc \bigcup_{b^j} B^*(x_{b^j}, \gamma^j)\bigcup\bigcup_{d^j} B^*(x_{d^j}, \gamma^j)\bigcup\bigcup_{e^j}B^*(x_{e^j}, \gamma^j) \rc\rc.
	\end{align*}
From our construction, $\NNN$ is a $(k,\delta, \cc, 1)$-quotient cylindrical neck region regarding $\bar{\mathcal C}^m_k(\Gamma)$, with centers given by $\CCC_0$ together with the centers of the associated balls, and the radius function chosen to be the corresponding radii of those balls.

We rewrite the above decomposition as:
	\begin{equation*}
		B^*(z,2)\subset \left(\CCC_0\bigcup\NNN\right)\bigcup\bigcup_b B^*(x_b,r_b)\bigcup \bigcup_d B^*(x_d,r_d)\bigcup \bigcup_e B^*(x_e,r_e).
	\end{equation*}
	
The quotient cylindrical neck region $\NNN$ we construct satisfies Ahlfors regularity with a constant $C(Y, \cc)>1$, in the same sense as Theorem \ref{ahlforsregucyl1} and Proposition \ref{ahlforsregucyl1quo}. This follows directly from the construction of $\NNN$ and the argument used in the proof of Theorem \ref{ahlforsregucyl1}.

Thus, the following content estimate holds:
	\begin{align*}
\HHH^{k} \lc \CCC_0 \bigcap B^*(z,3/2) \rc+\sum_{x_b \in B^*(z,3/2)} r_b^{k}+\sum_{x_d \in B^*(z,3/2)}r_d^{k}+\sum_{x_e \in B^*(z,3/2)} r_e^{k}\leq C(Y, \cc),
	\end{align*}
which completes the proof of the type-I $c$-ball decomposition.
\end{proof}

For type-II $c$-balls, we consider the cases $k=1$ and $k=2$ separately.

\begin{prop}[Decomposition of type-II $c$-balls: $k=2$] \label{cballdecomposition2}
	For $z\in B^*(z_0, 4 r_0)$ and $ s \le r_0-d_Z(z, z_0)/4$, let $B^*(z,s)$ be a type-II $c$-ball with $k=2$. For any constants $\delta>0$, $\eta>0$, $\cc \in (0, 10^{-40})$ and $\beta>0$, if $\cc \le \cc( Y)$ and $\zeta \le \zeta( Y, \delta, \eta, \cc, \beta)$, then we have the decomposition
	\begin{equation*}
		B^*(z,2s)\subset \left(\CCC_0\bigcup\NNN'\right)\bigcup\bigcup_b B^*(x_b,r_b)\bigcup \bigcup_d B^*(x_d,r_d)\bigcup \bigcup_e B^*(x_e,r_e)\bigcup \bigcup_{e'} B^*(x_{e'},r_{e'})
	\end{equation*}
	satisfying
	\begin{enumerate}[label=\textnormal{(\roman{*})}]
		\item for each $b$, $B^*(x_b,r_b)$ is a $b$-ball with $r_b \le \cc^5 s;$
				\item for each $d$, $B^*(x_d,r_d)$ is a $d$-ball with $r_d \le \cc^5 s;$
		\item for any $e$ and $e'$, the balls $B^*(x_e,r_e)$ and $B^*(x_{e'}, r_{e'})$ are $e$-balls with $\max\{r_e,r_{e'}\} \le \cc^5 s;$

		\item $\displaystyle \NNN:=B^*(z,2s)\setminus \lc\CCC_0\bigcup\bigcup_b B^*(x_b,r_b)\bigcup \bigcup_d B^*(x_d,r_d)\bigcup \bigcup_e B^*(x_e,r_e)\rc$ is a $(\delta, \cc, s)$-flat neck region regarding some $\mathcal F^0(\Gamma);$
		
		\item any $x \in \displaystyle \NNN':=\NNN \setminus \lc \bigcup_{e'} B^*(x_{e'},r_{e'}) \rc$ is $(3, \delta,  \delta d_Z(x, \CCC))$-symmetric\emph{;}

		\item the following content estimate holds:
						\begin{align*}
&\HHH^{2}\lc \CCC_0 \bigcap B^*(z, 3s/2) \rc+\sum_{x_b \in B^*(z, 3s/2)} r_b^{2}+\sum_{x_d \in B^*(z, 3s/2)}r_d^{2} \\
+&\sum_{x_e \in B^*(z, 3s/2)} r_e^{2}+\sum_{x_{e'} \in B^*(z, 3s/2)} r_{e'}^{2}\leq C(Y, \cc) s^2.
	\end{align*}			
	\end{enumerate}		
\end{prop}

\begin{proof}
		Without loss of generality, we assume $s=1$. In the proof, we choose a small parameter $\ep \ll \delta$ and assume that Lemma \ref{contentconesplitting} holds for constants $\ep$, $\eta$, $\beta$ and $\zeta$ by default. Moreover, we may assume $\delta \ll \cc$ and define $\gamma:=\cc^{5}$.
		
	Since $B^*(z,1)$ is a type-II $c$-ball, the time-function $t_s(z, 1)$ (see \eqref{eq:statictime}) is well defined, and \eqref{eq:timea1} holds. For simplicity, we call a type-II $c$-ball $B^*(x, r)$ \textbf{static} if $t_s(x, r) \ge \t(x)+(2.5 r)^2$ and \textbf{quasi-static} if $t_s(x, r) < \t(x)+(2.5 r)^2$.
	
	By the definition of $t_s$, there exists $q\in \PP_{\zeta,1}(z) \cap B^*(z, 1)$ such that $q$ is $(\ep, 1)$-close to $\mathcal F^a(\Gamma)$ with respect to $\LL_{q, 1}$, where $a=t_s(z,1)-\t(q)$. We define $\tilde\LL_{z,1} \subset \LL_{q, 1}$\index{$\tilde \LL_{z, 1}$} as follows.
\begin{itemize}
\item If $t_p(z,1)\geq \sup\{\t(y) \mid y\in B^*(z, 2.3)\}-\xi$, we set 
		\begin{align}\label{equ:defLL1}
			\tilde\LL_{z,1}:=\LL_{q, 1} \bigcap B^*(z,2.2).
		\end{align} 
	
\item If $t_p(z,1)< \sup\{\t(y) \mid y\in B^*(z, 2.3)\}-\xi$, we set
		\begin{align}\label{equ:defLL2}
			\tilde\LL_{z,1}:=\LL_{q, 1} \bigcap B^*(z,2.3).
		\end{align} 
\end{itemize}
Here, $\xi=\xi(Y) \in (0, 10^{-4})$ is a small constant to be determined.

We set $L^1:=\tilde\LL_{z,1}$ and choose a maximal $2\cc^2 \gamma$-separated set $\{x_{f^1}\} \subset L^1$. In particular, $\{B^*(x_{f^1}, \cc^2 \gamma)\}$ are pairwise disjoint and
\begin{align*}
L^1 \subset \bigcup_{f^1}B^*(x_{f^1}, 2 \cc^2\gamma).
	\end{align*}	

According to the type of $B^*(x_{f^1},\gamma)$, we can write
		\begin{align*}
L^1 \subset \bigcup_{b^1} B^*(x_{b^1}, \gamma)\bigcup\bigcup_{c^1} B^*(x_{c^1}, \gamma)\bigcup\bigcup_{d^1} B^*(x_{d^1}, \gamma)\bigcup\bigcup_{e^1} B^*(x_{e^1}, \gamma).
		\end{align*}	

Arguing as in the proof of Proposition \ref{cballdecomposition1}, we conclude that each $c$-ball $B^*(x_{c^1}, \gamma)$ is of type-II with the same group $\Gamma$. We can reindex the $c$-balls $\bigcup_{c^1}B^*(x_{c^1},\gamma)$ by 
		\begin{align*}
			\bigcup_{c^{1,s}}B^*(x_{c^{1,s}},\gamma)\bigcup \bigcup_{c^{1,q}}B^*(x_{c^{1,q}},\gamma)
		\end{align*}
based on whether they are static or quasi-static.		

By a limiting argument, we conclude that if $\t(x_{c^1}) \le t_p(z, 1)-9 \gamma^2$, then $B^*(x_{c^1},\gamma)$ is a static $c$-ball, provided that $\zeta \le \zeta(Y, \ep, \eta, \cc, \beta)$. Moreover, by the definition of $t_p$, any ball $B^*(x_{f^1},\gamma)$ with $\t(x_{f^1}) \ge t_p(z, 1)+9 \gamma^2$ must be an $e$-ball. Consequently, any quasi-static $c$-ball $B^*(x_{c^{1,q}},\gamma)$ must satisfy 
		\begin{align*}
|\t(x_{c^{1,q}})-t_p(z, 1)| \le 9 \gamma^2.
		\end{align*}	
Thus, the number of $\{c^{1,q}\}$ is at most $C(Y) \cc^{-4}$.

Furthermore, by a limiting argument and using pseudolocality theorem, if $\ep \le \ep(Y, \delta, \cc)$ and $\zeta \le \zeta(Y, \ep, \eta, \cc, \beta)$, then we can choose a small constant $\xi=\xi(Y)>0$ such that
		\begin{align*}
y\in Z_{(-\infty, t_s(z,1)+2\xi^2\gamma^2)}\bigcap B^*(z,2)
		\end{align*}	
 with $d_Z(y, \tilde \LL_{z, 1})=s\geq \gamma$ is $(3,\delta, \delta  s)$-symmetric.

Next, we choose a cover $\{B^*(x_{e^{(1)}},\xi \gamma)\}$ of 
		\begin{align*}
\lc Z_{(t_s(z,1)+2\xi^2 \gamma^2,\infty)}\setminus \bigcup_{f^1}B^*(x_{f^1},\gamma)\rc\bigcap B^*(z,2)
		\end{align*}	
		such that $\{B^*(x_{e^{(1)}}, \xi\gamma/2)\}$ are pairwise disjoint. Note that if $B^*(z, 1)$ itself is static, then the above set is empty. By \eqref{eq:timea1}, all these balls are $e$-balls, and the total number of $\{B^*(x_{e^{(1)}}, \xi\gamma)\}$ is, by Proposition \ref{prop:volumebound}, bounded by $C(Y) \gamma^{-6}$. 

	Now we set 
		\begin{align*}
			\NNN^1:= B^*(z,2)\setminus \lc\bigcup_{b^1} B^*(x_{b^1}, \gamma)\bigcup\bigcup_{c^1} B^*(x_{c^1}, \gamma)\bigcup\bigcup_{d^1} B^*(x_{d^1}, \gamma)\bigcup\bigcup_{e^1} B^*(x_{e^1}, \gamma)\rc,
		\end{align*}
		and 
		\begin{align*}
			\NNN^{1,\prime}:=\NNN^1\setminus \lc\bigcup_{e^{(1)}}B^*(x_{e^{(1)}}, \xi\gamma)\rc.
		\end{align*} 	
		Moreover, we define $\CCC^1:= \{x_{b^1},x_{c^1},x_{d^1},x_{e^1}\}$ with radius $r_x:= \gamma$. 
		
		\begin{claim}\label{checkneck}
If $\ep \le \ep(Y, \delta, \cc)$ and $\zeta \le \zeta(Y, \ep, \eta, \cc, \beta)$, then $\NNN^1$ is a $(\delta, \cc, 1)$-flat neck region regarding $\mathcal F^0(\Gamma)$. Furthermore, any point in $\NNN^{1,\prime}$ is $(3, \delta,  \delta  d_Z(x, \CCC^1))$-symmetric.
		\end{claim}
		
Indeed, property $(n1)$ follows from the construction. For $(n2)$, since $q \in \PP_{\zeta,1}(z)$, we have
		\begin{align*}
			\left|\widetilde \WW_q(\tau)-\bar W\right|\leq \delta^4,\quad \forall \tau \in \left[(\delta\gamma)^2,\delta^{-2}\right],
		\end{align*}
if $\zeta$ is sufficiently small. Moreover, it follows from Proposition \ref{prop:almostconstant1} (i) that for any $x\in\tilde\LL_{z,1} $,
		\begin{align*}
			\left|\widetilde \WW_x(\tau)-\widetilde \WW_q(\tau)\right|\leq \delta^4,\quad \forall \tau \in \left[(\delta\gamma)^2,\delta^{-2}\right].
		\end{align*}
		Combining the above two inequalities, we have for any $x \in \tilde\LL_{z,1}$,
		\begin{align*}
			\left|\widetilde \WW_x(\tau)-\bar W\right|\leq \delta^3, \quad \forall \tau \in \left[(\delta\gamma)^2,\delta^{-2}\right],
		\end{align*}
		which establishes $(n2)$. $(n3)$ and $(n4)$ follow from our construction and Proposition \ref{prop:almostconstant1} (iii). Finally, the last conclusion follows from our construction. This completes the proof of Claim \ref{checkneck}. 

For each $c$-ball $B^*(x_{c^1},\gamma)$, we can define the set $\tilde \LL_{x_{c^1}, \gamma}$ as in \eqref{equ:defLL1} and \eqref{equ:defLL2}. Then, we define
		\begin{align*}
		L^2:=	\bigcup_{c^{1}} \tilde\LL_{x_{c^{1}},\gamma} \setminus\lc\bigcup_{b^1} B^*(x_{b^1}, 2 \cc^2 \gamma)\bigcup\bigcup_{d^1} B^*(x_{d^1}, 2  \cc^2  \gamma)\bigcup\bigcup_{e^1} B^*(x_{e^1}, 2  \cc^2  \gamma)\rc
		\end{align*}
and choose a maximal $2\cc^2 \gamma^2$-separated set $\{x_{f^2}\} \subset L^2$. In particular, $\{B^*(x_{f^2}, \cc^2 \gamma^2)\}$ are pairwise disjoint and
\begin{align*}
L^2 \subset \bigcup_{f^2}B^*(x_{f^2}, 2 \cc^2\gamma^2).
	\end{align*}	

Next, we choose a cover $\{B^*(x_{e^{(2)}},  \xi\gamma^2)\}$ of 
		\begin{align*}
			\bigcup_{c^{1,q}}\lc Z_{\big(t_s(x_{c^{1,q}},\gamma)+2\xi^2\gamma^4,\infty\big)}\setminus \bigcup_{f^2}B^*(x_{f^2},\gamma^2)\rc\bigcap B^*(x_{c^{1,q}}, 2\gamma),
		\end{align*}
		such that $\{B^*(x_{e^{(2)}},  \xi\gamma^2/2)\}$ are pairwise disjoint. Note that each $B^*(x_{e^{(2)}},  \xi\gamma^2)$ is an $e$-ball by definition and \eqref{eq:timea1}.
		
We reindex all $f^2$-balls according to their types:
		\begin{align*}
B^*(x_{b^2},\gamma^2),\,B^*(x_{c^2},\gamma^2),\,B^*(x_{d^2},\gamma^2),\, B^*(x_{e^2},\gamma^2).
		\end{align*}
Moreover, we can reindex the $c$-balls $\{B^*(x_{c^2},\gamma^2)\}$ by
		\begin{align*}
			\bigcup_{c^{2,s}}B^*(x_{c^{2,s}},\gamma^2)\bigcup \bigcup_{c^{2,q}}B^*(x_{c^{2,q}},\gamma^2)
		\end{align*}
	based on whether they are static or quasi-static.		

Now, we set 
		\begin{align*}
			\NNN^2:=B^*(z,2)\setminus \lc \bigcup_{c^2}B^*(x_{c^2},\gamma^2)\bigcup \bigcup_{1\leq j\leq 2} \lc \bigcup_{b^j}B^*(x_{b^j},\gamma^j)\bigcup \bigcup_{d^j}B^*(x_{d^j},\gamma^j)\bigcup \bigcup_{e^j}B^*(x_{e^j},\gamma^j) \rc \rc,
		\end{align*}
with $\CCC^2:= \{x_{b^1}, x_{b^2},x_{d^1},x_{d^2},x_{e^1},x_{e^2}, x_{c^2}\}$ and the radius function chosen to be the corresponding radii of the associated balls. Moreover, we set
		\begin{align*}
			\NNN^{2,\prime}:=\NNN^2\setminus\lc \bigcup_{1\leq j\leq 2}\bigcup_{e^{(j)}} B^*(x_{e^{(j)}}, \xi\gamma^j)\rc.
		\end{align*} 
Arguing as in Claim \ref{checkneck}, we know that $\NNN^2$ is a $(\delta,\cc, 1)$-flat neck region regarding $\mathcal F^0(\Gamma)$. Furthermore, any point in $\NNN^{2,\prime}$ is $(3, \delta, \delta  d_Z(x, \CCC^2))$-symmetric.

For all $c$-balls $B^*(x_{c^2},\gamma^2)$, we can repeat the above decomposition process. After $l$ steps, we obtain
		\begin{align*}
			\NNN^l:= B^*(z,2)\setminus\lc\bigcup
			_{c^l} B^*(x_{c^l}, \gamma^l)\bigcup\bigcup_{1\leq j\leq l}\lc\bigcup_{b^j} B^*(x_{b^j}, \gamma^j)\bigcup\bigcup_{d^j} B^*(x_{d^j}, \gamma^j)\bigcup\bigcup_{e^j}B^*(x_{e^j}, \gamma^j)\rc\rc,
		\end{align*}
		and corresponding $\bigcup_{1\leq j\leq l}\bigcup_{e^{(j)}}B^*(x_{e^{(j)}},\xi\gamma^j)$. 	It follows as before that $\NNN^l$ is a $(\delta,\cc, 1)$-flat neck region regarding $\mathcal F^0(\Gamma)$ with center $\CCC^l$ consisting of the centers of the balls and the radius function chosen to be the corresponding radii of the associated balls. Moreover, we define
			\begin{align*}
\NNN^{l,\prime}:=\NNN^l\setminus\lc\bigcup_{1\leq j\leq l}\bigcup_{e^{(j)}}B^*(x_{e^{(j)}},\xi\gamma^j)\rc. 
		\end{align*}	
	Then each point in $\NNN^{l,\prime}$ is $(3, \delta, \delta  d_Z(x, \CCC^l))$-symmetric. As before, we can reindex the $c$-balls $\bigcup_{c^l} B^*(x_{c^l}, \gamma^l)$ by 
		$$\bigcup_{c^{l,s}}B^*(x_{c^{l,s}},\gamma^l)\bigcup \bigcup_{c^{l,q}}B^*(x_{c^{l,q}},\gamma^l)$$
	based on whether they are static or quasi-static.

Next, we carry out the content estimates for all these balls up to the $l$-th step. For simplicity, we use $r_{b^j},\,r_{c^j},\,r_{d^j},\,r_{e^j},\,r_{e^{(j)}}$ to denote the corresponding radii of these balls. Moreover, we call a $c$-ball of the form $B^*(x_{c^j}, \gamma^j)$ a \textbf{$c^{j}$-ball}. Similarly, $B^*(x_{c^{j, s}}, \gamma^j)$ and $B^*(x_{c^{j, q}}, \gamma^j)$ are referred to as a \textbf{$c^{j, s}$-ball} and a \textbf{$c^{j, q}$-ball}, respectively. The notations $b^{j}$-ball, $d^j$-ball, $e^j$-ball and $e^{(j)}$-ball are defined in the same way. Moreover, we say that a $c^j$-ball \textbf{produces} a $b^{j+1}$-ball, $c^{j+1}$-ball, $d^{j+1}$-ball, $e^{j+1}$-ball, or $e^{(j+1)}$-ball if such a ball arises in the decomposition of the $c^j$-ball. 
		
From our construction, we obtain several facts that will be used below:
\begin{itemize}
\item Only a $c^{j, q}$-ball can produce $e^{(j+1)}$-balls. In this case, the total number of $e^{(j+1)}$-balls produced by a $c^{j, q}$-ball is at most $C(Y) \gamma^{-6}$.
	
\item Each $c^j$-ball can produce at most $C(Y) \cc^{-4} \gamma^{-2}$ $c^{j+1, s}$-balls.

\item Each $c^j$-ball can produce at most $C(Y) \cc^{-4}$ $c^{j+1, q}$-balls.

\item If a $c^{j, s}$-ball can produce a $c^{j+1, q}$-ball, then $t_p(x_{c^{j, s}}, \gamma^j)- \sup\{\t(y) \mid y\in B^*(x_{c^{j, s}}, 2.3\gamma^j)\}<-\xi \gamma^{2j}$.
\end{itemize}
		
For each $e^{(l)}$-ball, we can assign an $(l-1)$-tuple $(a_0,\ldots,a_{l-2})$ with $a_i \in \{0, 1\}$ as follows. We know that the $e^{(l)}$-ball is produced by some $c^{l-1,q}$-ball. If this $c^{l-1,q}$-ball is produced by a $c^{l-2,s}$-ball, we set $a_{l-2}=1$. Otherwise, if it is produced by a $c^{l-2,q}$-ball, we set $a_{l-2}=0$. 

In general, for any $1\leq k\leq l-2$, if the $c^k$-ball in the chain is produced by a $c^{k-1,s}$-ball, we set $a_{k-1}=1$; if it is produced by a $c^{k-1,q}$-ball, we set $a_{k-1}=0$.

For any $e^{(l)}$-ball, let $S_{e^{(l)}}$ denote the largest number $k \in [0, l-2]$ such that $a_{k}=1$. If no such $k$ exists, we set $S_{e^{(l)}}=-1$. Then it is clear that the number of $e^{(l)}$-balls such that $S_{e^{(l)}}=k$ is at most
		\begin{align*}
			N_q^{k,s}\cdot (C(Y) \cc^{-4})^{l-1-k}\cdot C(Y)\gamma^{-6} \le N_q^{k,s} (C(Y))^{l-k} \cc^{-4(l-1-k)}\gamma^{-6},
		\end{align*}
where $N_q^{k,s}$ is the total number of $c^{k,s}$-balls that can produce $c^{k+1,q}$-balls. Here, we set $N^{-1,s}_q=1$.

Therefore, we have the following content estimate:
		\begin{align}\label{equ:econt1}
			\sum_{e^{(l)}}r_{e^{(l)}}^{2}=\sum_{-1\leq k\leq l-2}\sum_{S_{e^{(l)}}=k}r_{e^{(l)}}^{2}\leq \sum_{-1\leq k\leq l-2}N_q^{k,s} (C(Y))^{l-k} \gamma^{2l-6-4(l-k-1)/5},
		\end{align}
where we used the fact that $\cc=\gamma^{1/5}$.

		On the other hand, we claim that for each $c^{k,s}$-ball $B^*(x_{c^{k,s}},\gamma^k)$ that can produce a $c^{k+1,q}$-ball, there must exist a ball of the form
			\begin{align*}
B^*(x_{b^{k}}, \gamma^k),\,B^*(x_{d^{k}},\gamma^k),\,B^*(x_{e^{k}}, \gamma^k),\, \text{ or }B^*(x_{e^{k+1}},\gamma^{k+1})
		\end{align*}	
contained in $B^*(x_{c^{k,s}},10\gamma^k)$. Indeed, since $B^*(x_{c^{k,s}},\gamma^k)$ can produce a $c^{k+1,q}$-ball, we know $t_p(x_{c^{k,s}}, \gamma^k)- \sup\{\t(y) \mid y\in B^*(x_{c^{k,s}}, 2.3\gamma^k)\}<-\xi \gamma^{2k}$. If no ball of the form
			\begin{align*}
B^*(x_{b^{k}}, \gamma^k),\,B^*(x_{d^{k}}, \gamma^k),\,B^*(x_{e^{k}}, \gamma^k)
		\end{align*}	
lies inside $B^*(x_{c^{k,s}},10\gamma^k)$, then by \eqref{equ:defLL2}, we can find an $e^{k+1}$-ball with its center on $\tilde \LL_{x_{c^{k,s}},\gamma^k}$, provided $\cc \le \cc(Y)$. This implies such an $e^{k+1}$-ball is contained in $B^*(x_{c^{k,s}},10\gamma^k)$. This establishes the claim.
	
For each $c^{k,s}$-ball that can produce a $c^{k+1,q}$-ball, we assign to it a $b^k$-, $d^k$-, $e^k$- or $e^{k+1}$-ball as described in the claim above. This assignment defines a map:
			\begin{align*}
\iota_k: \left\{1, \ldots, N^{k, s}_q \right\} \longrightarrow  \left \{b^k, d^k, e^k, e^{k+1} \right \}.
		\end{align*}	
By Proposition \ref{prop:volumebound} and condition $(n1)$ of Definition \ref{defnneckgeneral}, each element in the image of $\iota_k$ has at most $C(Y, \gamma)$ preimages.

Consequently, we obtain that for $k\in \{0,1,\ldots,l-2\}$,
		\begin{align}\label{equ:Nksq}
N_q^{k,s} \le C(Y, \gamma) \lc \gamma^{-2k} \lc\sum_{b^k}r_{b^k}^{2}+\sum_{d^k}r_{d^k}^{2}+\sum_{e^{k}}r_{e^{k}}^{2}\rc+\gamma^{-2k-2} \sum_{e^{k+1}}r_{e^{k+1}}^{2} \rc,
		\end{align}		
where $\{x_{b^0}, x_{d^0}, x_{e^0}\}$ is empty.

	Note that
			\begin{align*}
\sum_{S_{e^{(l)}}=-1} r_{e^{(l)}}^2\leq (C(Y)\cc^{-4})^{l}\cdot C(Y)\gamma^{-6}\cdot \gamma^{2l}\leq C(Y,\gamma)(C(Y) \gamma^{6/5})^{l}.
		\end{align*}		
Combining this with \eqref{equ:econt1} and \eqref{equ:Nksq}, we have
		\begin{align}\label{equ:econt2}
		\sum_{e^{(l)}}r_{e^{(l)}}^{2}	\leq & C(Y,\gamma)(C(Y) \gamma^{6/5})^{l} \notag \\
		&+C(Y, \gamma)\sum_{0\leq k\leq l-2}(C(Y) \gamma^2)^{l-k} \gamma^{-4(l-k-1)/5} \lc  \gamma^{-6}\lc\sum_{b^k}r_{b^k}^{2}+\sum_{d^k}r_{d^k}^{2}+\sum_{e^{k}}r_{e^{k}}^{2}\rc+\gamma^{-8} \sum_{e^{k+1}}r_{e^{k+1}}^{2}\rc\nonumber\\
			\leq &C(Y, \gamma)\sum_{0\leq k\leq l-1}(C(Y) \gamma^{6/5})^{l-k}\lc\sum_{b^k}r_{b^k}^{2}+\sum_{d^k}r_{d^k}^{2}+\sum_{e^{k}}r_{e^{k}}^{2}\rc+C(Y,\gamma)(C(Y) \gamma^{6/5})^{l}.
		\end{align}  
		
Set $\mathcal G^l:=\bigcup_{c^l} \{x_{c^l}\}$. It is clear from our construction that $\mathcal G^{l+1} \subset B^*_{2\gamma^l}(\mathcal G^l)$. Then we denote the Hausdorff limit of $\mathcal G^l$ by $\CCC_0$. We define
		\begin{align*}
			\NNN:= B^*(z,2)\setminus\lc\CCC_0\bigcup\bigcup_{1\leq j<\infty}\lc\bigcup_{b^j} B^*(x_{b^j}, \gamma^j)\bigcup\bigcup_{d^j} B^*(x_{d^j}, \gamma^j)\bigcup\bigcup_{e^j}B^*(x_{e^j}, \gamma^j)\rc\rc.
		\end{align*}
		
	From our construction, $\NNN$ is a $(\delta,\cc, 1)$-flat neck region regarding $\mathcal F^0(\Gamma)$ with centers given by $\CCC_0$ together with the centers of the associated balls, and the radius function chosen to be the corresponding radii of those balls. Moreover, we define
			\begin{align*}
\NNN':=\NNN \setminus\lc\bigcup_{1\leq j < \infty}\bigcup_{e^{(j)}}B^*(x_{e^{(j)}},\xi\gamma^j)\rc. 
		\end{align*}	
Note that each point in $\NNN'$ is $(3, \delta,  \delta  d_Z(x, \CCC))$-symmetric. 		
		
We rewrite the above decomposition as:
		\begin{equation*}
			B^*(z,2)\subset \left(\CCC_0\bigcup\NNN' \right)\bigcup\bigcup_b B^*(x_b,r_b)\bigcup \bigcup_d B^*(x_d,r_d)\bigcup \bigcup_e B^*(x_e,r_e)\bigcup \bigcup_{e'}B^*(x_{e'}, r_{e'}),
		\end{equation*}
where $\bigcup_{e'}B^*(x_{e'},r_{e'}):=\bigcup_{1\leq j < \infty}\bigcup_{e^{(j)}}B^*(x_{e^{(j)}},\xi\gamma^j)$. 

By Proposition \ref{ahlforsregforstaticneck}, the following content estimate holds:
		\begin{align}\label{equ:ahcont}
			\HHH^{2}\lc \CCC_0 \bigcap B^*(z, 1.6) \rc+\sum_{x_b \in B^*(z, 1.6)}r_b^{2}+\sum_{x_d \in B^*(z, 1.6)}r_d^{2}+\sum_{x_e \in B^*(z, 1.6)}r_e^{2}\leq C(Y, \cc).
		\end{align}
		Moreover, by \eqref{equ:econt2} and \eqref{equ:ahcont}, the following content estimate holds:
		\begin{align*}
			&\sum_{x_{e'} \in B^*(z, 1.5)}r_{e'}^{2}=\sum_{l=1}^\infty \sum_{e^{(l)}} r_{e^{(l)}}^{2} \\
			\leq& C(Y, \gamma)\sum_{l=1}^\infty \sum_{0\leq k\leq l} (C(Y) \gamma^{6/5})^{l-k}\lc\sum_{x_{b^k} \in B^*(z, 1.6)}r_{b^k}^{2}+\sum_{x_{d^k} \in B^*(z, 1.6)}r_{d^k}^{2}+\sum_{x_{e^k} \in B^*(z, 1.6)}r_{e^{k}}^{2}\rc+C(Y,\gamma)\sum_{l=1}^\infty (C(Y) \gamma^{6/5})^{l}\nonumber\\
			\leq& C(Y, \gamma)\sum_{k=0}^\infty \sum_{k \le l <\infty}(C(Y) \gamma^{6/5})^{l-k}\lc\sum_{x_{b^k} \in B^*(z, 1.6)}r_{b^k}^{2}+\sum_{x_{d^k} \in B^*(z, 1.6)}r_{d^k}^{2}+\sum_{x_{e^k} \in B^*(z, 1.6)}r_{e^{k}}^{2}\rc+C(Y,\gamma)\nonumber\\
			\leq& C(Y, \gamma)\sum_{k=1}^\infty\lc\sum_{x_{b^k} \in B^*(z, 1.6)}r_{b^k}^{2}+\sum_{x_{d^k} \in B^*(z, 1.6)}r_{d^k}^{2}+\sum_{x_{e^k} \in B^*(z, 1.6)}r_{e^{k}}^{2}\rc+C(Y,\gamma)\leq C(Y, \gamma)=C(Y, \cc).
		\end{align*}
		Here, we have chosen $\cc \le \cc(Y)$ such that $C(Y) \gamma^{6/5} \leq 1/2$.
		This completes the proof of the decomposition.
	\end{proof}

\begin{prop}[Decomposition of type-II $c$-balls: $k=1$] \label{cballdecomposition3}
	For $z\in B^*(z_0, 4 r_0)$ and $ s \le r_0-d_Z(z, z_0)/4$, let $B^*(z,s)$ be a type-II $c$-ball with $k=1$. For any constants $\delta>0$, $\eta>0$, $\cc \in (0, 10^{-40})$ and $\beta>0$, if $\cc \le \cc( Y, \eta)$ and $\zeta \le \zeta( Y, \delta, \eta, \cc, \beta)$, then we have the decomposition
	\begin{equation*}
		B^*(z,2s)\subset \left(\CCC_0\bigcup\NNN'\right)\bigcup\bigcup_b B^*(x_b,r_b) \bigcup \bigcup_d B^*(x_d,r_d)\bigcup \bigcup_e B^*(x_e,r_e)\bigcup \bigcup_{e'} B^*(x_{e'},r_{e'})
	\end{equation*}
	satisfying
	\begin{enumerate}[label=\textnormal{(\roman{*})}]
		\item for each $b$, $B^*(x_b,r_b)$ is a $b$-ball with $r_b \le \cc^5 s;$
				\item for each $d$, $B^*(x_d,r_d)$ is a $d$-ball with $r_d \le \cc^5 s;$
		\item for any $e$ and $e'$, the balls $B^*(x_e,r_e)$ and $B^*(x_{e'}, r_{e'})$ are $e$-balls with $\max\{r_e,r_{e'}\} \le \cc^5 s;$

		\item $\displaystyle \NNN:=B^*(z,2s)\setminus \lc\CCC_0\bigcup\bigcup_b B^*(x_b,r_b)\bigcup\bigcup_{b'} B^*(x_{b'},r_{b'})\bigcup \bigcup_d B^*(x_d,r_d)\bigcup \bigcup_e B^*(x_e,r_e)\rc$ is a $(\delta, \cc, s)$-flat neck region regarding some $\mathcal F^0(\Gamma)$ and for each $b'$, $B^*(x_{b'},r_{b'})$ is a $b$-ball with $r_{b'} \le \cc^5 s;$
		
		\item any point $x$ in $\displaystyle \NNN':= \lc \NNN  \bigcup \bigcup_{ b'} B^*(x_{ b'},r_{ b'}) \rc \setminus \lc \bigcup_{e'} B^*(x_{e'},r_{e'}) \rc$ is $(2, \delta, \delta d_Z(x, \CCC'))$-symmetric, where
			\begin{equation*}
\CCC':=\CCC_0 \bigcup_{b, d, e} \{x_b, x_d, x_e\};
	\end{equation*}
		
\item the following content estimates hold\emph{:}
						\begin{align*}
\sum_{x_b \in B^*(z, 3s/2)} r_b+\sum_{x_d \in B^*(z, 3s/2)}r_d+\sum_{x_e \in B^*(z, 3s/2)} r_e \le s \quad \text{and} \quad \sum_{x_{e'} \in B^*(z, 3s/2)} r_{e'} \leq C(Y, \cc) s.
	\end{align*}		
	
\item for any $r \in (0, 1]$, we have
\begin{align*}
\abs{B^*_{rs} \lc \CCC' \bigcap B^*(z, 3s/2) \rc} \le  C(Y, \eta) r^{5.1} s^6;
	\end{align*}		
	in particular, $\HHH^{1}\lc \CCC' \bigcap B^*(z, 3s/2) \rc=0$.
	
	\end{enumerate}		
\end{prop}

\begin{proof}
Without loss of generality, we assume $s=1$. The construction of the flat neck region $\NNN$ is analogous to that in Proposition \ref{cballdecomposition2}, and we only sketch the proof. As before, we assume $\delta \ll \cc$ and set $\gamma:=\cc^5$.

Since $B^*(z,1)$ is a type-II $c$-ball, by the definition of $t_s$, there exists $q\in \PP_{\zeta,1}(z) \cap B^*(z, 1)$ such that $q$ is $(\ep, 1)$-close to $\mathcal F^a(\Gamma)$ with respect to $\LL_{q, 1}$, where $a=t_s(z,1)-\t(q)$. We define
\begin{align*}
			\tilde\LL_{z,1}:=\LL_{q, 1} \bigcap B^*(z,2.2) \bigcap Z_{(-\infty, t_p(z, 1)]}.
	\end{align*}			

We set $L^1=\tilde \LL_{z, 1}$ and choose a maximal $2\cc^2 \gamma$-separated set $\{x_{f^1}\} \subset L^1$. In particular, $\{B^*(x_{f^1}, \cc^2 \gamma)\}$ are pairwise disjoint and
\begin{align*}
L^1 \subset \bigcup_{f^1}B^*(x_{f^1}, 2 \cc^2\gamma).
	\end{align*}	

By a limiting argument, if $\cc \le \cc(\eta)$, $\ep \le \ep(Y,  \eta, \cc)$ and $\zeta \le \zeta(Y, \ep, \eta, \cc, \beta)$, then any ball $B^*(x_{f^1}, \gamma)$ satisfying $\t(x_{f^1})<t_p(z, 1)-2 \eta^{-1} \gamma^2$ must be a $b$-ball.

According to the type of $B^*(x_{f^1},\gamma)$, we can write
		\begin{align*}
L^1 \subset \bigcup_{b^1} B^*(x_{b^1}, \gamma)\bigcup\bigcup_{c^1} B^*(x_{c^1}, \gamma)\bigcup\bigcup_{d^1} B^*(x_{d^1}, \gamma)\bigcup\bigcup_{e^1} B^*(x_{e^1}, \gamma).
		\end{align*}	

Then we set 
		\begin{align*}
			\NNN^1:= B^*(z,2)\setminus \lc\bigcup_{b^1} B^*(x_{b^1}, \gamma)\bigcup\bigcup_{c^1} B^*(x_{c^1}, \gamma)\bigcup\bigcup_{d^1} B^*(x_{d^1}, \gamma)\bigcup\bigcup_{e^1} B^*(x_{e^1}, \gamma)\rc.
		\end{align*}
Moreover, we define $\CCC^1:= \{x_{b^1},x_{c^1},x_{d^1},x_{e^1}\}$ with radius $r_x:= \gamma$. As in the proof of Proposition \ref{cballdecomposition2}, $\NNN^1$ is a $(\delta, \cc, 1)$-flat neck region regarding $\mathcal F^0(\Gamma)$.

We can reindex the $b$-balls $\bigcup_{b^1}B^*(x_{b^1},\gamma)$ as 
		\begin{align*}
			\bigcup_{b^{1,s}}B^*(x_{b^{1,s}},\gamma)\bigcup \bigcup_{b^{1,q}}B^*(x_{b^{1,q}},\gamma)
		\end{align*}
in the following way: if a $b$-ball $B^*(x_{b^1},\gamma)$ satisfies $\t(x_{b^1}) \le t_p(z, 1)-2 \eta^{-1} \gamma^2$, then it is designated as $B^*(x_{b^{1,s}},\gamma)$; otherwise, it is designated as $B^*(x_{b^{1,q}},\gamma)$. 

By another limiting argument, if $\cc \le \cc(\eta)$, $\ep \le \ep(Y, \delta, \eta, \cc)$ and $\zeta \le \zeta(Y, \ep, \eta, \cc, \beta)$, then there exists $\xi=\xi(Y)>0$ such that any point
		\begin{align*}
y\in Z_{(-\infty, t_p(z,1)+2\xi^2\gamma^2)}\bigcap B^*(z,2)
		\end{align*}	
 with $d_Z(y, \CCC^{1, \prime})=s\geq \gamma$ is $(2,\delta, \delta  s)$-symmetric for some constant $c(Y)>0$, where 
 		\begin{align*}
\CCC^{1, \prime}:=\CCC^1 \setminus \bigcup_{b^{1, s}}\{x_{b^{1, s}}\}.
		\end{align*}

 Next, we choose a cover $\{B^*(x_{e^{(1)}},\xi\gamma)\}$ of 
		\begin{align*}
\lc Z_{(t_p(z,1)+2 \xi^2\gamma^2,\infty)}\setminus \bigcup_{f^1}B^*(x_{f^1},\gamma)\rc\bigcap B^*(z,2)
		\end{align*}	
		such that $\{B^*(x_{e^{(1)}}, \xi\gamma/2)\}$ are pairwise disjoint. By \eqref{eq:timea1}, all these balls are $e$-balls, and the total number of $\{B^*(x_{e^{(1)}}, \xi\gamma)\}$ is bounded by $C(Y) \gamma^{-6}$. 
		
We define
  		\begin{align*}
\NNN^{1, \prime}=\lc \NNN^1 \bigcup \bigcup_{b^{1, s}} B^*(x_{b^{1,s}},\gamma) \rc \setminus \lc\bigcup_{e^{(1)}}B^*(x_{e^{(1)}}, \xi\gamma)\rc.
		\end{align*}	
 Then it is clear from our construction that any point $y$ in $\NNN^{1, \prime}$ is $(2,\delta, \delta d_Z(y, \CCC^{1, \prime}))$-symmetric.
  
It is clear that if $x_{f^1} \in \{x_{b^{1,q}}, x_{c^1}, x_{d^1}, x_{e^1}\}$, then $\t(x_{f^1}) \ge  t_p(z, 1)-2\eta^{-1}\gamma^2$. Thus, we have the content estimate
 		\begin{align} \label{eq:extraconten1}
\sum_{b^{1,q}} r_{b^{1, q}}+\sum_{c^1} r_{c^1}+\sum_{d^1} r_{d^1}+\sum_{e^1} r_{e^1} \le C(Y) \gamma \eta^{-1}\cc^{-4}
		\end{align}	
		and
	 		\begin{align} \label{eq:extraconten1a}
\sum_{e^{(1)}} r_{e^{(1)}} \le C(Y) \gamma^{-5}.
		\end{align}		
		
Similar to the proof of Proposition \ref{cballdecomposition2}, after repeating the above decomposition for $l$ times, we obtain \begin{align*}
			\NNN^l:= B^*(z,2)\setminus\lc\bigcup
			_{c^l} B^*(x_{c^l}, \gamma^l)\bigcup\bigcup_{1\leq j\leq l} \lc \bigcup_{b^j} B^*(x_{b^j}, \gamma^j)\bigcup\bigcup_{d^j} B^*(x_{d^j}, \gamma^j)\bigcup\bigcup_{e^j}B^*(x_{e^j}, \gamma^j) \rc \rc,
		\end{align*}
together with the corresponding collection$\bigcup_{1\leq j\leq l}\bigcup_{e^{(j)}}B^*(x_{e^{(j)}}, \xi\gamma^j)$. 	

As before, $\NNN^l$ is a $(\delta,\cc, 1)$-flat neck region regarding $\mathcal F^0(\Gamma)$, with center $\CCC^l$ consisting of the centers of the balls and radius function given by their associated radii.

Moreover, we define
  		\begin{align*}
\NNN^{l, \prime}=\lc \NNN^l \bigcup \lc \bigcup_{1 \le j \le l} \bigcup_{b^{j, s}} B^*(x_{b^{j,s}},\gamma^j) \rc \rc \setminus \lc \bigcup_{1 \le j \le l}\bigcup_{e^{(j)}}B^*(x_{e^{(j)}}, \xi\gamma^j)\rc.
		\end{align*}		
Then any point $y$ in $\NNN^{l, \prime}$ is $(2,\delta, \delta d_Z(y, \CCC^{l, \prime}))$-symmetric, where
		  		\begin{align*}
\CCC^{l, \prime}=\CCC^l \setminus \lc \bigcup_{1 \le j \le l} \bigcup_{b^{j,s}} \{x_{b^{j, s}}\} \rc.
		\end{align*}		
		
Similar to \eqref{eq:extraconten1} and \eqref{eq:extraconten1a}, we obtain the following facts that will be used below:
\begin{itemize}
\item For each $j \in [1, l]$, the total number of $\{x_{b^{j, q}}, x_{c^j}, x_{d^j}, x_{e^j}\}$ is at most $\lc C(Y) \eta^{-1}\cc^{-4} \rc^{j}$.
	
\item The total number of $\{x_{e^{(j)}}\}$ is at most $C(Y) \gamma^{-6}\lc C(Y) \eta^{-1} \cc^{-4}  \rc^{j-1}$.
\end{itemize}		
		
Consequently, we compute
\begin{align} \label{eq:extraconten2}
\sum_{b^{j,q}} r_{b^{j, q}}+\sum_{d^j} r_{d^j}+\sum_{e^j} r_{e^j} \le \lc C(Y) \eta^{-1} \cc^{-4} \gamma \rc^j=\lc C(Y)\eta^{-1} \gamma^{1/5} \rc^j 
		\end{align}
	where we used the fact that $\cc=\gamma^{1/5}$. Also, we obtain
\begin{align} \label{eq:extraconten3}
\sum_{e^{(j)}} r_{e^{(j)}} \le C(Y) \gamma^{-5} \lc C(Y) \eta^{-1}\cc^{-4} \gamma \rc^{j-1} =C(Y) \gamma^{-5}\lc C(Y) \eta^{-1}\gamma^{1/5}  \rc^{j-1}.
		\end{align}	
	
Set $\mathcal G^l:=\bigcup_{x_{c^l}} \{x_{c^l}\}$. It is clear from our construction that $\mathcal G^{l+1} \subset B^*_{2\gamma^l}(\mathcal G^l)$. Then we denote the Hausdorff limit of $\mathcal G^l$ by $\CCC_0$. We define
		\begin{align*}
			\NNN:= B^*(z,2)\setminus\lc\CCC_0\bigcup\bigcup_{1\leq j<\infty} \lc \bigcup_{b^j} B^*(x_{b^j}, \gamma^j)\bigcup\bigcup_{d^j} B^*(x_{d^j}, \gamma^j)\bigcup\bigcup_{e^j}B^*(x_{e^j}, \gamma^j) \rc \rc.
		\end{align*}
		
	From our construction, $\NNN$ is a $(\delta,\cc, 1)$-flat neck region regarding $\mathcal F^0(\Gamma)$ with centers given by $\CCC_0$ together with the centers of the associated balls, and the radius function chosen to be the corresponding radii of those balls. Moreover, we define
  		\begin{align*}
\NNN'=\lc \NNN \bigcup \lc \bigcup_{1 \le j <\infty} \bigcup_{b^{j, s}} B^*(x_{b^{j,s}},\gamma^j) \rc \rc\setminus \lc \bigcup_{1 \le j <\infty}\bigcup_{e^{(j)}}B^*(x_{e^{(j)}},\xi\gamma^j)\rc.
		\end{align*}		
Note that each point $x$ in $\NNN'$ is $(2, \delta,  \delta  d_Z(x, \CCC'))$-symmetric, where
		  		\begin{align*}
\CCC'=\CCC \setminus \lc \bigcup_{1 \le j <\infty}\bigcup_{b^{j, s}} \{x_{b^{j, s}}\} \rc.
		\end{align*}				
		
We rewrite the above decompositions as:
		\begin{equation*}
			B^*(z,2)\subset \left(\CCC_0\bigcup\NNN' \right)\bigcup\bigcup_b B^*(x_b,r_b)\bigcup \bigcup_d B^*(x_d,r_d)\bigcup \bigcup_e B^*(x_e,r_e)\bigcup \bigcup_{e'}B^*(x_{e'}, r_{e'})
		\end{equation*}
and
\begin{align*}
\NNN'= \lc \NNN  \bigcup \bigcup_{ b'} B^*(x_{ b'},r_{ b'}) \rc \setminus \lc \bigcup_{e'} B^*(x_{e'},r_{e'}) \rc,
\end{align*}
where 
		  		\begin{align*}
\bigcup_{b}B^*(x_{b},r_{b}):=&\bigcup_{1\leq j < \infty}\bigcup_{b^{j, q}}B^*(x_{b^{j, q}},\gamma^j),\\
\bigcup_{b'}B^*(x_{b'},r_{b'}):=&\bigcup_{1\leq j < \infty}\bigcup_{b^{j, s}}B^*(x_{b^{j, s}},\gamma^j),	  \\		
\bigcup_{e'}B^*(x_{e'},r_{e'}):=&\bigcup_{1\leq j < \infty}\bigcup_{e^{(j)}}B^*(x_{e^{(j)}},\xi \gamma^j).
		\end{align*}	

From \eqref{eq:extraconten2}, we compute
\begin{align*}
&\sum_{x_b \in B^*(z, 3/2)} r_b+\sum_{x_d \in B^*(z, 3/2)}r_d+\sum_{x_e \in B^*(z, 3/2)} r_e \le \sum_{1 \le j <\infty} \lc C(Y) \eta^{-1} \gamma^{1/5} \rc^j  \le 1,
	\end{align*}		
where we choose $\cc \le \cc(Y, \eta)$ so that $C(Y) \eta^{-1} \gamma^{1/5}  \le 1/2$.

Similarly, by \eqref{eq:extraconten3}, we obtain
\begin{align*}
\sum_{x_{e'} \in B^*(z, 3/2)} r_{e'} \le \sum_{1 \le j <\infty} C(Y) \gamma^{-5}\lc C(Y) \eta^{-1} \gamma^{1/5}  \rc^{j-1} \le C(Y, \cc).
	\end{align*}	

To complete the proof, it remains to show (vii). Indeed, since the total number of elements in $\CCC^{l,\prime}$ is at most $\lc C(Y) \eta^{-1} \cc^{-4} \rc^l$, we conclude that
\begin{align*}
\abs{B^*_{\gamma^l} \lc \CCC' \bigcap B^*(z, 3/2) \rc} \le \abs{B^*_{4\gamma^l}\lc \CCC^{l,\prime} \bigcap B^*(z, 1.6) \rc} \le C(Y) \lc C(Y) \eta^{-1}\cc^{-4} \rc^{l} \gamma^{6l} \le C(Y) \lc C(Y) \eta^{-1} \gamma^{1/5} \rc^l \gamma^{5l}.
	\end{align*}	
We choose $\cc \le \cc(Y, \eta)$ such that 	$C(Y) \eta^{-1} \gamma^{1/5} \le \gamma^{1/10}$. From this, the conclusion easily follows.

In sum, the proof is complete.
\end{proof}

Based on Propositions \ref{cballdecomposition2} and \ref{cballdecomposition3}, it is natural to introduce the following concept.

\begin{defn} \label{def:modi}\index{modified $\delta$-region associated with $\NNN$}
The set $\NNN'$ obtained in Proposition \ref{cballdecomposition2} or Proposition \ref{cballdecomposition3} is called the \textbf{modified $\delta$-region associated with} $\NNN$. 
\begin{itemize}
\item In the case $k=2$, we have $\NNN' \subset \NNN$, and any point $x \in \NNN'$ is $(3, \delta, \delta d_Z(x, \CCC))$-symmetric.
	
\item In the case $k=1$, we have $\CCC' \subset \CCC$, and any point $x \in \NNN'$ is $(2, \delta, \delta d_Z(x, \CCC'))$-symmetric. Here, $\CCC'$ is called the \textbf{modified center}.
\end{itemize}
\end{defn}

\begin{prop}[Decomposition of $d$-balls]\label{dballdecomposition}
	For $z\in B^*(z_0, 4 r_0)$ and $ s \le r_0-d_Z(z, z_0)/4$, let $B^*(z,s)$ be a $d$-ball with $k \in \{1, 2\}$. If $\beta \le \beta(Y)$, we have the decomposition\emph{:}
	\begin{equation*}
B^*(z,s)\subset S_d^k\bigcup\bigcup_b B^*(x_b,r_b)\bigcup\bigcup_c B^*(x_c,r_c)\bigcup\bigcup_e B^*(x_e,r_e)
	\end{equation*}
	satisfying
	\begin{enumerate}[label=\textnormal{(\roman{*})}]
		\item for each $b$, $B^*(x_b,r_b)$ is a $b$-ball with $r_b \le \beta s;$
\item for each $c$, $B^*(x_c,r_c)$ is a $c$-ball with $r_c \le \beta s;$
		\item for each $e$, $B^*(x_e,r_e)$ is an $e$-ball with $r_e \le \beta s;$
		\item $S_d^k \subset \MS$ with $\HHH^{k}(S_d^k)=0;$
		\item the following content estimates hold:
		\begin{align*}
			\sum_b r_b^{k}+\sum_e r_e^{k}\leq C(Y) \beta^{k-6} s^k \quad \text{and} \quad \sum_c r_c^{k}\leq C(Y)\beta s^k.
		\end{align*}
	\end{enumerate}
\end{prop}
\begin{proof}
Without loss of generality, we assume $s=1$. 

We first choose a maximal cover $\{B^*(x_{f^1}, \beta)\}$ of $B^*(z,1)$ with $x_{f^1} \in B^*(z, 1)$ so that $\{B^*(x_{f^1},\beta/2)\}$ pairwise disjoint. According to their types, we can write
		\begin{align} \label{eq:decomd001}
B^*(z,1)\subset \bigcup_{b^1} B^*(x_{b^1},\beta)\bigcup\bigcup_{c^1} B^*(x_{c^1},\beta)\bigcup\bigcup_{d^1} B^*(x_{d^1},\beta)\bigcup\bigcup_{e^1} B^*(x_{e^1},\beta).
		\end{align}
By Proposition \ref{prop:volumebound}, we have the following content estimate:
		\begin{align*}
	\sum_{b^1} \beta^{k}+\sum_{c^1} \beta^{k}+\sum_{d^1} \beta^{k}+\sum_{e^1} \beta^{k}\leq C(Y) \beta^{k-6}=:C_1.
		\end{align*}
	
	On the other hand, since $B^*(z,1)$ is a $d$-ball, we have $\abs{B^*_{3 \beta}\lc\mathcal{P}_{\zeta,1}(z) \rc} \leq D \beta^{7-k}$ by \eqref{eq:extravol2}. Then for $c$-ball $B^*(x_{c^1},\beta)$ and $d$-ball $B^*(x_{d^1},\beta)$, we know that neither $\mathcal{P}_{\zeta,\beta}(x_{c^1})$ nor $\mathcal{P}_{\zeta,\beta}(x_{d^1})$ is empty. Hence, by definition, we have
	\begin{align*}
		\mathcal{P}_{\zeta,\beta}(x_{c^1})\bigcup\mathcal{P}_{\zeta,\beta}(x_{d^1})\subset \mathcal{P}_{\zeta,1}(z)
	\end{align*}
which implies
	\begin{align*}
	B^*(x_{c^1}, \beta/2) \bigcup B^*(x_{d^1}, \beta/2) \subset B^*_{3 \beta}\lc\mathcal{P}_{\zeta,1}(z) \rc.
	\end{align*}
	Therefore, if we denote the number of $c$-balls and $d$-balls in \eqref{eq:decomd001} by $N_c$ and $N_d$, respectively, it follows from Proposition \ref{prop:volumebound} that
			\begin{align*}
			(N_c+N_d)\beta^{6}\leq C(Y) \beta^{7-k},			
		\end{align*}
		which implies that $N_c+N_d\leq C(Y) \beta^{1-k}$.	Thus, we have the following better content estimate for $c$-balls and $d$-balls:
	\begin{align}\label{eq:decomd002}
		\sum_{c^1}\beta^{k}+\sum_{d^1} \beta^{k}\leq C(Y) \beta.
	\end{align}
	
For each $d$-ball $B^*(x_{d^1},\beta)$, we can repeat the above decomposition to get
			\begin{align*}
\bigcup_{d^1}B^*(x_{d^1},\beta)\subset\bigcup_{b^2} B^*(x_{b^2},\beta^2)\bigcup\bigcup_{c^2} B^*(x_{c^2},\beta^2)\bigcup\bigcup_{d^2} B^*(x_{d^2},\beta^2)\bigcup\bigcup_{e^2} B^*(x_{e^2},\beta^2).
		\end{align*}
		
By Proposition \ref{prop:volumebound} and \eqref{eq:decomd002}, we obtain
	\begin{align*}
		\sum_{d^2}(\beta^2)^{k} \leq \big(C(Y) \beta \big)^2, \quad \sum_{b^2}(\beta^2)^{k}+\sum_{e^2}(\beta^2)^{k} \leq C_1 C(Y) \beta, \quad \text{and} \quad	\sum_{c^2}(\beta^2)^{k}\leq (C(Y)\beta)^2.
	\end{align*}

We write the decomposition of $B^*(z,1)$ as		
					\begin{align*}
B^*(z,1)\subset\bigcup_{d^2}B^*(x_{d^2},\beta^2)\bigcup_{1\leq j\leq 2}\lc\bigcup_{b^j} B^*(x_{b^j},\beta^j)\bigcup\bigcup_{c^j} B^*(x_{c^j},\beta^j)\bigcup\bigcup_{e^j} B^*(x_{e^j},\beta^j)\rc
		\end{align*}
with the content estimates:
	\begin{align*}
		\sum_{d^2}(\beta^2)^{k}&\leq \big(C(Y)\beta\big)^2,\\
		\sum_{1 \le j \le 2} \lc \sum_{b^j}(\beta^j)^{k}+\sum_{e^j}(\beta^j)^{k} \rc&\leq C_1+C_1 C(Y) \beta,\\
		\sum_{1 \le j \le 2}\sum_{c^j}(\beta^j)^{k}&\leq C(Y)\beta+(C(Y)\beta)^2.
	\end{align*}

	After repeating the above decomposition for $l$ steps, we obtain
						\begin{align*}
B^*(z,1)\subset\bigcup_{d^l}B^*(x_{d^l},\beta^l)\bigcup_{1\leq j\leq l}\lc\bigcup_{b^j} B^*(x_{b^j},\beta^j)\bigcup\bigcup_{c^j} B^*(x_{c^j},\beta^j)\bigcup\bigcup_{e^j} B^*(x_{e^j},\beta^j)\rc,
		\end{align*}
	with the content estimates:
	\begin{align}
		\sum_{d^l}(\beta^l)^{k}&\leq \big(C(Y)\beta\big)^l, \label{eq:decomd003a}\\
		\sum_{1 \le j \le l} \lc \sum_{b^j}(\beta^j)^{k}+\sum_{e^j}(\beta^j)^{k} \rc&\leq  C_1 \sum_{0 \le j \le l-1} (C(Y)\beta)^j, \label{eq:decomd003b}\\
		\sum_{1 \le j \le l}\sum_{c^j}(\beta^j)^{k}&\leq \sum_{1 \le j \le l} (C(Y) \beta)^j. \label{eq:decomd003c}
	\end{align}
	We assume that $\beta$ is small enough so that $C(Y)\beta<1/2$. From \eqref{eq:decomd003b} and \eqref{eq:decomd003c}, we have for any $l \ge 1$,
\begin{align*}
		\sum_{1 \le j \le l} \lc \sum_{b^j}(\beta^j)^{k}+\sum_{e^j}(\beta^j)^{k} \rc\leq  C_1 C(Y) \quad \text{and} \quad	\sum_{1 \le j \le l}\sum_{c^j}(\beta^j)^{k} \leq C(Y) \beta. 
		\end{align*}
		
	Set $S^{k, l}_d:=\bigcup_{d^l}\{x_{d^l}\}$. By our construction, we have $S^{k, l+1}_d \subset B^*_{\beta^l}\lc S^{k, l}_d \rc$. Then we denote the Hausdorff limit of $S^{k, l}_d$ by $S^k_d$. From \eqref{eq:decomd003a}, we obtain
		\begin{align*}
		\abs{B^*_{\beta^l} \lc S^{k, l}_d  \rc} \le \sum_{d^l} \abs{B^*(x_{d^l}, \beta^l)} \le C(Y) \sum_{d^l}\beta^{6l}\leq C(Y)\left(C(Y) \beta \right)^l\beta^{(6-k)l}.
	\end{align*}
	Thus, we obtain $\HHH^{k}(S^k_d)=0$. Moreover, since any $x_{d^l}$ is not $(k+1, \eta, \beta^l)$-symmetric, it is easy to show that $S^k_d \subset \MS$. 
	
	In sum, the proof is complete.
\end{proof}

\begin{prop}[Inductive decomposition]\label{inductivedecomposition}
Given $k \in \{1, 2\}$, $z\in B^*(z_0, 4r_0)$ and $ s \le r_0-d_Z(z, z_0)/4$, for any constants $\delta>0$, $\eta>0$, $\cc \in (0, 10^{-40})$ and $\beta>0$, if $\cc \le \cc(Y, \eta)$ \emph{(}for $k=2$ it suffices to assume $\cc \le \cc(Y)$\emph{)}, $\beta\le \beta(Y, \cc)$ and $\zeta \le \zeta( Y, \delta, \eta, \cc, \beta)$, we have the decomposition:
	\begin{equation*}
		B^*(z,s)\subset \bigcup_a \left((\CCC_{0,a}\bigcup\NNN'_a)\bigcap B^*(x_a,r_a)\right)\bigcup \bigcup_b B^*(x_b,r_b)\bigcup\bigcup_e B^*(x_e,r_e)\bigcup S_0,
	\end{equation*}
	satisfying
	\begin{enumerate}[label=\textnormal{(\roman{*})}]
		\item for all $a$, $\NNN_a\subset B^*(x_a,2r_a)$ is either a $(k, \delta, \cc, r_a)$-quotient cylindrical neck region or a $(\delta, \cc, r_a)$-flat neck region with center $\CCC_a=\CCC_{a, 0} \bigcup \CCC_{a, +}$ \emph{(}in the former case, we set $\NNN_a'=\NNN_a;$ in the latter, $\NNN_a'$ is the modified $\delta$-region associated with $\NNN_a$\emph{);}		
		\item for all $b$, $B^*(x_b,r_b)$ is a $b$-ball\emph{;}
		\item for all $e$, $B^*(x_e,r_e)$ is an $e$-ball\emph{;}
		\item $S_0\subset \MS$ with $\HHH^{k}(S_0)=0;$
		\item the following content estimate holds:
		\begin{align*}
			\sum_a \lc r_a^{k}+\HHH^k \lc \CCC_{0, a} \bigcap B^*(x_a,r_a) \rc \rc+\sum_b r_b^{k}+\sum_e r_e^{k}\leq C(Y, \cc, \beta) s^k.
		\end{align*}
	\end{enumerate}
\end{prop}
\begin{proof}
This follows from the iterative decompositions of $c$-balls and $d$-balls. Without loss of generality, we assume $s=1$.
	
	We consider $B^*(z,1)$ by types. If $B^*(z,1)$ is a $b$-ball or an $e$-ball, then we are done. So we only need to consider the remaining two cases in the following argument. Assume first that $B^*(z,1)$ is a $d$-ball. By Proposition \ref{dballdecomposition}, we can write
	\begin{align*}
		B^*(z,1)\subset \bigcup_{b'} B^*(x_{b'},r_{b'})\bigcup \bigcup_{c'} B^*(x_{c'},r_{c'})\bigcup \bigcup_{e'} B^*(x_{e'},r_{e'})\bigcup \tilde{S}^1,
	\end{align*}
	with $\tilde{S}^1\subset \MS$ satisfying $\HHH^k(\tilde{S}^1)=0$ and the following content estimates:
	\begin{align*}
		\sum_{b'} r_{b'}^{k}+\sum_{e'} r_{e'}^{k}\leq C(Y , \beta) \quad \text{and} \quad	\sum_{c'} r_{c'}^{k}&\leq C(Y) \beta.
	\end{align*}
	
For each $c$-ball $B^*(x_{c'},r_{c'})$, we can apply Proposition \ref{cballdecomposition1}, \ref{cballdecomposition2}, or \ref{cballdecomposition3} to obtain
	\begin{align*}
		B^*(x_{c'},r_{c'})\subset \lc (\CCC_{0, c'}\bigcup\NNN_{c'}')\bigcap B^*(x_{c'},r_{c'})\rc \bigcup\bigcup_{b''} B^*(x_{b''},r_{b''})\bigcup\bigcup_{d''} B^*(x_{d''},r_{d''})\bigcup\bigcup_{e''} B^*(x_{e''},r_{e''}),
	\end{align*} 
together with the content estimate:
	\begin{align*}
\HHH^{k}\lc\CCC_{0, c'} \bigcap B^*(x_{c'},r_{c'}) \rc+\sum_{b''} r_{b''}^{k}+\sum_{d''} r_{d''}^{k}+\sum_{e''}r_{e''}^{k}\leq C(Y, \cc)r_{c'}^k.
	\end{align*}
	Here, $\NNN_{c'}'=\NNN_{c'}$ if $\NNN_{c'} \subset B^*(x_{c'}, 2r_{c'})$ is a quotient cylindrical neck region, while $\NNN_{c'}'$ denotes the modified $\delta$-region associated with $\NNN_{c'}$ if $\NNN_{c'}$ is a flat neck region.

		Combining the above two decompositions and according to the type of these balls, we write
	\begin{align*}
		B^*(z,1)\subset &\bigcup_{a^1}\left((\CCC_{0,a^1}\bigcup \NNN'_{a^1})\bigcap B^*(x_{a^1}, r_{a^1})\right)\bigcup \bigcup_{b^1}B^*(x_{b^1},r_{b^1})\bigcup \bigcup_{d^1}B^*(x_{d^1},r_{d^1})\\
		&\bigcup \bigcup_{e^1}B^*(x_{e^1},r_{e^1})\bigcup \tilde{S}^1
	\end{align*}
	with $\tilde{S}^1\subset \MS$ satisfying $\HHH^k(\tilde{S}^1)=0$ and the following content estimate:
	\begin{align*}
		&\sum_{b^1}r_{b^1}^{k}+\sum_{e^1}r_{e^1}^{k} \leq C(Y, \beta)+C(Y,\cc)\beta=:C_1,\\
		&\sum_{a^1}\lc r_{a^1}^{k}+\HHH^{k}\lc\CCC_{0,a^1} \bigcap B^*(x_{a^1}, r_{a^1})\rc\rc+\sum_{d^1}r_{d^1}^{k}\leq C(Y, \cc) \beta.
	\end{align*}
	
	Now we apply the above decomposition to each $d$-ball $B^*(x_{d^1},r_{d^1})$ to get the second decomposition
	\begin{align*}
		B^*(z,1)\subset &\bigcup_{1\leq j\leq 2}\lc\bigcup_{a^j}\big((\CCC_{0,a^j}\bigcup \NNN'_{a^j})\bigcap B^*(x_{a^j}, r_{a^j})\big)\bigcup \bigcup_{b^j}B^*(x_{b^j},r_{b^j})\bigcup \bigcup_{e^j}B^*(x_{e^j},r_{e^j})\bigcup \tilde{S}^j\rc\\
		&\bigcup \bigcup_{d^2}B^*(x_{d^2},r_{d^2})
	\end{align*}
	with $\bigcup_{1\leq j\leq 2} \tilde{S}^j\subset \MS$ satisfying $\HHH^k(\bigcup_{1\leq j\leq 2}\tilde{S}^j)=0$ and the following content estimate:
	\begin{align*}
	\sum_{1\leq j\leq 2} \sum_{a^j}\lc r_{a^j}^{k}+\HHH^{k}\lc \CCC_{0,a^j} \bigcap B^*(x_{a^j}, r_{a^j}) \rc\rc &\le \sum_{1\le j \le 2} \lc C(Y, \cc) \beta \rc^j,\\
	\sum_{1\leq j\leq 2}\lc\sum_{b^j}r_{b^j}^{k}+\sum_{e^j}r_{e^j}^{k} \rc&\leq C_1 \sum_{0 \le j \le 1}\lc C(Y, \cc) \beta \rc^j,\\
		\sum_{d^2}r_{d^2}^{k}&\leq \left(C(Y, \cc) \beta \right)^2.
	\end{align*}
	
	Repeating this decomposition $l$ times, we get
	\begin{align*}
		B^*(z,1)\subset &\bigcup_{1\leq j\leq l}\lc\bigcup_{a^j}\big((\CCC_{0,a^j}\bigcup \NNN'_{a^j})\bigcap B^*(x_{a^j}, r_{a^j})\big)\bigcup \bigcup_{b^j}B^*(x_{b^j},r_{b^j})\bigcup \bigcup_{e^j}B^*(x_{e^j},r_{e^j})\bigcup \tilde{S}^j\rc\\
		&\bigcup \bigcup_{d^l}B^*(x_{d^l},r_{d^l})
	\end{align*}
	with $\bigcup_{1\leq j\leq l} \tilde{S}^j \subset\MS$ satisfying $\HHH^k(\bigcup_{1\leq j\leq l}\tilde{S}^j)=0$ and the following content estimate:
	\begin{align}
	\sum_{1\leq j\leq l} \sum_{a^j} \lc r_{a^j}^{k}+\HHH^{k}\lc\CCC_{0,a^j} \bigcap B^*(x_{a^j}, r_{a^j}) \rc\rc &\le \sum_{1\le j \le l} \lc C(Y, \cc) \beta \rc^j, \label{dcontentinduca}\\
	\sum_{1\leq j\leq l}\lc\sum_{b^j}r_{b^j}^{k}+\sum_{e^j}r_{e^j}^{k} \rc&\leq C_1 \sum_{0 \le j \le l-1}\lc C(Y, \cc) \beta \rc^j, \label{dcontentinducb}\\
		\sum_{d^l}r_{d^l}^{k}&\leq \lc C(Y, \cc) \beta \rc^l \label{dcontentinducc}.
	\end{align}
	
	Choose $\beta$ small such that $C(Y, \cc) \beta  \leq 1/2$. Set $\tilde{S}_{d^l}=\bigcup_{d^l}\{x_{d^l}\}$ and $(\tilde{S}_{d^l})_{2}:=\bigcup_{d^l}B^*(x_{d^l}, 2 r_{d^l})$. By construction, $(\tilde{S}_{d^{l+1}})_{2}\subset(\tilde{S}_{d^l})_{2}$. Set 
	\begin{align*}
		\tilde{S}_d:=\bigcap_{l=1}^\infty (\tilde{S}_{d^l})_{2}.
	\end{align*}
Thus, by construction, we have $\tilde{S}_d\subset\MS$. Define
	\begin{align*}
S_0:=\tilde{S}_d\bigcup\bigcup_{l=1}^\infty\tilde{S}^l.
	\end{align*}
 It is clear from \eqref{dcontentinducc} that $S_0 \subset \MS$ with $\HHH^k(S_0)=0$. Moreover, the content estimate follows from \eqref{dcontentinduca} and \eqref{dcontentinducb}.

	If $B^*(z,1)$ is a $c$-ball, we first apply Proposition \ref{cballdecomposition1}, \ref{cballdecomposition2}, or \ref{cballdecomposition3} to obtain a decomposition involving only $b$-balls, $d$-balls, $e$-balls and neck regions. Then, applying the above decomposition procedure recursively to each $d$-ball completes the proof.
\end{proof}

Now, we are ready to prove Theorem \ref{neckdecomgeneral}.

\begin{proof}[Proof of Theorem \ref{neckdecomgeneral}]

For given constants $\delta$ and $\eta$, we fix the constants as follows.
\begin{itemize}
\item If $k=1$, set $\cc=\cc(Y, \eta)$, $\beta=\beta(Y, \cc)=\beta(Y, \eta)$ and $\zeta \le \zeta( Y, \delta, \eta, \cc, \beta)=\zeta(Y, \delta, \eta)$ so that the conclusions of Proposition \ref{inductivedecomposition} hold.
	
\item If $k=2$, set $\cc=\cc(Y)$, $\beta=\beta(Y, \cc)=\beta(Y)$ and $\zeta \le \zeta( Y, \delta, \eta, \cc, \beta)=\zeta(Y, \delta, \eta)$ so that the conclusions of Proposition \ref{inductivedecomposition} hold.
\end{itemize}	

Moreover, in the following proof, the constant $C=C(Y)$ if $k=2$ and $C=C(Y, \eta)$ if $k=1$.

For simplicity, we consider the function
\begin{align*}
	\bar W_{x, r}(\zeta):=\inf _{y \in B^*(x,4 r)}\widetilde \WW_y(\zeta^{-2} r^2)
\end{align*}
for any $x \in B^*(z_0, 4r_0)$ and $r \le r_0-d_Z(x, z_0)/4$. In particular, $\bar W_{z_0, r_0}(\zeta)=\bar W$ in \eqref{infNash}.

	Choose a maximal cover $\{B^*(x_i,\zeta r_0)\}$ of $B^*(z_0, r_0)$ such that $x_i \in B^*(z_0, r_0)$ and $\{B^*(x_i,\zeta r_0/2)\}$ are pairwise disjoint. We set $r_1:=\zeta r_0$. Note that
	\begin{align}\label{decom1}
		\bar W_{x_i,r_1}(\zeta)=\inf_{y\in B^*(x_i,4\zeta r_0)}\widetilde \WW_y(r_0^2) \geq \inf_{y\in B^*(z_0,4r_0)}\widetilde \WW_y(r_0^2)=\bar W_{z_0,r_0}(1).
	\end{align}
	
	Apply Proposition \ref{inductivedecomposition} to each $B^*(x_i, r_1)$. In this context, we consider a new quantity $\bar W_1:=\bar W_{x_i,r_1}(\zeta)$, and define the updated pinching sets and localized pinching set as
	\begin{equation*}
		\PP'_{\zeta,r}:= \{y\in B^*(x_i, 4 r_1) \mid \widetilde \WW_y(\zeta^2 r^2)\leq \bar W_1+\zeta\}.
	\end{equation*}
	and
		\begin{equation*}
	\PP'_{\zeta,r}(x):= \PP'_{\zeta,r} \bigcap B^*(x, 2.5 r).
	\end{equation*}

Accordingly, the definitions of $b$-balls, $c$-balls, $d$-balls and $e$-balls are now made with respect to $\bar W_1$ and the updated pinching set $\PP'_{\zeta,r}(x)$.
	
From Proposition \ref{inductivedecomposition}, we obtain
	\begin{align*}
		B^*(z_0,r_0)\subset \bigcup_{a^1}\lc(\NNN'_{a^1}\bigcup\CCC_{0,a^1})\bigcap B^*(x_{a^1},r_{a^1})\rc\bigcup\bigcup_{b^1} B^*(x_{b^1},r_{b^1})\bigcup\bigcup_{e^1} B^*(x_{e^1}',r_{e^1}')\bigcup S_0^1,
	\end{align*}
	with $r_{a^1},r_{b^1},r_{e^1}'\leq \zeta r_0$, $S_0^1\subset \MS$ and $\HHH^{k}(S_0^1)=0$. Moreover, we have the following content estimates:
			\begin{align*}
\sum_{a^1} (r_{a^1})^k+\sum_{b^1} (r_{b^1})^k+\sum_{e^1}(r_{e^1}')^k+\HHH^k\left(\bigcup_{a^1} \lc \CCC_{0,{a^1}} \bigcap B^*(x_{a^1},r_{a^1}) \rc \right)\leq C r_0^k.
	\end{align*}	
	
	For each $e$-ball, by \eqref{decom1}, we have
	\begin{align}\label{decom3}
		 \bar W_{x'_{e^1},r_{e^1}'}(\zeta^{-1})=\inf_{y\in B^*(x_{e^1}', 4r_{e^1}')}\widetilde \WW_y\left((\zeta r'_{e^1})^2\right)\geq \inf_i\left\{\bar W_{x_i, r_1}(\zeta)\right\}+\zeta\geq \bar W_{z_0,r_0}(1)+\zeta.
	\end{align}
	Moreover, for each $e$-ball $B^*(x_{e^1}', r_{e^1}')$, choose a maximal cover $\{B^*(x_{e^1}'',\zeta r_{e^1}')\}$ with $x_{e^1}'' \in B^*(x_{e^1}', r_{e^1}')$ such that $x_{e^1}''\in B^*(x_{e^1}', r_{e^1}')$ and $\{B^*(x_{e^1}'',\zeta r_{e^1}'/2)\}$ are pairwise disjoint. Set $r_{e^1}'':=\zeta r_{e^1}'$. Therefore by \eqref{decom3}, it follows that
	\begin{align*}
		\bar W_{x_{e^1}'',r_{e^1}''}(1)=\inf_{y\in B^*(x_{e^1}'', 4r''_{e^1})}\widetilde \WW_y\left((\zeta r'_{e^1})^2\right)\geq \bar W_{z_0,r_0}(1)+\zeta.
	\end{align*}
	
	Combining the above two decompositions, we obtain
		\begin{align*}
		B^*(z_0,r_0)\subset \bigcup_{a^1}\lc(\NNN'_{a^1}\bigcup\CCC_{0,a^1})\bigcap B^*(x_{a^1},r_{a^1})\rc\bigcup\bigcup_{b^1} B^*(x_{b^1},r_{b^1})\bigcup\bigcup_{e^1} B^*(x_{e^1},r_{e^1})\bigcup S_0^1,
	\end{align*}
	with $S_0^1\subset \MS$ and $\HHH^{k}(S_0^1)=0$. Moreover, we have the following content estimates:
				\begin{align*}
\sum_{a^1} r_{a^1}^k+\sum_{b^1} r_{b^1}^k+\sum_{e^1} r_{e^1}^k+\HHH^k\left(\bigcup_{a^1}\lc \CCC_{0,{a^1}} \bigcap B^*(x_{a^1},r_{a^1}) \rc\right)\leq C r_0^k.
	\end{align*}	
And for each $e$-ball, $\bar W_{x_{e^1},r_{e^1}}(1)\geq \bar W_{z_0,r_0}(1)+\zeta$.
	
	Repeat the above process for each $e$-ball $B^*(x_{e^1},r_{e^1})$, we get the following second decomposition
		\begin{align*}
		B^*(z_0,r_0)\subset \bigcup_{a^2}\lc(\NNN'_{a^2}\bigcup\CCC_{0,a^2})\bigcap B^*(x_{a^2},r_{a^2})\rc\bigcup\bigcup_{b^2} B^*(x_{b^2},r_{b^2})\bigcup\bigcup_{e^2} B^*(x_{e^2},r_{e^2})\bigcup S_0^2,
	\end{align*}
	with $S_0^2\subset \MS$ and $\HHH^{k}(S_0^2)=0$. Moreover, we have the following content estimates:
\begin{align*}
&\sum_{a^2} r_{a^2}^k+\sum_{b^2} r_{b^2}^k+\sum_{e^2} r_{e^2}^k+\HHH^k\left(\bigcup_{a^2} \lc \CCC_{0,{a^2}} \bigcap B^*(x_{a^2},r_{a^2}) \rc\right)\leq C r_0^k.
	\end{align*}	
Moreover, for each $e$-ball $B^*(x_{e^2},r_{e^2})$, we have $\bar W_{x_{e^2},r_{e^2}}(1)\geq \bar W_{z_0,r_0}(1)+2\zeta$.
	
We apply Proposition \ref{inductivedecomposition} to each $e$-ball $B^*(x_{e^2},r_{e^2})$ to continue the decomposition process. This process must terminate in at most $C(Y)/\zeta$ steps, since our Ricci flow limit space $Z$ is obtained from a sequence of closed Ricci flows in $\MM(n, Y, T)$. In other words, after the final step, no $e$-balls remain. 

Define
	\begin{align*}
		\tilde{S}^{k,\delta,\eta}:=\bigcup_i S_0^i.
	\end{align*}
After re-indexing the balls, we obtain the following decomposition:
	\begin{align*}
		B^*(z_0,r_0)\subset \bigcup_a\big((\NNN'_a\bigcup\CCC_{0,a})\bigcap B^*(x_a,r_a)\big)\bigcup\bigcup_b B^*(x_b,r_b)\bigcup \tilde{S}^{k,\delta,\eta},
	\end{align*}
	with $\tilde{S}^{k,\delta,\eta}\subset\MS$ and $\HHH^{k}(\tilde{S}^{k,\delta,\eta})=0$. Moreover, the following content estimates hold:
	\begin{align*}
\sum_a r_a^{k}+\sum_b r_b^{k}+\HHH^{k}\left(\bigcup_a\lc \CCC_{0, a} \bigcap B^*(x_a,r_a) \rc\right)\leq C r_0^k.
	\end{align*}	

Finally, set
	\begin{align*}
S^{k,\delta,\eta}=\tilde{S}^{k,\delta,\eta}\bigcup \bigcup_a\CCC_{0,a}.
	\end{align*}
This completes the proof.
\end{proof}

\subsection{Rectifiability and volume estimates of singular sets}

Recall that for each stratum $\MS^k$, we have
\begin{align*}
\MS^{k}=\bigcup_{\ep \in (0, 1)} \bigcap_{0<r<\ep } \MS^{\ep,k}_{r, \ep}.
\end{align*}
where $\MS^{\ep,k}_{r_1,r_2}$ is the quantitative singular stratum in Definition \ref{introdefnquantiSS}.

We now fix all related constants by the following lemma, whose proof is straightforward from our definition of neck regions and a limiting argument.

\begin{lem}\label{lem:constant}
Given $\ep>0$, there exist small positive constants $\delta=\delta(Y, \ep)$, $\eta=\eta(Y, \ep)$ and $\zeta=\zeta(Y, \ep)$ such that the following properties hold:
\begin{enumerate}[label=\textnormal{(\roman{*})}]
		\item Theorem \ref{neckdecomgeneral} holds with the constants $\delta$, $\eta$ and $\zeta$.

		\item If $\NNN_a=B^*(x_a,2r_a)\setminus B^*_{r_x}(\CCC_a)$ is a $(k,\delta, \cc(Y), r_a)$-quotient neck region for $k \in \{1, 2\}$, then any $y\in B^*(x_a, 3r_a/2) \setminus B^*_{r_x/4}(\CCC_a)$ is $(3,\ep,  s)$-symmetric for any $s \le \ep d_Z(y,\CCC_a)$.
		
		\item If $\NNN_a=B^*(x_a,2r_a)\setminus B^*_{r_x}(\CCC_a)$ is a $(\delta, \cc(Y), r_a)$-flat neck region and $k=2$, then any $y\in \NNN'_a \bigcup \lc B^*(x_a, 3r_a/2) \bigcap \lc B^*_{r_x}(\CCC_a) \setminus B^*_{r_x/4}(\CCC_a) \rc \rc$ is $(3,\ep, s)$-symmetric for any $s \le \ep d_Z(y,\CCC_a)$, where $\NNN_a'$ is the modified $\delta$-region associated with $\NNN_a$.
		
		\item If $\NNN_a=B^*(x_a,2r_a)\setminus B^*_{r_x}(\CCC_a)$ is a $(\delta, \cc(Y), r_a)$-flat neck region and $k=1$, then any $y\in \NNN'_a \bigcap  B^*(x_a, 3r_a/2)$ is $(2,\ep, s)$-symmetric for $s = \ep d_Z(y,\CCC'_a)$, where $\NNN_a'$ is the modified $\delta$-region associated with $\NNN_a$.
		
		\item If $B^*(x_b, r_b)$ is a $b$-ball, then any $y \in B^*(x_b, 3r_b/2)$ is $(k+1, \ep, \ep r_b)$-symmetric if $k=2$, and $(k+1, \ep, r)$-symmetric for some $r \in (c(Y, \ep) r_b, r_b)$ if $k=1$.
	\end{enumerate}	
\end{lem}

Based on Lemma \ref{lem:constant}, we define
\begin{align} \label{eq:largerneck}
\NNN_a'':=
\begin{dcases}
 B^*(x_a, 3r_a/2) \setminus B^*_{r_x/4}(\CCC_a), \quad &\text{for case (ii)};\\
 \NNN'_a \bigcup \lc B^*(x_a, 3r_a/2) \bigcap \lc B^*_{r_x}(\CCC_a) \setminus B^*_{r_x/4}(\CCC_a) \rc \rc, \quad &\text{for case (iii)}.
\end{dcases}
\end{align}

Next, we prove the rectifiability for the singular strata $\MS^0$, $\MS^1$ and $\MS^2$.

\begin{thm} \label{thm:rec1}
For any $k \in \{ 0, 1, 2\}$, $\MS^k$ is parabolic $k$-rectifiable.
\end{thm}

\begin{proof}
We prove the cases $k \in \{ 0, 1, 2\}$ separately.

\textbf{Case $k=0$:} This has already been proved in Proposition \ref{prop:countable}.

\textbf{Case $k=1$:} Since $\MS^1 \setminus \MS^0 = \MS^1_{\mathrm{qc}}\setminus \MS^0_{\mathrm{qc}}$ and $\HHH^1(\MS^0)=0$, it follows from Theorem \ref{thmrectquotient} that $\MS^1$ is horizontally parabolic $1$-rectifiable.

\textbf{Case $k=2$:} Recall that we have
\begin{align*}
\MS^2 \setminus \MS^1 = \lc \MS^2_{\mathrm{qc}}\setminus \MS^1_{\mathrm{qc}}\rc \sqcup \MS^2_{\mathrm{F}}.
\end{align*}
By Theorem \ref{thmrectquotient}, it remains only to show that $\MS^2_{\mathrm{F}}$ is vertically parabolic $2$-rectifiable. 

Fix any $z_0 \in \MS^2_{\mathrm{F}}$. By Theorem \ref{neckdecomgeneral} (with $k=2$), there exists a sufficiently small $r_0>0$ such that
\begin{align} \label{eq:4drec003}
\MS^2_{\mathrm{F}} \bigcap B^*(z_0,r_0)\subset \bigcup_a \lc (\NNN'_a \bigcup \CCC_{0,a}) \bigcap B^*(x_a,r_a) \rc \bigcup \bigcup_b B^*(x_b,r_b)\bigcup \tilde S^{2,\delta,\eta}.
\end{align}
By Lemma \ref{lem:constant}, choosing $\ep$ sufficiently small ensures that
\begin{align*}
\bigcup_a\lc \NNN'_a \bigcap B^*(x_a,r_a)\rc \bigcup \bigcup_b B^*(x_b,r_b)
\end{align*}
consists entirely of regular points. Therefore, from \eqref{eq:4drec003}, we obtain
\begin{align*}
\MS^2_{\mathrm{F}} \bigcap B^*(z_0,r_0) \subset \bigcup_a\lc \CCC_{0,a} \bigcap B^*(x_a,r_a) \rc \bigcup \tilde{S}^{2,\delta,\eta}.
\end{align*}

It is clear from the definition of $\MS^2_{\mathrm{F}}$ that if $\MS^2_{\mathrm{F}} \bigcap \CCC_{0,a} \ne \emptyset$, then $\NNN_a$ is a static $(\delta, \cc(Y), r_a)$-flat neck region, which implies, by Lemma \ref{lem:bilitpsta}, that $\CCC_{0,a}$ is vertically parabolic $2$-rectifiable. Thus, we conclude that $\MS^2_{\mathrm{F}} \bigcap B^*(z_0,r_0)$ is vertically parabolic $2$-rectifiable.

A standard covering argument then shows that $\MS^2_{\mathrm{F}}$ is vertically parabolic $2$-rectifiable, completing the proof.
\end{proof}

\begin{cor} \label{cor:tangunique}
For $\HHH^2$-a.e. $x \in \MS_{\mathrm{F}}^2$, the tangent flow at $x$ is backward unique in the sense that the negative part of any tangent flow at $x$ is isometric to $\R^4/\Gamma \times \R_-$.
\end{cor}

\begin{proof}
As in the proof of Theorem \ref{thm:rec1}, for any $z_0 \in \MS_{\mathrm{F}}^2$, there exists $r_0>0$ such that 
\begin{align*}
\MS^2_{\mathrm{F}} \bigcap B^*(z_0,r_0) \subset \bigcup_a\lc\CCC_{0,a} \bigcap B^*(x_a,r_a)\rc \bigcup \tilde{S}^{2,\delta,\eta}.
\end{align*}
For any $x \in \MS^2_{\mathrm{F}} \bigcap B^*(z_0,r_0) \bigcap \CCC_{0,a} \bigcap B^*(x_a,r_a)$, it follows from the definition of the flat neck region that $x$ is $(\ep, r)$-close to $\mathcal F^0(\Gamma)$ for any $0<r \le r_a$. Therefore, the negative part of any tangent flow at $x$ is given by $\R^4/\Gamma \times \R_-$. 

A standard covering argument then completes the proof.
\end{proof}

Next, we prove the optimal volume estimates for quantitative singular strata.

\begin{thm}\label{thmrectifiabledim4}
For any $z_0 \in Z$ with $\t(z_0)-2 r_0^2 \in \III^-$, we have
	\begin{align}\label{contentestiS0}
\abs{B^*_{rr_0}\lc\MS^{\ep,2}_{rr_0,\ep r_0}\rc\bigcap B^*(z_0, r_0) } \leq C(Y,\ep)r^{4} r_0^{6}
	\end{align}
for any $r \in (0, \ep )$. In particular, we have
	\begin{align}\label{contentestiS0a}
		\HHH^2\left(\bigcap_{0<r<\ep} \MS^{\ep,2}_{rr_0,\ep r_0}\bigcap B^*(z_0, r_0)\right)\leq C(Y,\ep) r_0^{6}.
	\end{align}
\end{thm}

\begin{proof}
Given $\ep>0$, we choose $\delta$, $\eta$ and $\zeta$ according to Lemma \ref{lem:constant}. Without loss of generality, we assume $r_0=1$ and $r \le \ep/100$. Moreover, by a standard covering argument, we may assume $\t(z_0)-2 \zeta^{-2} \in \III^-$.

By Theorem \ref{neckdecomgeneral} (with $k=2$), we have
\begin{align} \label{eq:4drec004}
B^*(z_0,2)\subset \bigcup_a\lc \NNN'_a\bigcap B^*(x_a,r_a)\rc \bigcup \bigcup_b B^*(x_b,r_b)\bigcup S^{2,\delta,\eta}
\end{align}
	so that $S^{2,\delta,\eta}$ is parabolic $2$-rectifiable and
\begin{align}\label{contentesti1}
			\sum_a r^2_a+\sum_b r^2_b+\HHH^2\left(S^{2,\delta,\eta}\right)\leq C(Y).
		\end{align}	
Since $r \le 1/100$, we have $B^*_{r}(\MS^{\ep,2}_{r,\ep}) \bigcap B^*(z_0, 1) \subset B^*_r \lc \MS^{\ep,2}_{r, \ep} \cap B^*(z_0, 3/2) \rc$. 

For any $y\in \MS^{\ep,2}_{r,\ep }\cap B^*(z_0, 3/2) \cap \NNN''_a$, where $\NNN''_a$ is defined in \eqref{eq:largerneck}, it follows from Lemma \ref{lem:constant} (ii) (iii) that $d_Z(y,\CCC_a)<2 \ep^{-1} r$, since otherwise $y$ would be a $(3, \ep, 2 r)$-symmetric point, contradicting the fact that $y \in \MS^{\ep,2}_{r,\ep}$.
	
	If $r_a \ge 100 \ep^{-1} r$, we set $W_a=\{z \in \CCC_a \mid d_Z(y,z)=d_Z(y,\CCC_a) \text{ for some } y \in \MS^{\ep,2}_{r,\ep} \bigcap \NNN''_a\}$. Define $s=16 \ep^{-1} r$. By a standard covering argument, there exists a maximal set $\{x_i \mid x_i\in W_a\}_{1\leq i\leq K_a}$ such that $\{B^*(x_i,s)\}_{1\leq i\leq K_a}$ covers $W_a$ and $\{B^*(x_i,s/2)\}_{1\leq i\leq K_a}$ are pairwise disjoint. In particular, we have
	\begin{align*}
B^*(x_i,s) \subset B^*(x_a,5r_a/3)
\end{align*}
by our assumption. Let $\mu_a$ denote the corresponding packing measure on $\NNN_a$. Then, by Proposition \ref{ahlforsregucyl1quo} or \ref{ahlforsregforstaticneck}, we have
	\begin{align*}
	0<	K_ac(Y)s^2 \leq \sum_{i=1}^{K_a} \mu_a\big(B^*(x_i,s/2)\big) \le \mu_a \big(B^*(x_a,5r_a/3)\big)\leq C(Y)r^2_a,
	\end{align*}
where the last inequality can be derived from Ahlfors regularity after covering $B^*(x_a,5r_a/3)$ by balls of size $r_a/100$. Thus, we conclude that $K_a \le C(Y) r^2_a r^{-2}$. In addition, we have 
	\begin{align} \label{eq:subextra1}
\MS^{\ep,2}_{r,\ep} \bigcap B^*(z_0, 3/2) \bigcap \NNN''_a \subset \bigcup_{i=1}^{K_a} B^*(x_i,2s),
	\end{align}
which implies
	\begin{align*}
\abs{B^*_r \lc \MS^{\ep,2}_{r,\ep} \bigcap B^*(z_0, 3/2) \bigcap \NNN''_a \rc} \leq\sum_{i=1}^{K_a} \abs{B^*(x_i, 3s)} \leq C(Y, \ep) K_a r^{6}\leq C(Y, \ep) r^2_a r^{4}.
	\end{align*}
	
Next, we set
	\begin{align*}
		\MS':= \MS^{\ep,2}_{r,\ep}\bigcap B^*(z_0, 3/2) \bigcap \bigcup_{r_a \ge 100 \ep^{-1}r} \NNN''_a.
	\end{align*}
	In sum, we obtain
	\begin{equation}\label{sub1}
		\abs{B_r^* \lc\MS' \rc} \leq C(Y, \ep) \sum_{a} r^2_a r^{4} \le C(Y, \ep) r^{4}, 
	\end{equation}
	where we have used \eqref{contentesti1} for the final inequality.
	
	Next, we take a maximal $r$-separated set $\{y_i\}_{1 \le i \le K}$ of $\MS^{\ep,2}_{r,\ep} \bigcap B^*(z_0, 3/2) \setminus \MS'$.  We define the subset
		\begin{equation*}
		I^a:=\left\{ 1 \le i \le K \mid B^*(y_i,r/2) \bigcap \NNN'_a \bigcap B^*(x_a,r_a) \ne \emptyset \text{ and } r_a \ge 100 \ep^{-1} r \right\}.
	\end{equation*}
	
It follows from our definition of $\MS'$ that for $i \in I^a$, $y_i \in B^*(x_a, 3r_a/2) \bigcap B^*_{r_x/4}(\CCC_a)$. We choose $z_i \in \CCC_a$ such that $d_i:=d_Z(z_i, y_i) \le r_{z_i}/4$. Then it is clear that $d_i+r/2 \ge r_{z_i}$ since otherwise $B^*(y_i, r/2) \bigcap \NNN'_a =\emptyset$. From this, we obtain that
	\begin{align}\label{sub1b}
	r \ge \frac{3}{2} r_{z_i} \quad \text{and} \quad d_i \le \frac{r}{6}.
	\end{align}	
For $i, j \in I^a$, it follows from \eqref{sub1b} that $d_Z(z_i, z_j) \ge d_Z(y_i, y_j)-d_i-d_j \ge 2r/3$. In other words, the balls $\{B^*(z_i, r/3)\}_{i \in I^a}$ are mutually disjoint. By Proposition \ref{prop:volumebound}, we conclude that any point in $\bigcup_{i \in I^a} B^*(z_i, 2r/3)$ belongs to at most $C(Y)$ balls in $\{B^*(z_i, 2r/3)\}_{i \in I^a}$. Thus, it follows from Proposition \ref{ahlforsregucyl1quo} or \ref{ahlforsregforstaticneck} that
	\begin{align*}
	0<	|I^a| c(Y) r^2 \leq \sum_{i \in I^a} \mu_a\big(B^*(z_i,2r/3)\big) \le C(Y) \mu_a \big(B^*(x_a,5r_a/3)\big)\leq C(Y)r^2_a,
	\end{align*}
	which implies $|I^a| \le C(Y) r^2_a r^{-2}$. Taking the sum for all possible $a$ and using \eqref{contentesti1}, we obtain
	\begin{align}\label{sub1c}
\abs{\bigcup_{a} I^a} \le C(Y)  r^{-2}.
	\end{align}

Next, we set $J:=\{1, \ldots, K\} \setminus \bigcup_{a} I^a$. For $i \in J$, if $B^*(y_i,r/2) \bigcap \NNN'_a \bigcap B^*(x_a,r_a) \ne \emptyset$, then $r_a \le 100 \ep^{-1} r$. On the other hand, if $B^*(y_i,r/2) \bigcap B^*(x_b,r_b) \ne \emptyset$, then it follows from Lemma \ref{lem:constant} (v) that $r_b \le 2 \ep^{-1} r$. Otherwise, $y_i \in B^*(x_b, 3r_b/2)$ and hence is $(3, \ep, \ep r_b)$-symmetric. However, this contradicts the fact that $y_i \in \MS^{\ep,2}_{r,\ep}$. 
	
It follows from the decomposition \eqref{eq:4drec004} that
	\begin{align*}
		\bigcup_{i \in J} B^*(y_i,r/2) \subset & \bigcup_{r_a \le 100 \ep^{-1} r} B^*(x_a,r_a)\bigcup\bigcup_{r_b \le 2 \ep^{-1} r}B^*(x_b,r_b)\bigcup S^{2,\delta,\eta}.
	\end{align*}
	By comparing the volumes, it follows that
	\begin{align*}
		0<c(Y)|J| r^6 \leq &\sum_{i \in J}\abs{B^*(y_i,r/2)}		\leq \sum_{r_a \le 100 \ep^{-1} r} \abs{B^*(x_a,r_a)}+\sum_{r_b \le 2 \ep^{-1} r}\abs{B^*(x_b,r_b)}+\abs{S^{2,\delta,\eta}}\\
		\leq & C(Y) \lc \sum_{r_a \le 100 \ep^{-1} r}r_a^{6}+\sum_{r_b \le 2 \ep^{-1} r}r_b^{6} \rc	\leq  C(Y,\ep)r^{4} \lc \sum_{a}r^2_a+\sum_{b}r^2_b \rc\leq C(Y, \ep)r^{4},
	\end{align*}
	where we have used the fact that $r_a\leq 100 \ep^{-1} r$, $r_b\leq 2 \ep^{-1} r$ and $\abs{S^{2,\delta,\eta}}=0$ since $S^{2,\delta,\eta}$ is parabolic $2$-rectifiable.
	
Consequently, $|J| \leq C(Y, \ep) r^{-2}$. Combining this with \eqref{sub1c} and using the fact that all $\{B^* (y_i, 2r)\}_{1\leq i\leq K}$ cover $B^*_r\lc \MS^{\ep,2}_{r,\ep} \bigcap B^*(z_0, 3/2) \setminus \MS' \rc$, we obtain
	\begin{align}\label{sub2}
	\abs{B^*_r\lc \MS^{\ep,2}_{r,\ep} \bigcap B^*(z_0, 3/2) \setminus \MS' \rc } \le \sum_{i=1}^K \abs{B^*(y_i,2r)} \leq C(Y)Kr^{6}\leq C(Y, \ep)r^{4}.
	\end{align}
	Combining \eqref{sub1} and \eqref{sub2}, we conclude that for any $r \in (0,\ep)$,
	\begin{equation} \label{eq:main1x}
\abs{B_r^* \lc \MS^{\ep,2}_{r,\ep} \bigcap B^*(z_0, 3/2) \rc}\leq C(Y, \ep) r^{4},
	\end{equation}
	which completes the proof of \eqref{contentestiS0}.
	
Since $\bigcap_{0<r<\ep} \MS^{\ep,2}_{r,\ep} \subset \MS^{\ep, 2}_{r, \ep}$ for any $r \in (0, \ep)$, it follows from \eqref{eq:main1x} that
	\begin{equation*}
\abs{ B^*_r\lc \bigcap_{0<r<\ep} \MS^{\ep,2}_{r,\ep} \bigcap B^*(z_0,3/2) \rc}\leq C(Y, \ep) r^{4}
	\end{equation*}
	for any $r \in (0, \ep)$. From this, it is clear that
	\begin{align*}
		\HHH^2 \left(\bigcap_{0<r<\ep} \MS^{\ep,2}_{r,\ep} \bigcap B^*(z_0,3/2) \right)\leq C(Y,\epsilon),
	\end{align*}	
	which yields \eqref{contentestiS0a} and completes the proof.
\end{proof}

Similarly, we have

\begin{thm}\label{thmrectifiabledim4x}
For any $z_0 \in Z$ with $\t(z_0)-2 r_0^2 \in \III^-$, we have
	\begin{align*}
\abs{B^*_{rr_0}\lc\MS^{\ep,1}_{rr_0,\ep r_0}\rc\bigcap B^*(z_0, r_0) } \leq C(Y,\ep)r^{5} r_0^{6}
	\end{align*}
for any $r \in (0, \ep )$. In particular, we have
	\begin{align*}
		\HHH^1\left(\bigcap_{0<r<\ep} \MS^{\ep,1}_{rr_0,\ep r_0}\bigcap B^*(z_0, r_0)\right)\leq C(Y,\ep) r_0^{6}.
	\end{align*}
\end{thm}

\begin{proof}
The proof is similar to that of Theorem \ref{thmrectifiabledim4}, and we only sketch it. As before, we assume $r_0=1$, $r \le \ep/100$ and fix constants $\delta$, $\eta$ and $\zeta$.

By Theorem \ref{neckdecomgeneral} (with $k=1$), we have
\begin{align*} 
B^*(z_0,2)\subset \bigcup_a\lc \NNN'_a\bigcap B^*(x_a,r_a)\rc \bigcup \bigcup_b B^*(x_b,r_b)\bigcup S^{1,\delta,\eta}.
\end{align*}
so that $S^{1,\delta,\eta}$ is parabolic $1$-rectifiable and
\begin{align}\label{eq:k1000}
			\sum_a r_a+\sum_b r_b+\HHH^1\left(S^{1,\delta,\eta}\right)\leq C(Y, \ep).
		\end{align}	
		
We reindex 
\begin{align*} 
\bigcup_a\big(\NNN'_a\bigcap B^*(x_a,r_a)\big)=\bigcup_{a'}\big(\NNN_{a'}\bigcap B^*(x_{a'},r_{a'})\big) \bigcup \bigcup_{a''}\big(\NNN'_{a''}\bigcap B^*(x_{a''},r_{a''})\big)
\end{align*}
so that each $\NNN_{a'}$ is a quotient cylindrical neck region and each $\NNN_{a''}$ is a flat neck region. For simplicity, we set
\begin{align*} 
S_0:= \MS^{\ep,1}_{r,\ep } \bigcap B^*(z_0, 3/2)\bigcap B^*_r \lc \bigcup_{r_{a''} \ge 100 \ep^{-1}r}  \lc \NNN'_{a''} \bigcap B^*(x_{a''},r_{a''}) \rc  \rc.
\end{align*}

According to the definition, for any $y \in S_0$, there exists $a''$ with $r_{a''} \ge 100 \ep^{-1}r$ such that $d_Z \lc y,  \NNN'_{a''} \bigcap B^*(x_{a''},r_{a''}) \rc \le r$. Then it follows from Lemma \ref{lem:constant} (ii) (iv) that $d_Z(y,\CCC'_{a''})<2 \ep^{-1} r$, since otherwise $y$ would be a $(2, \ep , \ep d_Z(y,\CCC'_{a''}))$-symmetric point, contradicting the fact that $y \in \MS^{\ep,1}_{r,\ep}$.

Thus, we conclude from Proposition \ref{cballdecomposition3} that 
	\begin{align} \label{eq:k1001}
\abs{B_r^*\lc S_0 \rc} \le &  \sum_{r_{a''} \ge 100 \ep^{-1}r} \abs{B_{3\ep^{-1} r}^*\lc \CCC'_{a''} \rc} \le \sum_{r_{a''} \ge 100 \ep^{-1}r} C(Y, \ep) r^{5.1} r^{0.9}_{a''} \notag \\
 \le & \sum_{r_{a''} \ge 100 \ep^{-1}r} C(Y, \ep) r^{5} r_{a''} \le C(Y, \ep) r^5,
	\end{align}
	where we used \eqref{eq:k1000} for the last inequality.

Next, we set
	\begin{align*}
S_1=\MS^{\ep,1}_{r,\ep} \bigcap B^*(z_0, 3/2) \setminus S_0.
	\end{align*}
Then it follows from the definition of $S_0$ that
\begin{align*} 
B^*_r \lc S_1 \rc \subset \bigcup_{a'}\lc \NNN_{a'}\bigcap B^*(x_{a'},r_{a'})\rc \bigcup  \bigcup_{r_{a''} \le 100 \ep^{-1} r}\lc \NNN'_{a''}\bigcap B^*(x_{a''},r_{a''})\rc \bigcup \bigcup_b B^*(x_b,r_b)\bigcup S^{1,\delta,\eta}.
\end{align*}

By the same argument as in the proof of Theorem \ref{thmrectifiabledim4}, we have
	\begin{equation}\label{eq:k1003}
		\abs{B_r^* \lc S_1 \rc} \leq C(Y, \ep) r^5.
	\end{equation}
Combining \eqref{eq:k1001} and \eqref{eq:k1003}, the proof is complete.
\end{proof}

As an application of Theorem \ref{thmrectifiabledim4}, we have

\begin{thm}
For any $z_0 \in Z$ with $\t(z_0)-2r_0^2 \in \III^-$, we have
	\begin{align*}
\HHH^2 \lc \MS \bigcap B^*(z_0, r_0) \rc \leq C(Y) r_0^{6}.
	\end{align*}
\end{thm}

\begin{proof}
It follows from \cite[Theorem 8.14]{fang2025RFlimit} that if $\ep \le \ep(Y)$, $\MS \bigcap B^*(z_0, r_0) \subset \MS^{\ep, 2}_{r r_0, \ep r_0}$ for any $r \in (0, \ep)$. Thus, for some fixed $\ep=\ep(Y)$,
	\begin{align*}
\MS \bigcap B^*(z_0, r_0) = \bigcap_{0<r<\ep} \MS^{\ep,2}_{r r_0,\ep r_0} \bigcap B^*(z_0, r_0)
	\end{align*}	
which yields the conclusion by \eqref{contentestiS0a}.
\end{proof}

We also have the following estimate for each time-slice.

\begin{thm}\label{thmvolumeslice}
Fix $k\in\{1,2\}$. For any $z_0 \in Z$ with $\t(z_0)-2 r_0^2 \in \III^-$ and $t \in \R$, we have
	\begin{align*}
\abs{B^*_{rr_0}\lc\MS^{\ep,k}_{rr_0,\ep r_0}\rc\bigcap B^*(z_0, r_0) \bigcap Z_t}_t \leq C(Y,\ep)r^{4-k} r_0^{4}, \quad \forall r \in (0, \ep).
	\end{align*}
\end{thm}

\begin{proof}
Without loss of generality, we set $r_0 = 1$. In the proof below, we focus on the case $k = 2$, using the same notation as in the proof of Theorem \ref{thmrectifiabledim4}, since the case $k = 1$ can be treated in the same way.

It follows from \eqref{eq:subextra1} and \cite[Proposition 5.36]{fang2025RFlimit} that if $r_a\geq 100\ep^{-1}r$, then
	\begin{align*}
\abs{B^*_r \lc \MS^{\ep,2}_{r,\ep} \bigcap B^*(z_0, 3/2) \bigcap \NNN''_a \rc \bigcap Z_t}_t \leq\sum_{i=1}^{K_a} \abs{B^*(x_i, 3s) \bigcap Z_t}_t \leq C(Y, \ep) K_a r^{4}\leq C(Y, \ep) r^2_a r^{2},
	\end{align*}
which implies
	\begin{equation}\label{subextra001}
		\abs{B_r^* \lc\MS' \rc \bigcap Z_t}_t \leq C(Y, \ep) \sum_{a} r^2_a r^{2} \le C( Y,\ep) r^{2}. 
	\end{equation}
By the same argument, we obtain as \eqref{sub2} that
	\begin{align}\label{subextra002}
	\abs{B^*_r\lc \MS^{\ep,2}_{r,\ep} \bigcap B^*(z_0, 3/2) \setminus \MS' \rc \bigcap Z_t }_t \le \sum_{i=1}^K \abs{B^*(y_i,2r) \bigcap Z_t}_t \leq CKr^{4}\leq C(Y, \ep)r^{2}.
	\end{align}
Combining \eqref{subextra001} and \eqref{subextra002}, we conclude that for any $r \in (0, \ep)$,
	\begin{equation*}
\abs{B_r^* \lc \MS^{\ep,2}_{r,\ep} \bigcap B^*(z_0, 3/2) \rc \bigcap Z_t}_t \leq C(Y, \ep) r^{2}, 
	\end{equation*}
	which completes the proof.
\end{proof}

As a corollary, we prove

\begin{cor} \label{cor:curvatureradius}
For any $z_0 \in Z$ with $\t(z_0)-2 r_0^2 \in \III^-$, $r <1$ and $t \in \R$, we have
	\begin{align*}
\abs{\{r_{\Rm}< r r_0 \} \bigcap B^*(z_0, r_0)\bigcap Z_t}_t \le C(Y) r^2 r_0^4.
	\end{align*}
\end{cor}

\begin{proof}
Without loss of generality, we assume $r_0=1$ and $r \le r(Y) \ll 1$.

By \cite[Theorem 8.14]{fang2025RFlimit}, we have
	\begin{align*}
\{r_{\Rm}< r \} \subset \MS^{\ep, 2}_{ r, \ep},
	\end{align*} 
if $\ep = \ep(Y)$ is sufficiently small. By Theorem \ref{thmvolumeslice}, we obtain
	\begin{align*}
\abs{\{r_{\Rm}< r \} \bigcap B^*(z_0, 1) \bigcap Z_t}_t \leq C(Y) r^2.
	\end{align*}
\end{proof}

Thanks to Corollary \ref{cor:curvatureradius}, the cutoff functions constructed in \cite[Proposition 8.20]{fang2025RFlimit} admit improved estimates in dimension four:
\begin{prop}\label{pro:cutoffdim4}
For any $z \in Z_{\III^-}$, we can find a family of smooth cutoff functions $\{\eta_{r, A}\}$ satisfying the following properties for $r \le r(Y)$.
	\begin{enumerate}[label=\textnormal{(\arabic{*})}]
		\item $r_{\Rm}\geq r$ and $d_Z(z, \cdot) \le 2A$ on $\{\eta_{r, A}>0\}$.
		\item $\eta_{r, A}=1$ on $\{r_{\Rm}\geq 2r\} \cap B^*(z, A)$.
		\item $\displaystyle r|\na \eta_{r, A}|+r^2|\partial_\t\eta_{r, A}|+r^2|\na^2\eta_{r, A}|\leq C$, where $C$ is a universal constant.
		\item For any $L$ with $\t(z)-2L^2 \in \III^-$, there exists a constant $C=C(Y, L, A)>0$ such that
		 		\begin{align*}
			\iint_{\RR_{[\t(z)-L^2, \t(z)+L^2]} \cap \{0<\eta_{r, A}<1\}} 1\,\mathrm{d}V_{g^Z_t}\mathrm{d}t\leq Cr^{4}.
		\end{align*}
			Moreover, for any $t \in [\t(z)-L^2, \t(z)+L^2]$, we have
		\begin{align*}
			\int_{\RR_t \cap \{0<\eta_{r, A}<1\}}1 \,\mathrm{d}V_{g^Z_t} \leq Cr^{2}.
		\end{align*}
	\end{enumerate}
\end{prop}

We also remark that, by the same reason, the cutoff functions on Ricci shrinker spaces constructed in \cite[Proposition 8.23]{fang2025RFlimit} also have better estimates in dimension four; we leave the precise formulation to the interested readers.

Next, we show that in dimension four, the modified pointed $\widetilde{\mathcal W}$-entropy in Definition \ref{defnWRFlimit1} coincides with the pointed $\mathcal W$-entropy in Definition \ref{defnentropy}:
 \begin{prop}\label{prop:agree}
 Let $(Z, d_Z, \t)$ be a noncollapsed Ricci flow limit space arising as the pointed Gromov--Hausdorff limit of a sequence in $\MM(4, Y, T)$. Then for any $z\in Z_{\III^-}$ and any $\tau>0$ such that $\t(z)-\tau\in \III^-$, we have
 	\begin{align*}
 		\widetilde{\mathcal W}_z(\tau)=\WW_z(\tau).
 	\end{align*}
 \end{prop} 
 \begin{proof}
 	By Definitions \ref{defndiscrete} and \ref{defnWRFlimit1}, it suffices to show that for any $z\in Z_{\III^-}$ and any $\tau>0$ such that $\t(z)-\tau\in \III^-$, it holds that 
 	\begin{equation}\label{equ:entropycoincide1}
\int_{\RR_{\t(z)-\tau}} \Delta f_z-|\na f_z|^2 \,\mathrm{d}\nu_{z;\t(z)-\tau} = 0.
		\end{equation}
In fact, let $\eta_{r,A}$ be the cutoff functions from Proposition \ref{pro:cutoffdim4}. Then, using integration by parts, we obtain
\begin{align}\label{equ:coin1}
	\int_{\RR_{\t(z)-\tau}}\eta_{r,A}\Delta f_z\,\mathrm{d}\nu_{z;\t(z)-\tau}=\int_{\RR_{\t(z)-\tau}}\eta_{r,A}|\na f_z|^2 \,\mathrm{d}\nu_{z;\t(z)-\tau}-\int_{\RR_{\t(z)-\tau}}\la \na \eta_{r,A}, \na  f_z\ra \,\mathrm{d}\nu_{z;\t(z)-\tau}.
\end{align}
By Proposition \ref{integralbound} and smooth convergence on the regular part, we have
\begin{align}\label{equ:coincide4}
	\int_{\RR_{\t(z)-\tau}}|\Delta f_z|+|\na f_z|^2\,\mathrm{d}\nu_{z;\t(z)-\tau} <\infty,
\end{align}
which implies 
\begin{align}\label{equ:coincide2}
	\lim_{\substack{r\to 0\\ A\to\infty}}\int_{\RR_{\t(z)-\tau}}\eta_{r,A}\lc \Delta f_z-|\na f_z|^2 \rc\,\mathrm{d}\nu_{z;\t(z)-\tau}=\int_{\RR_{\t(z)-\tau}}\Delta f_z-|\na f_z|^2\,\mathrm{d}\nu_{z;\t(z)-\tau}.
\end{align}

On the other hand, by Proposition \ref{pro:cutoffdim4}, it follows that there exists a constant $C>0$ independent of $A$ and $r$ such that 
\begin{align}\label{equ:coincide3}
&\left|\int_{\RR_{\t(z)-\tau}}\la \na \eta_{r,A}, \na  f_z\ra \,\mathrm{d}\nu_{z;\t(z)-\tau}\right| \leq  Cr^{-1}\int_{\RR_{\t(z)-\tau}\bigcap  \{0<\eta_{r, A}<1\}}|\na f_z|\,\mathrm{d}\nu_{z;\t(z)-\tau}\nonumber\\
\leq& C r^{-1}\lc \int_{\RR_{\t(z)-\tau} \bigcap  \{0<\eta_{r, A}<1\}}|\na f_z|^2\,\mathrm{d}\nu_{z;\t(z)-\tau}\rc^{1/2}\lc \int_{\RR_{\t(z)-\tau}\bigcap  \{0<\eta_{r, A}<1\}} 1 \,\mathrm{d}\nu_{z;\t(z)-\tau}\rc^{1/2}\nonumber\\
\leq& C \lc \int_{\RR_{\t(z)-\tau} \bigcap  \{0<\eta_{r, A}<1\}}|\na f_z|^2\,\mathrm{d}\nu_{z;\t(z)-\tau}\rc^{1/2}.
\end{align}
Combining \eqref{equ:coincide4} with \eqref{equ:coincide3}, we have
\begin{align}\label{equ:coincide5}
	\lim_{A\to\infty}\lim_{r\to 0}\int_{\RR_{\t(z)-\tau}}\la \na \eta_{r,A}, \na  f_z\ra \,\mathrm{d}\nu_{z;\t(z)-\tau}=0.
\end{align}
From \eqref{equ:coin1}, \eqref{equ:coincide2} and \eqref{equ:coincide5}, we conclude that \eqref{equ:entropycoincide1} holds, which completes the proof.
 \end{proof}

\subsection{\texorpdfstring{$L^1$}{L1}-curvature bounds for four-dimensional closed Ricci flows}

Let $\XX=\{M^4, (g(t))_{t\in [-T, 0)}\}$ be a four-dimensional closed Ricci flow with entropy bounded below by $-Y$ such that $T<\infty$ and $0$ is the first singular time. Assume that $(Z, d_Z, \t)$ is the $d^*$-completion of $\XX_{[-0.98T, 0)}$. 

\begin{thm}\label{thm:RmL1y}
With the above assumptions, there exists a constant $C>0$ depending on $Y$, $T$ and $\mathrm{diam}_{g(-0.99T)}(M)$ such that for any $t\in [-T/2,0)$, we have
	\begin{align*}
		\int_{M} |\Rm|\,\mathrm{d}V_{g(t)} \le \int_{M} r^{-2}_{\Rm}\,\mathrm{d}V_{g(t)}\leq C.
	\end{align*}
\end{thm}

\begin{proof}
For small parameters $\delta>0$ and $\eta>0$ to be determined later, we choose $\zeta=\zeta(Y,\delta,\eta)$ such that Theorem \ref{neckdecomgeneral} holds with $k=2$. Specifically, for any $z_0\in Z_{[-T/2 ,0]}$, set $r_0:=\zeta\sqrt{T/8}$. Then we obtain the following decomposition:
\begin{align*}
	&B^*(z_0,r_0)\subset \bigcup_a\big(\NNN'_a\bigcap B^*(x_a,r_a)\big)\bigcup\bigcup_b B^*(x_b,r_b)\bigcup S^{2,\delta,\eta},\\
	&S^{2,\delta,\eta}\subset \bigcup_a\big(\CCC_{0,a}\bigcap B^*(x_a,r_a)\big)\bigcup\tilde{S}^{2,\delta,\eta},
\end{align*}
with the following properties:
\begin{enumerate}[label=\textnormal{(\alph{*})}]
	\item For each $a$, $\NNN_a=B^*(x_a,2r_a)\setminus B^*_{r_x}(\CCC_a)$ is either a $(2, \delta, \cc, r_a)$-quotient cylindrical neck region or a $(\delta, \cc, r_a)$-flat neck region, where $\cc=\cc(Y)$. In the former case, we set $\NNN_a'=\NNN_a$; in the latter, $\NNN_a'$ denotes the modified $\delta$-region associated with $\NNN_a$.
	
	\item For each $b$, there exists a point in $B^*(x_b,2 r_b)$ which is $(3,\eta,r_b)$-symmetric.
	
	\item The following content estimate holds:
	\begin{align}\label{equ:contenta}
		\sum_a r_a^2+\sum_b r_b^2+\HHH^2(S^{2,\delta,\eta})\leq C(Y) r_0^2.
	\end{align}	
\end{enumerate}

Fix $t\in [-T/2,0)$. Since $S^{2,\delta,\eta} \subset Z_0$ and $Z_t$ is smooth, we have
\begin{align*}
	B^*(z_0,r_0)\bigcap Z_t\subset \bigcup_a\big(\NNN'_a\bigcap B^*(x_a,r_a)\bigcap Z_t\big)\bigcup\bigcup_b \lc B^*(x_b,r_b)\bigcap Z_t\rc.
\end{align*}

From Lemma \ref{lem:constant} (v), it follows that for any $\ep>0$, if $\eta\leq\eta(Y,\ep)$, then any point $y\in B^*(x_b,3r_b/2)$ is $(3,\ep,\ep r_b)$-symmetric. If $\ep=\ep(Y)$ is small, then by \cite[Theorem 8.13]{fang2025RFlimit}, we obtain that $r^{-2}_{\Rm}(y) \leq C(Y)r_b^{-2}$. 
Therefore, by \cite[Proposition 3.14 (i)]{fang2025RFlimit}, we have
\begin{align}\label{equ:RmL1a}
	\int_{B^*(x_b,r_b)\bigcap Z_t}r^{-2}_{\Rm}\,\mathrm{d}V_{g(t)}\leq C(Y)r_b^{-2}\cdot r_b^4\leq C(Y)r_b^2.
\end{align}
Next, we consider the $a$-balls in the following two cases. 

\textbf{Case 1:} $\NNN_a=B^*(x_a,2r_a)\setminus B^*_{r_x}(\CCC_a)$ is a $(2, \delta, \cc, r_a)$-quotient cylindrical neck region.

In this case, $\NNN_a'=\NNN_a$. For any $x\in \NNN_a\bigcap B^*(x_a,r_a)\bigcap Z_t=:Z_{a,t}$, we set $d_x:=d_Z(x,\CCC_a)=d_Z(x,x')$ for some $x'\in \CCC_a$. By definition, we have $d_x\geq r_{x'}$. Applying the Vitali covering argument, we can choose $\{x_i\}\subset Z_{a,t}$ such that for any $i\neq j$, $d_Z(x_i,x_j)\geq 2(d_{x_i}+d_{x_j})$,
and moreover,
\begin{align}\label{equ:RmL1d}
	Z_{a,t}\subset \bigcup_i B^*(x_i, 10d_{x_i}).
\end{align}

By Definition \ref{defiofcylneckregionquo} $(n3)$, for each $i$, $x_i'\in Z$ is $(2,\delta,d_{x_i})$-quotient cylindrical (regarding $\bar \CC^2$ or $\bar \CC^2_2(\mathbb{Z}_2)$). Moreover, by Definition \ref{defiofcylneckregionquo} $(n4)$ and the fact $d_Z(x_i,\CCC_a)=d_{x_i}$, there exists a constant $c(Y)>0$ such that $t=\t(x_i)\leq \t(x_i')-c(Y)d_{x_i}^2$. Hence, by the quotient version of \cite[Lemma 2.27]{FLloja05}, if $\delta\leq\delta(Y,\ep)$, then every $y\in B^*(x_i,10d_{x_i})\bigcap Z_t$ is $(3,\ep,\ep d_{x_i})$-symmetric for $\ep\leq\ep (Y)$, and therefore satisfies the curvature bound $r^{-2}_{\Rm}(y) \leq C(Y)d_{x_i}^{-2}$. 

Applying \cite[Proposition 3.14 (i)]{fang2025RFlimit}, we obtain
\begin{align}\label{equ:RmL1b}
	\int_{B^*(x_i,10d_{x_i})\bigcap Z_{a,t}}r^{-2}_{\Rm}\,\mathrm{d}V_{g(t)}\leq C(Y)d_{x_i}^2.
\end{align}

On the other hand, by Ahlfors regularity (Proposition \ref{ahlforsregucyl1quo}), if $d_{x_i}\leq r_a/100$, 
\begin{align*}
\mu \lc B^*(x_i',d_{x_i}) \rc \geq c(Y)d_{x_i}^2>0.
\end{align*}
Combining this with \eqref{equ:RmL1b} gives
\begin{align}\label{equ:RmL1c}
	\int_{B^*(x_i,10d_{x_i})\bigcap Z_{a,t}}r^{-2}_{\Rm}\,\mathrm{d}V_{g(t)}\leq C(Y)\mu \lc B^*(x_i',d_{x_i}) \rc.
\end{align}
By construction, for $i\neq j$ we have $B^*(x_i',d_{x_i})\bigcap B^*(x_j',d_{x_j})=\emptyset$ and if $d_{x_i}\leq r_a/100$, then $B^*(x_i',d_{x_i})\subset B^*(x_a,3r_a/2)$. Moreover, by Proposition \ref{prop:volumebound}, the number of indices $i$ with $d_{x_i}\geq r_a/100$ is finite and bounded by $C(Y)$. For these indices, the integrals as in \eqref{equ:RmL1b} can be estimated directly by $C(Y)r_a^2$. 

Finally, combining \eqref{equ:RmL1d} with \eqref{equ:RmL1c}, we obtain
\begin{align}\label{equ:RmL1g}
	\int_{Z_{a,t}}r^{-2}_{\Rm}\,\mathrm{d}V_{g(t)} \leq & C(Y)\sum_{d_{x_i}\leq r_a/100}\mu \lc B^*(x_i',d_{x_i}) \rc+C(Y)r_a^2 \notag \\
	 \leq & C(Y)\mu(B^*(x_a,3r_a/2))+C(Y)r_a^2\leq C(Y)r_a^2,
\end{align}
where the last inequality follows from Ahlfors regularity after covering $B^*(x_a,3r_a/2)$ by balls of radius $r_a/100$.

\textbf{Case 2:} $\NNN_a=B^*(x_a,2r_a)\setminus B^*_{r_x}(\CCC_a)$ is a $(\delta, \cc, r_a)$-flat neck region.

In this case, every point $y\in\NNN_a'$ is $(3,\delta, \delta d_Z(y,\CCC_a))$-symmetric. Hence, if we choose $\delta=\delta(Y)$ and compare with the model space $\R^4/\Gamma\times (-\infty,0]$ at scale $d_Z(y,\CCC_a)$, we obtain
\begin{align}\label{equ:RmL1fex01}
r^{-2}_{\Rm}(y) \leq C(Y)d_Z(y,\CCC_a)^{-2}.
\end{align}

For any $r\in (0,r_a]$, set 
\begin{align*}
A_{a,r}:=\{y\in \NNN'_a\bigcap B^*(x_a,r_a)\bigcap Z_t \mid d_Z(y,\CCC_a)\in [r/2,2r]\}.
\end{align*}
We choose a small constant $c_1=c_1(Y)>0$ to be determined later.

If $r \in [c_1 r_a, r_a]$, then by \cite[Proposition 3.14(i)]{fang2025RFlimit} we have $|A_{a, r}|_t \le C r_a^4$ for a universal constant $C$. Hence, by \eqref{equ:RmL1fex01},
\begin{align}\label{equ:RmL1fex02}
	\int_{A_{a, r}} r^{-2}_{\Rm}\,\mathrm{d}V_{g(t)}\leq C(Y) r^2.
\end{align}

If $r\in (0, c_1 r_a]$, fix $y \in A_{a, r}$ and choose $y' \in \CCC_a$ realizing the distance, i.e. $d_Z(y, \CCC_a)=d_Z(y, y')$. For any $x \in A_{a, r}$, pick $x' \in \CCC_a$ with $d_Z(x, \CCC_a)=d_Z(x, x')$. Then $\abs{\t(x)-\t(x')} \le d_Z^2(x, x') \le 4r^2$. By Lemma \ref{lem:bilitpsta}, it follows that $d^2_Z(y', x') \le C(Y)|\t(x')-\t(y')| \le C(Y) r^2$. Note that if $c_1=c_1(Y)$ is chosen to be small, the assumption of Lemma \ref{lem:bilitpsta} can be satisfied. Hence, $A_{a, r} \subset B^*(y', C(Y) r)$. Applying \cite[Proposition 3.14(i)]{fang2025RFlimit} again yields $|A_{a, r}|_t \le C(Y) r^4$, which, together with \eqref{equ:RmL1fex01}, implies \eqref{equ:RmL1fex02} as well.

Choose $r=2^{-i}r_a$ for $i \ge 0$. Then summing \eqref{equ:RmL1fex02} gives
\begin{align}\label{equ:RmL1e}
	\int_{\NNN_a'\bigcap Z_t}r^{-2}_{\Rm}\,\mathrm{d}V_{g(t)}\leq\sum_{i=0}^\infty\int_{A_{a, 2^{-i}r_a}} r^{-2}_{\Rm}\,\mathrm{d}V_{g(t)}\leq C(Y) r_a^2.
\end{align}

Combining \eqref{equ:RmL1a}, \eqref{equ:RmL1g}, and \eqref{equ:RmL1e}, together with the content estimate \eqref{equ:contenta}, we obtain
\begin{align}\label{equ:RmL1l}
	\int_{B^*(z_0,r_0)\bigcap Z_t}r^{-2}_{\Rm}\,\mathrm{d}V_{g(t)}\leq C(Y)\lc\sum_{a}r_a^2+\sum_b r_b^2\rc \leq C(Y)r_0^2.
\end{align}

We now fix constants $\delta=\delta(Y),\,\eta=\eta(Y)$ and $\zeta=\zeta(Y)$ so that all the above estimates hold. By \cite[Equation (9.1)]{fang2025RFlimit}, the diameter of $(Z,d_Z)$ is bounded by a constant depending on $Y$, $T$ and $\mathrm{diam}_{g(-0.99T)}(M)$. Consequently, for any $t \in [-T/2, 0)$, $Z_t$ can be covered by at most $C(Y, T, \mathrm{diam}_{g(-0.99T)}(M))$ balls $B^*(z_i, r_0)$ with $z_i \in Z_{[-T/2, 0]}$. Applying \eqref{equ:RmL1l} on each ball yields
\begin{align*}
	\int_{Z_t}r^{-2}_{\Rm}\,\mathrm{d}V_{g(t)}\leq C,
\end{align*}
where $C$ depends only on $Y$, $T$ and $\mathrm{diam}_{g(-0.99T)}(M)$.
\end{proof}

\begin{rem}
With the same argument, Theorem \ref{thm:RmL1y} extends to all $4$-dimensional Ricci flow limit spaces. More precisely, let $(Z,d_Z,\t)$ be a noncollapsed Ricci flow limit space arising as a pointed Gromov--Hausdorff limit of a sequence in $\MM(4,Y,T)$. Then for any $z_0\in Z$ with $\t(z_0)-2r_0^2\in\III^-$ and for any $t\in\R$, one has
\begin{align*}
\int_{B^*(z_0,r_0) \cap \RR_t} r_{\Rm}^{-2}\,\mathrm{d}V_{g_t^Z}\le C(Y) r_0^2.
\end{align*}
\end{rem}

By essentially the same argument as in Theorem \ref{thm:RmL1y}, one likewise obtains integral estimates for $\na^k \Rm$ for any $k \ge 1$.

\begin{thm}\label{thm:RmLky}
With the above assumptions, for any integer $k \ge 1$, there exists a constant $C_k>0$ depending on $k$, $Y$, $T$ and $\mathrm{diam}_{g(-0.99T)}(M)$ such that for any $t\in [-T/2,0)$, we have
	\begin{align*}
		\int_{M} |\na^k\Rm|^{\frac{2}{k+2}}\,\mathrm{d}V_{g(t)}\leq C_k.
	\end{align*}
\end{thm}

Using Theorem \ref{thm:RmL1y}, we obtain a refinement of \cite[Corollary 9.4 and Proposition 9.5]{fang2025RFlimit}. We note that the corresponding estimate is already known in the setting of compact Kähler Ricci flows, where it follows from properties of the Kähler forms.

\begin{cor} \label{cor:volume4d}
There exists a constant $C$ depending on $Y$, $T$ and $\mathrm{diam}_{g(-0.99T)}(M)$ such that for any $s, t \in [-T/2, 0)$,
	\begin{align*}
\abs{|M|_t-|M|_s} \le C |t-s|.
	\end{align*} 
\end{cor}

\begin{proof}
This follows directly from Theorem \ref{thm:RmL1y} together with the identity
	\begin{align*}
\diff{}{t} |M|_t=-\int_M \scal \,\mathrm{d}V_{g(t)}.
	\end{align*} 
\end{proof}

\newpage

\appendixpage
\addappheadtotoc
\appendix
\section{Geometric transformation theorem for almost splitting maps} 
\label{app:A} 

In this appendix, we derive a geometric transformation theorem for almost splitting maps on closed Ricci flows and noncollapsed Ricci flow limit spaces. Roughly speaking, under suitable conditions, if a map is almost splitting at a given scale, then at each smaller scale, it can be transformed by a lower triangular matrix so that it remains almost splitting at that scale. Moreover, the growth of these transformations can be quantitatively controlled. Note that a similar geometric transformation theorem was established in the setting of noncollapsed Ricci limit spaces in \cite[Theorem 7.2]{Cheeger2018RectifiabilityOS}.

We begin with the following elementary lemma of linear algebra, whose proof is straightforward and omitted.

\begin{lem} \label{lem:element}
Let $A$ be a $k \times k$ symmetric matrix and $T$ a lower triangular $k\times k$ matrix with positive diagonal such that $TAT^{t}=\mathrm{Id}$. Then the following statements hold.
	\begin{enumerate}[label=\textnormal{(\roman{*})}]
		\item If $\|A-\mathrm{Id}\| \le \ep$, then $\|T-\mathrm{Id}\| \le C(k) \ep$.
		
		\item If all diagonal entries of $A$ lie in $[-\Lambda, \Lambda]$, then each diagonal entry of $T$ has value bounded below by $\Lambda^{-\frac 1 2}$.		
	\end{enumerate}	
	Here, the norm $\|\cdot\|$\index{$\|\cdot\|$} is the operator norm of the matrix.
\end{lem}

\begin{lem}\label{lem:compmatrix}
	Let $\flow$ be a closed Ricci flow and $x_0^*=(x_0,t_0)\in \XX$ with $[t_0-10r^2,t_0]\subset I$. Assume that $\vec u=(u_1,\ldots, u_k):M\times [t_0-10r^2,t_0]$ is a $(k, \ep,r)$-splitting map. If there is a lower triangular $k\times k$-matrix $T$ with positive diagonal such that $\vec u':= T\vec u$ is a $(k,\ep, s)$-splitting map at $x_0^*$ for some $s \in [r/2, r]$, then $\rVert T-\mathrm{Id}\rVert\leq C(n)\ep$.
\end{lem}
\begin{proof}
We may assume $t_0=0$. By \cite[Proposition 10.2]{fang2025RFlimit}, we see that for any $1\leq i,j\leq k$,
	\begin{align*}
\abs{\aint_{-2r^2}^{-r^2}\int_M \la \nabla u_i,\nabla u_j\ra  \,\mathrm{d}\nu_{x_0^*;t}\mathrm{d}t-\delta_{ij}} \le 2\ep \quad \text{and} \quad \abs{\aint_{-2r^2}^{-r^2}\int_M \la \nabla u'_i,\nabla u'_j\ra  \,\mathrm{d}\nu_{x_0^*;t}\mathrm{d}t-\delta_{ij}} \le 2\ep.
	\end{align*}	
	
Thus, we can find lower triangular matrices $T_1,T_2$ with positive diagonals such that for any $1\leq i,j\leq k$,
		\begin{align*}
\aint_{-2r^2}^{-r^2}\int_M\left\la \nabla (T_1\vec u)_i,\nabla (T_1\vec u)_j\right\ra \,\mathrm{d}\nu_{x_0^*;t}\mathrm{d}t=\delta_{ij} \quad \text{and} \quad
\aint_{-2r^2}^{-r^2}\int_M\left\la \nabla (T_2\vec u')_i,\nabla (T_2\vec u')_j\right\ra \,\mathrm{d}\nu_{x_0^*;t}\mathrm{d}t=\delta_{ij}.
	\end{align*}	

By Lemma \ref{lem:element} (i), we have $\rVert T_1-\mathrm{Id}\rVert+\rVert T_2-\mathrm{Id}\rVert\leq C(n)\ep$.
	
Since $T_2\vec u'=T_2 T\vec u$, it follows from the uniqueness of Cholesky decomposition that $T_1=T_2 T$. Therefore, we obtain
		\begin{align*}
\rVert T-\mathrm{Id}\rVert =\rVert T_2^{-1} T_1-\mathrm{Id}\rVert\leq C(n)\ep,
	\end{align*}		
	where we used the elementary estimate
		\begin{align} \label{eq:elementary}
\|A_1 A_2-\mathrm{Id}\| \le \|A_1-\mathrm{Id}\|+\|A_2-\mathrm{Id}\|+\|A_1-\mathrm{Id}\| \cdot \|A_2-\mathrm{Id}\|
	\end{align}			
	for any $k \times k$ matrices $A_1$ and $A_2$.	
\end{proof}

\begin{thm}[Geometric transformation theorem]\label{thm:existtransRF}
	Given constants $\eta>0$ and $\epsilon>0$, there exists $C=C(n)$ such that the following holds.
	
	Let $\XX=\{M^n,(g(t))_{t\in I}\}$ be a closed Ricci flow with entropy bounded below by $-Y$ and $x_0^*=(x_0, t_0)\in\XX$ with $[t_0-10r^2,t_0]\subset I$.    
	 Suppose for any $s\in [r_0,  r]$, $x_0^*$ is $(\delta, s)$-selfsimilar, $(k,\delta,s)$-splitting but not $(k+1,\eta,s)$-splitting. Assume that $\vec u=(u_1,\ldots,u_k):M\times [t_0-10r^2,t_0]\to\R^k$ is a $(k,\delta,r)$-splitting map. If $\delta \le \delta(n,Y,\eta, \ep)$, then for each $s\in [r_0,r]$, there exists a lower triangular $k\times k$ matrix $T_s$ with positive diagonal such that
	\begin{enumerate}[label=\textnormal{(\roman{*})}]
		\item $\vec u^s=T_s\vec u$ is a $(k,\ep,s)$-splitting map at $x_0^*$.
		\item For any $r_0 \leq s_1\leq s_2\leq r$, we have
		\begin{align*}
\max \left\{ \rVert T_{s_1}^{-1} \circ T_{s_2}\rVert, \rVert T_{s_2}^{-1} \circ T_{s_1}\rVert \right\} \le (1+C\ep)\lc\frac{s_2}{s_1}\rc^{C \ep}.
	\end{align*}			
	
		\item Each diagonal entry of $T_s$ has value bounded below by $1/2$.
	\end{enumerate}
\end{thm}
\begin{proof}
Without loss of generality, we assume $t_0=0$ and $\ep \le \ep(n, Y, \eta)$ is sufficiently small. We set $\mathrm{d}\nu_t=\mathrm{d}\nu_{x_0^*;t}=(4\pi\tau)^{-n/2}e^{-f}\,\mathrm{d}V_{g(t)}$, where $\tau=-t$. Throughout the proof, constants denoted by $C$ depend only on $n$, and may vary from line to line.
	
Define $r'$ to be the smallest number so that for any $s \in [r', r]$, there exists a lower triangular $k\times k$ matrix $T_s$ with positive diagonal for which $\vec u^s:=T_s\vec u$ is a $(k,\ep,s)$-splitting map at $x_0^*$. 

By the definition of $r'$ and Lemma \ref{lem:compmatrix}, for any $r' \le s \le s' \le \min\{2s, r\}$, we have
	\begin{align}\label{equ:compmatrix}
	\rVert T^{-1}_{s}\circ T_{s'}-\mathrm{Id}\rVert+\rVert T^{-1}_{s'}\circ T_{s}-\mathrm{Id}\rVert\leq C \ep.
	\end{align}
It then follows that for any $r'\leq s_1\leq s_2\leq r$,
	\begin{align}\label{equ:trans1}
	\max \left\{\rVert T^{-1}_{s_1}\circ T_{s_2}\rVert,\rVert T^{-1}_{s_2}\circ T_{s_1}\rVert \right\}\leq (1+C\ep)\lc\frac{s_2}{s_1}\rc^{C \ep}.
	\end{align}
Indeed, it follows from \eqref{eq:elementary}, \eqref{equ:compmatrix} and an induction argument that for any $l \ge 1$
	\begin{align*}
	\rVert T^{-1}_{2^{-l}s}\circ T_{s}-\mathrm{Id}\rVert  \leq \lc 1+ C \ep \rc^l-1
	\end{align*}
	as long as $[2^{-l}s, s] \subset [r', r]$. Similarly, we have
		\begin{align*}
\rVert T^{-1}_{s}\circ T_{2^{-l}s}-\mathrm{Id}\rVert \leq \lc 1+ C \ep \rc^l-1.
	\end{align*}
	From the above two inequalities, \eqref{equ:trans1} easily follows.

On the other hand, since $\square |\na u_i|^2 \le 0$, it follows from \cite[Proposition 10.2]{fang2025RFlimit} that
	\begin{align}\label{equ:trans113}
		\aint_{-10s^2}^{-s^2/10}\int_M |\na u_i|^2\,\mathrm{d}\nu_t \mathrm{d}t\leq 1+2\delta.
	\end{align}
Moreover, for any $s \in [r', r]$ and $1\leq i,j\leq k$, we have
	\begin{align}\label{equ:trans114}
		\aint_{-10s^2}^{-s^2/10}\int_M \la\na u^s_i,\na u^s_j\ra\,\mathrm{d}\nu_t \mathrm{d}t=\delta_{ij}.
	\end{align}
	Combining \eqref{equ:trans113} and \eqref{equ:trans114}, it follows from Lemma \ref{lem:element} (ii) that each diagonal entry of $T_s$ has value bounded below by $1/2$, provided $\delta \le \delta (n, Y)$.
	
Therefore, to complete the proof, it remains to show that $r'=r_0$. Suppose, for contradiction, that $r'>r_0$.

\begin{claim}\label{cla:eigen}
Let $0<\lambda_1(t)\leq\lambda_2(t)\leq\ldots$ be the nonzero eigenvalues of $\Delta_f=\Delta-\la \na\cdot,\na f\ra$ at time $t$, counted by multiplicities. Then, we have
	\begin{align}\label{equ:eigen00}
		\abs{s^2\lambda_{i}(-s^2)-\frac{1}{2}} \le \Psi(\delta) 
	\end{align} 
for any $s \in [r_0, r]$ and $1 \le i \le k$, and there exists $c_0=c_0(n, Y, \eta)>0$ such that
	\begin{align}\label{equ:eigen1}
		s^2\lambda_{k+1}(-s^2)\geq \frac{1+c_0}{2}.
	\end{align} 
\end{claim}

First, by Theorem \ref{poincareinequ}, we have $\tau\lambda_i(-\tau) \ge 1/2$. Thus, \eqref{equ:eigen00} follows from \cite[Proposition 10.7, Corollary C.5]{fang2025RFlimit}.

Suppose $s^2\lambda_{k+1}(-s^2)\le 1/2+ \zeta$, then it follows from \cite[Corollary C.4]{fang2025RFlimit} that we can find a $(k+1, C\zeta, s/\sqrt{10})$-splitting map at $x_0^*$, where $C$ is a universal constant. Thus, by Proposition \ref{equivalencesplitsymetric1} (i), we conclude that $x_0^*$ is $(k+1, \Psi(\delta+\zeta), s)$-splitting. This contradicts our assumption if $\zeta\leq\zeta(n,Y,\eta)$ and $\delta\leq\delta(n,Y,\eta)$, which completes the proof of Claim \ref{cla:eigen}.

Next, we set $\vec{u}':=T_{r'} \vec{u}=(u'_1, \ldots, u'_k)$.

\begin{claim}\label{cla:growthrate}
		For each $s\in [r',r]$ and $1\leq i\leq k$, we have
		\begin{align*}
			\int_M (u_i')^2\,\mathrm{d}\nu_{-s^2} \leq Cs^2\left(\frac{s}{r'}\right)^{C \ep}.
		\end{align*}
	\end{claim}
	
For $\vec{u}^s=(u^s_1, \ldots, u^s_k)$, by \cite[Proposition 10.2]{fang2025RFlimit}, we have
	\begin{align}\label{equ:trans115a}
		\int_M |\na u^s_i|^2\,\mathrm{d}\nu_{t}\leq 1+2\ep.
	\end{align}
for any $-10s^2\leq t\leq -s^2/10$. Thus, it follows from Theorem \ref{poincareinequ} that
	\begin{align}\label{equ:trans115}
		\int_M (u_i^s)^2\,\mathrm{d}\nu_{-s^2}\leq 2s^2\int_M |\na u^s_i|^2\,\mathrm{d}\nu_{-s^2}\leq 4s^2.
	\end{align}
Thus, Claim \ref{cla:growthrate} follows from \eqref{equ:trans1}, \eqref{equ:trans115a} and \eqref{equ:trans115} directly.

Next, we set $\phi_i(t)$ for $i \in \mathbb N$ to be the orthonormal basis of eigenfunctions of $\Delta_f$ corresponding to $\lambda_i(t)$. In the following, we assume $r'=1$ for simplicity and choose constants $\delta' \ll L^{-1} \ll 1$, both of which are $\Psi(\delta)$ and will be determined later.

Since $x_0^*$ is $(\delta, s)$-selfsimilar and $(k,\delta,s)$-splitting for any $s\in [r_0,  r]$, it follows from Proposition \ref{equivalencesplitsymetric1} (i) and \cite[Proposition 12.21]{bamler2020structure} that we can find $\vec y=(y_1,\ldots,y_k):M\times [-L,0]\to\mathbb R^k$ such that
	\begin{enumerate}[label=\textnormal{(\alph{*})}]
		\item $\square y_i=0$;
		
		\item $y_i(x_0^*)=0$;
		
		%\item $\displaystyle \int _{-L}^{-L^{-1}}\int_M \left|\nabla^2 y_i\right|^2\,\mathrm{d}\nu_t \mathrm{d}t\leq \delta'$;
		
		\item $\displaystyle \abs{\int_M y_i^2 \,\mathrm{d}\nu_t -2\tau} \le \delta'$ for any $t \in [-L, -L^{-1}]$;
		
		\item $\displaystyle \int_M \left|\la\nabla y_i,\nabla y_j\ra -\delta_{ij}\right|\,\mathrm{d}\nu_t \leq \delta'$ for any $t \in [-L, -L^{-1}]$.
	\end{enumerate}	

For simplicity, we set
	\begin{align*}
y_i'=\frac{1}{\sqrt{2\tau}} y_i, \quad \text{and decompose} \quad y_i'=\sum_{l=1}^{\infty} Y_{i,l} \phi_l,
	\end{align*}
where the convergence is in the $L^2$-sense. Note that $Y_{i,l}$ depends on $t$.

	\begin{claim}\label{cla:frea}
For any $t \in [-L, -1]$ and $i,j \in \{1, \ldots, k\}$, we have
	\begin{align}\label{equ:trans8}
\abs{\sum_{l=1}^k Y_{i,l} Y_{j,l}-\delta_{ij} }+\sum_{l=k+1}^\infty Y_{i, l}^2 \le C(n, Y, \eta) \delta'.
	\end{align}
	\end{claim}	
By (c) and (d) above, we have
	\begin{align}\label{equ:ort1}
		\left|\int_M y_i^2\,\mathrm{d}\nu_t-2\tau\right|\leq \delta'\quad  \mathrm{and}\quad \left|\int_M |\na y_i|^2\,\mathrm{d}\nu_t-1\right|\leq \delta'.
	\end{align}
Since $y_i'=\sum_{l=1}^{\infty} Y_{i,l} \phi_l$, \eqref{equ:ort1} implies
	\begin{align*}
		\left|\sum_{l=1}^{\infty}Y_{i,l}^2-1\right|\leq \frac{\delta'}{2\tau}\quad  \mathrm{and}\quad \left|\sum_{l=1}^{\infty}\lambda_l Y_{i,l}^2-\frac{1}{2\tau}\right|\leq \frac{\delta'}{2\tau}.
			\end{align*}
	Combining these estimates with \eqref{equ:eigen1}, it follows that
	\begin{align}
		\frac{1+\delta'}{2\tau}\geq \sum_{l=1}^{\infty}\lambda_l Y_{i,l}^2\geq & \frac{1}{2\tau}\sum_{l=1}^kY_{i,l}^2+\frac{1+c_0}{2\tau}\sum_{l=k+1}^\infty Y_{i,l}^2 \notag \\
		\geq & \frac{1}{2\tau}\lc 1-\frac{\delta'}{2\tau}-\sum_{l=k+1}^\infty Y_{i,l}^2\rc+\frac{1+c_0}{2\tau}\sum_{l=k+1}^\infty Y_{i,l}^2, \label{equ:ort2a}
	\end{align}
	which implies
	\begin{align*}
	\sum_{l=k+1}^\infty Y_{i,l}^2\leq c_0^{-1} \delta'(1+1/2\tau) \le C(n, Y, \eta) \delta'.
	\end{align*}
Moreover, it follows from \eqref{equ:ort2a} again that
		\begin{align*}
1+\delta'\geq 1-\frac{\delta'}{2\tau}-\sum_{l=k+1}^\infty Y_{i,l}^2+2\tau\sum_{l=k+1}^\infty \lambda_l Y_{i,l}^2,
	\end{align*}
	which implies
		\begin{align}\label{equ:ort2b}
2\tau \sum_{l=k+1}^\infty \lambda_l Y_{i,l}^2 \le C(n, Y, \eta) \delta'.
	\end{align}
	
Next, it follows from (d) above that for $1 \le i \le j \le k$,
		\begin{align*}
\abs{2\tau \sum_{l=1}^\infty \lambda_l Y_{i,l} Y_{j,l}-\delta_{ij} }\le \delta',
	\end{align*}
which, when combined with \eqref{equ:ort2b}, yields
		\begin{align*}
\abs{2\tau \sum_{l=1}^k \lambda_l Y_{i,l} Y_{j,l}-\delta_{ij} }\le C(n, Y, \eta) \delta'.
	\end{align*}
	Since $\lambda_l$ is close to $\frac{1}{2\tau}$ by Claim \ref{cla:eigen}, it is easy to see that \eqref{equ:trans8} holds. This completes the proof of Claim \ref{cla:frea}.

For the rest of the proof, we consider $w:=u'_a$ for some $1 \le a \le k$ and write 
	\begin{align*}
w(t)=w_0(t)+\sum_{i=1}^k H_i(t)y_i(t)=w_0(t)+\sum_{i=1}^k H'_i(t)y'_i(t),
	\end{align*}
for any $t \in [-L, -1]$, where $H'_i(t)=\sqrt{2\tau} H_i(t)$ is chosen so that
	\begin{align}\label{equ:trans7}
		\int_M w_0 y_i\,\mathrm{d}\nu_t=0, \quad \forall 1\le i \le k.
	\end{align} 
Notice that by Claim \ref{cla:frea}, such choices are possible.

	\begin{claim}\label{cla:frebex}
For any $t \in [-L, -1]$ and $1 \le i \le k$,
\begin{align}\label{equ:ortex001}
\abs{H'_i(t)-\int_M wy'_i \, \mathrm{d}\nu_t} \le C(n, Y, \eta) \delta' \lc \int_M w^2\,\mathrm{d}\nu_t \rc^{\frac 1 2}.
	\end{align}
	\end{claim}	

By our definition of $H_i$, we have
	\begin{align*}
\int_M w y_i\, \mathrm{d}\nu_t=\sum_{j=1}^k H_j \int_M y_i y_j \,\mathrm{d}\nu_t=2\tau \sum_{j=1}^k \sum_{l=1}^\infty H_j Y_{i,l}Y_{j,l}
	\end{align*}
for any $t \in [-L, -1]$ and $1 \le i \le k$. By Claim \ref{cla:frea}, the proof of \eqref{equ:ortex001} is clear.

Next, we decompose
	\begin{align*}
w(t)=\sum_{i=1}^\infty J_i(t)\phi_i(t),
	\end{align*}
where the sum is in the $L^2$-sense. Note that by Claim \ref{cla:growthrate}, we have
	\begin{align}\label{equ:trans117}
		\sum_{i=1}^\infty J^2_i(t)=\int_M w^2(t)\,\mathrm{d}\nu_t\leq C\tau^{1+C \ep}\leq C\tau^{2}.
	\end{align}

	\begin{claim}\label{cla:freb}
For any $t \in [-L, -1]$,
\begin{align}\label{equ:ort4}
\int_M \lc \sum_{i=1}^k (J_i \phi_i-H_i y_i) \rc^2 \,\mathrm{d}\nu_t \le C(n, Y, \eta) \tau^2 \delta'
	\end{align}
	and
	\begin{align}\label{equ:ort4a}
\int_M \abs{\na \lc \sum_{i=1}^k (J_i \phi_i-H_i y_i) \rc}^2 \,\mathrm{d}\nu_t \le C(n, Y, \eta) \tau \delta'.
	\end{align}
	\end{claim}	

We compute
	\begin{align*}
\int_M w y'_i\,\mathrm{d}\nu_t=\sum_{l=1}^\infty \int_M w Y_{i,l} \phi_l\,\mathrm{d}\nu_t=\sum_{l=1}^\infty Y_{i,l} J_l.
	\end{align*}

By Claim \ref{cla:frea} and \eqref{equ:trans117}, 
	\begin{align*}
\abs{\sum_{l=k+1}^\infty Y_{i,l} J_l} \le \lc \sum_{l=k+1}^\infty Y^2_{i,l} \rc^{\frac 1 2} \lc \sum_{l=k+1}^\infty  J^2_l \rc^{\frac 1 2} \le C(n, Y, \eta) \tau (\delta')^{\frac 1 2}.
	\end{align*}
Thus, we obtain
	\begin{align*} 
\abs{\int_M w y'_i\,\mathrm{d}\nu_t-\sum_{l=1}^k Y_{i,l} J_l} \le C(n, Y, \eta) \tau (\delta')^{\frac 1 2}.
	\end{align*}
Combining this with Claim \ref{cla:frebex}, we have
	\begin{align} \label{eq:transextra002}
\abs{H_i'-\sum_{l=1}^k Y_{i,l} J_l} \le C(n, Y, \eta) \tau (\delta')^{\frac 1 2}.
	\end{align}

Thus, we obtain
	\begin{align} 
\sum_{i=1}^k H'_i y'_i=&\sum_{i=1}^k \lc H'_i-\sum_{l=1}^k Y_{i,l} J_l \rc y_i'+\sum_{i=1}^k \sum_{l=1}^k Y_{i,l} J_l y_i' \notag \\
=& \sum_{i=1}^k \lc H'_i-\sum_{l=1}^k Y_{i,l} J_l \rc y_i'+\sum_{i=1}^k \sum_{l=1}^k \sum_{j=1}^\infty Y_{i,l} Y_{i,j} J_l \phi_j \notag \\
=& \sum_{i=1}^k \lc H'_i-\sum_{l=1}^k Y_{i,l} J_l \rc y_i'+\sum_{i=1}^k \sum_{l=1}^k \sum_{j=1}^k Y_{i,l} Y_{i,j} J_l \phi_j+\sum_{i=1}^k \sum_{l=1}^k \sum_{j=k+1}^\infty Y_{i,l}Y_{i,j}J_l \phi_j \notag \\
=&\sum_{i=1}^k \lc H'_i-\sum_{l=1}^k Y_{i,l} J_l \rc y_i'+ \sum_{l=1}^k J_l \phi_l+\sum_{i=1}^k \sum_{l=1}^k \sum_{j=k+1}^\infty Y_{i,l}Y_{i,j}J_l \phi_j+\sum_{l=1}^k \sum_{j=1}^k \lc \sum_{i=1}^k Y_{i,l}Y_{i,j}-\delta_{jl} \rc J_l \phi_j. \label{eq:transextra003}
	\end{align}
Combining Claim \ref{cla:frea}, \eqref{eq:transextra002} and \eqref{eq:transextra003}, \eqref{equ:ort4} is easily derived. On the other hand, we have
	\begin{align*} 
\int_M |\na \phi_i|^2\,\mathrm{d}\nu_t=\lambda_i 
	\end{align*}
and, by \eqref{equ:ort1}, 
	\begin{align*} 
\left|\int_M |\na y'_i|^2\,\mathrm{d}\nu_t-\frac{1}{2\tau}\right|\leq \frac{\delta'}{2\tau}.
	\end{align*}
Thus, it follows from \eqref{equ:ort2b}, \eqref{eq:transextra002} and \eqref{eq:transextra003} that
	\begin{align*}
& \int_M \abs{\na \lc \sum_{i=1}^k J_i \phi_i-H'_i y'_i \rc}^2 \,\mathrm{d}\nu_t \\
\le & C(n, Y, \eta) \tau \delta'+C(n, Y, \eta) \tau^2 \sum_{i=1}^k\sum_{j=k+1}^\infty Y_{i,j}^2 \lambda_j+C(n, Y, \eta)\tau (\delta')^2 \le C(n, Y, \eta) \tau \delta'.
	\end{align*}
In sum, the proof of Claim \ref{cla:freb} is complete.

	\begin{claim}\label{cla:freca8}
For any $t \in [-L, -1]$, we have
	\begin{align} \label{eq:transextra003a}
\abs{ \sum_{i=1}^k \sum_{j=k+1}^\infty \int_M J_j\phi_j H_i' y_i' \,\mathrm{d}\nu_t} \le C(n, Y, \eta) \tau^2 (\delta')^{\frac 1 2}
	\end{align}
and for any $\sigma>0$,
	\begin{align} \label{eq:transextra003b}
\abs{\sum_{i=1}^k \sum_{j=k+1}^\infty \lambda_j\int_M J_j\phi_j H_i' y_i' \,\mathrm{d}\nu_t} \le \sigma \sum_{j=k+1}^\infty \lambda_j J_j^2+\frac{C(n, Y, \eta) \tau \delta'}{\sigma}.
	\end{align}
	\end{claim}	

From \eqref{equ:trans8}, we have
	\begin{align*}
\abs{ \sum_{i=1}^k \sum_{j=k+1}^\infty \int_M J_j\phi_j H_i' y_i' \,\mathrm{d}\nu_t}=&\abs{ \sum_{i=1}^k \sum_{j=k+1}^\infty J_j H_i' Y_{i,j}}
 \le  C(n, Y, \eta) \tau \sum_{i=1}^k \sum_{j=k+1}^\infty |J_j|  |Y_{i,j}| \le   C(n, Y, \eta) \tau^2 (\delta')^{\frac 1 2},
\end{align*}
where we used the Cauchy-Schwarz inequality for the last line. Similarly, we have by \eqref{equ:ort2b},
	\begin{align*}
\abs{\sum_{i=1}^k \sum_{j=k+1}^\infty \lambda_j\int_M J_j\phi_j H_i' y_i' \,\mathrm{d}\nu_t} =\abs{\sum_{i=1}^k \sum_{j=k+1}^\infty \lambda_j Y_{i,j}J_j H_i'} \le \sigma \sum_{j=k+1}^\infty \lambda_j J_j^2+\frac{C(n, Y, \eta) \tau \delta'}{\sigma}.
	\end{align*}
In sum, the proof of Claim \ref{cla:freca8} is complete.

Next, we consider the frequency function of $w_0$:
		\begin{align*}
F_{w_0}(t):=\frac{\tau \int_M |\na w_0|^2 \, \mathrm{d}\nu_t}{ \int_M w_0^2 \,\mathrm{d}\nu_t}.
	\end{align*}
Recall that 
		\begin{align*}
w_0=\sum_{i=1}^k(J_i\phi_i-H_iy_i)+\sum_{i=k+1}^{\infty} J_i\phi_i=\sum_{i=1}^k(J_i\phi_i-H'_iy'_i)+\sum_{i=k+1}^{\infty} J_i\phi_i.
	\end{align*}
We compute
		\begin{align}
			\int_M w_0^2\,\mathrm{d}\nu_{t}=&\sum_{i=k+1}^{\infty}J_i^2 +\int_M \lc\sum_{i=1}^k (J_i\phi_i-H'_iy'_i) \rc^2\,\mathrm{d}\nu_{t}+2\sum_{i=1}^k \sum_{j=k+1}^\infty \int_M J_j\phi_j(J_i\phi_i-H'_iy'_i)\,\mathrm{d}\nu_{t}\notag \\
			=&\sum_{i=k+1}^{\infty}J_i^2 +\int_M \lc\sum_{i=1}^k (J_i\phi_i-H'_iy'_i) \rc^2\,\mathrm{d}\nu_{t}-2\sum_{i=1}^k \sum_{j=k+1}^\infty \int_M J_j \phi_j H'_i y'_i\,\mathrm{d}\nu_{t} \notag \\
			\leq& \sum_{i=k+1}^{\infty}J_i^2+C(n, Y, \eta) \tau^2 (\delta')^{\frac 1 2}, \label{eq:transextra004a}
		\end{align}
where we used \eqref{equ:ort4} and \eqref{eq:transextra003a} for the last inequality.

Moreover, we calculate
		\begin{align}
			\int_M |\na w_0|^2\,\mathrm{d}\nu_{t}= &\sum_{i=k+1}^{\infty}\lambda_iJ_i^2 +\int_M \left|\na \lc \sum_{i=1}^k (J_i\phi_i-H'_iy'_i) \rc\right|^2\,\mathrm{d}\nu_{t}-2 \sum_{i=1}^k \sum_{j=k+1}^\infty \int_M J_j\Delta_f\phi_j(J_i\phi_i-H'_iy'_i)\,\mathrm{d}\nu_{t}\notag \\
=&\sum_{i=k+1}^{\infty}\lambda_iJ_i^2 +\int_M \left|\na \lc \sum_{i=1}^k (J_i\phi_i-H'_iy'_i) \rc\right|^2\,\mathrm{d}\nu_{t}-2 \sum_{i=1}^k \sum_{j=k+1}^\infty \int_M \lambda_jJ_j\phi_j H'_iy'_i\,\mathrm{d}\nu_{t}\notag \\
			\geq&(1- \sigma)\sum_{i=k+1}^{\infty}\lambda_iJ_i^2 -C(n, Y, \eta) \tau \delta' \sigma^{-1}, \label{eq:transextra004b}
		\end{align} 
where we used \eqref{eq:transextra003b} for the last inequality.

Combining \eqref{eq:transextra004a} and \eqref{eq:transextra004b}, we obtain
		\begin{align}
F_{w_0}(t) \ge & \dfrac{(1- \sigma)\sum_{i=k+1}^{\infty}\lambda_i \tau J_i^2 -C(n, Y, \eta) \tau^2 \delta' \sigma^{-1}}{\sum_{i=k+1}^{\infty}J_i^2+C(n, Y, \eta) \tau^2 (\delta')^{\frac 1 2}} \notag \\
\ge & \dfrac{\frac{(1- \sigma)(1+c_0)}{2}\sum_{i=k+1}^{\infty}J_i^2 -C(n, Y, \eta) \tau^2 \delta' \sigma^{-1}}{\sum_{i=k+1}^{\infty}J_i^2+C(n, Y, \eta) \tau^2 (\delta')^{\frac 1 2}}, \label{eq:transextra005a}
		\end{align} 
where we used \eqref{equ:eigen1} for the last inequality. Now, we fix the constant $\sigma=\sigma(n, Y, \eta)$ so that
		\begin{align*}
\frac{(1- \sigma)(1+c_0)}{2}=\frac{1}{2}+\frac{c_0}{3}.
		\end{align*} 
Moreover, we choose $L=(\delta')^{-\frac{1}{10}}$. Thus, \eqref{eq:transextra005a} becomes
		\begin{align}
F_{w_0}(t) \ge &  \lc \frac{1}{2}+\frac{c_0}{3} \rc \frac{\sum_{i=k+1}^{\infty}J_i^2-C(n, Y, \eta) (\delta')^{\frac 4 5}}{\sum_{i=k+1}^{\infty}J_i^2+C(n, Y, \eta)  (\delta')^{\frac{3}{10}}}. \label{eq:transextra005b}
		\end{align} 

From \eqref{eq:transextra005b}, the next claim is immediate.

	\begin{claim}\label{cla:frec}
For any $t \in [-(\delta')^{-\frac{1}{10}}, -1]$, if $\delta' \le \delta'(n, Y, \eta)$ and $\sum_{i=k+1}^{\infty}J_i^2 \ge (\delta')^{\frac 1 5}$, then
		\begin{align*}
F_{w_0}(t) \ge \frac{1}{2}+\frac{c_0}{4}.
		\end{align*} 
	\end{claim}	

	Since $\square w=0$ and $\square \vec y=0$, we have 
			\begin{align*}
\square w_0=-\square \lc \sum_{i=1}^k H_iy_i \rc=-\sum_{i=1}^k \lc\frac{\mathrm{d}}{\mathrm{d}t}H_i\rc y_i,
		\end{align*} 
		which implies
	\begin{align}
		\frac{\mathrm{d}}{\mathrm{d}t}\int_M w_0^2\,\mathrm{d}\nu_t=&\int_M 2w _0 \square w_0-2|\nabla w_0|^2\,\mathrm{d}\nu_t \notag \\
		=&\int_M -2 w_0 \lc\sum_{i=1}^k \lc\frac{\mathrm{d}}{\mathrm{d}t} H_i\rc y_i\rc-2|\nabla w_0|^2\,\mathrm{d}\nu_t
		=-2\int_M |\nabla w_0|^2\,\mathrm{d}\nu_t, \label{equ:trans9}
	\end{align}
	where in third equality, we used \eqref{equ:trans7}.

Next, we consider the following two cases.

\textbf{Case 1}: $\sum_{i=k+1}^{\infty}J_i^2 \ge (\delta')^{\frac 1 5}$ for all $t \in [-(\delta')^{-\frac{1}{10}}, -21]$.

In this case, we conclude from Claim \ref{cla:frec} that
	\begin{align}\label{equ:trans13}
		F_{w_0}(t)\geq\frac{1+c_1}{2},
	\end{align}
	where $c_1=c_0/2>0$.	Set $I(t)=\int_M w_0^2\,\mathrm{d}\nu_t$. Then by \eqref{equ:trans9} and \eqref{equ:trans13}, we have
	\begin{align*}
		\frac{\mathrm{d}}{\mathrm{d}t}\log I(t)=\frac{-2\int_M |\nabla w_0|^2\,\mathrm{d}\nu_t}{\int_M w_0^2\,\mathrm{d}\nu_t} 
		= -\frac{2F_{w_0}(t)}{\tau}
		\leq \frac{1+c_1}{t}.
	\end{align*}
By integration, we have for all $t\in [-(\delta')^{-1/10},-21]$ and $s\in [-21,-20]$,
	\begin{align}\label{equ:trans14}
		(21)^{1+c_1}I(t)\geq I(s)(-t)^{1+c_1}.
	\end{align}
	
	On the other hand, by Claim \ref{cla:growthrate}, we have
	\begin{align}\label{equ:trans15}
		I(t)\leq \int_M w^2\,\mathrm{d}\nu_t\leq C(-t)^{1+C \ep}.
	\end{align}
	Therefore, by choosing $t=-(\delta')^{-1/10}$ in \eqref{equ:trans14} and \eqref{equ:trans15}, we obtain that for any $s\in [-21,-20]$,
	\begin{align*}
		I(s)\leq C(\delta')^{\frac{c_1-C \ep}{10}}.
	\end{align*}
		By \eqref{equ:trans9}, we can find $t_1 \in [-21,-20]$ such that
	\begin{align}\label{equ:trans112}
		\int_M |\nabla w_0|^2\,\mathrm{d}\nu_{t_1}\leq C (\delta')^{\frac{c_1-C \ep}{10}} \le C(\delta')^{\frac{c_1}{20}},
	\end{align}
since $\ep \le \ep(n, Y, \eta)$. Thus, we compute
	\begin{align*}
	F_{w}(t_1)=\dfrac{|t_1|\int_M |\na w|^2\,\mathrm{d}\nu_{t_1}}{\int_M w^2\,\mathrm{d}\nu_{t_1}} \le \dfrac{|t_1|\int_M \abs{\na \lc w_0+\sum_{i=1}^k H_i y_i \rc }^2\,\mathrm{d}\nu_{t_1}}{\int_M \lc \sum_{i=1}^k H_i y_i \rc^2\,\mathrm{d}\nu_{t_1}}.
	\end{align*}
Combining \eqref{equ:ort4}, \eqref{equ:ort4a} and \eqref{equ:trans112}, this implies
		\begin{align}
	F_{w}(t_1) \le  \dfrac{|t_1|\int_M \abs{\na \lc \sum_{i=1}^k J_i \phi_i \rc}^2\,\mathrm{d}\nu_{t_1}}{\int_M \lc \sum_{i=1}^k J_i \phi_i \rc^2\,\mathrm{d}\nu_{t_1}}+\Psi(\delta') \le \frac{1}{2}+\Psi(\delta'), \label{equ:trans12a}
	\end{align}
	where we used \eqref{equ:eigen00} for the last inequality since $\delta \ll \delta'$.
	
Then, one can use the same argument as in \cite[Proposition C.3]{fang2025RFlimit} to conclude that
		\begin{align} \label{equ:trans12exa}
\int_{-20}^{-1/100}\int_M |\na^2 w|^2\,\mathrm{d}\nu_t \mathrm{d}t\leq \Psi(\delta').
	\end{align}

\textbf{Case 2}: There exists $t_2\in [-(\delta')^{-1/10},-21]$ such that $\sum_{i=k+1}^{\infty}J_i^2 < (\delta')^{\frac 1 5}$ at $t_2$.

In this case, we have
	\begin{align*}
\int_M w_0^2\,\mathrm{d}\nu_{t_2}=&\int_M w^2\,\mathrm{d}\nu_{t_2}- \int_M \lc \sum_{i=1}^k H_i y_i \rc^2 \,\mathrm{d}\nu_{t_2} \\
=& \sum_{i=1}^k J_i^2(t_2)-\int_M \lc \sum_{i=1}^k H_i y_i \rc^2 \,\mathrm{d}\nu_{t_2}+\sum_{i=k+1}^\infty J_i^2(t_2) \\
\le & \lc \int_M \lc \sum_{i=1}^k J_i \phi_i-H_i y_i \rc^2 \,\mathrm{d}\nu_{t_2}  \rc^{\frac 1 2} \lc \int_M \lc \sum_{i=1}^k J_i \phi_i+H_i y_i \rc^2 \,\mathrm{d}\nu_{t_2}  \rc^{\frac 1 2}+(\delta')^{\frac 1 5} \\
\le & C(n,Y, \eta) (\delta')^{\frac{3}{10}}+(\delta')^{\frac 1 5} \le (\delta')^{\frac 1 6},
	\end{align*}
where we used \eqref{equ:trans117} and \eqref{equ:ort4}. 

	By \eqref{equ:trans9}, for any $t\in [t_2,-20]$,
	\begin{align*}
		\int_M w_0^2\,\mathrm{d}\nu_t\leq (\delta')^{\frac 1 6}.
	\end{align*} 
	Moreover, using \eqref{equ:trans9} again, we can find $t_3\in [-21,-20]$ such that
	\begin{align}\label{equ:trans10}
		\int_M |\nabla w_0|^2\,\mathrm{d}\nu_{t_3}\leq (\delta')^{\frac 1 6}.
	\end{align}
Similar to \eqref{equ:trans12a} and \eqref{equ:trans12exa}, we obtain from \eqref{equ:trans10} that
	\begin{align}
		F_{w}(t_3)=\dfrac{|t_3|\int_M |\na w|^2\,\mathrm{d}\nu_{t_3}}{\int_M w^2\,\mathrm{d}\nu_{t_3}}\leq  \frac{1}{2}+\Psi(\delta') \quad \text{and} \quad \int_{-20}^{-1/100}\int_M |\na^2 w|^2\,\mathrm{d}\nu_t \mathrm{d}t\leq \Psi(\delta').\label{equ:trans12exb}
	\end{align}

	Since $w=u'_a$, combining \eqref{equ:trans12exa} and \eqref{equ:trans12exb}, we conclude that for any $s \in [1/2,1]$, there exists a lower triangular $k \times k$ matrix $T_s$ with positive diagonal such that $T_s \vec{u}'$ is a $(k, \ep, s)$-splitting map at $x_0^*$, provided that $\delta' \le \delta'(n, Y, \eta)$. This contradicts our assumption, thereby completing the proof.
\end{proof}
	
	Since Theorem \ref{thm:existtransRF} (ii) gives the boundedness of $T_s$ for any fixed scale $s$, we can obtain the same result for Ricci flow limit spaces by taking the limit.
	
	\begin{thm}\label{thm:existtransRFL}
	Given constants $\eta>0$ and $\ep>0$, there exists $C=C(n)$ such that the following holds.
	
	Let $(Z,d_Z,\t)$ be a noncollapsed Ricci flow limit space obtained as the limit of a sequence of closed Ricci flows in $\MM(n,Y,T)$. Suppose $z \in Z$ and for any $s\in [r_0,  r]$, $z$ is $(\delta, s)$-selfsimilar, $(k,\delta,s)$-splitting but not $(k+1,\eta,s)$-splitting. Assume that $\vec u=(u_1,\ldots,u_k): Z_{(\t(z)-10r^2,\t(z)]}\to\R^k$ is a $(k,\delta,r)$-splitting map. If $\delta \le \delta(n,Y,\eta, \ep)$, then for each $s\in [r_0,r]$, there exists a lower triangular $k\times k$ matrix $T_s$ with positive diagonal such that
	\begin{enumerate}[label=\textnormal{(\roman{*})}]
		\item $\vec u^s=T_s\vec u$ is a $(k,\ep,s)$-splitting map at $z$.
		\item For any $r_0 \leq s_1\leq s_2\leq r$, we have
\begin{align*}
\max \left\{ \rVert T_{s_1}^{-1} \circ T_{s_2}\rVert, \rVert T_{s_2}^{-1} \circ T_{s_1}\rVert \right\} \le (1+C\ep)\lc\frac{s_2}{s_1}\rc^{C \ep}.
	\end{align*}
		\item Each diagonal entry of $T_s$ has value bounded below by $1/2$.
	\end{enumerate}
	\end{thm}

\newpage
\printindex

\bibliographystyle{alpha}
\bibliography{lojaref3}

\vskip10pt

Hanbing Fang, Mathematics Department, Stony Brook University, Stony Brook, NY 11794, United States; Email: hanbing.fang@stonybrook.edu;\\

Yu Li, Institute of Geometry and Physics, University of Science and Technology of China, No. 96 Jinzhai Road, Hefei, Anhui Province, 230026, China; Hefei National Laboratory, No. 5099 West Wangjiang Road, Hefei, Anhui Province, 230088, China; Email: yuli21@ustc.edu.cn.

\end{CJK}
\end{document}